\documentclass{article}
\usepackage{amssymb,amsfonts,amsmath,amsthm,amsopn,amstext,amscd,latexsym,xy,stmaryrd,mathtools,color,enumitem}
\usepackage{hyperref}

\usepackage{xy}
\xyoption{all}

\newtheorem{theorem}{Theorem}[section]
\newtheorem{lemma}[theorem]{Lemma}

\newtheorem{proposition}[theorem]{Proposition}
\newtheorem{corollary}[theorem]{Corollary}
\newtheorem{definition}[theorem]{Definition}

\theoremstyle{definition}

\theoremstyle{remark}

\newtheorem*{remark}{Remark}
\newtheorem{claim}{Claim}
\numberwithin{equation}{section}
\theoremstyle{definition}

\renewcommand{\(}{\textup{(}}
\renewcommand{\)}{\textup{)}}

\DeclareMathOperator{\GL}{GL}
\DeclareMathOperator{\PGL}{PGL}
\DeclareMathOperator{\SL}{SL}
\DeclareMathOperator{\PSL}{PSL}
\DeclareMathOperator{\Ann}{Ann}

\DeclareMathOperator{\Supp}{Supp}
\DeclareMathOperator{\Def}{Def}

\DeclareMathOperator{\Ass}{Ass}
\DeclareMathOperator{\ord}{ord}
\DeclareMathOperator{\Frac}{Frac}
\DeclareMathOperator{\Frob}{Frob}

\DeclareMathOperator{\Hom}{Hom}

\DeclareMathOperator{\Fil}{Fil}

\DeclareMathOperator{\Gal}{Gal}
\DeclareMathOperator{\Aut}{Aut}

\DeclareMathOperator{\im}{im}

\DeclareMathOperator{\ad}{ad}
\DeclareMathOperator{\diag}{diag}

\DeclareMathOperator{\Ext}{Ext}

\DeclareMathOperator{\End}{End}
\DeclareMathOperator{\tr}{tr}

\DeclareMathOperator{\CNL}{CNL}
\DeclareMathOperator{\Sets}{Sets}

\DeclareMathOperator{\Res}{Res}
\DeclareMathOperator{\val}{val}

\DeclareMathOperator{\Ind}{Ind}
\DeclareMathOperator{\disc}{disc}
\DeclareMathOperator{\coker}{coker}
\DeclareMathOperator{\Tor}{Tor}
\DeclareMathOperator{\Pic}{Pic}
\DeclareMathOperator{\PSp}{PSp}
\DeclareMathOperator{\Art}{Art}

\DeclareMathOperator{\rec}{rec}
\DeclareMathOperator{\Spec}{Spec}

\newcommand{\cC}{{\mathcal C}}
\newcommand{\cD}{{\mathcal D}}
\newcommand{\cE}{{\mathcal E}}

\newcommand{\cH}{{\mathcal H}}

\newcommand{\cJ}{{\mathcal J}}

\newcommand{\cL}{{\mathcal L}}
\newcommand{\cM}{{\mathcal M}}

\newcommand{\cO}{{\mathcal O}}

\newcommand{\cQ}{{\mathcal Q}}

\newcommand{\cS}{{\mathcal S}}
\newcommand{\cT}{{\mathcal T}}

\newcommand{\cV}{{\mathcal V}}

\newcommand{\cX}{{\mathcal X}}
\newcommand{\cY}{{\mathcal Y}}

\newcommand{\fra}{{\mathfrak a}}
\newcommand{\frb}{{\mathfrak b}}

\newcommand{\frg}{{\mathfrak g}}

\newcommand{\ffrm}{{\mathfrak m}}
\newcommand{\frn}{{\mathfrak n}}

\newcommand{\frp}{{\mathfrak p}}
\newcommand{\frq}{{\mathfrak q}}

\newcommand{\frt}{{\mathfrak t}}

\newcommand{\frz}{{\mathfrak z}}
\newcommand{\bbA}{{\mathbb A}}

\newcommand{\bbC}{{\mathbb C}}

\newcommand{\bbF}{{\mathbb F}}
\newcommand{\bbG}{{\mathbb G}}

\newcommand{\bbL}{{\mathbb L}}

\newcommand{\bbP}{{\mathbb P}}
\newcommand{\bbQ}{{\mathbb Q}}
\newcommand{\bbR}{{\mathbb R}}

\newcommand{\bbT}{{\mathbb T}}

\newcommand{\bbZ}{{\mathbb Z}}

\newcommand{\plim}{\mathop{\varprojlim}\limits}

\newcommand{\loc}{{\text{loc}}}

\title{Modularity of $\GL_2(\bbF_p)$-representations over CM fields}

\begin{document}
	\author{Patrick B. Allen, Chandrashekhar Khare,  and Jack A. Thorne}
\maketitle

\begin{abstract}
	We prove that many representations $\overline{\rho} : \Gal(\overline{K} / K) \to \GL_2(\bbF_3)$, where $K$ is a CM field, arise from modular elliptic curves. We prove similar results when the prime $p = 3$ is replaced by $p = 2$ or $p = 5$. As a consequence, we prove that a positive proportion of elliptic curves over any CM field not containing a 5th root of unity are modular. 
\end{abstract}
\tableofcontents

\section{Introduction}\label{sec_Introduction}

In this paper we take the first steps towards proving the modularity (as opposed to potential modularity) of elliptic curves over CM fields (for example, imaginary quadratic fields). We say that an elliptic curve $E$ over a number field $K$ is modular if either $E$ has complex multiplication, or there exists a cuspidal, regular algebraic automorphic representation $\pi$ of $\GL_2(\bbA_K)$ such that $E$ and $\pi$ have the same $L$-function.
This implies, for example, that the $L$-function $L(E, s)$ has an analytic continuation to the whole complex plane, and therefore that the Birch--Swinnerton-Dyer conjecture for $E$ can be formulated unconditionally. The modularity of all elliptic curves over a given number field is known to have profound Diophantine consequences \cite{Wil95, Kan16, Sen18}.

As an illustration of our techniques, we prove:
\begin{theorem}[Special case of Theorem \ref{thm_example_application}]\label{thm_intro_example_application}
	Let $K$ be an imaginary quadratic field. Then a positive proportion of elliptic curves over $K$ are modular. 
\end{theorem}
In fact, we give precise conditions that imply the modularity of an elliptic curve over any CM field not containing a primitive 5th root of unity. It suffices to impose local conditions at finitely many places, leading to the theorem stated above. 

The first general results concerning modularity of elliptic curves over number fields were proved in the case $K = \bbQ$ \cite{Wil95, Tay95}. In these works, the modularity of all semistable elliptic curves over $\bbQ$ is established broadly following two steps. In the first step, modularity lifting theorems are proved for 2-dimensional Galois representations $\rho : \Gal(\overline{\bbQ} / \bbQ) \to \GL_2(\bbZ_p)$. These theorems state that if $\rho$ satisfies a list of conditions, chief among them that the residual representation $\overline{\rho} : \Gal(\overline{\bbQ} / \bbQ) \to \GL_2(\bbF_p)$ is already known to arise from modular forms (say holomorphic of weight 2), then $\rho$ also arises from modular forms. Since modularity of an elliptic curve $E$ is equivalent to modularity of any one of its associated $p$-adic Galois representations $\rho_{E, p}$, this largely reduces the problem of showing modularity of $E$ to proving modularity of some $\overline{\rho}_{E, p}$.

In the second step, modularity of $\overline{\rho}_{E, p}$ is proved to hold if $p = 3$. The first main observation is that the Langlands--Tunnell theorem (which asserts the modularity of certain Artin representations with soluble image) implies that $\overline{\rho}_{E, 3}$ arises from a holomorphic modular form of weight 1. Since any such form is congruent to a form of weight 2, one can apply the previously established modularity lifting theorem to deduce the modularity of $E$, provided that $\overline{\rho}_{E, 3}$ satisfies the technical conditions of the modularity lifting theorem. In particular, it should be irreducible. For elliptic curves which do not satisfy this irreducibility condition, a clever trick (the `3--5 switch') is used to show that $p = 5$ works instead. 

This two-pronged approach has proved very fruitful for studying the modularity of elliptic curves over totally real number fields, leading to a number of impressive results. For example, we mention the modularity of all elliptic curves over real quadratic fields \cite{Fre15}, and the potential modularity of all elliptic curves over totally real number fields, proved using Taylor's `potential $p$--$q$ switch' \cite{Tay02}, which supplants the second step described above in cases where the 3--5 switch is not applicable. 

Recent progress in our understanding of the Galois representations attached to torsion classes in the cohomology of arithmetic locally symmetric spaces has led to the possibility of studying the modularity of elliptic curves over CM fields. We recall that a CM field is, by definition, a totally imaginary quadratic extension of a totally real field. The simplest class of CM fields is that of the imaginary quadratic fields. Two groups of authors have now proved modularity lifting theorems that have the potential to be applied to proving the modularity of elliptic curves over CM fields \cite{Box19, 10authors}. Combined with Taylor's technique for verifying the potential modularity of a given residual representation, this leads to the potential modularity of all elliptic curves over CM fields \cite[Theorem 1.0.1]{10authors} and even over an arbitrary quadratic extension of a totally real field \cite[Theorem 1.1.4]{Box19}.

This naturally leads to the question of whether, for an elliptic curve $E$ over a CM field $K$, one can follow the same lines as Wiles by establishing e.g. the residual modularity of the modulo 3 representation $\overline{\rho}_{E, 3}$, and then use this to prove modularity of $E$ over its original field of definition. The following theorem, which is a special case of one of the main theorems of this paper, affirms that this is the case:
\begin{theorem}\label{thm_intro_residual_modularity}[Theorem \ref{thm_application_to_serre}]
Let $K$ be an imaginary quadratic field, and let $\overline{\rho} : G_K \to \GL_2(\bbF_3)$ be a continuous homomorphism of cyclotomic determinant. Suppose that for each place $v | 5$ of $K$, there exists a Tate elliptic curve $E_v$ over $K_v$ such that $\overline{\rho}|_{G_{K_v}} \cong \overline{\rho}_{E_v, 3}$. Then there exists a modular elliptic curve $E$ over $K$ such that $\overline{\rho} \cong \overline{\rho}_{E, 3}$.

In particular, $\overline{\rho}$ is modular, in the sense that it arises from a regular algebraic, cuspidal automorphic representation of $\GL_2(\bbA_K)$ of weight 2. 
\end{theorem}
In fact, we prove a similar theorem where $K$ is allowed to be any imaginary CM field such that $\zeta_5 \not\in K$, and for mod 2 and mod 5 representations as well as mod 3 representations; see Theorem \ref{thm_application_to_serre}. We use ``Tate elliptic curve'' as a synonym for ``elliptic curve with split multiplicative reduction'' and ``weight 2'' as a synonym for ``cohomological''. We note in particular that the local condition at the places $v|5$ of $K$ in Theorem \ref{thm_intro_residual_modularity} is always satisfied after passage to a soluble CM extension. We therefore expect that, combined with base change, this theorem will have many applications to proving modularity of elliptic curves beyond those presented in this paper. 

As mentioned above, the modularity of a residual representation $\overline{\rho} : \Gal(\overline{K} / K) \to \GL_2(\bbF_3)$ can be proved in the case that $K$ is a totally real field using the Langlands--Tunnell theorem (see \cite[\S 5.3]{Ser87}). The same chain of reasoning no longer applies when $K$ is a CM field. Indeed, a $\PGL_2(\bbF_3)$-representation can still be lifted to an Artin representation in characteristic 0, as in \cite{Ser87}, and the automorphy of this lift proved using the Langlands--Tunnell theorem. However, there is no known method to construct congruences between the resulting automorphic representation and one which is cohomological. More informally, one does not know how to go from weight 1 to weight 2. Note that such a method would presumably also imply the existence of Galois representations attached to algebraic Maass forms for $\GL_2(\bbA_\bbQ)$! We must therefore find a different route. 

Our argument is inspired by the `3-5' switch of Wiles, and in particular makes crucial use of the fact that certain modular curves of low level can have infinitely many rational points (because they have genus 0 or 1). A first approximation to the argument is that instead of using $p = 3$, we want to use $p = 2$. Indeed, the mod 2 representation of an elliptic curve is always dihedral, if it is irreducible, and one knows that automorphic induction can lead to cohomological automorphic representations, in contrast to the automorphic representations asserted to exist by the Langlands--Tunnell theorem. 

We therefore first prove a 2-adic automorphy lifting theorem for ordinary 2-adic Galois representations, which applies to the case of soluble (and therefore dihedral -- see Proposition \ref{prop_soluble_means_dihedral}) residual image. This takes up \S \S \ref{sec_Galois_deformation_theory} -- \ref{sec_deduction_of_main_ALTs} of the paper, the final statement being Theorem \ref{thm_main_automorphy_theorem}. Unfortunately, we must impose a local condition at the places $v|2$ of $K$ which precludes the existence of an ordinary, cohomological, dihedral lift of the given dihedral residual representation. A weaker version of this condition is necessary to force the dihedral locus in the ordinary deformation space to be sufficiently small, following a strategy of Skinner--Wiles \cite{Ski01}. The stronger version that we impose further ensures that the ordinary Galois deformation problems at the 2-adic places are formally smooth, a condition that we need in order to control the relevant global Selmer groups.
We therefore impose the condition that $\overline{\rho}$ extends to a representation of the Galois group of the maximal totally real subfield $K^+$ of $K$. Under this condition, we can verify the residual automorphy by using automorphic induction in weight 1 over $K^+$, passing to weight 2 over $K^+$ (using Hida theory in the same way as in \cite{All14}), and then using base change to get the residual automorphy over $K$. This leads to a modularity theorem for ordinary 2-adic Galois representations of $\Gal(\overline{K} / K)$ that does not have any assumption of residual modularity (but does have an assumption that the residual representation extends to $\Gal(\overline{K} / K^+)$). 

This seems at first to be catastrophic for e.g.\ the proof of Theorem \ref{thm_intro_residual_modularity}, since the condition that $\overline{\rho}_{E, 2}$ extends to $\Gal(\overline{K} / K^+)$ is a Diophantine condition on elliptic curves $E$ over $K$. Note that the modular curve  parameterizing elliptic curves $E$ with fixed mod 2 and mod 3 representations (which is a twist of $X(6)$) has genus 1, and often has no rational points at all. The key to getting around this is Lemma \ref{lem_existence_of_elliptic_curves}, which shows that the main obstruction to the existence of rational points is the image of the discriminant of the elliptic curve $E$ in $K^\times / (K^\times)^6$; this can be read off from the action of the Galois group on $E[6] = E[2] \times E[3]$. 

In order to prove Theorem \ref{thm_intro_residual_modularity}, we carefully construct a soluble CM extension $L / K$ over which this obstruction can be shown to vanish. This implies the modularity of $\overline{\rho}|_{G_L}$, and then a further argument using soluble base change and the rationality of $X(3)$ gives the modularity of $\overline{\rho}$ itself. Similar arguments can then be used to prove the analogue of Theorem \ref{thm_intro_residual_modularity} for the primes $p = 2$ and $p = 5$.

\subsection*{Structure of this paper}

We now describe the contents of this paper. In \S \ref{sec_group_theory} and \S \ref{sec_commutative_algebra} we establish some useful background results; these can be omitted on a first reading. In \S \ref{sec_Galois_deformation_theory} -- \S \ref{sec_deduction_of_main_ALTs} we develop Galois deformation theory, with particular attention paid to $p = 2$, and prove our main 2-adic automorphy lifting theorem. The techniques used are a mix of those of \cite{All14} and \cite{10authors}. Two important novelties here are the use of non-neat level subgroups (forced on us by the rather degenerate residual representations we use), which requires a vanishing theorem for the non-Eisenstein part of the cohomology at torsion levels, and a modification of the Skinner--Wiles patching argument where we must control what happens at multiple dimension 1 primes of the Hecke algebra simultaneously. See Theorem \ref{thm_boundedness_of_good_dihedral_cohomology} for the first statement and Proposition \ref{prop_connectedness_of_nice_prime_graph} for an illustration of the latter technique. 

In \S \ref{sec_odd_aut_lift} we state some new modularity lifting theorems for ordinary $p$-adic Galois representations, where $p$ is odd. These are closer to the results of \cite{10authors}, but again our need to deal with more degenerate residual representations means we have to work a little harder. We relegate the proofs to an appendix, which could be skipped on first reading, although they are necessary in order to achieve our main applications. 

In \S \ref{sec_2-3_switch} we carry out the `2--3 switch' argument sketched above and prove Theorem \ref{thm_intro_residual_modularity} and its analogues. Finally, in \S \ref{sec_application_to_modularity} we prove Theorem \ref{thm_intro_example_application}, which we hope will be the first of many applications of the results of this paper to the problem of modularity of elliptic curves over CM fields. 

\subsection{Acknowledgments}

We would like to thank Tom Fisher for useful conversations about the modular curve $X(6)$, in particular for showing us an example of a twist of $X(6)$ which does not have any rational points.
We would also like to thank an anonymous referee for a careful reading and helpful comments on an earlier version of this manuscript.

J.T.'s work received funding from the European Research Council (ERC) under the European Union's Horizon 2020 research and innovation programme (grant agreement No 714405).  This research was begun during the period that J.T. served as a Clay Research Fellow. 

P.A. was supported by Simons Foundation Collaboration Grant 527275, NSF grant DMS-1902155, and NSERC. He would like to thank J.T. and Cambridge University for hospitality during a visit where some of this work was completed. Parts of this work were completed while P.A. was a visitor at the Institute for Advanced Study, where he was partially supported by the NSF. He would like to thank the IAS for providing excellent working conditions during his stay. 

C.K. was partially supported by NSF grant DMS-2200390 and by a Simons Fellowship.

\subsection{Notation}\label{subsec_notation}

If $K$ is a field, then we write $G_K$ for the absolute Galois group of $K$ (with respect to a fixed choice of separable closure $K^s / K$). If $E$ is an elliptic curve over $K$ and $p$ is a prime not dividing the characteristic of $K$, we write $\overline{\rho}_{E, p} : G_K \to \GL_2(\bbF_p)$ for the Galois representation on $E[p](K^s)$ determined by a fixed choice of basis. 

If $K$ is a number field and $v$ is a place of $K$, then we write $K_v$ for the corresponding completion. We will fix an embedding $K^s \to K_v^s$ extending $K \to K_v$. This choice determines an embedding $G_{K_v} \to G_{K}$. If moreover $v$ is a finite place then we write $\frp_v \subset \cO_K$ for the corresponding prime ideal, $\cO_{K_v} \subset K_v$ for the ring of integers, $\varpi_v \in \cO_{K_v}$ for a fixed choice of uniformizer, $k(v) = \cO_K / \frp_v$ for the residue field, and $q_v = \# k(v)$ for the cardinality of the residue field. We write $\Frob_v \in G_{K_v}$ for a geometric Frobenius lift at $v$. If $p$ is a prime, we write $\epsilon : G_K \to \bbZ_p^\times$ for the $p$-adic cyclotomic character (the prime $p$ will always be clear from the context). 

If $A$ is a ring, and $P$ is a prime ideal, then we will write $A_{(P)}$ for the localization of $A$ at $P$ and $A_P$ for the completion of this local ring. Similarly if $M$ is an $A$-module, we write $M_{(P)}$ for the localization at $P$ and $M_P$ for the completion with respect to the $P$-adic topology. If $A$ is a local ring, then we write $\ffrm_A$ for its maximal ideal.

Throughout this paper we use ``complex'' as a synonym for ``cochain complex'', and identify any naturally occurring chain complex $C_\bullet$ with the cochain complex $C^\bullet$ defined by the formula $C^i = C_{-i}$. If $A$ is a ring then we write $\mathbf{D}(A)$ for the derived category of cochain complexes of $A$-modules, and $\mathbf{D}(A)^+, \mathbf{D}(A)^- \subset \mathbf{D}(A)$ for its full subcategories of bounded below (resp. bounded above) complexes. Every object of $\mathbf{D}(A)$ with bounded above cohomology groups is isomorphic in $\mathbf{D}(A)$ to a bounded above complex of projective $A$-modules. If $A$ is a Noetherian commutative local ring, we define a minimal complex of $A$-modules to be a bounded complex $C$ of finite projective $A$-modules such that the differentials on $C \otimes_A A/\ffrm_A$ are all $0$.

If $A$ is a complete Noetherian local ring, then we will write $\CNL_A$ for the category of complete Noetherian local $A$-algebras $R$ with residue field $A / \ffrm_A$. If $G$ is a group scheme over $A$, then we write $\widehat{G}$ for the functor $\CNL_A \to \Sets$ which sends $R \in \CNL_A$ to the group $\ker(G(R) \to G(R / \ffrm_R))$.  

If $A$ is a discrete valuation ring with fraction field $E$, and $M$ is an $A$-module, then we call $M^\vee = \Hom_A(M, E/A)$ the Pontryagin dual of $M$. 

In situations where we have fixed a prime $p$, we will fix an algebraic closure $\overline{\bbQ}_p$ of $\overline{\bbQ}$, and call subfields $E \subset \overline{\bbQ}_p$ which are finite over $\bbQ_p$ coefficient fields. We will generally write $\cO \subset E$ for the ring of integers of a coefficient field, $k$ for its residue field, and $\varpi \in \cO$ for a choice of uniformizer. We call $\cO$ a coefficient ring.

If $L$ is a characteristic $0$ local field and $\chi_1, \chi_2 : L^\times \to \bbC^\times$ are smooth characters, we write $i_B^G \chi_2 \otimes \chi_2$ for the normalized induction of $\chi_1 \otimes \chi_2$ from the upper triangular subgroup $B \subseteq \GL_2(L)$ to $G = \GL_2(L)$. 
Specifically, this is the space of locally constant functions $f : \GL_2(L) \to \bbC$ such that $f (\begin{psmallmatrix} a & b \\ 0 & d \end{psmallmatrix} g ) = \lvert \frac{a}{d} \rvert^{1/2} \chi_1(a) \chi_2(b) f(g)$ for all $\begin{psmallmatrix} a & b \\ 0 & d \end{psmallmatrix} \in B$ and $g \in G$. 
Here $\lvert \cdot \rvert$ is normalized absolute value on $L$.

We write $\bbZ_+^n$ for tuples of integers $(\lambda_1, \ldots, \lambda_n)$ with $\lambda_1 \ge \cdots \lambda_n$. 
We write $\bbZ_{+, 0}^n \subset \bbZ_+^n$ for the subset of $(\lambda_1, \ldots, \lambda_n)$ with $\lambda_1 + \cdots + \lambda_n = 0$. We identify $\bbZ_+^n \subset \bbZ^n$ with the set of dominant weights of $\GL_n$ with respect to the usual diagonal torus and upper triangular Borel subgroup. 

If $L$ is a non-archimedean local field, we write $\Art_L : L^\times \to G_L$ for the Artin map normalized so that uniformizers are taken to geometric Frobenii. If $\rho : G_L \to \GL_n(\overline{\bbQ}_p)$ is a potentially semistable representation, we write $\operatorname{WD}(\rho)$ for the associated Weil--Deligne representation. 
If $L$ has residual characteristic $p$, we say that $\rho$ is ordinary of weight $\lambda \in (\bbZ_+^n)^{\Hom_{\bbQ_p}(L, \overline{\bbQ}_p)}$ if there is a full flag $0 = \Fil_0 \subset \cdots \subset \Fil_n = \overline{\bbQ}_p^n$ stabilized by $\rho$ such that letting $\psi_i : G_L \to \overline{\bbQ}_p^\times$ be the character giving the action of $\rho$ on $\Fil_i/\Fil_{i-1}$, we have
\[ \psi_i\circ \Art_L(x) = \prod_{\tau \in \Hom_{\bbQ_p}(L, \overline{\bbQ}_p)} \tau(x)^{-(\lambda_{\tau,n-i+1} + i - 1)} \]
for all $x$ in some open subgroup of $\cO_L^\times$. If $K$ is a number field and $\rho : G_K \to \GL_n(\overline{\bbQ}_p)$ is a continuous representation, we say that $\rho$ is ordinary of weight $\lambda \in (\bbZ^n_+)^{\Hom(L, \overline{\bbQ}_p)}$ if for each place $v | p$ of $K$, $\rho|_{G_{K_v}}$ is ordinary of weight $\lambda_v = (\lambda_\tau)_{\tau \in \Hom_{\bbQ_p}(K_v, \overline{\bbQ}_p)}$.

Let $K$ be a number field. If $\pi$ is an irreducible admissible representation of $\GL_n(\bbA_K)$, we say that $\pi$ is of weight $\lambda$ if $\pi_\infty$ has the same infinitesimal character as the dual of the irreducible algebraic representation of $\GL_n(K \otimes_\bbQ \bbC)$ of highest weight $\lambda$. Note that this reflects a different choice of normalisation to that used in the introduction, which is adapted to classical holomorphic modular forms; thus ``weight 2'' in the introduction corresponds to ``weight 0'' in the body of the paper.

If $K$ is a CM number field and $\pi$ is a cuspidal automorphic representation of $\GL_n(\bbA_K)$ of weight $\lambda$, then for each isomorphism $\iota : \overline{\bbQ}_p \to \bbC$, there exists a continuous semisimple representation $r_\iota(\pi) : G_K \to \GL_n(\overline{\bbQ}_p)$, characterized up to isomorphism by the characteristic polynomials of Frobenius elements at all but finitely many places \cite[Theorem A]{hltt}. An elliptic curve $E$ over $K$ without complex multiplication is modular if and only if there exists a cuspidal automorphic representation $\pi$ of weight $0$ and an isomorphism $\iota : \overline{\bbQ}_p \to \bbC$ such that $\rho_{E, p} \otimes \epsilon^{-1} \cong r_\iota(\pi)$. See Lemma \ref{lem_more_on_modularity} for more discussion of this notion.

\section{Group theory}\label{sec_group_theory}
In this section, we prove some group theoretic lemmas and propositions that will be useful later.

\subsection{Dihedral representations}

We first collect some useful facts about dihedral Galois representations. Let $K$ be a number field.
\begin{definition}\label{dfn:dihedral}
    Let $\overline{\rho} : G_K \to \GL_2(\overline{\bbF}_2)$ be a continuous representation. We say that $\overline{\rho}$ is dihedral if the image of $\overline{\rho}(G_K)$ in $\PGL_2(\overline{\bbF}_2)$ is isomorphic to a dihedral group.
\end{definition}
\begin{proposition}\label{prop_soluble_means_dihedral}
    Let $\overline{\rho} : G_K \to \GL_2(\overline{\bbF}_2)$ be a continuous representation. The following are equivalent:
    \begin{enumerate}
        \item $\overline{\rho}$ is irreducible and $\overline{\rho}(G_K)$ is soluble.
        \item $\overline{\rho}$ is dihedral.
        \item There is an isomorphism $\overline{\rho} \cong \Ind_{G_L}^{G_K} \overline{\chi}$, where $L / K$ is a quadratic extension and $\overline{\chi} : G_L \to \overline{\bbF}_2^\times$ is a continuous character such that $\overline{\chi}^\sigma \neq \overline{\chi}$, where $\sigma \in \Gal(L / K)$ is the non-trivial element.
    \end{enumerate}
    In case 3, the extension $L / K$ is uniquely determined by $\overline{\rho}$.
\end{proposition}
\begin{proof}
    It is clear that $3 \Rightarrow 2 \Rightarrow 1$. To show that $1 \Rightarrow 3$, we use the classification of finite subgroups of $\PGL_2(\overline{\bbF}_2)$, which shows that any soluble, irreducible subgroup of $\PGL_2(\overline{\bbF}_2)$ is, after conjugation, contained in the normalizer of the standard maximal torus of $\PGL_2$, and surjects onto the Weyl group of this torus. In particular, $\overline{\rho}(G_K)$ contains a unique index two subgroup, corresponding to a quadratic extension $L / K$, and there is an isomorphism $\overline{\rho}|_{G_L} \cong \overline{\chi} \oplus \overline{\chi}^\sigma$. If $\overline{\chi} = \overline{\chi}^\sigma$ then $\overline{\rho}$ preserves a line, so is reducible. 
\end{proof}
We recall (cf. \cite{Car17}) that a representation $\overline{\rho} : G_K \to \GL_2(\overline{\bbF}_2)$ is said to be decomposed generic if there exists a prime $l \neq 2$ which splits in $K$ and such that for each place $v | l$ of $K$, $\overline{\rho}|_{G_{K_v}}$ is unramified and the eigenvalues $\alpha_v, \beta_v$ of $\overline{\rho}(\Frob_v)$ satisfy $\alpha_v / \beta_v \neq q_v^{\pm 1}$. (In fact, in characteristic 2 this just means that $\alpha_v \neq \beta_v$.) 
\begin{lemma}\label{lem:decomposed_generic_after_extension}
    Let $\overline{\rho} : G_K \to \GL_2(\overline{\bbF}_2)$ be dihedral and decomposed generic. Then for any $m \geq 1$, $\overline{\rho}|_{G_{K(\zeta_{2^m})}}$ is decomposed generic.
\end{lemma}
\begin{proof}
    Let $L / K$ denote the quadratic extension associated to $\overline{\rho}$, and let $\widetilde{L}$ denote the Galois closure of $L / \bbQ$. Let $\overline{\psi} = \overline{\chi} / \overline{\chi}^\sigma$. If $\overline{\rho}$ is decomposed generic, then there exists $g \in G_{\widetilde{L}}$ such that for all $h \in G_\bbQ$, $\overline{\psi}^h(g) \neq 1$. Indeed, we can just choose $g$ to be in the conjugacy class of a Frobenius element at the prime $l$. Then $g^{2^m} \in G_{\widetilde{L}(\zeta_{2^m})}$ and for all $h \in G_{\bbQ}$, $\overline{\psi}^h(g^{2^m}) \neq 1$ (as $\overline{\psi}$ has odd order). Let $E / \bbQ$ be a Galois extension containing the extension of $L(\zeta_{2^m})$ cut out by $\overline{\rho}$. Let $l'$ be a prime unramified in $E$ and with the conjugacy class of $\Frob_{l'}$ in $\Gal(E / \bbQ)$ equal to that of the image of $g^{2^m}$. The existence of $l'$ implies that $\overline{\rho}|_{G_{K(\zeta_{2^m})}}$ is decomposed generic.
\end{proof}
\begin{lemma}\label{lem:not_bad_dihedral_after_extension}
    Suppose that $\overline{\rho} : G_K \to \GL_2(\overline{\bbF}_2)$ is dihedral but that the corresponding quadratic extension $L/K$ is not contained in $K(\zeta_{2^m})$ for any $m \geq 1$. Then for any $m \geq 1$, $\overline{\rho}|_{G_{K(\zeta_{2^m})}}$ is dihedral.
\end{lemma}
\begin{proof}
    To show that $\overline{\rho}|_{G_{K(\zeta_{2^m})}}$ is dihedral, it is enough to show that $\overline{\rho}(G_{K(\zeta_{2^m})}) = \overline{\rho}(G_K)$. Equivalently, 
that $L / K$ and $K(\zeta_{2^m})$ are linearly disjoint. The only possibilities for the intersection are $L \cap K(\zeta_{2^m}) = K$ or $L \cap K(\zeta_{2^m}) = L$. In the first case we're done; in the second, we get $L \subset K(\zeta_{2^m})$, contradicting our assumption.
\end{proof}

\subsection{Continuous cochain cohomology}\label{sec:cohomology}
Let $F$ be a local field of residual characteristic $p$ and with ring of integers $A$.
We review some facts about continuous group cohomology for profinite groups acting continuously on topological $A$-modules. 
This is all standard when $F$ is a local field of characteristic $0$ but we will also need these facts when $F$ has characteristic $p$. 
The proofs are the same as in the characteristic $0$ case. 
We include them here for for lack of a reference.

Let $\Gamma$ be a profinite topological group and let $M$ be a topological $A$-module equipped with a continuous $A$-linear $\Gamma$-action. 
Then we can form the continuous cohomology groups $H^i(\Gamma, M)$ computed using continuous cochains (see \cite[Chapter II, \S7]{nsw}, where these groups are denoted by $H_{cts}^i(\Gamma, M)$). 
In particular, the group $H^1(\Gamma, M)$ is given by the quotient $Z^1(\Gamma, M)/B^1(\Gamma, M)$ where $Z^1(\Gamma, M)$ is the $A$-module of continuous $1$-cocycles (continuous functions $f : \Gamma \to M$ satisfying $f(gh) = gf(h) + f(g)$) and $B^1(\Gamma, M)$ is the submodule of $1$-cocycles of the form $g \mapsto gm - m$ for fixed $m \in M$.
If the $A$-module $M$ is further an $F$-vector space, then so is $H^i(\Gamma, M)$.

Let 
\[ \xymatrix@1{ 0 \ar[r] & L \ar[r] & M \ar[r] & N \ar[r] & 0 }\]
be a short exact sequence of topological $A$-modules equipped with continuous $A$-linear $\Gamma$-actions. 
If the topology on $L$ is induced from that on $M$ and the map $M \to N$ has a continuous set-theoretic section, then there is a long exact sequence of cohomology groups
\[ \xymatrix@1{ \cdots \ar[r] & H^i(\Gamma, L) \ar[r] & H^i(\Gamma, M) \ar[r] & H^i(\Gamma, N) \ar[r] & H^{i+1}(\Gamma, M) \ar[r] & \cdots }. \]
We have cup products as well as inflation, restriction, and corestriction satisfying all the usual properties (see \textit{loc. cit.}). 
In particular, we have the five term inflation-restriction exact sequence as well as the restriction-corestriction formula. 

The following is a special case of \cite[Corollary~2.7.6]{nsw}.

\begin{proposition}\label{prop:inverse-limit}
Let $t \in \ffrm_A$ be nonzero.
Assume that $M$ is a finitely generated $A$-module and that $H^{i-1}(\Gamma, M/(t^N))$ is finite for all $N \ge 1$. 
Then the natural map
\[ H^i(\Gamma, M) \to \varprojlim_{N} H^i(\Gamma, M/(t^N)).  \]
is an isomorphism.
\end{proposition}

\begin{lemma}\label{lem:no-divisible}
Let $t \in \ffrm_A$ be nonzero.
Assume that $M$ is a finitely generated $A$-module and let $Y$ be a finitely generated $A$-submodule of $H^i(\Gamma, M)$. 
Then $H^i(\Gamma, M)/Y$ contains no nontrivial $t$-divisible submodule.
\end{lemma}

\begin{proof}
Take $x_n \in H^i(\Gamma, M)$, $n\ge 0$, such that $x_n \equiv tx_{n+1} \bmod Y$. 
We want to show that $x_0 \in Y$. 
Let $g_1,\ldots,g_r$ be cocycles representing generators for $Y$, and for each $n \ge 0$, let $f_n$ be a cocycle representing $x_n$. 
Then for each $n\ge 0$, then is an $(i-1)$-cochain $h_n$ and elements $a_{nm} \in A$ such that
\[ f_n = tf_{n+1} + \sum_{m = 1}^r a_{nm} g_m + dh_n. \]
The space of continuous $i$-cochains of $\Gamma$ with values in $M$ is the inverse limit in $n$ of the space of continuous $i$-cochains of $\Gamma$ with coefficients in $M/(t^n)$. 
So defining $h = \sum_{n\ge 0} t^n h_n$ and $a_m = \sum_{n \ge 0} t^n a_{nm}$, $1 \le m \le r$, we have $f_0 = \sum_{m=1}^r a_m g_m + dh$ and $x_0 \in Y$.
\end{proof}

\begin{proposition}\label{prop:local-field-cohom}
Let $t \in \ffrm_A$ be nonzero.
Assume that $M$ is a finitely generated $A$-module and set $V = M \otimes_A F$. 
\begin{enumerate}
	\item\label{cohom:finite-gen} $H^i(\Gamma, M)$ is a finitely generated $A$-module if and only if $H^i(\Gamma, M)/(t)$ is finite. 
	\item\label{cohom:finite-implies-finite} If $M$ is $t$-torsion free and $H^i(\Gamma, M/(t))$ is finite, then $H^i(\Gamma, M)/(t)$ is finite. 
	\item\label{cohom:vector-space} The canonical map $H^i(\Gamma, M) \otimes_A F \to H^i(\Gamma, V)$ is an isomorphism.
\end{enumerate}
\end{proposition}

\begin{proof}
Assume that $H^i(\Gamma, M)/(t)$ is finite. 
Let $Y \subseteq H^i(\Gamma, M)$ be a finitely generated $A$-submodule surjecting onto $H^i(\Gamma, M)/(t)$.
Then $H^i(\Gamma, M)/Y$ is $t$-divisible, so $Y = H^i(\Gamma, M)$ by Lemma~\ref{lem:no-divisible}, proving part~\ref{cohom:finite-gen}.

Part~\ref{cohom:finite-implies-finite} follows easily from the long exact sequence in cohomology resulting from the short exact sequence
\[ \xymatrix@1{ 0 \ar[r] & M \ar[r]^t & M \ar[r] & M/(t) \ar[r] & 0. } \]

We turn to part~\ref{cohom:vector-space}. 
Let $M_{\mathrm{tor}} \subseteq M$ be the $\ffrm_A$-torsion (equivalently, the $t$-torsion) submodule of $M$. 
Then $M$ is finite, hence annihilated by some fixed power $\ffrm_A^s$ of $\ffrm_A$. 
It follows that $\ffrm_A^s$ also annihilates $H^i(\Gamma, M_{\mathrm{tor}})$ for every $i \ge 0$. 
Considering the long exact sequences arising from
\[ \xymatrix@1{ 0 \ar[r] & M_{\mathrm{tor}} \ar[r] & M \ar[r] & M/M_{\mathrm{tor}} \ar[r] & 0, } \]
the natural map $H^i(\Gamma, M) \otimes_A F \to H^i(\Gamma, M/M_{\mathrm{tor}}) \otimes_A F$ is an isomorphism. 
We are thus reduced to the case that $M$ is torsion free. 
We then consider the short exact sequence
\[ \xymatrix@1{ 0 \ar[r] & M \ar[r] & V \ar[r] & V/M \ar[r] & 0. } \]
Since $V/M$ is discrete and $\Gamma$ is compact, any (continuous) $i$-cocyle $f : \Gamma^i \to V/M$ takes only finitely many values, 
so is annihilated by some power of $\ffrm_A$. 
It follows that $H^i(\Gamma, V/M) \otimes_A F = 0$ for all $i$, so the natural map $H^i(\Gamma, M) \otimes_A F \to H^i(\Gamma, V)$ is an isomorphism.
\end{proof}

\subsection{Subgroups of $\GL_2(\bbF\llbracket t \rrbracket)$}

Let $\bbF / \bbF_2$ be a finite extension, and let $A = \bbF\llbracket t \rrbracket$, $F = A[t^{-1}]$. If $R$ is a ring, let $\ad_R$ denote the $R$-module of $2 \times 2$ matrices with entries in $R$, on which $\GL_2(R)$ acts by conjugation, and let $\frz_R \subset \ad_R$ denote the scalar matrices. If $R$ is a local ring and $l \geq 1$ is an integer, then we write $\Gamma(\ffrm_R^l)$ for the principal congruence subgroup of $\SL_2(R)$ of level $l$.
\begin{lemma}\label{lem_no_invariants_in_ad_PGL_2}
    Let $\Gamma \subset \GL_2(A)$ be a compact subgroup with Zariski dense image in $\PGL_{2, F}$. Then $H^0(\Gamma, \ad_A / \frz_A) = 0$.
\end{lemma}
\begin{proof}
    It suffices to show that $H^0(\PGL_2(F), \ad_F / \frz_F) = 0$. Let $X \in \ad_F$ map to $H^0(\PGL_2(F), \ad_F / \frz_F)$. Then for all $\gamma \in \PGL_2(F)$, there exists $z_\gamma \in \frz_F$ such that $\gamma X \gamma^{-1} - X = z_\gamma$. Choosing $\gamma = \diag(\alpha, 1)$ with $\alpha \neq 1$, this forces $X$ to be diagonal. Choosing $\gamma = \left( \begin{array}{cc} 1 & \alpha \\ 0 & 1 \end{array}\right)$ with $\alpha \neq 0$, this forces $X \in \frz_F$. This completes the proof.
\end{proof}
\begin{lemma}\label{lem_open_subgroup_of_image}
    Let $\Gamma \subset \GL_2(A)$ be a compact subgroup containing a non-trivial unipotent element, and with Zariski dense image in $\PGL_{2, F}$. Then there exists $g \in \GL_2(F)$, a closed subfield $F_0 \subset F$ with $F / F_0$ finite and valuation ring $A_0 \subset A$, and an integer $l \geq 1$ such that $ \Gamma(\ffrm_{A_0}^l) \subset g \overline{[\Gamma, \Gamma]} g^{-1} \subset \SL_2(A_0)$. 
\end{lemma}
\begin{proof}
    Let $\Gamma'$ denote the image of $\Gamma$ in $\PGL_2(F)$. We apply \cite[Theorem 0.2]{Pink_compact} to $\Gamma'$ to deduce the existence of a subfield $F_0 \subset F$ with $F / F_0$ finite, a linear algebraic group $H$ over $F_0$, and an isomorphism $\varphi : H_F \to \PGL_{2, F}$ with the following properties: 
    \begin{itemize}
        \item We have $\Gamma' \subset \varphi(H(F_0))$.
        \item Let $\widetilde{H}$ denote the simply connected cover of $H$, and let $\widetilde{\varphi} : \widetilde{H}_F \to \SL_{2, F}$ denote the induced isomorphism. Then the closure of $[\Gamma, \Gamma]$ in $\SL_2(F)$ is an open subgroup of $\widetilde{\varphi}(\widetilde{H}(F_0))$. 
    \end{itemize}
    Since $\Gamma$ contains a non-trivial unipotent element, so does $\Gamma'$ and hence so does $H(F_0)$. It follows that $H$ is split, hence we can identify $H = \PGL_{2, F_0}$ and $\varphi = \operatorname{Ad}(g)$ for some $g \in \GL_2(F)$. Since $\SL_2(F_0)$ has a unique $\GL_2(F_0)$-conjugacy class of maximal compact subgroups, we can moreover assume that $g \overline{ [ \Gamma, \Gamma ]} g^{-1} \subset \SL_2(A_0)$.
\end{proof}
\begin{lemma}\label{lem_injection_of_cokernels}
    Let $\Gamma \subset \GL_2(A)$ be a compact subgroup, and let $\Delta \subset \Gamma$ be a closed normal subgroup with Zariski dense image in $\PGL_{2, F}$. Then the map 
    \[  \coker(H^1(\Gamma, \frz_A) \to H^1(\Gamma, \ad_A))\to \coker(H^1(\Delta, \frz_A) \to H^1(\Delta, \ad_A)) \]
    is injective.
\end{lemma}
\begin{proof}
	It suffices to show that the map $H^1(\Gamma, \ad_A / \frz_A) \to H^1(\Delta, \ad_A / \frz_A)$ is injective. Its kernel is $H^1(\Gamma / \Delta, (\ad_A / \frz_A)^{\Delta})$, and this is zero, by Lemma \ref{lem_no_invariants_in_ad_PGL_2}.
\end{proof}
\begin{lemma}\label{lem_explicit_cocycle}
    Let $\Gamma \subset \GL_2(A)$ be a compact subgroup containing a non-trivial unipotent element and with Zariski dense image in $\PGL_{2, F}$. Then 
    \[ \coker(H^1(\Gamma, \frz_A) \to H^1(\Gamma, \ad_A)) \]
    has $A$-rank at most 1. If $\Gamma$ is an open subgroup of $\SL_2(A)$, then the rank equals 1, a non-$t$-torsion element being given by the cocycle $\phi(g) = D(g) g^{-1}$, where $D : M_2(A) \to M_2(A)$ is the derivation given by the formula
    \[ D : \left( \begin{array}{cc} a(t) & b(t) \\ c(t) & d(t) \end{array}\right) \mapsto \left( \begin{array}{cc} a'(t) & b'(t) \\ c'(t) & d'(t) \end{array}\right) \]
    (in other words, differentiate each matrix entry with respect to $t$).
\end{lemma}
\begin{proof}
    We first show that the rank is at most 1. By Lemma \ref{lem_open_subgroup_of_image} and Lemma \ref{lem_injection_of_cokernels}, we can assume that there exists $g \in \GL_2(F)$ and a subfield $F_0 \subset F$ with $F / F_0$ finite and an integer $l \geq 1$ such that $\Gamma = g \Gamma(\ffrm_{A_0}^l) g^{-1}$. Let $M = g \ad_A g^{-1} \subset \ad_F$. We can find an integer $N \geq 1$ such that 
    \[ t^N \ad_A \subset M \subset t^{-N} \ad_A. \]
    This implies that the multiplication by $t^{2N}$ endomorphism of $H^1(\Gamma, \ad_A) / H^1(\Gamma, \frz_A)$ factors through $H^1(\Gamma, M) / H^1(\Gamma, \frz_A)$. We are therefore free to replace $\Gamma$ by $g^{-1} \Gamma g$. Since $A_0 \to A$ is flat, we finally reduce to the case that $\Gamma = \Gamma(\ffrm_A^l)$ is a principal congruence subgroup of $\SL_2(A)$.
    
    We now use the cocycle computations done in the proof of \cite[Lemma 3.1.6]{All14}. We can summarise these as follows. Let $\phi \in Z^1(\Gamma, \ad_A)$ be a cocycle. Then:
    \begin{enumerate}
        \item After multiplying $\phi$ by $t^{3l}$ and changing by a coboundary, we can assume that $\phi$ has the following `standard' form: for all $\alpha \in 1 + \ffrm_A^l$, $x \in \ffrm_A^l$, we have
        \[ \phi\left( \begin{array}{cc} \alpha & 0 \\ 0 & \alpha^{-1} \end{array}\right) = \left( \begin{array}{cc} a_\alpha & 0 \\ 0 & a_\alpha \end{array}\right), \]
        \[ \phi\left( \begin{array}{cc} 1 & x \\ 0 & 1 \end{array}\right) = \left( \begin{array}{cc} a_x & b_x \\ 0 & a_x \end{array}\right),\] 
        and
        \[ \phi\left( \begin{array}{cc} 1 & 0 \\ x & 1 \end{array}\right) = \left( \begin{array}{cc} d_x & 0 \\ b_x & d_x \end{array}\right). \]
        \item For $\phi$ in the above form, the image of $\phi$ in $Z^1(\Gamma, \ad_A / \frz_A)$ is uniquely determined by the function $b_x : \ffrm_A^l \to A$, and the image of $\phi$ in $Z^1(\Gamma(\ffrm_{A}^{3l}), \ad_A / \frz_A)$ is uniquely determined by the value $b_{t^{l+1}}$. 
    \end{enumerate}
    This has the following consequence: if $[\phi], [\phi'] \in H^1(\Gamma, \ad_A)$ are two cohomology classes which have non-$t$-torsion image in $H^1(\Gamma, \ad_A / \frz_A)$, then their images in $H^1(\Gamma(\ffrm_A^{3l}), \ad_A / \frz_A)$ must be proportional.  The composite map
    \[  \coker(H^1(\Gamma, \frz_A) \to H^1(\Gamma, \ad_A)) \to H^1(\Gamma, \ad_A / \frz_A) \to H^1(\Gamma(\ffrm_{A}^{3l}), \ad_A / \frz_A) \]
    is a composite of two injections, so is injective. This proves the first assertion of the lemma. 
    
    For the second, we now assume that $\Gamma \subset \SL_2(A)$ is an open subgroup. It is easy to see that the given formula for $\phi(g)$ defines an element of $Z^1(\Gamma, \ad_A)$. We just need to check that its image in $H^1(\Gamma, \ad_A / \frz_A)$ is not $t$-torsion. 
    
    Suppose then for contradiction that there exists $N \geq 1$ and $X = \left( \begin{array}{cc} a & b \\ c & d \end{array}\right) \in \ad_A$ such that for all $g \in \Gamma$, we have $t^N D(g) g^{-1} \equiv g X g^{-1} - X \text{ mod }\frz_A$. Considering an element of the form $g = \diag( 1 + t^k, (1+t^k)^{-1} )$ with $k > l$ an even integer, we find that $D(g) = 0$, hence 
    \[ g X g^{-1} - X = \left(\begin{array}{cc} 0 &((1 + t^{k})^2 - 1) b \\ ((1+t^{k})^{-2} - 1)c & 0 \end{array}\right) \in \frz_A. \]
    This forces $b = c = 0$. Considering $g = \left( \begin{array}{cc} 1 & t^k \\ 0 & 1 \end{array}\right)$ with $k > l$ an even integer, we find again $D(g) = 0$, hence
      \[ g X g^{-1} - X  = \left(\begin{array}{cc} 0 &t^k(a+d) \\ 0 & 0 \end{array}\right) \in \frz_A. \]
      This forces $a = d$, hence $X \in \frz_A$ and $D(g) g^{-1} \in \frz_A$ for all $g \in \Gamma$. Finally, considering an element of the form $g = \left( \begin{array}{cc} 1 & t^k \\ 0 & 1 \end{array}\right)$ with $k > l$ an odd integer gives
          \[ D(g) g^{-1} = \left( \begin{array}{cc} 0 & t^{k-1} \\ 0 & 0 \end{array}\right) \not\in \frz_A, \]
    a contradiction. This completes the proof of the lemma.
\end{proof}
\begin{lemma}\label{lem_special_element_sl_2}
    Let $\Gamma \subset \SL_2(A)$ be an open subgroup, and let $\sigma_0 \in \Gamma$ be an element such that $\overline{\sigma}_0 \in \SL_2(\bbF)$ is regular semisimple. Let $[\phi] \in H^1(\Gamma, \ad_A)$ be a class which has non-$t$-torsion image in $H^1(\Gamma, \ad_A) / H^1(\Gamma, \frz_A)$. Then for any open neighbourhood $U$ of $\sigma_0$ in $\Gamma$, there exists $\sigma \in U$ such that $\phi(\sigma) \not\in (\sigma - 1)\ad_A$.
\end{lemma}
\begin{proof}
    After conjugating by an element of $\GL_2(A)$, we can assume that $\sigma_0$ is diagonal. Observe that $(\sigma_0 - 1)\ad_A \subset \ad_A$ is then the rank 2 direct summand consisting of matrices with zero diagonal entries. 
    
    Let $l \geq 1$ be an integer such that $\Gamma(\ffrm_A^l) \subset \Gamma$ and such that for any $z \in Z^1(\Gamma, \frz_A)$, we have $z(g) = 0$ for all $g \in \Gamma(\ffrm_A^l)$. We will show that for any odd $k > l$, there exists $\tau \in \Gamma(\ffrm_A^k)$ such that if $\sigma = \tau \sigma_0$, then $\phi(\sigma) \not\in (\sigma - 1) \ad_A$. This will imply the result. If $\phi(\sigma_0) \not\in (\sigma_0 - 1) \ad_A$, then we're done. We can therefore assume without loss of generality that $\phi(\sigma_0) \in (\sigma_0 - 1) \ad_A$. We are also free to replace $\phi$ by a multiple by a non-zero element of $A$, and to change it by a coboundary. By Lemma \ref{lem_explicit_cocycle}, we can therefore assume that $\phi(g) = t^m D(g) g^{-1} + z(g)$, where $D : M_2(A) \to M_2(A)$ is the derivation given by differentiating each matrix entry, $m \geq 0$ is an integer, and $z \in Z^1(\Gamma, \frz_A)$. Note that $\phi(\sigma_0) = t^m D(\sigma_0) \sigma_0^{-1} + z(\sigma_0)$ has diagonal entries equal to 0; this means that $z(\sigma_0) \equiv 0 \text{ mod }t^m \ad_A$ and therefore that $\phi(\sigma_0)$ is divisible by $t^m$ in $\ad_A$.
    
    Let $k > l$ be an odd integer, let $\tau \in \Gamma(\ffrm_A^k)$, and consider $\sigma = \tau \sigma_0$. We have $\phi(\sigma) = \phi(\tau \sigma_0) = \phi(\tau) + {}^\tau \phi(\sigma_0)$. A computation shows that $\phi(\tau) \equiv t^m D(\tau) \text{ mod } t^{m+k} \ad_A$ and ${}^\tau \phi(\sigma_0) \equiv \phi(\sigma_0) \text{ mod }t^{m+k} \ad_A$, hence $\phi(\sigma) \equiv t^m D(\tau) + \phi(\sigma_0) \text{ mod }t^{m+k} \ad_A$. It would be enough to show that we can choose $\tau$ so that $\phi(\sigma) \text{ mod }t^{m+k} \ad_A \not\in (\sigma-1) \ad_{A / t^{m+k}}$, or equivalently such that $D(\tau) + \phi(\sigma_0) / t^m \text{ mod }t^k \ad_A \not \in (\sigma - 1) \ad_{A / t^k} = (\sigma_0 - 1) \ad_{A / t^k}$, or in other words such that
    \[ D(\tau) \text{ mod }t^k \ad_A \not\in \phi(\sigma_0) / t^m + (\sigma_0 - 1) \ad_{A / t^k}. \]    
The possible values of $D(\tau) \text{ mod }\ffrm_A^k$ span a 3-dimensional $\bbF$-vector subspace of $\ad_{A / (t^{k})}[t]$. Since $(\sigma_0 - 1) \ad_{A / (t^k)}[t]$ is 2-dimensional, there must exist a choice of $\tau$ satisfying the required condition. This completes the proof.
\end{proof}
\begin{lemma}\label{lem_special_element}
    Let $\Gamma \subset \GL_2(A)$ be a compact subgroup containing a non-trivial unipotent element, and with Zariski dense image in $\PGL_{2, F}$. Let $\sigma_0 \in \overline{[\Gamma, \Gamma]}$ be an element such that $\overline{\sigma}_0 \in \GL_2(\bbF)$ is regular semisimple. Suppose that there exists a class $[\phi] \in H^1(\Gamma, \ad_A)$ which has non-$t$-torsion image in $H^1(\Gamma, \ad_A) / H^1(\Gamma, \frz_A)$. Then for any open neighbourhood $U$ of $\sigma_0$ in $\overline{[\Gamma, \Gamma]}$, there exists $\sigma \in U$ such that $\phi(\sigma) \not\in (\sigma - 1)\ad_A$.
\end{lemma}
\begin{proof}
By Lemma \ref{lem_open_subgroup_of_image}, there exists a subfield $F_0 \subset F$ with $F / F_0$ finite, an element $g_0 \in \GL_2(F)$, and an integer $l \geq 1$ such that we have
    \[ \Gamma(\ffrm_{A_0}^l) \subset g_0 \overline{[ \Gamma, \Gamma ]} g_0^{-1} \subset \SL_2(A_0). \]
    Let $U$ be an open neighbourhood of $\sigma_0$ in $\overline{[\Gamma, \Gamma]}$, and let $M = g_0 \ad_A g_0^{-1} \subset \ad_F$. After shrinking $U$, we can assume that for any $\sigma \in U$, we have $\det(\lambda - \sigma) \equiv \det(\lambda - \sigma_0) \text{ mod }(t)$. Choose an integer $N \geq 1$ such that 
    \[ t^N \ad_A \subset M \subset t^{-N} \ad_A. \]
    Let $\psi$ denote the image of $t^N \phi$ in $Z^1(\overline{[\Gamma, \Gamma]}, M)$. By Lemma \ref{lem_special_element_sl_2}, we can find an element $\sigma \in U$ such that $\psi(\sigma) \not\in (\sigma - 1)M$. We claim that this implies $\phi(\sigma) \not\in (\sigma - 1) \ad_A$. Indeed, since $\overline{\sigma}$ is regular semisimple, $(\sigma - 1)\ad_A$ is an $A$-direct summand of $\ad_A$, and $(\sigma - 1)\ad_A = (\sigma - 1)\ad_F \cap \ad_A$. We therefore have $\phi(\sigma) \in (\sigma - 1)\ad_A$ if and only if $\phi(\sigma) \in (\sigma - 1)\ad_F$. Since $M \otimes_A K = \ad_A \otimes_A K = \ad_F$, we find that $\phi(\sigma) \in (\sigma - 1) \ad_A$ if and only if $\psi(\sigma) \in (\sigma - 1)M$. This completes the proof. 
\end{proof}

\begin{proposition}\label{prop:cokernel_class}
    Let $\Gamma \subset \GL_2(A)$ be a compact subgroup containing a non-trivial unipotent element and with Zariski dense image in $\PGL_{2, F}$. Suppose that  $\Gamma_0 = \mathrm{im}(\Gamma \to \PGL_2(\bbF))$ is a dihedral group. Then for any class $[\phi] \in H^1(\Gamma, \ad_A)$ which has non-$t$-torsion image in $\coker(H^1(\Gamma, \frz_A) \to H^1(\Gamma, \ad_A))$, there exists $\sigma \in \overline{[\Gamma, \Gamma]}$ such that $\overline{\sigma} \in \GL_2(\bbF)$ is regular semisimple and $\phi(\sigma) \not\in (\sigma - 1) \ad_A$.
\end{proposition}
\begin{proof}
    The odd dihedral group $\Gamma_0$ has abelianization of order 2, and $[\Gamma_0 , \Gamma_0 ]$ contains regular semisimple elements. Let $\sigma_0 \in \overline{[\Gamma, \Gamma]}$ be an element whose image in $\GL_2(\bbF)$ is regular semisimple. By Lemma \ref{lem_special_element}, we can find an element $\sigma \in \overline{[\Gamma, \Gamma]}$ with the same image in $\GL_2(\bbF)$ as $\sigma_0$ and such that $\phi(\sigma) \not\in (\sigma - 1) \ad_A$. This completes the proof.
\end{proof}

\section{Commutative algebra}\label{sec_commutative_algebra}

In this short section we recall a few useful facts from commutative algebra. 
\begin{lemma}\label{lem_existence_of_bounded_above_resolution}
	Let $A$ be a ring and let $C$ be a cochain complex of $A$-modules. Suppose that each group $H^i(C)$ is a finitely generated $A$-module, and that $H^i(C) = 0$ for all sufficiently large $i$. Then we can find a quasi-isomorphism $P \to A$, where $P$ is a bounded above complex of finite free $A$-modules.
\end{lemma}
\begin{proof}
	This is standard. After replacing $C$ by its truncation, we can assume that $C$ is bounded above. The existence of $P$ then follows from the existence of the Cartan--Eilenberg resolution (see e.g.\ \cite[\S 5.7]{Wei94}).
\end{proof}
\begin{lemma}\label{lem_existence_of_minimal_resolution}
	Let $A$ be a local ring, and let $P$ be a bounded above complex of finite free $A$-modules. Then we can find a direct summand subcomplex $Q \subset P$ such that the inclusion $Q \to P$ is a quasi-isomorphism and $Q$ is minimal, in the sense that the differentials in $Q \otimes_A A / \ffrm_A$ are all zero.
\end{lemma}
\begin{proof}
	We recursively modify $P$ so that the differentials $d_i : P^i \to P^{i+1}$ are zero modulo $\ffrm_A$ for each $i \geq n$. Suppose that $P$ has the property that the differentials $d_i : P^i \to P^{i+1}$ are zero modulo $\ffrm_A$ for each $i \geq n$. We can choose bases for $P^{n-1}$ and $P^n$ so that $d_{n-1}$ is given by a block matrix of the form
	\[ d_{n-1} = \left( \begin{array}{cc} 1_s & 0 \\ 0 & M \end{array}\right), \]
	where $M$ has entries in $\ffrm_A$. This determines a direct sum decomposition $P^{n-1} = X^{n-1} \oplus Y^{n-1}$, $P^n = X^n \oplus Y^n$, where $d_{n-1}$ determines an isomorphism $X^{n-1} \to X^n$ and satisfies $d_{n-1}(Y^{n-1}) \subset \ffrm_A Y^n$. One can check that if we let $Q$ denote the complex with $Q^{n-1} = Y^{n-1}$ and $Q^n = Y^n$, then $Q \subset P$ is a direct summand subcomplex such that the inclusion $Q \subset P$ is a quasi-isomorphism, and the differentials $d_i : Q^i \to Q^{i+1}$ are zero modulo $\ffrm_A$ for each $i \geq n-1$. 
\end{proof}

\begin{lemma}\label{lem_support_and_derived_tensor_product}
	Let $A \to B$ be a homomorphism of  Noetherian rings. Let $C$ be a bounded complex of finite free $A$-modules, and suppose given a map $B \to \End_{\mathbf{D}(A)}(C)$ of $A$-algebras. Let $J \subset A$ be an ideal. Then we have
	\[ \Supp_{B / J B} H^\ast(C \otimes^\bbL_A A / J) = \Supp_B H^\ast(C) \cap \Spec B / J B. \]
\end{lemma}
\begin{proof}
	There is a spectral sequence
	\[ E_2^{p, q} = \Tor_{-p}^A(H^q(C), A / J) \Rightarrow H^{p + q}(C \otimes^\bbL_A A / J), \]
	\[ d_2^{p, q} : E_2^{p, q} \to E_2^{p+2, q-1} \]
	of $B$-modules. If $P \in \Spec B - \Supp_B H^\ast(C)$, then $E_2^{p, q}$ vanishes after localization at $P$, showing that $P \not\in \Supp_{B / JB} H^\ast(C \otimes^\bbL_A A / J)$. Suppose conversely that $P \in \Supp_B H^\ast(C) \cap \Spec B / J B$; we must show that $P \in \Supp_{B / J B} H^\ast(C \otimes^\bbL_A A / J)$. Let $q$ be the greatest integer such that $H^q(C)_{(P)} \neq 0$. Then $H^q(C)_{(P)} / J H^q(C)_{(P)} \neq 0$, and $H^q(C)_{(P)} / J H^q(C)_{(P)} = (E_2^{0, q})_{(P)} =  (E_\infty^{0, q})_{(P)}$ is a stable entry of the spectral sequence after localization at $P$. Therefore $H^q(C)_{(P)} / J H^q(C)_{(P)}$ is a non-zero subquotient of $H^\ast(C \otimes^\bbL_A A / J)_{(P)}$ and $P \in \Supp_{B / J B} H^\ast(C \otimes^\bbL_A A / J)$, as required.
\end{proof}
\begin{lemma}\label{lem_primes_and_supports}
	Let $A$ be a Noetherian ring, and let $M$ be a finite $A$-module. Let $P \subset A$ be a prime ideal. We recall (see \S \ref{subsec_notation}) that $A_P$ denotes the completion of the local ring $A_{(P)}$. Then:
	\begin{enumerate}
		\item If $Q \subset A_P$ is a minimal prime, then the pullback $Q'$ of $Q$ to $A$ is a minimal prime of $A$ which is contained in $P$. Every minimal prime of $A$ which is contained in $P$ arises in this way. If $Q \in \Supp_{A_P} M_P$, then $Q' \in \Supp_{A} M$. 
		\item Let $\overline{M} = \operatorname{im}(M \to M_{(P)})$. Then $\Ass_A(\overline{M}) = \Ass_{A_{(P)}} M_{(P)}$. In particular, if $Q \in \Ass_A \overline{M}$ (e.g.\ if $Q$ is a prime minimal in the support of $\overline{M}$), then $Q \subset P$.
	\end{enumerate}
\end{lemma}
\begin{proof}
	The first part follows from the going down theorem and the equality $\Ann_{A_P} M_P = (\Ann_A M) A_P$. The second part follows from standard properties of associated primes (see  \cite[\href{https://stacks.math.columbia.edu/tag/00L9}{Section 00L9}]{stacks-project}).
\end{proof}
\begin{lemma}\label{lem_finite_end_group}
	Let $R$ be a Noetherian ring, and let $G$ be a (possibly non-abelian) finite group. Let $C$ be a complex of $R[G]$-modules such that $H^\ast(C)$ is a finite $R[G]$-module. Then:
	\begin{enumerate}
		\item $\End_{\mathbf{D}(R[G])}(C)$ is a finite $R$-module.
		\item Suppose moreover that $H^i(C) = 0$ if $i \not\in [0, d]$. Let $A \subset \End_{\mathbf{D}(R[G])}(C)$ be a commutative $R$-subalgebra, and let $I = \ker(A \to \End_{R[G]}(H^\ast(C))$. Then $I^{d+1} = 0$. 
	\end{enumerate}
\end{lemma}
\begin{proof}
	For the first part we can assume, after translation and truncation, that $C$ is concentrated in positive degrees. We first show that if $M$ is a finite $R[G]$-module and $j \in \bbZ$, then $\Hom_{\mathbf{D}(R[G])}(C, M[j])$ is a finite $R$-module, by showing using induction on $i$ that $\Hom_{\mathbf{D}(R[G])}(\tau_{\leq i} C, M[j])$ is a finite $R$-module for each $i \geq 0$. For each $i \geq 0$ there is an exact triangle in $\mathbf{D}(R[G])$:
	\[ \tau_{\leq i} C \to \tau_{\leq i+1} C \to H^{i+1}(C)[-(i+1)] \to \tau_{\leq i}(C)[1]. \]
	There is an associated exact sequence
	\[ \Hom_{\mathbf{D}(R[G])}(H^i(C)[-(i+1)], M[j]) \to \Hom_{\mathbf{D}(R[G])}(\tau_{\leq i+1} C, M[j]) \to \Hom_{\mathbf{D}(R[G])}(\tau_{\leq i} C, M[j])  \]
	The first module in this sequence is a finite $R$-module, since it can be calculated as an Ext group between finite $R[G]$-modules. By induction, we reduce to the case $i = 0$, in which case $\tau_{\leq 0} C$ is isomorphic in $\mathbf{D}(R[G])$ to a finite $R[G]$-module. 
	
	We now show by induction on $i$ that $\Hom_{\mathbf{D}(R[G])}(C, \tau_{\leq i} C)$ is a finite $R$-module. The case $i = 0$ follows from the previous paragraph. In general, we use the exact sequence
	\[ \Hom_{\mathbf{D}(R[G])}(C, \tau_{\leq i} C) \to \Hom_{\mathbf{D}(R[G])}(C, \tau_{\leq i+1} C) \to \Hom_{\mathbf{D}(R[G])}(C, H^i(C)[-(i+1)]). \]
	This proves the first part of the lemma. For the second, let $f_0, \dots, f_d \in I$. We must show that $f_0 \cdots f_d = 0$. We show by induction on $i \geq 0$ that $\tau_{\leq i} f_0 \cdots f_i = 0$ in $\End_{\mathbf{D}(R[G])}(\tau_{\leq i} C)$. The case $i = 0$ follows immediately. In general, we use the exact sequence
	\[ \Hom_{\mathbf{D}(R[G])}(\tau_{\leq i+1} C, \tau_{\leq i} C) \to \Hom_{\mathbf{D}(R[G])}(\tau_{\leq i+1} C, \tau_{\leq i+1} C) \to \Hom_{\mathbf{D}(R[G])}(\tau_{\leq i+1} C,  H^{i+1}(C)[-(i+1)]). \]
	Let $\alpha_i : \tau_{\leq i} C \to \tau_{\leq i+1} C$ be the natural map. The exactness of this sequence implies that there is $g \in \Hom_{\mathbf{D}(R[G])}(\tau_{\leq i+1} C, \tau_{\leq i} C)$ such that $\alpha_i \circ g = \tau_{\leq i+1} f_{i+1}$. Using the induction hypothesis $\tau_{\leq i} f_0 \cdots f_i = 0$ and the functorial nature of truncation, we find
	\[ 0 = \alpha_i \circ \tau_{\leq i} f_0 \cdots f_i \circ g = \tau_{\leq i+1} f_0 \cdots f_i \circ \alpha_i \circ g = \tau_{\leq i+1} f_0 \cdots f_{i+1}. \]
	This completes the proof.
\end{proof}
We now discuss ordinary parts. Let $A$ be a complete Noetherian local ring. In \cite{Kha17} we defined the ordinary part of a perfect complex of $A$-modules with respect to an endomorphism $T$. We now need to extend this to complexes which are not necessarily perfect.
\begin{proposition}\label{prop_ord_summand}
	Let $A$ be a complete Noetherian local ring, and let $C$ be a complex of $A$-modules such that for each $i \in \bbZ$, $H^i(C)$ is a finitely generated $A$-module. Then:
	\begin{enumerate}
		\item Suppose that $T \in \End_{\mathbf{D}(A)}(C)$ is an endomorphism such that $A[T] \subset \End_{\mathbf{D}(A)}(C)$ is a finite $A$-algebra. Then there is a unique idempotent $e \in A[T]$ such that for all $i \in \bbZ$, $H^i(T)$ acts invertibly on $e H^i(C)$ and topologically nilpotently on $(1 - e) H^i(C)$. Consequently, there is a canonical $T$-invariant direct sum decomposition $C = C^\text{ord} \oplus C^\text{non-ord}$ in $\mathbf{D}(A)$ with the property that the composite
		\[ C \to C^\text{ord} \to C \]
		equals the idempotent $e$, that $T$ acts invertibly on each $H^i(C^\text{ord})$, and acts topologically nilpotently on $H^i(C^\text{non-ord})$.
		\item Suppose that $R \subset \End_{\mathbf{D}(A)}(C)$ is a finite $A$-algebra. Then for each maximal ideal $\ffrm \subset R$ there is a unique idempotent $e_\ffrm \in R$ such that for all $i \in \bbZ$, $\ffrm$ lies in the support of $e_\ffrm H^i(C)$ and does not lie in the support of $(1 - e_\ffrm) H^i(C)$. Consequently, there is a canonical $R$-invariant direct sum decomposition $C = \oplus_\ffrm C_\ffrm$ in $\mathbf{D}(A)$ with the property that for any maximal ideal $\ffrm \subset R$, the composite
		\[ C \to C_\ffrm \to C \]
		equals $e_\ffrm$, and $\ffrm$ is the only maximal ideal of $R$ in the support of each non-zero group $H^i(C_\ffrm)$. 
	\end{enumerate}
\end{proposition}
\begin{proof}
	This follows from the fact that the derived category $\mathbf{D}(A)$ is idempotent complete, as we now explain. The second part is very similar to the first, so we justify only the first. Since $A[T]$ is a finite $A$-algebra, it decomposes as a direct product $A[T] = \prod_i A_i$ of complete local $A$-algebras. We let $e$ denote the idempotent corresponding to the product of those factors in which $T$ is a unit. The idempotent completeness of $\mathbf{D}(A)$ implies that there is a direct sum decomposition $C = C_1 \oplus C_2$ in $\mathbf{D}(A)$ corresponding to $e$. In other words, there are morphisms $p_i : C \to C_i$ and $j_i : C_i \to C$ such that $p_i \circ j_i = \operatorname{id}_{C_i}$, $j_1 p_1 + j_2 p_2 = \operatorname{id}_C$, and $e = j_1 p_1$. We claim that we can take $C^\text{ord} = C_1$ and $C^{\text{non-ord}} = C_2$. We just need to check that for each $i \in \bbZ$, $T$ acts invertibly on $H^i(C_1)$ and topologically nilpotently on $H^i(C_2)$.
	
	However, it follows from the definition that for each $i \in \bbZ$ we have $H^i(C) = H^i(C_1) \oplus H^i(C_2)$, that $(1-e)H^i(C_1) = 0$, and $e H^i(C_2) = 0$. Since $T$ is a unit in $e A[T]$ and topologically nilpotent in $(1-e) A[T]$, the result follows. 
\end{proof}

\section{Galois deformation theory}\label{sec_Galois_deformation_theory}

In this section we define and study some Galois deformation rings. In order to fix ideas, we recall the basic definitions and theorems. We then go on to study some particular local deformation problems and the global structure of certain Galois deformation rings. Throughout, we consider deformations of 2-dimensional Galois representations without fixing determinants. Since we will ultimately be focused on the case $p = 2$, we specialize to this case when it simplifies arguments. 

\subsection{Basics}\label{sec_deformation_basics}

Let $K$ be a number field, let $p$ be a prime, and let $\cO$ be a coefficient ring. Let $\overline{\rho} : G_K \to \GL_2(k)$ be a continuous, absolutely irreducible representation. Let $S_p$ denote the set of $p$-adic places of $K$, and let $S$ be a finite set of finite places of $K$ containing $S_p$ and all the places at which $\overline{\rho}$ is ramified. We also fix for each $v \in S$ a ring $\Lambda_v \in \CNL_\cO$, and set $\Lambda = \widehat{\otimes}_{v \in S} \Lambda_v$, the completed tensor product being relative to $\cO$.

Let $v \in S$. By definition, a local deformation problem at $v$ is a representable subfunctor of the functor $\Def^\square_v : \CNL_{\Lambda_v} \to \Sets$ of all liftings of $\overline{\rho}|_{G_{K_v}}$, which is invariant under the action of $\widehat{\GL}_2$ by conjugation (see \S \ref{subsec_notation} for the notation $\widehat{\GL}_2$).

A global deformation problem is a tuple 
\[ \cS = ( K, \overline{\rho}, S, \{ \Lambda_v \}_{v \in S}, \{ \cD_v \}_{v \in S} ), \]
where for each $v \in S$, $\cD_v$ is a local deformation problem at $v$. If $R \in \CNL_\Lambda$, we say that a continuous lifting $\rho_R : G_K \to \GL_2(R)$ of $\overline{\rho}$ is of type $\cS$ if $\rho_R$ is unramified outside $S$, and $\rho_R \in \cD_v(R)$ for each $v \in S$. This allows us to define $\Def^\square_\cS$, the functor of all liftings of type $\cS$, and $\Def_\cS$, the functor of all deformations (i.e.\ strict equivalence classes, or equivalently $\widehat{\GL}_2(R)$-conjugacy classes, of liftings) of type $\cS$. These functors are representable, and we write $R_\cS^\square$, $R_\cS \in \CNL_\cO$ for the representing objects.

We also introduce the notion of $T$-framed lifting and $T$-framed deformation. Let $T \subset S$ be a subset. By definition, a $T$-framed lifting of $\overline{\rho}$ to $R \in \CNL_\Lambda$ is a pair $(\{\alpha_v\}_{v \in T}, \rho_R)$, where $\alpha_v \in \widehat{\GL}_2(R)$ for each $v \in T$ and $\rho_R : G_K \to \GL_2(R)$ is a lifting of $\overline{\rho}$. We say that the $T$-framed lifting is of type $\cS$ if $\rho_R$ is. Two $T$-framed liftings $(\{\alpha_v\}_{v \in T}, \rho_R)$, $(\{\alpha'_v\}_{v \in T}, \rho'_R)$ are said to be strictly equivalent if there exists $\beta \in \widehat{\GL}_2(R)$ such that $\alpha'_v = \beta \alpha_v$ for each $v \in T$ and $\beta \rho_R \beta^{-1} = \rho_R'$. 

We observe that if $v \in T$ then the lifting $\alpha_v \rho_R|_{G_{K_v}} \alpha_v^{-1}$ of $\overline{\rho}|_{G_{K_v}}$ depends only on the strict equivalence class of the $T$-framed lifting. We call a strict equivalence class of $T$-framed liftings of $\overline{\rho}$ a $T$-framed deformation. The functor $\Def_\cS^T : \CNL_{\Lambda} \to \Sets$ of $T$-framed deformations of type $\cS$ is representable, and we write $R_\cS^T \in \CNL_{\Lambda} $ for the representing object. 

\subsection{Some local deformation problems}

\subsubsection{An ordinary deformation problem}\label{subsubsec_ordinary_deformation_problems}

Let $v$ be a place of $K$ dividing $p$, and suppose that $\overline{\rho}|_{G_{K_v}}$ has the form
\[ \overline{\rho}|_{G_{K_v}}(\sigma) = \left( \begin{array}{cc} 1 & \alpha(\sigma) \\ 0 & 1 \end{array} \right) \]
for some \emph{ramified} homomorphism $\alpha : G_{K_v} \to k$. We assume moreover that $\zeta_p \in K_v$, and that $\alpha$ is the Kummer character attached to an element $\beta \in K_v^\times$ of valuation not divisible by $p$; in particular, it has image of order $p$. (One could define ordinary deformation problems much more generally, as is done for example in \cite{Ger18}. The reason for imposing these hypotheses is: the ramifiedness of $\alpha$ implies the representability of the functor $\cD_v^\text{ord}$ in the form defined below. The special form of $\alpha$ implies that the representing object is in fact a formally smooth $\Lambda_v$-algebra.)

Let $\Lambda_v = \cO\llbracket (\cO_{K_v}^\times)^f \times (\cO_{K_v}^\times)^f \rrbracket$, where $(\cO_{K_v}^\times)^f$ denotes the maximal torsion-free quotient of $\cO_{K_v}^\times$. We identify $(\cO_{K_v}^\times)^f$ with a quotient of the group $I_{K_v}$, via class field theory. We write $\phi_{v, 1}, \phi_{v, 2} : I_{K_v} \to \Lambda_v^\times$ for the two universal characters. 

We define a functor $\cD_v^\text{ord} \subset \Def^\square_v$ as follows: if $R \in \CNL_{\Lambda_v}$ and if $\rho_R : G_{K_v} \to \GL_2(R)$ is a lifting of $\overline{\rho}|_{G_{K_v}}$, then $\rho_R \in \cD_v^\text{ord}(R)$ if and only if $\rho_R$ is $\widehat{\GL}_2(R)$-conjugate to a lifting of the form
\begin{equation}\label{eqn_upper_triangular_lifting} \rho'_R(\sigma) = \left( \begin{array}{cc} \psi_{v, 1}(\sigma) & a(\sigma) \\ 0 & \psi_{v, 2}(\sigma) \end{array}\right),
\end{equation}
where we have $\psi_{v, 1}|_{I_{K_v}} = \phi_{v, 1}$ and $\psi_{v, 2}|_{I_{K_v}} = \phi_{v, 2}$.
\begin{proposition}\label{prop_representability_of_ordinary_lifting_problem}
The functor $\cD_v^\text{ord}$ is a local deformation problem. 
\end{proposition}
\begin{proof}
Let $\cD' \subset \cD_v^\text{ord}$ denote the subfunctor of upper-triangular liftings of the form (\ref{eqn_upper_triangular_lifting}), and let $\cE \subset \widehat{\GL}_2$ denote the subfunctor of lower-triangular matrices with 1's on the diagonal. Then there is a natural transformation $\cD' \times \cE \to \cD_v^\text{ord}$, $(\rho, u) \mapsto u \rho u^{-1}$. Since $\cD'$ and $\cE$ are obviously representable, to prove the proposition it suffices to show that this natural transformation is an isomorphism, i.e.\ that it is bijective on $R$-points for every $R$. 

The surjectivity follows from the equality $\cE(R) \widehat{B}(R) = \widehat{\GL}_2(R)$ for every $R \in \CNL_{\Lambda_v}$. To show injectivity, it suffices to show the following statement: let $R \in \CNL_{\Lambda_v}$, and suppose that we have a lifting
\[ \rho_R = \left( \begin{array}{cc} \psi_{v, 1}(\sigma) & a(\sigma) \\ 0 & \psi_{v, 2}(\sigma) \end{array}\right), \]
where we have $\psi_{v, 1}|_{I_{K_v}} = \phi_{v, 1}$ and $\psi_{v, 2}|_{I_{K_v}} = \phi_{v, 2}$. Suppose that there exists $g \in \widehat{\GL}_2(R)$ such that we have
\[ g \rho_R g^{-1} = \left( \begin{array}{cc} \psi'_{v, 1}(\sigma) & a'(\sigma) \\ 0 & \psi'_{v, 2}(\sigma) \end{array}\right), \]
where we have $\psi'_{v, 1}|_{I_{K_v}} = \phi_{v, 1}$ and $\psi'_{v, 2}|_{I_{K_v}} = \phi_{v, 2}$. Then $g$ is upper triangular. 

To prove this, we can assume that $R$ is Artinian, and reduce to the case that $\ffrm_R^{i+1} = 0$ for some $i \geq 1$ and $g = 1 + X$ for some $X \in M_2(\ffrm_R^i)$. Let us write
\[ X = \left( \begin{array}{cc} A & B \\ C & D \end{array}\right). \]
Then we have for any $\sigma \in G_{K_v}$, we have
\[ g \rho_R(\sigma) g^{-1} =  \left( \begin{array}{cc} \psi_{v, 1}(\sigma) - C a(\sigma) & (1 + A-D) a(\sigma)  \\ 0 & \psi_{v, 2}(\sigma) + C a(\sigma) \end{array}\right). \]
Choose $\sigma \in I_{K_v}$ such that $\alpha(\sigma) \neq 0$. Then we have  $\psi_{v, 1}(\sigma) - C a(\sigma) = \psi_{v, 1}(\sigma)$, by construction, hence $C = 0$, showing that $g$ is upper triangular. This completes the proof.
\end{proof}
We write $R_v^\text{ord} \in \CNL_{\Lambda_v}$ for the representing object of $\cD_v^\text{ord}$. The definition of $\cD_v^\text{ord}$ implies that there are uniquely determined characters $\psi_{v, 1}, \psi_{v, 2} : G_{K_v} \to (R_v^{\text{ord}})^\times$ with the property that the universal lifting is $\widehat{\GL}_2(R_v^{\text{ord}})$-conjugate to an upper-triangular one with diagonal characters $\psi_{v, 1}, \psi_{v, 2}$.
\begin{proposition}\label{prop_characterization_of_ordinary_deformations_by_pseudocharacter_relations}
Let $R \in \CNL_{\Lambda_v}$, and suppose given a lifting $\rho_R : G_{K_v} \to \GL_2(R)$ of $\overline{\rho}|_{G_{K_v}}$ and characters $\widetilde{\psi}_{v, 1}, \widetilde{\psi}_{v, 2} : G_{K_v} \to R^\times$ extending the respective pushforwards of $\phi_{v, 1}$, $\phi_{v, 2}$, all satisfying the following conditions:
\begin{enumerate}
\item For all $g \in G_{K_v}$, we have $\det( X - \rho_R(g) ) = (X - \widetilde{\psi}_{v, 1}(g)) ( X - \widetilde{\psi}_{v, 2}(g) )$.
\item For all $g_1, g_2 \in G_{K_v}$, we have $(\rho_R(g) - \widetilde{\psi}_{v, 1}(g_1))(\rho_R(g_2) - \widetilde{\psi}_{v, 2}(g_2)) = 0$.
\end{enumerate}
Then $\rho_R$ is of type $\cD_v^\text{ord}$.
\end{proposition}
\begin{proof}
Choose $g_0 \in I_{K_v}$ such that $\alpha(g_0) \neq 0$. Let $T = \rho_R(g_0) - \phi_{v, 2}(g_0) \in M_2(R)$. Let $f_1, f_2$ denote the standard basis of $R^2$, and let $e_2 = f_2$, $e_1 = T e_2$. Then, using overline to denote reduction modulo $\ffrm_R$, we have $\overline{e}_1 = \overline{f}_1$ and $\overline{e}_2 = \overline{f}_2$. After conjugating $\rho_R$ by an element of $\widehat{\GL}_2(R)$, we can therefore assume that $e_1 = f_1$ and $e_2 = f_2$. 

In this case, we will show that $\rho_R$ has the form
\[ \rho_R(\sigma) = \left( \begin{array}{cc} \widetilde{\psi}_{v, 1}(\sigma) & a(\sigma) \\ 0 & \widetilde{\psi}_{v, 2}(\sigma) \end{array} \right) \]
for all $\sigma \in G_{K_v}$. This will imply the truth of the proposition.

By hypothesis we have $(T + \phi_{v, 2}(g_0) - \phi_{v, 1}(g_0))T = 0$. This implies that $T = \left( \begin{array}{cc} x & 1 \\ 0 & 0 \end{array}\right)$ for some $x \in R$. The relation
\[ (\rho_R(\sigma) - \widetilde{\psi}_{v, 1}(\sigma))(\rho_R(g_0) - \phi_{v, 2}(g_0)) = (\rho_R(\sigma) - \widetilde{\psi}_{v, 1}(\sigma))T = 0 \]
therefore implies that the first column of $\rho_R(\sigma)$ is as claimed. The relation $\det \rho_R(\sigma) = \widetilde{\psi}_{v, 1}(\sigma) \widetilde{\psi}_{v, 2}(\sigma)$ implies that the second column of $\rho_R(\sigma)$ is as claimed.
\end{proof}
\begin{remark}
We could generalize the proposition as follows: suppose given an injection $f : R \to S$ in $\CNL_{\Lambda_v}$, and suppose given homomorphisms $\rho_R : G_{K_v} \to \GL_2(R)$ and $\widetilde{\psi}_{v, 1}, \widetilde{\psi}_{v, 2} : G_{K_v} \to S^\times$ such that $f \circ \rho_R$ satisfies the hypotheses of the proposition. Then $\rho_R$ is of type $\cD_v^\text{ord}$. Indeed, the proposition implies that $f \circ \rho_R \in \cD_v^\text{ord}(S)$, which implies in turn that $\rho_R \in \cD_v^\text{ord}(R)$ (because $\cD_v^\text{ord}$ is representable).
\end{remark}
\begin{proposition}\label{prop_smoothness_of_ordinary_lifting_problem}
$R_v^\text{ord}$ is a formally smooth $\Lambda_v$-algebra of dimension $5 + 3 [K_v : \bbQ_p]$.
\end{proposition}
\begin{proof}
It is easy to show that $R_v^\text{ord}$ has dimension at least $5 + 3 [K_v : \bbQ_p]$. For example, we can appeal to \cite[Lemma 3.7]{Ger18} to see that $R_v^\text{ord}[1/p]$ is formally smooth of dimension $4 + 3[K_v : \bbQ_p]$ at the closed point corresponding to the Tate module of an elliptic curve with split multiplicative reduction (and Tate parameter $q = \beta$). The key point then is to show that $R_v^\text{ord} / (\ffrm_{\Lambda_v})$ has tangent space of dimension equal to $4 + [K_v : \bbQ_p]$. The ring $R_v^\text{ord} / (\ffrm_{\Lambda_v})$ represents the functor $\CNL_k \to \Sets$ of liftings of $\overline{\rho}$ to homomorphisms $\rho_R : G_{K_v} \to \GL_2(R)$ which are $\widehat{\GL}_2(R)$-conjugate to one of the form 
\[ \rho_R(\sigma) = \left( \begin{array}{cc} \psi_{v, 1}(\sigma) & a(\sigma) \\ 0 & \psi_{v, 2}(\sigma) \end{array}\right)  \]
where the characters $\psi_{v, 1}, \psi_{v, 2}$ are unramified. 

Let $\frb \subset \ad$ denote the subspace of upper-triangular matrices, and let $\frb = \frt \oplus \frn$ denote its direct sum decomposition into diagonal and nilpotent upper-triangular matrices. Then there are short exact sequences of $k[G_{K_v}]$-modules
\[ \xymatrix@1{ 0\ar[r] & \frb \ar[r] & \ad \ar[r] & \ad / \! \frb \ar[r] & 0 } \]
and
\[ \xymatrix@1{ 0 \ar[r] & \frn \ar[r] & \frb \ar[r] & \frt \ar[r] & 0. } \]
Let $\cL_\frb \subset H^1(K_v, \frb)$ denote the pre-image of the unramified subspace of $H^1(K_v, \frt)$. Let $\cL$ denote the image of $\cL_\frb$ inside $H^1(K_v, \ad)$, and let $\cL^1$ denote the pre-image of $\cL$ in $Z^1(K_v, \ad)$ (the space of continuous 1-cocycles). Then the tangent space to $R_v^\text{ord} / (\ffrm_{\Lambda_v})$ is of dimension equal to $\dim_k \cL^1 = \dim_k \cL + 2$. We must therefore show that $\dim_k \cL = 2 + [K_v : \bbQ_p]$.

We fix a basis $e, x, y, f$ of $\ad$ consisting of the matrices
\[ e = \left( \begin{array}{cc} 0 & 1 \\ 0 & 0 \end{array}\right),x = \left( \begin{array}{cc} 1 & 0 \\ 0 & 0 \end{array}\right), y = \left( \begin{array}{cc} 0 &0 \\ 0 & 1 \end{array}\right),f = \left( \begin{array}{cc} 0 & 0 \\1 & 0 \end{array}\right) .\]
The map $H^1(K_v, \frb) \to H^1(K_v, \ad)$ has a 1-dimensional kernel, spanned by the class $\sigma \mapsto \alpha(\sigma)(x-y-\alpha(\sigma)e)$. Since we have assumed that $\alpha$ is ramified, this does not intersect $\cL_\frb$ (in fact, we have already used essentially the same fact in the proof of Proposition \ref{prop_representability_of_ordinary_lifting_problem}). We thus have $\dim_k \cL = \dim_k \cL_\frb$. On the other hand, we have
\[ \dim_k \cL_\frb = -1 + h^1(K_v, \frn) + \dim_k( \ker \delta \cap H^1_\text{ur}(K_v, \frt) ), \]
where $\delta : H^1(K_v, \frt) \to H^2(K_v, \frn)$ is the connecting homomorphism. Using Tate duality, we can identify the dual of $\delta$ with the connecting homomorphism $\delta^\vee : H^0(K_v, \ad / \frb) \to H^1(K_v, \frt)$ attached to the exact sequence
\[ \xymatrix@1{ 0 \ar[r] & \frt \ar[r] & \ad / \frn \ar[r] & \ad / \frb \ar[r] & 0. } \]
We have a formula
\[ \dim_k (\ker \delta \cap H^1_\text{ur}(K_v, \frt)) = \dim_k (\delta^\vee)^{-1}(H^1_\text{ur}(K_v, \frt)^\perp) + \dim_k H^1_\text{ur}(K_v, \frt) - \dim_k H^2(K_v, \frn). \]
Putting these formulae together, and using Tate's Euler characteristic formula, we get an equation
\begin{align*} \dim_k \cL &= -1 + h^1(K_v, \frn) - h^2(K_v, \frn) + \dim_k (\delta^\vee)^{-1}(H^1_\text{ur}(K_v, \frt)^\perp) + \dim_k H^1_\text{ur}(K_v, \frt) \\ 
&= 2 + [K_v : \bbQ_p] + \dim_k (\delta^\vee)^{-1}(H^1_\text{ur}(K_v, \frt)^\perp). \end{align*}
The proof will thus be complete if we can show that $\dim_k (\delta^\vee)^{-1}(H^1_\text{ur}(K_v, \frt)^\perp) = 0$. The 1-dimensional image of $\delta^\vee$ is spanned by the class of the cocycle $\sigma \mapsto \alpha(\sigma)(x-y)$, so it is equivalent to ask that $\alpha \cup \phi \neq 0$, where $\phi : G_{K_v} \to k$ is a non-trivial unramified character, or (using the explicit description of the cup product in terms of the Artin map) that $\phi(\Art_{K_v}(\beta)) \neq 0$. Since $\beta$ has valuation not divisible by $p$, by hypothesis, we see that this is non-trivial. This concludes the proof of the proposition.
\end{proof}

\subsubsection{Level-raising deformation problems}\label{subsubsec_level_raising_deformations}

Let $v$ be a finite place of $K$ such that $q_v \equiv 1 \text{ mod } p$ and $\overline{\rho}|_{G_{K_v}}$ is trivial. Let $\Lambda_v = \cO$. Given characters $\chi_{v, 1}, \chi_{v, 2} : k(v)^\times \to \cO^\times$ which are trivial mod $(\varpi)$, we define $\cD_v^{\chi_v}$ to be the functor of liftings $\rho_R : G_{K_v} \to \GL_2(R)$ of $\overline{\rho}|_{G_{K_v}}$ such that for all $\sigma \in I_{K_v}$, the characteristic polynomial of $\rho_R(\sigma)$ equals
\[ (X - \chi_{v, 1}(\Art_{K_v}^{-1}(\sigma))) (X - \chi_{v, 2}(\Art_{K_v}^{-1}(\sigma))). \]
The following proposition needs no proof.
\begin{proposition}
The functor $\cD_v^{\chi_v}$ is a local deformation problem.
\end{proposition}
We write $R_v^{\chi_v}$ for the representing object. We note that the ring  $R_v^{\chi_v} / (\varpi)$ is canonically independent of $\chi_v = (\chi_{v, 1}, \chi_{v, 2})$ (because the restriction of the functor $\cD_v^{\chi_v}$ to $\CNL_k \subset \CNL_\cO$ is independent of the choice of $\chi_v$).

In order to study the properties of the rings $R_v^{\chi_v}$, it is helpful to introduce as well a finite type model. Let us write $\cM$ for the closed subscheme of $(\GL_2 \times \GL_2)_\cO$ consisting of pairs of matrices $(P, S)$ satisfying the relation $P S = S^{q_v} P$. Fix a generator $\sigma \in k(v)^\times$, and let $\cM^{\chi_v} \subset \cM$ denote the closed subscheme where the characteristic polynomial of $S$ equals
\[  (X - \chi_{v, 1}(\sigma))(X - \chi_{v, 2}(\sigma)). \]
\begin{lemma}\label{lem_ihara_avoidance_ring_1}
Assume $p = 2$. Let $\overline{R}_v^1$ denote the reduced quotient of $R_v^1$. Then:
\begin{enumerate}
\item $\overline{R}_v^1$ is $\cO$-flat and equidimensional of dimension 5 and $\overline{R}_v^1/(\varpi)$ is generically reduced.
\item $R_v^1$ has two distinct minimal primes $Q_1$, $Q_2$. Both quotients $R_v^1 / (Q_1, \varpi)$ and $R_v^1 / (Q_2, \varpi)$ have irreducible spectrum of dimension 4. 
\item We have $\dim R_v^1 / (\varpi, Q_1, Q_2) = 3$. In particular, each minimal prime of $R_v^1$ minimal over $\varpi$ contains a unique minimal prime of $R_v^1$. 
\end{enumerate}
\end{lemma}
\begin{proof}
\cite[Lemma 3.15]{Tho12} shows that $\cM^1$ has two irreducible components $Z_1, Z_2$, where $Z_1$ is defined by the relation $S = 1$, and $Z_2$ is the closure of the open subscheme defined by the relation $S \neq 1$; that these components (with their reduced scheme structure) are $\cO$-flat of dimension 5; and that the reduced subschemes underlying the special fibres $Z_{1, k}$ and $Z_{2, k}$ are the distinct irreducible components of $\cM^1_k$, and are generically reduced. Note that $Z_1$ is smooth over $\cO$, while $Z_2$ is smooth over $\cO$ away from the closed subscheme $Z_{1, k} \cap Z_{2, k} \subset Z_{2, k}$. In particular, if $\overline{\cM}^1$ denotes the nilreduction of $\cM^1$, then $\overline{\cM}^1_k$ is generically reduced.

The ring $R_v^1$ can be identified with the completed local ring of $\cM^1$ at the point $\overline{x} \in \cM^1(k)$ corresponding to the matrices $P = S = 1$. 
Then $\overline{R}_v^1/(\varpi)$ is generically reduced by \cite[\href{https://stacks.math.columbia.edu/tag/033A}{Tag 033A}]{stacks-project}. 
To prove the remainder of the lemma, we need to check that $Z_1, Z_2$, and $Z_{1, k}$ and $Z_{2, k}$ are still irreducible after passage to completion. By \cite[Lemma 2.7]{Tay08} each prime of $R_v^1$ minimal over $\varpi$ is contained in a unique minimal prime, and each minimal prime is contained in a prime minimal over $\varpi$. If we can show that $Z_{1, k}$ and $Z_{2, k}$ are irreducible after passage to completion then so too will be $Z_1$ and $Z_2$.

Since $Z_{1, k}$ is smooth over $k$, it has the desired property. We verify the analogous statement for $Z_{2, k}$ by direct computation. Let $S = \cO[ A, B, C, D, X, Y, Z, W]$. Then $\cM^1_k$ may be identified with spectrum of the ring
\[ S_0 = S / (\varpi, BZ - CY, (A - D)Y - B(X-W), (A-D)Z - C(X-W), X^2 + YZ, W^2 + YZ, (X+W)Y, (X+W)Z)[\det(P)^{-1}] \]
where we write
\[  P = 1 + \left( \begin{array}{cc} A& B \\ C & D \end{array}\right), S = 1 + \left( \begin{array}{cc} X & Y \\ Z & W \end{array}\right). \]
Let $\mathfrak{p}$ denote the prime ideal of $S_0$ corresponding to $Z_{2, k}$. Then $Y, Z \not\in \mathfrak{p}$ (as we can write down points of $Z_{2, k}$ with non-zero $Y$ and $Z$ co-ordinates), so (using that the characteristic of $k$ is 2) we get $X+W, A + D \in \mathfrak{p}$. We see that $Z_{2, k}$ can be identified with the open subscheme of $\Spec R$ obtained by inverting $(A^2 + BC + 1)$, where now
\[  R = k[A, B, C, X, Y, Z] / ( X^2 + YZ ,  BZ + CY). \]
The point $\overline{x}$ corresponds to the ideal $(A, B, C, X, Y, Z)$. By \cite[Corollary 5.5]{Eis95}, the completion of $Z_{2, k}$ at the point $\overline{x}$ will be integral if $R$ is an integral domain. To show this, it suffices to show that $\operatorname{Proj} R$ is integral, and this is easy to check directly. 
\end{proof}

In the statement of the next proposition, we write $c(R)$ for the connectedness dimension of a complete Noetherian local $\cO$-algebra as in \cite[Definition 1.7]{Tho15}. 
It is defined by
\[  c(R) = \inf_{\cC_1, \cC_2} \{ \dim_{Z_1 \in \cC_1, Z_2 \in \cC_2} Z_1 \cap Z_2\}, \]
where the infimum is over all partitions $\cC_1 \sqcup \cC_2$ of the set of irreducible components of $\Spec R$ into two non-empty subsets.

\begin{proposition}\label{prop_connectedness_dimension_of_local_lifting_ring}
	Suppose that $p = 2$. Suppose given a finite set $R = \{ v_1, \dots, v_r \}$ of finite places of $K$ such that for each $v \in R$, $q_v \equiv 1 \text{ mod }p$ and $\overline{\rho}|_{G_{K_v}}$ is trivial. Choose for each $v \in R$ characters $\chi_{v, 1}, \chi_{v, 2} : k(v)^\times \to \cO^\times$ which are trivial mod $(\varpi)$.  Let $A_R = \widehat{\otimes}_{v \in R} R_v^{\chi_v}$, the completed tensor product being over $\cO$. Then $c(A_R / (\varpi)) = 4 |R| - 1$.
\end{proposition}
\begin{proof}
	Since $A_R / (\varpi)$ depends only on the reduction modulo $\varpi$ of the characters $\chi_{v, i}$, we can assume that $\chi_{v, 1} = \chi_{v, 2} = 1$ for each $v \in R$. By Lemma \ref{lem_ihara_avoidance_ring_1} and \cite[Lemma 1.4]{Tho15}, there is a bijection between the minimal primes of $A_R / (\varpi)$ and the set $\{ 1, 2 \}^r$. Suppose given a decomposition $\{ 1, 2 \}^r = S_1  \sqcup S_2$ into two disjoint non-empty subsets. Then there exists $I \in S_1$, $J \in S_2$ such that $I$ and $J$ differ in exactly one entry. If $Z_I, Z_J$ denote the corresponding irreducible components, then we have $\dim Z_I \cap Z_J = 4(|R| - 1) + 3 = 4 |R| - 1$. This implies that $c(A_R / (\varpi)) = 4|R| - 1$.
\end{proof}
\begin{lemma}\label{lem_ihara_avoidance_ring_chi}
Suppose that $\chi_{v, 1} \neq \chi_{v, 2}$. Let $\overline{R}_v^{\chi_v}$ denote the underlying reduced quotient of $R^{\chi_v}_v$. Then $\overline{R}_v^{\chi_v}$ is an $\cO$-flat domain of dimension 5. Moreover, $R_v^{\chi_v}[1/p]$ is formally smooth over $\cO[1/p]$.
\end{lemma}
\begin{proof}
This is part of \cite[Proposition 3.15]{Tho15}.
\end{proof}

\begin{proposition}\label{prop_irreducibility_of_completed_local_lifting_ring}
Suppose given a finite set $R = \{ v_1, \dots, v_r \}$ of finite places of $K$ such that for each $v \in R$, $q_v \equiv 1 \text{ mod }p$ and $\overline{\rho}|_{G_{K_v}}$ is trivial. Choose for each $v \in R$ characters $\chi_{v, 1}, \chi_{v, 2} : k(v)^\times \to \cO^\times$ which are trivial mod $(\varpi)$. Let $A_R = \widehat{\otimes}_\cO R_v^{\chi_v}$. Let $\frp \subset A_R \llbracket X_1, \dots, X_g \rrbracket$ be a prime ideal of dimension 1 and characteristic $p$, and let $A = A_R \llbracket X_1, \dots, X_g \rrbracket / \frp$. Suppose that for each $v \in R$, the induced homomorphism $r_\frp : G_{K_v} \to \GL_2(A)$ is trivial. Let $A'_R = A_R \llbracket X_1, \dots, X_g \rrbracket_\frp$ (localization and completion, cf. \S \ref{subsec_notation}).
\begin{enumerate}
\item Suppose that for each $v \in R$, $\chi_{v, 1} = \chi_{v, 2} = 1$. Then each minimal prime of $A'_R$ is of characteristic 0, and each prime of $A'_R$ minimal over $(\varpi)$ contains a unique minimal prime of $A'_R$.
\item  Suppose that for each $v \in R$, $\chi_{v, 1} \neq \chi_{v, 2}$. Then $A'_R$ contains a unique minimal prime, which is of characteristic 0.
\end{enumerate}
\end{proposition}
\begin{proof}
This is an easier version of \cite[Lemma 3.40]{Tho15}.

First assume that $\chi_{v, 1} = \chi_{v, 2} = 1$ for all $v \in R$. Let $\overline{A}_R$ denote the nilreduction of $A_R$.
By Lemma~\ref{lem_ihara_avoidance_ring_1} and \cite[Lemma~1.4]{Tho15}, each minimal prime of $A_R$ has characteristic $0$, each prime of $A_R$ minimal over $(\varpi)$ contains a unique minimal prime, and $\overline{A}_R/(\varpi)$ is generically reduced. 
The same properties then hold for $\overline{A}_R \llbracket X_1, \ldots, X_g \rrbracket_{(\frp)}$.
Applying \cite[Proposition~1.6]{Tho15} to $\overline{A}_R$, we get the desired properties of $A_R'$.

Now assume that $\chi_{v, 1} \ne \chi_{v, 2}$ for each $v \in R$. 
By Lemma~\ref{lem_ihara_avoidance_ring_chi} and \cite[Lemma~1.4]{Tho15}, $A_R$ has a unique minimal prime, which is of characteristic $0$. 
The same is then true for $A_R \llbracket X_1,\ldots,X_g \rrbracket_{(\frp)}$, and the proof will be finished if we show that $A_R'[1/p]$ is a domain.
The maps $A_R \to A_R\llbracket X_1, \ldots, X_g \rrbracket_{(\frp)}$ and $A_R\llbracket X_1,\ldots, X_g \rrbracket_{(\frp)} \to A_R'$ are both regular, so $A_R \to A_R'$ is regular by \cite[\href{https://stacks.math.columbia.edu/tag/07QI}{Tag 07QI}]{stacks-project}. 
This then implies that $A_R[1/p] \to A_R'[1/p]$ is a regular ring map and since $A_R[1/p]$ is a regular ring (by Lemma~\ref{lem_ihara_avoidance_ring_chi}), \cite[\href{https://stacks.math.columbia.edu/tag/033A}{Tag 033A}]{stacks-project} implies that $A_R'[1/p]$ is also a regular ring. 
Since $r_\frp : G_{K_v} \to \GL_2(A)$ is trivial for each $v \in R$,
the pullback of $\frp$ to $A_R$ is the maximal ideal $\ffrm_{A_R}$. 
So under the isomorphism $A_R\llbracket X_1,\ldots, X_g \rrbracket \cong A_R \widehat{\otimes}_\cO \cO \llbracket X_1, \ldots, X_g \rrbracket$, $\frp$ corresponds to $(\ffrm_{A_R}, \frq) \subset A_R \widehat{\otimes}_\cO \cO \llbracket X_1, \ldots, X_g \rrbracket$ with $\frq \subset \cO \llbracket X_1,\ldots, X_g \rrbracket$ a prime ideal of dimension 1 and characteristic $p$. 
So repeatedly applying \cite[Proposition~3.16]{Tho15}, we see that $\Spec A_R'[1/p]$ is connected. 
Since $A_R'[1/p]$ is regular with connected spectrum, it is a domain.
\end{proof}

\subsubsection{Taylor--Wiles deformation problem}\label{subsubsec_taylor_wiles_deformation_problems}

Let $v$ be a finite place of $K$ and let $N \geq 1$ be an integer such that $q_v \equiv 1 \text{ mod }p^N$ and $\overline{\rho}|_{G_{K_v}} = \alpha_v \oplus \beta_v$ is a direct sum of distinct unramified characters. Let $\Lambda_v = \cO$. Then we define $\Delta_v = (k(v)^\times \times k(v)^\times) / p^{N}$; this group is isomorphic to $(\bbZ / p^{N})^2$.

The following lemma is well-known.
\begin{lemma}
Let $R \in \CNL_\cO$ and $\rho_R \in \cD_v^\square(R)$. Then there are uniquely determined lifts $A_v, B_v : G_{K_v} \to R^\times$ of $\alpha_v, \beta_v$ such that $\rho_R$ is $\widehat{\GL}_2(R)$-conjugate to $A_v \oplus B_v$. 
\end{lemma}
It follows that $R_v^\square$ has a structure of $\cO[(k(v)^\times \times k(v)^\times)(p)]$-algebra, which comes from restricting the pair of characters $A_v \circ \Art_{K_v}$, $B_v \circ \Art_{K_v}$ to $k(v)^\times$. We define $\cD_v^{\text{TW}, N}$ to be the deformation problem which is represented by the quotient ring 
\[ R_v^{\text{TW}, N} = R_v^\square \otimes_{ \cO[(k(v)^\times \times k(v)^\times)(p)] } \cO[\Delta_v]. \]
If $\cS = (K, \overline{\rho}, S, \{ \cD_v \}_{v \in S})$ is a deformation problem,  then we call a Taylor--Wiles datum for $\cS$ a tuple $(Q, N, \{ \alpha_v, \beta_v \}_{v \in Q})$, where:
\begin{enumerate}
	\item $Q$ is a finite set of finite places of $K$.
	\item $N \geq 1$ is an integer. 
	\item $\alpha_v, \beta_v : G_{K_v} \to k^\times$ are continuous characters.
	\end{enumerate}
 We require that the following conditions are satisfied:
\begin{enumerate}
\item $Q \cap S = \emptyset$.
\item For each $v \in Q$, $q_v \equiv 1 \text{ mod }p^N$.
\item For each $v \in Q$, $\overline{\rho}|_{G_{K_v}} \cong \alpha_v \oplus \beta_v$ is a direct sum of distinct unramified characters.
\end{enumerate}
We call $N$ the level of the Taylor--Wiles datum. If $(Q, N, \{ \alpha_v, \beta_v \}_{v \in Q})$ is a Taylor--Wiles datum, then we define the augmented deformation problem
\[ \cS_Q = (K, \overline{\rho}, S \cup Q,  \{ \Lambda_v \}_{v \in S \cup Q}, \{ \cD_v \}_{v \in S} \cup \{ \cD_v^{\text{TW}, N} \}_{v \in Q}). \]
The deformation ring $R_{\cS_Q}$ has a natural structure of $\cO[\Delta_Q]$-algebra (where $\Delta_Q = \prod_{v \in Q} \Delta_v$); by construction, $\Delta_Q$ is a free $\bbZ / p^N \bbZ$-module of rank $2 |Q|$.
\subsection{Geometry of deformation rings}\label{sec_geo_def_rings_p_even}

We now define Selmer groups and dual Selmer groups, and use these to get coarse information about the size of Galois deformation spaces. Suppose given a deformation problem $\cS = (K, \overline{\rho}, S,  \{ \Lambda_v \}_{v \in S}, \{ \cD_v \}_{v \in S})$. Let $T \subset S$, and let $\Lambda_T = \widehat{\otimes}_{v \in T} \Lambda_v$. We let $R_v \in \CNL_{\Lambda_v}$ denote the representing object of $\cD_v$ ($v \in S$), and define $A_\cS^T = \Lambda \widehat{\otimes}_{\Lambda_T} ( \widehat{\otimes}_{v \in T} R_v )$, the completed tensor products being over $\cO$. Then $A_\cS^T \in \CNL_{\Lambda}$. Note that there is a canonical map of $\Lambda$-algebras $A_\cS^T \to R_\cS^T$, corresponding to the natural transformation $( ( \alpha_v )_{v \in T}, \rho ) \mapsto ( \alpha_v \rho \alpha_v^{-1} )_{v \in T}$ at the level of $T$-framed liftings.

Suppose given an object $A \in \CNL_\Lambda$ and a type $\cS$ lifting $\rho_A : G_K \to \GL_2(A)$. 
If $B$ is a finite $A$-module, we let $\ad\rho_B = \ad\rho\otimes_A B$ and
we define a $\CNL_\Lambda$-algebra $A\oplus \epsilon B$ by $\epsilon^2 = 0$. 
Then for each $v \in S - T$, the fibre of $\rho_A$ under $\cD_v(A \oplus \epsilon B) \to \cD_v(A)$ is naturally identified with a $A$-submodule $\cL_v^1(\rho_B) \subset Z^1(K_v, \ad \rho_B)$, and we let $\cL_v(\rho_B)$ be its image in $H^1(K_v,\ad\rho_B)$. 
In the case that $B$ is a finite $A$-algebra, these $A$-modules inherit a compatible $B$-module structure.

We define the groups $H^i_{\cS, T}(\ad \rho_B)$ to be the cohomology groups of the complex $C_{\cS, T}(\ad \rho_B)$ which is determined by the formula (cf. \cite[\S 5.3]{Tho16}; the differentials are defined as in \emph{loc. cit.}):
\[ C^i_{\cS, T}(\ad\rho_B) = \left\{ \begin{array}{ll} C^0(K_S / K, \ad \rho_B) & i = 0;   \\
C^1(K_S / K, \ad \rho_B) \oplus \bigoplus_{v \in T} C^0(K_v, \ad \rho_B) & i = 1 ; \\
C^2(K_S / K, \ad \rho_B) \oplus \bigoplus_{v \in T} C^1(K_v, \ad \rho_B) \oplus \bigoplus_{v \in S - T} C^1(K_v, \ad \rho_B) / \cL_v^1(\rho_B) & i = 2; \\
C^i(K_S / K, \ad \rho_B) \oplus \bigoplus_{v \in S} C^{i-1}(K_v, \ad \rho_B) & i > 2 .
\end{array}\right. \]
Here $C^i$ denotes the group of continuous inhomogeneous cochains with values in $\ad\rho_B$. 
Note that if $B'$ is another finite $A$-module, an $A$-module map $B \to B'$ induces $A$-module 
maps $H_{\cS,T}^i(\ad\rho_B) \to H_{\cS,T}^i(\ad\rho_{B'})$ for each $i$.

\begin{proposition}\label{prop:tangent_space}
Let $A$ be a $\CNL_\Lambda$-algebra and let $\rho_A$ be a type $\cS$ lifting. 
Let $B$ be a finite $A$-module (resp. a finite $A$-algebra).
Consider the $\CNL_\Lambda$-morphism $R_\cS^T \to A$ given by $(\rho_A, (1)_{v\in T})$. 
Assume that the composite $A_\cS^T \to R_S^T \to A$ is surjective. 
Let $\mathfrak{r} = \ker(R_S^T \to A)$ and $\mathfrak{a} = \ker(A_S^T \to A)$. 
There is a canonical isomorphism of $A$-modules (resp. of $B$-modules)
	\[ \Hom_A(\mathfrak{r}/(\mathfrak{r}^2,\mathfrak{a}), B) \cong H_{\cS,T}^1(\ad\rho_B). \]
\end{proposition}

\begin{proof}
Give $A\oplus \epsilon B$ the $A_\cS^T$-algebra structure by 
$A_\cS^T \to A \hookrightarrow A \oplus \epsilon B$.
Let $\cL^1$ be the set of
$T$-framed liftings $(\rho,(\alpha_v)_{v\in T})$ of $\overline{\rho}$ to $A\oplus \epsilon B$ 
with $(\rho,(\alpha_v)_{v\in T}) \bmod \epsilon = (\rho_A,(1)_{v\in T})$, and let $\cL$ denote the set of $T$-framed deformations containing an element of $\cL^1$. 
Then $\cL$ is identified with the set of $A_\cS^T$-algebra homomorphisms $f \colon R_\cS^T \rightarrow A \oplus \epsilon B$
such that $f \bmod \epsilon$ is the map $R_\cS^T \to A$ given by $(\rho_A,(1)_{v\in T})$.
The assumption that $A_\cS^T \to A$ is surjective implies that we have a decomposition of $A$-algebras $R_\cS^T/(\mathfrak{r}^2,\mathfrak{a}) \cong A \oplus \mathfrak{r}/(\mathfrak{r}^2,\mathfrak{a})$.
It follows that the restriction $f \mapsto f|_{\mathfrak{r}}$ determines a bijection $\cL \to \Hom_A(\mathfrak{r}/(\mathfrak{r}^2,\mathfrak{a}), B)$.

Any $T$-framed lifting in $\cL^1$ can be written uniquely as $((1 +\epsilon \phi)\rho_A, (1 + \epsilon \beta_v)_{v \in T})$, 
with $\phi \in Z^1(K_S/K, \ad \rho_B)$ and $\beta_v \in \ad \rho_B$. 
The condition that this lifting lies in $\cD_v(A \oplus \epsilon B)$ for $v \in S - T$ 
is equivalent to $\phi|_{G_{K_v}} \in \cL_v^1(\rho_B)$.
The condition that it give the trivial lift at $v \in T$ is equivalent to the condition
  \[ 
  (1 - \epsilon \beta_v) ((1 + \epsilon \phi)\rho_A)|_{G_{K_v}} (1 + \epsilon \beta_v) = \rho_A|_{G_{K_v}}. 
  \]
Putting it another way, the elements of $\cL$ are in bijection with the tuples $(\phi, (\beta_v)_{v \in T})$, 
where $\phi \in Z^1(K_S/K, \ad \rho_B)$ is such that $\phi|_{G_{K_v}} \in \cL_v^1(\rho_B)$ for each $v\in S - T$, 
and $\beta_v \in \ad \rho_B$ is such that for each $v \in T$ we have the equality
\[ \phi|_{G_{K_v}} = (\ad \rho_B|_{G_{K_v}} - 1) \beta_v. \]
Two tuples $(\phi, (\beta_v)_{v \in T})$ and $(\phi', (\beta'_v)_{v \in T})$ give rise to strictly equivalent $T$-framed liftings if and only if there exists $X \in \ad \rho_B$ satisfying
\begin{gather*}
\phi' = \phi + (1 - \ad \rho_B)X, \\
\beta'_v = \beta_v + X
\end{gather*}
for each $v \in T$. 
We then have bijections
  \[
   \Hom_A(\mathfrak{r}/(\mathfrak{r}^2,\mathfrak{a}), B) \cong \cL \cong H_{\cS,T}^1(\ad\rho_B),
  \]
and chasing through the description of these bijections, the $A$-module (resp. $B$-module) structures coincide.
\end{proof}

A special case of the above is $A = B = k$, which gives the usual cohomological description of the relative tangent space.
In this case, the trace pairing $(X,Y) \to \tr(XY)$ on $\ad\overline{\rho}$ is perfect, so for each $v\in S$, we can define $\cL_v(\overline{\rho})^\perp \subset H^1(K_v, \ad\overline{\rho}(1))$ to be the orthogonal complement of $\cL_v(\overline{\rho})$ with respect to the Tate duality pairing.
We then define a group $H^1_{\cS^\perp, T}(\ad \overline{\rho}(1))$ by the formula
\[ H^1_{\cS^\perp, T}(\ad\overline{\rho}(1)) = \ker \left( H^1(K_S / K, \ad \overline{\rho}(1)) \to \prod_{v \in S - T} H^1(K_v, \ad \overline{\rho}(1)) / \cL_v(\overline{\rho})^\perp \right). \]
The following result is proved exactly as in \cite[\S 5.3]{Tho16}, using the fundamental duality theorems in local and global Galois cohomology. (We use $h^i$ as a shorthand for $\dim_k H^i$.)
\begin{proposition}\label{prop_presentation_by_generators_and_relations}
\begin{enumerate}
\item There is a surjective map $A_\cS^T \llbracket X_1, \dots, X_g \rrbracket \to R_\cS^T$ of $A_\cS^T$-algebras, where $g = h^1_{\cS, T}(\ad \overline{\rho})$. If $R_v$ is a formally smooth $\Lambda_v$-algebra for each $v \in S - T$, then the kernel of this map can be generated by $r$ elements, where $r = h^1_{\cS^\perp, T}(\ad \overline{\rho}(1))$.
\item If $v \in S$, let $\ell_v = \dim_k \cL_v(\overline{\rho})$. Let $\delta_T = 1$ if $T$ is empty, and $\delta_T = 0$ otherwise. Then we have the formula
\[ h^1_{\cS, T}(\ad \overline{\rho}) = h^1_{\cS^\perp, T}(\ad \overline{\rho}(1)) - h^0(K, \ad \overline{\rho}(1)) - \sum_{v | \infty} h^0(K_v, \ad \overline{\rho}) + \sum_{v \in S - T} (\ell_v - h^0(K_v, \ad \overline{\rho})) + \delta_T. \]
\end{enumerate}
\end{proposition}

\begin{corollary}\label{cor_connectedness_dimension_Ihara_case}
Suppose that the following conditions are satisfied:
\begin{enumerate}
\item $p = 2$ and $i \in K$ (in particular, $K$ is totally imaginary). \item For each $v \in S_p$, $\cD_v = \cD_v^\text{ord}$ is of the type considered in \S \ref{subsubsec_ordinary_deformation_problems}.
\item For each $v \in S - S_p$, $\cD_v = \cD_v^{\chi_v}$ for some $\chi_v = (\chi_{v, 1}, \chi_{v, 2})$, as in \S \ref{subsubsec_level_raising_deformations}. Write $R = S - S_p$.
\end{enumerate}
Then:
\begin{enumerate} \item For any irreducible component $Z$ of $\Spec R_{\cS} / (\varpi)$, we have $\dim Z \geq [K : \bbQ]$.
	\item For any irreducible component $Z$ of $\Spec R_{\cS}$, we have $\dim Z \geq [K : \bbQ] + 1$.
\item Let $\cC$ denote the set of irreducible components of $\Spec R_\cS / (\varpi)$. Then For any partition $\cC = \cC_1 \sqcup \cC_2$ into two disjoint non-empty subsets, we can find $Z_1 \in \cC_1$, $Z_2 \in \cC_2$ such that $\dim Z_1 \cap Z_2 \geq [K : \bbQ] - 2$. 
\end{enumerate}
\end{corollary}
\begin{proof}
Let $T = S$. Then each minimal prime of $A_\cS^T / (\varpi)$ has dimension $4|T| + 3 [K : \bbQ]$. According to Proposition \ref{prop_presentation_by_generators_and_relations}, there is an isomorphism $R^T_{\cS} \cong A^T_{\cS} \llbracket X_1, \dots, X_g \rrbracket / (f_1, \dots, f_r)$, where $g - r = -1 - 2[K : \bbQ]$. It follows that each minimal prime of $R^T_{\cS} / (\varpi)$ has dimension at least $4|T| - 1 + [K : \bbQ]$ and each minimal prime of $R^T_\cS$ has dimension at least $4|T| + [K : \bbQ]$. Since $R^T_{\cS} \cong R_\cS \llbracket Y_1, \dots, Y_{4|T| - 1} \rrbracket$, this implies the first two assertions. For the third, we note that by Proposition \ref{prop_connectedness_dimension_of_local_lifting_ring}, we have $c( A_\cS^T / (\varpi) ) =  4|T| + 3 [K : \bbQ] - 1$. It then follows from \cite[Theorem 2.4]{Bro86} that $c( R_{\cS}^T / (\varpi)) \geq 4|T| - 1 + [K : \bbQ] - 2$, hence $c( R_{\cS} / (\varpi) ) \geq [K : \bbQ] - 2$.
\end{proof}
\begin{remark}
The lower bound on the dimension of $\Spec R_\cS$ appearing in Corollary \ref{cor_connectedness_dimension_Ihara_case} is not expected to be sharp, but rather 1 less than the sharp bound. This is due to the presence of the non-zero term $h^0(K, \ad \overline{\rho}(1))$.
\end{remark}

\subsection{Selmer groups for truncated discrete valuation rings}\label{subsec:dvr_selmer}

Throughout this subsection, we suppose given a deformation problem 
$\cS = (K, \overline{\rho}, S,  \{ \Lambda_v \}_{v \in S}, \{ \cD_v \}_{v \in S})$ 
satisfying the following:
\begin{itemize}
\item For each $v \in S_p$, $\cD_v = \cD_v^\text{ord}$ is of the type considered in \S \ref{subsubsec_ordinary_deformation_problems}.
\item If $v \in S - S_p$, then $\Lambda_v = \cO$.
\item $\zeta_p \in K$ and $i \in K$ if $p = 2$.
\end{itemize}
We set $T = S - S_p$.

We also suppose given a discrete valuation ring $A \in \CNL_\Lambda$ and a type $\cS$ lifting $\rho_A : G_K \to \GL_2(A)$. 
We fix nonzero $t \in \ffrm_A$ with $p \in tA$.  
For each $N \ge 1$, let $\rho_{A/t^N} = \rho_A \otimes_A A/t^N$. 
The trace pairing $(X,Y) \mapsto \tr(XY)$ is perfect on $\ad\rho_{A/t^N}$, so for each $v\in S_p$, 
we can let $\cL_v(\rho_{A/t^N})^\perp \subset H^1(K_v, \ad\rho_{A/t^N}(1))$ be the orthogonal complement of $\cL_v(\rho_{A/t^N})$ 
with respect to the Tate duality pairing. 
We can then define
\[ H^1_{\cS^\perp, T}(\ad\rho_{A/t^N}(1)) = \ker \left( H^1(K_S / K, \ad \rho_{A/t^N}(1)) \to \prod_{v \in S_p} H^1(K_v, \ad\rho_{A/t^N}(1)) / \cL_v(\rho_{A/t^N})^\perp \right). \]
To ease notation in some of what follows, we set $\cL_v(\rho_{A/t^N}) = 0$ for $v\in T$.

The purpose of this subsection is to estimate the size of a certain cohomology group 
$H_{\cS_{Q_N},T}^1(\ad\rho_{A/t^N})$ assuming that the corresponding dual Selmer group 
$H_{\cS_{Q_N}^\perp,T}(\ad\rho_{A/t^N}(1))$ is sufficiently small (Proposition~\ref{prop:selmer_with_TW_primes} below).

\begin{lemma}\label{lem:local_mod_N}
Let $v\in S_p$. 
\begin{enumerate}
\item\label{local_mod_N:fin_alg} For any finite $A$-algebra $B$, $\cL_v^1(\rho_B) \cong B^{4+[K_v:\bbQ_p]}$.
\item\label{local_mod_N:mod_N} For any $N\ge 1$, the natural map $\cL_v(\rho_A)/t^N \to \cL_v(\rho_{A/t^N})$ is an isomorphism. 
\end{enumerate}
\end{lemma}

\begin{proof}
Let $f : R_v^\mathrm{ord} \to A$ be the $\CNL_{\Lambda_v}$-algebra morphism corresponding to $\rho_A|_{G_{K_v}}$. 
For any finite $A$-algebra $B$, the $B$-module $\cL_v^1(\rho_B)$ equals the fibre of $f$ under $\Hom_{\CNL_{\Lambda_v}}(R_v^\mathrm{ord}, A \oplus \epsilon B) \to \Hom_{\CNL_{\Lambda_v}}(R_v^\mathrm{ord}, A)$. 
By Proposition~\ref{prop_smoothness_of_ordinary_lifting_problem}, $R_v^{\mathrm{ord}} \cong \Lambda_v \llbracket x_1,\ldots,x_{4+[K_v:\bbQ_p]} \rrbracket$, so $\cL_v^1(\rho_B) \cong B^{4+[K_v:\bbQ_p]}$.

Part~\ref{local_mod_N:fin_alg} then implies there is a commutative diagram
	\[ \xymatrix@1{ 
    (\ad\rho_A)/t^N \ar[r] \ar@{=}[d] & \cL_v^1(\rho_A)/t^N \ar[r] \ar@{=}[d] & \cL_v(\rho_A)/t^N \ar[r] \ar[d] & 0 \\
    \ad\rho_{A/t^N} \ar[r] & \cL_v^1(\rho_{A/t^N}) \ar[r] & \cL_v(\rho_{A/t^N}) \ar[r] & 0,
    }\]
which implies part~\ref{local_mod_N:mod_N}.
\end{proof}

Let $F$ denote the fraction field of $A$ and set $\ad\rho_{F/A} = \ad\rho_F/\ad\rho_A$. 
It is a divisible $A$-module with continuous $G_K$-action. 
The isomorphisms $\ad\rho_{F/A}[t^N] \cong \ad\rho_{A/t^N}$ give an isomorphism $H^i(K_S/K, \ad\rho_{F/A}) \cong 
\varinjlim_N H^i(K_S/K, \ad\rho_{A/t^N})$ and similarly for the local cohomology groups and with $\ad\rho_{F/A}(1)$-coefficients. 
We then define $\cL_v(\rho_{F/A}) = \varinjlim_N \cL_v(\rho_{A/t^N}) \subset H^1(K_v, \ad\rho_{F/A})$ for each $v\in S$,
	\[ H_{\cS, T}^1(\ad\rho_{F/A}) = \varinjlim_N H_{\cS, T}^1(\ad\rho_{A/t^N}), \]
and
	\[ H_{\cS^\perp, T}^1(\ad\rho_{F/A}(1)) = \varinjlim_N H_{\cS, T}^1(\ad\rho_{A/t^N}(1)). \]
Tate duality determines a perfect duality (for any $v \in S$):
	\[ H^1(K_v, \ad \rho_{A}) \times H^1(K_v, \ad \rho_{F / A}(1)) \to F / A, \]
	and we write $\cL_v(\rho_{A})^\perp \subset H^1(K_v, \ad \rho_{F / A}(1))$ for the orthogonal complement of $\cL_v(\rho_A)$. Then $\cL_v(\rho_A)^\perp = \varinjlim_N \cL_v(\rho_{A / t^N})^\perp$ and there is an isomorphism
	\[ \cL_v(\rho_A)^\vee \cong H^1(K_v, \ad\rho_{F/A}(1))/\cL_v(\rho_{A})^\perp. \]
There are exact sequences
	\[ \xymatrix@1{ 
\prod_{v\in T} H^0(K_v, \ad\rho_{F/A}) \ar[r] & H_{\cS, T}^1(\ad\rho_{F/A}) \ar[r] & H^1(K_S/K, \ad\rho_{F/A}) \ar[r] & 
\prod_{v \in S} \frac{H^1(K_v, \ad\rho_{F/A})}{\cL_v(\rho_{F/A})} }\]
and
	\[ \xymatrix@1{ 0 \ar[r] & H_{\cS^\perp, T}^1(\ad\rho_{F/A}(1)) \ar[r] & H^1(K_S/K, \ad\rho_{F/A}(1)) 
    \ar[r] & \prod_{v \in S_p} \frac{H^1(K_v, \ad\rho_{F/A}(1))}{\cL_v(\rho_{A})^\perp}. } \]
In particular, both $H_{\cS, T}^1(\ad\rho_{F/A})$ and $H_{\cS^\perp, T}^1(\ad\rho_{F/A}(1))$ 
are cofinitely generated.
\begin{lemma}\label{lem:H1_and_local_torsion}
\begin{enumerate}
 \item\label{H1_and_local_torsion:H1} 
 For any $N\ge 1$, the natural map 
 \[ H^1(K_S/K, \ad\rho_{A/t^N}) \to H^1(K_S/K, \ad\rho_{F/A})[t^N] \] 
 is an isomorphism.
 \item\label{H1_and_local_torsion:local} 
 For any $v\in S$, there is a nonnegative integer $b_v$, depending only on $\rho_A|_{G_{K_v}}$, 
 such that for any $N\ge 1$, the natural map 
 \[ \frac{H^1(K_v, \ad\rho_{A/t^N})}{\cL_v(\rho_{A/t^N})} \to \frac{H^1(K_v, \ad\rho_{F/A})}{\cL_v(\rho_{F/A})}[t^N] \]
has kernel annihilated by $t^{b_v}$.
\end{enumerate}
\end{lemma}

\begin{proof}
Applying cohomology to the exact sequence
  \[ \xymatrix@1{ 0 \ar[r] & \ad\rho_{A/t^N} \ar[r] & \ad\rho_{F/A} \ar[r]^{t^N} & \ad\rho_{F/A} \ar[r] & 0 } \]
we obtain a surjection $H^1(K_S/K, \ad\rho_{A/t^N}) \to H^1(K_S/K, \ad\rho_{F/A})[t^N]$
whose kernel is the cokernel of multiplication by $t^N$ on $H^0(K_S/K, \ad\rho_{F/A})$. 
Since $\overline{\rho}$ is absolutely irreducible, $H^0(K_S/K, \ad\rho_{F/A}) = F/A$, the submodule of diagonal matrices, 
on which multiplication by $t^N$ is surjective. This proves part~\ref{H1_and_local_torsion:H1}.

Now fix some $v\in S$.
We similarly have a surjection $H^1(K_v, \ad\rho_{A/t^N}) \to H^1(K_v, \ad\rho_{F/A})[t^N]$ whose kernel is the 
cokernel of multiplication by $t^N$ on $H^0(K_v, \ad\rho_{F/A})$. 
This $A$-module is cofinitely generated and we let $b_v$ be such that $t^{b_v}$ annihilates 
the quotient of $H^0(K_v, \ad\rho_{F/A})$ by its maximal divisible submodule. 
Then $t^{b_v}$ also annihilates the cokernel of multiplication by $t^N$ on $H^0(K_v, \ad\rho_{F/A})$.
If $v \notin S_p$, then part~\ref{H1_and_local_torsion:local} holds since $\cL_v = 0$ in this case. 
Now assume that $v\in S_p$, and 
set $\cL_v^1(\rho_{F/A}) = \varinjlim_N \cL_v^1(\rho_{A/t^N})$. 
It follows from Lemma~\ref{lem:local_mod_N} that the natural map $\cL_v^1(\rho_{A/t^N}) \to \cL_v^1(\rho_{F/A})[t^N]$ 
is an isomorphism. 
Take $\gamma \in \cL_v(\rho_{F/A})[t^N]$, and choose $\tilde{\gamma} \in \cL_v^1(\rho_{F/A})$ lifting $\gamma$. 
Then $t^N\tilde{\gamma} = dX$ with $X \in \ad\rho_{F/A}$. 
Since $\ad\rho_{F/A}$ is divisible, there is $Y \in \ad\rho_{F/A}$ with $t^N Y = X$. 
Then $\tilde{\gamma} - dY$ is another lift of $\gamma$ and lies in $\cL_v^1(\rho_{F/A})[t^N] \cong \cL_v^1(\rho_{A/t^N})$. 
Its image in $\cL_v(\rho_{A/t^N})$ then maps to $\gamma$, and $\cL_v(\rho_{A/t^N}) \to \cL_v(\rho_{F/A})[t^N]$ is surjective.
We can then apply the snake lemma to the commutative diagram
  \[ 
  \xymatrix@1{
  0 \ar[r] & \cL_v(\rho_{A/t^N}) \ar[r] \ar[d] & H^1(K_v, \ad\rho_{A/t^N}) \ar[r] \ar[d] & 
  \frac{H^1(K_v, \ad\rho_{A/t^N})}{\cL_v(\rho_{A/t^N})} \ar[r] \ar[d] & 0 \\
  0 \ar[r] & \cL_v(\rho_{F/A})[t^N] \ar[r] & H^1(K_v, \ad\rho_{F/A})[t^N] \ar[r] 
  & \frac{H^1(K_v, \ad\rho_{F/A})}{\cL_v(\rho_{F/A})}[t^N].  &
  }
  \]
\end{proof}

\begin{lemma}\label{lem:dual_H1_and_local_torsion}
\begin{enumerate}
 \item\label{dual_H1_and_local_torion:H1} 
 For any $N \geq 1$, the natural map 
 \[ H^1(K_S/K, \ad\rho_{A/t^N}(1)) \to H^1(K_S/K, \ad\rho_{F/A}(1))[t^N]\]
 is an isomorphism.
 \item For any $v\in S_p$ and $N \ge 1$, 
the natural map 
 \[ \frac{H^1(K_v, \ad\rho_{A/t^N}(1))}{\cL_v(\rho_{A/t^N})^\perp} \to 
 \frac{H^1(K_v, \ad\rho_{F/A}(1))}{\cL_v(\rho_{A})^\perp}[t^N] \] 
 is an isomorphism.
\end{enumerate}
\end{lemma}

\begin{proof}
Since $\zeta_p \in K$, we have $\ad \overline{\rho} = \ad \overline{\rho}(1)$. The first part follows in the same way as the first part of Lemma~\ref{lem:H1_and_local_torsion}. The map in part 2 is Pontryagin dual to the map $\cL_v(\rho_A)/t^N \to \cL_v(\rho_{A/t^N})$, 
which is an isomorphism by Lemma~\ref{lem:local_mod_N}. 
\end{proof}

\begin{lemma}\label{lem:selmer_torsion}
Let $N \ge 1$ be an integer.
\begin{enumerate}
 \item\label{selmer_torsion:selmer} There is a nonnegative integer $b$, depending only on $\cS$ and $\rho_A$ and not on $N$, such that $H_{\cS, T}^1(\ad\rho_{A/t^N}) \to H_{\cS, T}^1(\ad\rho_{F/A})[t^N]$ is injective with cokernel annihilated by $t^b$.
 \item\label{selmer_torsion:dual_selmer} $H_{\cS^\perp, T}^1(\ad\rho_{A/t^N}(1)) \to H_{\cS^\perp, T}^1(\ad\rho_{F/A}(1))[t^N]$ is an isomorphism. 
\end{enumerate}
\end{lemma}

\begin{proof}
We have a commutative diagram with exact rows
  \[ \xymatrix@1{ 
  0 \ar[r] & H^1_{\cS^\perp, T}(K_S/K, \ad\rho_{A/t^N}(1)) \ar[r] \ar[d] & H^1(K_S/K, \ad\rho_{A/t^N}(1)) \ar[r] \ar[d] & 
  \prod_{v\in S_p} \frac{H^1(K_v, \ad\rho_{A/t^N}(1))}{\cL_v(\rho_{A/t^N})^\perp} \ar[d] \\
  0 \ar[r] & H^1_{\cS^\perp, T}(K_S/K, \ad\rho_{F/A}(1))[t^N] \ar[r] & H^1(K_S/K, \ad\rho_{F/A}(1))[t^N] \ar[r] & 
  \prod_{v\in S_p} \frac{H^1(K_v, \ad\rho_{F/A}(1))}{\cL_v(\rho_{A})^\perp}[t^N]. 
   } \]
Part~\ref{selmer_torsion:dual_selmer} of the lemma then follows from Lemma~\ref{lem:dual_H1_and_local_torsion} and a diagram chase. 

We now turn to part~\ref{selmer_torsion:selmer}. Let
  \[
   H_{\cS, T}^1(K_S/K, \ad\rho_{A/t^N}) = \ker \left(H^1(K_S/K, \ad\rho_{A/t^N}) \to \prod_{v\in S} H^1(K_v, \ad\rho_{A/t^N})/\cL_v(\rho_{A/t^N}) \right)
  \]
and similarly for $H_{\cS, T}^1(K_S/K, \ad\rho_{F/A})$. 
Let $b_v$, $v\in S$, be as in part~\ref{H1_and_local_torsion:local} of 
Lemma~\ref{lem:H1_and_local_torsion} and set $b = \max_{v\in S}\{b_v\}$. 
There is a commutative diagram with exact rows
  \[ \xymatrix@1{ 
  0 \ar[r] & H^1_{\cS, T}(K_S/K, \ad\rho_{A/t^N}) \ar[r] \ar[d] & H^1(K_S/K, \ad\rho_{A/t^N}) \ar[r] \ar[d] & 
  \prod_{v\in S} \frac{H^1(K_v, \ad\rho_{A/t^N})}{\cL_v(\rho_{A/t^N})} \ar[d] \\
  0 \ar[r] & H^1_{\cS, T}(K_S/K, \ad\rho_{F/A})[t^N] \ar[r] & H^1(K_S/K, \ad\rho_{F/A})[t^N] \ar[r] & 
  \prod_{v\in S} \frac{H^1(K_v, \ad\rho_{F/A})}{\cL_v(\rho_{F/A})}[t^N].
   } \]
Lemma~\ref{lem:H1_and_local_torsion} and a diagram chase shows that 
$H_{\cS, T}^1(K_S/K, \ad\rho_{A/t^N}) \to H_{\cS, T}^1(K_S/K, \ad\rho_{F/A})[t^N]$ 
is injective with cokernel annihilated by $t^b$. 
If $T = \emptyset$, we are finished. 
Otherwise, note that $H^0(K, \ad\rho_{A/t^N}) = A/t^N$, the submodule of scalar matrices,
is a direct summand of $\prod_{v\in T} H^0(K_v, \ad\rho_{A/t^N})$, 
and similarly for $\ad\rho_{F/A}$. 
From this, it follows that we have equalities
	\[ \frac{\prod_{v\in T} H^0(K_v, \ad\rho_{F/A})}{H^0(K, \ad\rho_{F/A})}[t^N] = 
    \frac{\prod_{v\in T} H^0(K_v, \ad\rho_{F/A})[t^N]}{H^0(K, \ad\rho_{F/A})[t^N]} = 
    \frac{\prod_{v\in T} H^0(K_v, \ad\rho_{A/t^N})}{H^0(K, \ad\rho_{A/t^N})}.\]
Part 1 then follows from from the commutative diagram with exact rows
	\[ \xymatrix@1{ 
    0 \ar[r] & \frac{\prod_{v\in T} H^0(K_v, \ad\rho_{A/t^N})}{H^0(K, \ad\rho_{A/t^N})} \ar[r] \ar@{=}[d] 
    & H_{\cS, T}^1(\ad\rho_{A/t^N}) \ar[r] \ar[d] & H_{\cS, T}^1(K_S/K, \ad\rho_{A/t^N}) \ar[r] \ar[d] & 0 \\
    0 \ar[r] & \frac{\prod_{v\in T} H^0(K_v, \ad\rho_{F/A})}{H^0(K, \ad\rho_{F/A})}[t^N] \ar[r]  
    & H_{\cS, T}^1(\ad\rho_{F/A})[t^N] \ar[r] & H_{\cS, T}^1(K_S/K, \ad\rho_{F/A})[t^N].  & 
    }\]
\end{proof}

\begin{lemma}\label{lem:selmer_size}
\begin{enumerate}
\item\label{selmer_size:selmer} Let $h = \operatorname{corank}_A H_{\cS,T}^1(\ad\rho_{F/A})$. 
There is a nonnegative integer $b$, depending only on $\cS$ and $\rho_A$,
such that for each $N \ge 1$, there is an $A$-module map $H_{\cS,T}^1(\ad\rho_{A/t^N}) \to (A/t^N)^h$
with kernel and cokernel annihilated by $t^b$.
\item\label{selmer_size:dual_selmer} Let $q = \operatorname{corank}_A H_{\cS^\perp, T}^1(\ad\rho_{F/A}(1))$. 
There is a nonnegative integer $c$, depending only on $\cS$ and $\rho_A$, 
such that for each $N\ge 1$, there an injective $A$-module map $(A/t^N)^q \to H_{\cS^\perp,T}^1(\ad\rho_{A/t^N}(1))$
with cokernel annihilated by $t^c$.
\end{enumerate}
\end{lemma}

\begin{proof}
Let $c \ge 0$ be such that $t^c$ annihilates the quotient of $H_{\cS^\perp, T}^1(\ad_{F/A}(1))$ by its maximal divisible submodule. 
Part~\ref{selmer_size:dual_selmer} of the lemma holds with this choice of $c$ by part~\ref{selmer_torsion:dual_selmer} of Lemma~\ref{lem:selmer_torsion}.
Let $b_1 \ge 0$ be such that $t^{b_1}$ annihilates the quotient of 
$H^1_{\cS, T}(\ad\rho_{F/A})$ by its maximal divisible submodule. 
For each $N \ge 1$, we can find a surjective $A/t^N$-module map 
$H^1_{\cS,T}(\ad\rho_{F/A})[t^N] \to (A/t^N)^h$ with kernel annihilated by $t^{b_1}$.
By part~\ref{selmer_size:selmer} of Lemma~\ref{lem:selmer_torsion}, there is $b_2 \ge 0$ depending only on $\cS$ and $\rho_A$, 
such that for each $N\ge 1$, $H_{\cS,T}^1(\ad\rho_{A/t^N}) \to H_{\cS,T}^1(\ad\rho_{F/A})[t^N]$ 
is injective with cokernel annihilated by $t^{b_2}$, 
Part~\ref{selmer_size:selmer} of the lemma then with $b = b_1 + b_2$.
\end{proof}

\begin{lemma}\label{lem:greenberg_wiles}
Set $\delta_T = 1$ if $T$ is empty and $\delta_T = 0$ otherwise. 
Let $\delta_A = 1$ if $A$ has characteristic $p$, and $\delta_A = 0$ if $A$ has characteristic $0$.
For $h = \operatorname{corank} H_{\cS,T}^1(\ad\rho_{F/A})$ and 
$q = \operatorname{corank} H_{\cS^\perp,T}^1(\ad\rho_{F/A}(1))$, 
we have $q - h = \delta_A + [K:\bbQ] - \delta_T$.
\end{lemma}

\begin{proof}
By Lemma~\ref{lem:local_mod_N}, for any $N \ge 1$ and $v\in S_p$, 
	\[ 
    \frac{\lvert \cL_v(\rho_{A/t^N}) \rvert}{\lvert H^0(K_v, \ad\rho_{A/t^N}) \rvert} = 
    \frac{\lvert \cL_v^1(\rho_{A/t^N}) \rvert}{\lvert \ad\rho_{A/t^N} \rvert} = \lvert A/t^N \rvert^{[K_v:\bbQ_p]}.
    \]
Using the fundamental duality theorems in local and global Galois cohomology, 
as in \cite[\S 5.3]{Tho16}, we then have
	\[ \frac{\lvert H_{\cS, T}^1(\ad\rho_{A/t^N})\rvert}{\lvert H_{\cS^\perp, T}^1(\ad\rho_{A/t^N}(1))\rvert} = 
    \frac{\lvert A/t^N \rvert^{\delta_T - [K:\bbQ]}}{\lvert H^0(K, \ad\rho_{A/t^N}(1)) \rvert}
    \]
for all $N \ge 1$. 
Applying Lemma~\ref{lem:selmer_size}, it only remains to show that 
$\operatorname{corank}_A H^0(K_S/K, \ad\rho_{F/A}(1)) = \delta_A$.

If $A$ has characteristic $p$, then $H^0(K, \ad\rho_{F/A}(1)) = H^0(K, \ad\rho_{F/A}) \cong A/t^N$ is  
the submodule of scalar matrices, since $\overline{\rho}$ is absolutely irreducible. 
If $A$ has characteristic $0$, then we claim that $H^0(K, \ad\rho_{A/t^N}(1))$ has cardinality 
bounded independently of $N$. 
Since $F$ has characteristic $0$, $\epsilon^2 \ne 1$. 
This implies that there is no nontrivial map between the absolutely irreducible representations $\rho_F$ and $\rho_F(1)$, so $H^0(K, \ad\rho_A(1)) = 0$. 

Applying cohomology to 
	\[ \xymatrix@1{ 0 \ar[r] & \ad\rho_A(1) \ar[r]^{t^N} & \ad\rho_A(1) \ar[r] & \ad\rho_{A/t^N}(1) \ar[r] & 0, }\]
we see that $H^0(K, \ad\rho_{A/t^N}(1))$ injects into the torsion submodule of 
$H^1(K_S/K, \ad\rho_A(1))$, which has cardinality bounded independently of $N$.
\end{proof}

\begin{lemma}\label{lem:one_TW_prime}
Let $v$ be a prime of $K$ at which $\rho_A$ is unramified and such that $\overline{\rho}(\Frob_v)$ has distinct $k$-rational eigenvalues. 
Then for any quotient $B$ of $A$ (e.g. $B = A/t^N$ or $B = A$), the $B$-module 
$(\ad\rho_B)^{G_{F_v}}$ is free of rank $2$, is a direct summand of $\ad\rho_B$, 
and the map on cohomology
	\[ \Hom(\Gal(K_v^\mathrm{ur}/K), (\ad\rho_B)^{G_{F_v}}) = H^1(K_v^\mathrm{ur}/K_v, (\ad\rho_B)^{G_{F_v}}) \to H^1(K_v^\mathrm{ur}/K_v, \ad\rho_B) \]
is an isomorphism.
\end{lemma}

\begin{proof}
This is a straightforward computation in a basis that diagonalizes $\rho_B(\Frob_v)$.
\end{proof}

\begin{proposition}\label{prop:selmer_with_TW_primes}
Let $\delta_A = 1$ if $A$ has characteristic $p$, and $\delta_A = 0$ if $A$ has characteristic $0$.
Let $\delta_T = 1$ if $T$ is empty and $\delta_T = 0$ otherwise. 
Assume we are given an integer $q \ge [K:\bbQ] - \delta_T$
and an integer $c \ge 0$ satisfying the following property: 
for every integer $N\ge 1$ there is a Taylor--Wiles datum 
(see \S\ref{subsubsec_taylor_wiles_deformation_problems}) $\cQ_N = (Q_N, N, \{ \alpha_v, \beta_v \}_{v \in Q})$ of level $N$ 
such that:
\begin{enumerate}
\item $\lvert Q_N \rvert = q$.
\item Defining the augmented deformation problem
	\[ \cS_{Q_N} = (K, \overline{\rho}, S \cup Q_N,  \{ \Lambda_v \}_{v \in S \cup Q_N}, \{ \cD_v \}_{v \in S} \cup \{ \cD_v^{\mathrm{TW}, N} \}_{v\in Q_N}), \]
there is an $A$-module map $(A/t^N)^{\delta_A} \to H_{\cS_{Q_N}^\perp,T}(\ad\rho_{A/t^N}(1))$ 
with kernel and cokernel annihilated by $t^c$. 
\end{enumerate}
Then we can find a nonnegative integer $b$, depending only on $\cS$, $\rho_A$, and $c$, 
such that for each $N\ge 1$, there is an $A$-module map
	\[ H_{\cS_{Q_N},T}^1(\ad\rho_{A/t^N}) \to (A/t^N)^{2q + \delta_T - [K:\bbQ]} \]
whose kernel and cokernel annihilated by $t^b$.
\end{proposition}

\begin{proof}
Define
	\begin{align*}
    H_{\cS,T}^1(K_S/K, \ad\rho_{A/t^N}) & = 
    \ker\left(H^1(K_S/K, \ad\rho_{A/t^N}) \to \prod_{v\in S} H^1(K_v, \ad\rho_{A/t^N})/\cL_v(\rho_{A/t^N})\right) \\
    H_{\cS_{Q_N},T}^1(K_{S\cup Q_N}/K, \ad\rho_{A/t^N}) & = 
    \ker\left(H^1(K_{S\cup Q_N}/K, \ad\rho_{A/t^N}) \to \prod_{v\in S} H^1(K_v, \ad\rho_{A/t^N})/\cL_v(\rho_{A/t^N})\right)
    \end{align*}
There is an injective inflation map $i : H_{\cS,T}^1(K_S/K, \ad\rho_{A/t^N}) \to H_{\cS_{Q_N},T}^1(K_{S\cup Q_N}/K, \ad\rho_{A/t^N})$ 
inducing an injective map $i' : H_{\cS, T}^1(\ad\rho_{A/t^N}) \to H_{\cS_{Q_N}, T}^1(\ad\rho_{A/t^N})$, and we have a 
commutative diagram with exact rows
	\[ \xymatrix@1{ 
    \prod_{v\in T} H^0(K_v, \ad\rho_{A/t^N}) \ar[r] \ar@{=}[d] & H_{\cS, T}^1(\ad\rho_{A/t^N}) \ar[r] \ar[d]^{i'} & 
    H_{\cS,T}^1(K_S/K, \ad\rho_{A/t^N}) \ar[r] \ar[d]^{i} & 0 \\
    \prod_{v\in T} H^0(K_v, \ad\rho_{A/t^N}) \ar[r] & H_{\cS_{Q_N}, T}^1(\ad\rho_{A/t^N}) \ar[r] & 
    H_{\cS_{Q_N},T}^1(K_{S\cup Q_N}/K, \ad\rho_{A/t^N}) \ar[r] & 0.
    }\]
In particular, $\coker(i') \cong \coker(i)$. 

Let $q_0 = \operatorname{corank}_A H_{\cS^\perp,T}^1(\ad\rho_{F/A}(1)) - \delta_A$. 
By Lemma~\ref{lem:greenberg_wiles} and part~\ref{selmer_size:selmer} of Lemma~\ref{lem:selmer_size}, 
there is $b_1$, depending only on $\cS$ and $\rho_A$, 
and $A$-module maps $f_N : H_{\cS, T}^1(\ad\rho_{A/t^N}) \to (A/t^N)^{q_0 + \delta_T - [K:\bbQ]}$ 
with kernels and cokernels annihilated by $t^{b_1}$. 
Since $A/t^N$ is an injective $A/t^N$-module (its Pontryagin dual, isomorphic to itself, is a projective $A/t^N$-module), 
we can extend $f_N$ to an $A/t^N$-module map 
$\tilde{f}_N : H_{\cS_{Q_N},T}^1(\ad\rho_{A/t^N}) \to (A/t^N)^{q_0 + \delta_T - [K:\bbQ]}$.  
Say we can find a nonnegative integer $b_2$, depending only on $\cS$, $\rho_A$, and $c$, 
and for each $N \ge 1$, an $A/t^N$-module map 
$g_N :  \coker(i) \to (A/t^N)^{2q - q_0} $ with kernel and cokernel annihilated by $t^{b_2}$.  
Then letting $\tilde{g}_N$ be the composite
	\[ H_{\cS_{Q_N}, T}^1(\ad\rho_{A/t^N}) \to \coker(i') \cong \coker(i) \xrightarrow{g_N} (A/t^N)^{2q - q_0}, \]
the resulting $A$-module maps 
	\[ \tilde{f}_N \oplus \tilde{g}_N : H_{\cS_{Q_N},T}^1(\ad\rho_{A/t^N}) \to (A/t^N)^{2q + \delta_T - [K:\bbQ]} \]
will have kernel and cokernel annihilated by $t^{b_1+b_2}$. 

By the Poitou--Tate exact sequence, we have a commutative diagram with exact columns
\[ \xymatrix@1{ 
	0 \ar[d] & 0 \ar[d] \\
    H^1_{\cS,T}(K_S/K, \ad\rho_{A/t^N}) \ar[r]^{i} \ar[d] & H^1_{\cS_{Q_N},T}(K_{S\cup Q_N}/K, \ad\rho_{A/t^N}) \ar[d] \\
    H^1(K_{S\cup Q_N}/K, \ad\rho_{A/t^N}) \ar@{=}[r] \ar[d] & H^1(K_{S\cup Q_N}/K, \ad\rho_{A/t^N}) \ar[d] \\
    \prod_{v\in S} \frac{H^1(K_v, \ad\rho_{A/t^N})}{\cL_v(\rho_{A/t^N})} \times 
    \prod_{v\in Q_N} \frac{H^1(K_v, \ad\rho_{A/t^N})}{H^1(K_v^{\mathrm{ur}}/K_v, \ad\rho_{A/t^N})} \ar[r] \ar[d] &
    \prod_{v\in S} \frac{H^1(K_v, \ad\rho_{A/t^N})}{\cL_v(\rho_{A/t^N})} \ar[d] \\
    H^1_{\cS^\perp, T}(\ad\rho_{A/t^N}(1))^\vee \ar[r] & H^1_{\cS_{Q_N}^\perp, T}(\ad\rho_{A/t^N}(1))^\vee
    }\]
A diagram chase then shows that $\coker(i) \cong \ker(r^\vee)$ with $r^\vee$ the Pontryagin dual of the restriction map  
	\[ r : H^1_{\cS^\perp, T}(\ad\rho_{A/t^N}(1)) \to \prod_{v\in Q_N} H^1(K_v^{\mathrm{ur}}/K_v, \ad\rho_{A/t^N}(1)). \]
So it suffices to show that there is $b_2$, depending only on $\cS$, $\rho_A$, and $c$, and an $A/t^N$-module map 
$(A/t^N)^{2q - q_0} \to \coker(r) \cong \ker(r^\vee)^\vee$ with kernel and cokernel annihilated by $t^{b_2}$. 
By the assumptions at the beginning of this subsection, $p^N \in t^N A$, and by the definition of a Taylor--Wiles datum of level $N$, $q_v \equiv 1 \pmod {p^N}$ for each $v\in Q_N$. 
So for each $v\in Q_N$,  $\ad\rho_{A/t^N}(1) = \ad\rho_{A/t^N}$ as $G_{K_v}$-modules and Lemma~\ref{lem:one_TW_prime} 
implies that the codomain of $r$ is a free $A/t^N$-module of rank $2q$. 
By the assumptions of the lemma, 
there is a nonnegative integer $c$ and an $A/t^N$-module map 
$(A/t^N)^{\delta_A} \to H_{\cS_{Q_N}^\perp, T}^1(\ad\rho_{A/t^N}(1))$ with kernel and cokernel annihilated by $t^c$. 
By part~\ref{selmer_size:dual_selmer} of Lemma~\ref{lem:selmer_size} and the definition of $q_0$, 
there is $c_1$, depending only on $\cS$ and $\rho_A$, and an $A/t^N$-module map $(A/t^N)^{q_0 + \delta_A} \to 
H_{\cS^\perp, T}^1(\ad\rho_{A/t^N}(1))$ with kernel and cokernel annihilated by $t^{c_1}$. 
We take $b_2 = c + c_1$. 
Since $H^1_{\cS_{Q_N}^\perp, T}(\ad\rho_{A/t^N}(1)) = \ker(r)$, we can then find an $A/t^N$-module map 
$(A/t^N)^{q_0} \to \im(r)$ with kernel and cokernel annihilated by $t^{b_2}$, 
hence an $A/t^N$-module map $(A/t^N)^{2q - q_0} \to \coker(r)$ 
with kernel and cokernel annihilated by $t^{b_2}$. 
\end{proof}

\subsection{Nice primes, sweet primes, and their tangent spaces}\label{subsec_nice_and_sweet_primes}

We now fix a particular choice of global deformation problem
\[ \cS = (K, \overline{\rho}, S, \{ \Lambda_v \}_{v \in S}, \{ \cD_v \}_{v \in S}), \]
where:
\begin{itemize}
\item For each $v \in S_p$, $\cD_v = \cD_v^\text{ord}$ is of the type considered in \S \ref{subsubsec_ordinary_deformation_problems}.
\item For each $v \in S - S_p$, $\Lambda_v = \cO$.
\item $p = 2$, $i \in K$, and $\overline{\rho} = \Ind_{G_K}^{G_E} \overline{\chi}$ is dihedral.
\end{itemize}
For each $v \in S_p$, there are universal characters $\psi_{v, 1}, \psi_{v, 2} : G_{K_v} \to R_\cS^\times$ extending those arising from the $\Lambda_v$-algebra structure on $R_\cS$.
Throughout this subsection, if $A$ is an $R_\cS$-algebra, we let $\rho_A$ denote the pushforward of the universal type $\cS$-deformation to $A$. 
We say that $\rho_A$ is dihedral if there exists a quadratic extension $E' / K$ such that $\rho_A(G_{E'})$ is abelian; equivalently, if $\rho_A(G_E)$ is abelian.

\begin{definition}
Let $\frp \subset R_\cS$ be a prime. We say that $\frp$ is a nice prime if it satisfies the following conditions:
\begin{enumerate}
\item The ring $A = R_\cS / \frp$ is a domain of dimension 1 and characteristic $p$.
\item There exists a place $v \in S_p$ such that the quotient $\psi_{v, 1} / \psi_{v, 2}|_{I_{K_v}} \text{ mod }\frp$ has infinite order.
\item There is a nontrivial unipotent element in the image of $\rho_\frp$.
\end{enumerate}
\end{definition}

\begin{definition}\label{def:sweet-prime}
Let $\frp \subset R_\cS$ be a prime. We say that $\frp$ is a sweet prime if it satisfies the following conditions:
\begin{enumerate}
	\item The ring $A = R_\cS / \frp$ is a domain of dimension 1 and characteristic $0$.
	\item There exists a place $v \in S_p$ such that the quotient $\psi_{v, 1} / \psi_{v, 2}|_{I_{K_v}} \text{ mod }\frp$ has infinite order.
\end{enumerate}
\end{definition}

\begin{lemma}\label{lemma:not_dihedral}
Let $\frp \subset R_\cS$ be either a nice prime or a sweet prime. Then $A = R_\cS / \frp$ is a finite $\Lambda$-algebra and $\rho_\frp$ is not dihedral.
\end{lemma}
\begin{proof}
The normalization of $A$ is a DVR, and $A / (\ffrm_\Lambda)$ is a proper quotient of $A$. This quotient must therefore be an Artinian local $k$-algebra. The completed version of Nakayama's lemma then implies that $A$ is a finite $\Lambda$-algebra. 

Assume for contradiction that $\rho_\frp$ is dihedral, so that $\rho_\frp(G_E)$ is abelian. Then there exists a character $\chi : G_E \to \Frac(A)^\times$ such that $\rho_\frp$ is $\GL_2(A)$-conjugate to $\Ind_{G_E}^{G_K} \chi$. If $v | 2$ is a place of $K$, then $v$ is ramified in $E$ (by assumption: it is only in this case that $\cD_v^\text{ord}$ is defined). Let $w$ be the unique place of $E$ lying above $v$. There is an isomorphism $\rho_\frp|_{G_{K_v}} \cong \Ind_{G_{E_w}}^{G_{K_v}} \chi|_{G_{E_w}}$. Since the Jordan--H\"older factors of $\rho_\frp|_{G_{K_v}}$ are $\psi_{v, 1}$ and $\psi_{v, 2}$, we see that we must have $\psi_{v, 1} = \psi_{v, 2}$. Since there is at least one place $v | 2$ of $K$ such that $\psi_{v, 1} \neq \psi_{v, 2}$, this is the desired contradiction.
\end{proof}

\begin{lemma}\label{lemma:Zar_dense}
Let $\frp \subset R_\cS$ be either a nice prime or a sweet prime, let $A = R_\cS / \frp$ and let $H$ be a closed normal subgroup of $\Gamma = \im(\rho_\frp) \subset \GL_2(A)$ such that $\Gamma/H$ is abelian. 
Then the Zariski closure $\overline{H}$ of $H$ contains $\SL_2$.
\end{lemma}

\begin{proof}
	Let $\overline{\Gamma} \subset \GL_2$ denote the Zariski closure of $\Gamma$. Since $\rho_\frp$ is absolutely irreducible, 
the connected component $\overline{\Gamma}^0$ of the identity is a connected reductive group. The image of $\overline{\Gamma}^0$ in $\PGL_2$ contains a torus, because it contains a semisimple element of infinite order. If $\overline{\Gamma}^0$ has a root then it must equal $\PGL_2$. In this case we're done, since we have a chain of inclusions
\[ \overline{H} \supset \overline{[\Gamma, \Gamma]} = \overline{[\overline{\Gamma}, \overline{\Gamma}]} \supset \overline{[\overline{\Gamma}^0, \overline{\Gamma}^0]} = [ \overline{\Gamma}^0, \overline{\Gamma}^0 ] = \SL_2. \]
It remains to rule out the possibility that $\overline{\Gamma}^0$ is a torus. In this case, $\rho_\frp$ becomes reducible when restricted to the finite index subgroup $\Delta = \Gamma \cap \overline{\Gamma}^0$ of $\Gamma$. Let $l \subset \Frac(A)^2$ be a line invariant under $\Delta$, and let $\Delta'$ denote the stabilizer of $l$ in $\Gamma$. Since $\rho_\frp(\Delta)$ contains a semisimple element of infinite order, $\rho_\frp|_{\Delta'}$ is a direct sum of two distinct characters and $\Delta' \subset \Gamma$ has index 2. This contradicts Lemma \ref{lemma:not_dihedral}. 
\end{proof}

\begin{lemma}\label{lem:Gamma_stable_subgroup}
Let $A\in \CNL_\Lambda$ be a discrete valuation domain of characteristic $2$ with fraction field $F$. 
Let $t$ be a uniformizer for $A$ and let $\mathrm{val}$ be the normalized valuation on $A$.
Let $\Gamma$ be a subgroup of $\GL_2(A)$ and assume that there is a closed finite index subfield 
$F_0$ of $F$ and $g\in \GL_2(F)$ such that $g\Gamma g^{-1}$ contains an open subgroup of $\SL_2(F_0)$. 
Let $\ad = M_2(A)$ with adjoint $\Gamma$-action.
There is a nonnegative integer $d$ with the following property:
 
For any $\Gamma$-stable subgroup $M \subseteq \ad$ (not necessarily an $A$-module), 
if $M \not\subset t^m \ad$, then there is $X \in M$ with $X \notin t^{d+m}\ad$ such that either 
$\val(\tr(X)) < m+d$ or $X = z1 + Y$ with $\val(z) < m+d$ and $Y \in t^{\val(z)+1}\ad$.
\end{lemma}

\begin{proof} 
This is \cite[Lemma~3.2.5]{All14}. 
In \textit{loc. cit.}, it is assume that $\Gamma$ is the image of a certain Galois representation, 
however the proof of \textit{loc. cit.} only uses the fact that there is a finite index subfield 
$F_0$ of $\Frac(A)$ such that a conjugate of $\Gamma$ contains an open 
subgroup of $\SL_2(F_0)$, which we have taken as an assumption here.
\end{proof}

\begin{lemma}\label{lem:local_dual_centre}
Let $\rho_A$ be a type $\cS$ lifting of $\overline{\rho}$ to a characteristic $p$ discrete valuation ring $A \in \CNL_\Lambda$ 
and let $\frp \subset R_\cS$ be the kernel of the induced map $R_\cS \to A$.
Let $t$ be a uniformizer for $A$, and for any $N\ge 1$, let $\rho_{A/t^N} = \rho_A \bmod {t^N}$. 

Fix $v \in S_p$ and let $\kappa' : G_{K_v} \to \{\pm 1\}$ be the Kummer character attached to some $\theta \in K_v^\times$. 
Letting $\frz_{A/t^N}$ be the scalar matrices in $\ad\rho_{A/t^N}$, we again write $\kappa'$ for the 
class in $H^1(K_v,\frz_{A/t^N})$ given by $\sigma \mapsto \kappa'(\sigma)1$. 
We then let $\kappa \in H^1(K_v,\ad\rho_{A/t^N})$ be the image of $\kappa'$ 
under the natural map $H^1(K_v,\frz_{A/t^N}) \to H^1(K_v,\ad\rho_{A/t^N})$.
	\begin{enumerate}
		\item\label{local_dual_centre:in_selmer} If the valuation of $\theta$ is even, then $\kappa \in \cL_v(\rho_{A/t^N})^\perp$.
		\item\label{local_dual_centre:not_in_selmer} Assume the valuation of $\theta$ is odd 
		and that $\psi_{v, 1}/\psi_{v, 2} \bmod \frp$ has infinite order. 
		Then there is a nonnegative integer $c$, depending on $\rho_A|_{G_{K_v}}$ but not on $N$, 
		such that $t^m \kappa \in \cL_v(\rho_{A/t^N})^\perp$ only if $m \ge N - c$.
	\end{enumerate}
\end{lemma}

\begin{proof}
After conjugating $\rho_A$, we may assume that 
\[
\rho_A(\sigma) = \begin{pmatrix} \phi_1(\sigma) & \beta(\sigma) \\ & \phi_2(\sigma) \end{pmatrix}
\]
with $\phi_i = \psi_{v, i} \bmod \frp$.   
Let $\frb \subset \ad\rho_{A/t^N}$ and $\frn \subset \frb$ denote the submodules of upper triangular matrices and upper triangular nilpotent matrices, respectively, and let $\frt = \frb/\frn$. 
Let $\cL_\frb(\rho_{A/t^N}) \subset H^1(K_v,\frb)$ be the submodule of elements whose image in 
$H^1(K_v,\frt)$ lies in $H^1_\text{ur}(K_v, \frt)$. 
Then $\cL_v(\rho_{A/t^N})$ is the image of $\cL_\frb(\rho_{A/t^N})$ under $H^1(K_v,\frb) \to H^1(K_v,\ad\rho_{A/t^N})$. 
	
Let $\gamma \in \cL_v(\rho_{A/t^N})$, and choose a cocycle representing $\gamma$ of the form
\begin{equation}\label{eqn:borel_cocycle}
\sigma \mapsto \begin{pmatrix} x(\sigma) & y(\sigma) \\ & w(\sigma) \end{pmatrix}
\end{equation}
The trace pairing with $\frz_{A/t^N}$ annihilates the trace zero submodule $\ad^0\rho_{A/t^N}$ of $\ad\rho_{A/t^N}$, 
and we have a $G_{K_v}$-equivariant isomorphism $\ad\rho_{A/t^N}/\ad^0\rho_{A/t^N} \to A/t^N$ given by
\[ \begin{pmatrix} a & b \\ c & d \end{pmatrix} \mapsto a + d, \]
so $\kappa \cup \gamma = \kappa'\cup x + \kappa' \cup w$. 
As an abelian group, $A/t^N$ is isomorphic to $k^N$, 
so the explicit description of the cup product in terms of the Artin map gives
\[
\kappa \cup \gamma = \kappa' \cup x + \kappa' \cup w = x(\Art_{K_v}(\theta)) + w(\Art_{K_v}(\theta)).
\]
If $\theta$ has valuation divisible by $p$, then this is trivial since $x$ and $w$ are unramified, 
so $\kappa \in \cL_v(\rho_{A/t^N})^\perp$. 
	
Now assume that $\theta$ has valuation not divisible by $p$. 
Under the trace pairing $\frn$ is dual to $\ad\rho_{A/t^N}/\frb$, 
so by local Tate duality we have an isomorphism of $A/t^N$-modules 
$H^2(K_v,\frn) \cong H^0(K_v, \ad\rho_{A/t^N}/\frb)^\vee$. 
Since $\phi_1/\phi_2$ has infinite order, there is a nonnegative integer $c$, 
independent of $N$, such that $t^c$ annihilates $H^0(K_v, \ad\rho_{A/t^N}/\frb)$. 
Applying cohomology to the exact sequence
\[ \xymatrix@1{ 0 \ar[r] &  \frn \ar[r] & \frb \ar[r] & \frt \ar[r] & 0 }, \]
we find that the cokernel of $H^1(K_v, \frb) \to H^1(K_v, \frt)$ is annihilated by $t^c$. 
It follows that $t^c$ annihilates the cokernel of
\[
\cL_\frb(\rho_{A/t^N}) \to H_{\mathrm{ur}}^1(K_v, \frt).
\]
We can thus find $\gamma \in \cL_\frb(\rho_{A/t^N})$ and a cocycle representing $\gamma$ as in \eqref{eqn:borel_cocycle} with $w = 0$ and $x : G_{K_v} \to A/t^N$ 
an unramified homomorphism such that $t^{N-c-1}x \ne 0$. 
As $\theta$ has valuation not divisible by $p$, we find that
\[ 0 = t^m\kappa \cup \gamma = t^m x(\Art_{K_v}(\theta)) \]
only if $m \ge N - c$.
\end{proof}

\begin{lemma}\label{lem:all_TW_primes}
Let $A \in \CNL_\Lambda$ be a discrete valuation ring and let $\rho_A : G_K \to \GL_2(A)$ be a type $\cS$ lifting of $\overline{\rho}$ such that the kernel of the induced map $R_\cS \to A$ is a nice prime.
Fix a non-zero element $t \in \ffrm_A$. 
There is an integer $c \ge 0$, depending only on $\cS$ and $\rho_A$, with the following property.

Let $N\ge 1$ and let $\rho_{A/t^N} = \rho_A \bmod t^N$. 
Let $\cQ_N$ be the set of finite places $v$ of $K$ that split in $K(\zeta_{p^N})$, are disjoint from $S$, have degree $1$ over $\bbQ$, and such that $\overline{\rho}(\Frob_v)$ has distinct eigenvalues. 
Let $T \subseteq S - S_p$, let $V_N$ be any finite subset of $\cQ_N$, and let $Sh_{\cQ_N - V_N}$ be the kernel of the product of the restriction maps
 	\[ H_{\cS^\perp, T}^1(\ad\rho_{A/t^N}(1)) \to \prod_{v\in \cQ_N - V_N} H^1(F_v, \ad\rho_{A/t^N}(1)). \]
Then there is an $A/t^N$-module map $f_N : A/t^N \to Sh_{\cQ_N - V_N}$ whose kernel and cokernel is annihilated by $t^c$.
\end{lemma}

\begin{proof}
Since we are in characteristic $2$, Tate twists are trivial and we drop them from the notation. 
We can and do assume that $t$ is a uniformizer.
To ease notation, we write $\ad = \ad\rho_A$ and $\ad_{A/t^N} = \ad\rho_{A/t^N}$. 
We let $\frz \subset \ad$ be the submodule of diagonal matrices, 
and $\frz_{A/t^N} = \frz \otimes_A A/t^N$ the diagonal matrices in $\ad_{A/t^N}$. 
We also write $(\ad/\frz)_{A/t^N} = (\ad/\frz) \otimes_A A/t^N = \ad_{A/t^N}/\frz_{A/t^N}$.

Recall that $i \in K$, so there is a unique quadratic extension of $K$ contained in the cyclotomic extension. 
We let $N_0 \ge 1$ and $\zeta$ be a primitive $p^{N_0}$th root of unity 
such that $K(\zeta)/K$ is this quadratic extension. 
By ensuring $c \ge N_0$, we can and do assume that $N \ge N_0$ for the remainder of the proof. 
Let $\kappa_\zeta' : \Gal(K_S/K) \to \frz_{A/t^N}$ be the homomorphism that factors through 
$\Gal(K(\zeta)/K)$ and sends its nontrivial element to $1\in \frz_{A/t^N}$. 
Let $\kappa_\zeta$ be the image of $\kappa_\zeta'$ in $H^1(K_S/K, \ad_{A/t^N})$. 
By part~\ref{local_dual_centre:in_selmer} of Lemma~\ref{lem:local_dual_centre}, 
$\kappa_\zeta \in H_{\cS^\perp,T}(\ad_{A/t^N})$. 
Moreover, since $N \ge N_0$ and every $v\in \cQ_N$ splits in $K(\zeta_{p^N})$, 
we have $\kappa_\zeta \in Sh_{\cQ_N - V_N}$. 
We let $f_N : A/t^N \to Sh_{\cQ_N - V_N}$ be the $A/t^N$-module map given by $f(1) = \kappa_\zeta$. 
We claim that we can find an integer $c \ge N_0$, depending only on $\cS$ and $\rho_A$,
such that the kernel and cokernel of $f_N$ is annihilated by $t^c$. 

We first treat $\ker(f_N)$. 
By Lemma~\ref{lemma:Zar_dense}, the Zariski closure of $\im(\rho_A)$ contains $\SL_2$, 
so Lemma~\ref{lem_no_invariants_in_ad_PGL_2} implies that $H^0(K_S/K, \ad/\frz) = 0$. 
Applying cohomology to the exact sequence
	\[ \xymatrix@1{ 
    0 \ar[r] & \ad/\frz \ar[r]^{t^N} & \ad/\frz \ar[r] & (\ad/\frz)_{A/t^N} \ar[r] & 0},\]
we see that $H^0(K_S/K, (\ad\rho/\frz)_{A/t^N})$
injects into the $t$-power torsion of $H^1(K_S/K, \ad/\frz)$ for any $N \ge 1$. 
Let $c_0$ be the maximum of $N_0$ and the order of the $t$-power torsion of $H^1(K_S/K, \ad/\frz)$; 
it depends only $\cS$ and $\rho_A$. 
Then $t^{c_0}$ annihilates $H^0(K_S/K, (\ad/\frz)_{A/t^N})$,
hence also the kernel of
	\[ H^1(K_S/K, \frz_{A/t^N}) \to H^1(K_S/K, \ad_{A/t^N}). \]
Since $\kappa_\zeta'$ generates a free rank $1$ submodule of 
$H^1(K_S/K, \frz_{A/t^N})$ (assuming $N\ge N_0$), 
$t^{c_0}$ annihilates $\ker(f_N)$.

In order to treat $\coker(f_N)$, we proceed in steps. 
Fix some $\gamma \in Sh_{\cQ_N - V_N}$. We want to show that there is a 
positive integer $c \ge c_0$, depending only on $\cS$ and $\rho_A$,
such that $t^c\gamma \in \im(f_N)$. 
Let $K_\infty = \cup_{n\ge 1} K(\zeta_{p^n})$ and let
$L$ be the subfield of $K_S$ fixed by $\ker(\rho_A|_{G_{K_\infty}})$.

\begin{claim}\label{claim:in_im_cohom}
There is a positive integer $c_1$, depending only on $\rho_A$, such that 
$t^{c_1}\gamma$ is contained in the image of the inflation map 
$H^1(L/K, \ad_{A/t^N}) \to H^1(K_S/K, \ad_{A/t^N})$.
\end{claim}

\begin{proof}
Let $\kappa_\gamma$ be a cocycle representing $\gamma$. 
Since the projective image of $\overline{\rho}$ is a dihedral group of order $2n$ with $n$ odd and $K(\zeta_{p^N})/K$ is an extension of degree a power of $2$, 
there is $\sigma \in G_{K(\zeta_{p^N})}$ such that $\overline{\rho}(\sigma)$ has distinct eigenvalues.
By Chebotarev density, we can find $v\in \cQ_N - V_N$ such that $\rho_{A/t^N}(\sigma) = \rho_{A/t^N}(\Frob_v)$ 
and $\kappa_\gamma(\Frob_v) = \kappa_\gamma(\sigma)$. 
Since $\gamma \in Sh_{\cQ_N - V_N}$,
\[ \kappa_\gamma(\sigma) = \kappa_\gamma(\Frob_v) \in (\Frob_v - 1)\ad_{A/t^N} = (\sigma - 1)\ad_{A/t^N}. \] 
Consider the restriction to 
	\[ H^1(K_S/L, \ad_{A/t^N})^{\Gal(L/K)} = \Hom_{\Gal(L/K)}(\Gal(K_S/L), \ad_{A/t^N}), \]
and let $M_N = \kappa(G_L)$. 
Let $M$ be the inverse image of $M_N$ in $\ad$, it is a $\Gamma$-stable subgroup of $\ad$. 
Let $m \ge 1$ be the smallest integer such that $M \not\subset t^m \ad$. 
By Lemma~\ref{lem:Gamma_stable_subgroup}, there is a positive integer $c_1$, 
depending only on $\Gamma$, and $X \in M$ such that either $\val(\tr(X)) < m + c_1$ 
or $X = z1 + Y$ with $\val(z) < m + c_1$ and $Y \in t^{\val(z)+1}\ad$ 
(here $\val$ denotes the normalized valuation on $A$). 
Choose $\tau \in G_L$ such that $\kappa_\gamma(\tau) = X \bmod t^N$.
By Chebotarev density, we can find $w\in \cQ_N - V_N$ such that 
$\rho_{A/t^N}(\Frob_w) = \rho_{A/t^N}(\tau\sigma) = \rho_{A/t^N}(\sigma)$ and 
$\kappa_\gamma(\Frob_w) = \kappa_\gamma(\tau\sigma)$. 
Then $\kappa_\gamma(\Frob_w) \in (\Frob_w - 1)\ad_{A/t^N} = (\sigma - 1)\ad_{A/t^N}$. 
On the other hand, 
	\[ \kappa_\gamma(\Frob_w) = \kappa_\gamma(\tau) + \kappa_\gamma(\sigma) 
    \in \kappa_\gamma(\tau) + (\sigma - 1)\ad_{A/t^N} = 
    \kappa_\gamma(\tau) + (\Frob_w - 1)\ad_{A/t^N}, \]
and $\kappa_\gamma(\tau) = X \bmod t^N$. 
Lemma~\ref{lem:one_TW_prime} and the property of $X$ imply that 
$m + c_1 > N$. 
So $t^{c_1}M \subseteq t^N \ad$ and $t^{c_1}\kappa|_{G_L} = 0$. 
\end{proof}

Set $\gamma_1 = t^{c_1}\gamma$, and identify $\gamma_1$ with an element of $H^1(L/K, \ad_{A/t^N})$. 
Let $i : H^1(L/K, \frz) \to H^1(L/K, \ad)$ and $i_N : H^1(L/K, \frz_{A/t^N}) \to H^1(L/K, \ad_{A/t^N})$ 
denote the maps on cohomology induced by the respective inclusions $\frz \subset \ad$ and $\frz_{A/t^N} \subset \ad_{A/t^N}$. 

\begin{claim}\label{claim:in_centre}
There is a nonnegative integer $c_2$, depending only on $\rho_A$, 
such that $t^{c_2}\gamma_1\in \im(i_N)$.
\end{claim}

\begin{proof}
Applying cohomology to the short exact sequence
	\[ \xymatrix@1{
    0 \ar[r] & \ad \ar[r]^{t^N} & \ad \ar[r] & \ad_{A/t^N} \ar[r] & 0}, \]
we have an injection $H^1(L/K, \ad)/t^N \to H^1(L/K, \ad_{A/t^N})$ 
whose cokernel injects into the $t$-power torsion of $H^2(L/K, \ad)$. 
Let $d_1$ be the order of the $t$-power torsion of $H^2(L/K, \ad)$; it depends only on $\rho_A$. 
We can then lift $t^{d_1}\gamma_1$ to a class $\alpha \in H^1(L/K, \ad)$. 
By commutativity of the diagram
	\[ \xymatrixcolsep{4pc}
    \xymatrix@1{ H^1(L/K, \frz)/t^N \ar[r]^{i \bmod t^N} \ar[d] &  H^1(L/K, \ad)/t^N \ar[d] \\
    H^1(L/K, \frz_{A/t^N}) \ar[r]^{i_N} & H^1(L/K, \ad_{A/t^N}), } \]
it suffices to find a nonnegative integer $d_2$, depending only on $\rho_A$, 
such that 
	\[ t^{d_2}\alpha \in \im(i) + t^N H^1(L/K, \ad). \]
	
By Lemmas~\ref{lemma:Zar_dense} and~\ref{lem_no_invariants_in_ad_PGL_2}, $(\ad/\frz)^{\Gal(L/K_\infty)} = 0$, 
so the inflation-restriction exact sequence gives
an injection $H^1(L/K, \ad/\frz) \to H^1(L/K_\infty, \ad/\frz)$. 
Applying cohomology to
	\[ \xymatrix@1{ 
    0 \ar[r] & \frz \ar[r] & \ad \ar[r] & \ad/\frz \ar[r] &  0}, \]
we see that $\coker(i)$ 
injects into the cokernel of $j : H^1(L/K_\infty, \frz) \to H^1(L/K_\infty, \ad)$. 
By Lemma~\ref{lem_explicit_cocycle}, 
the $A$-rank of $\coker(j)$ is at most one, hence the same is true of $\coker(i)$. 
Let $d_3$ be the order of the $t$-power torsion in $\coker(i)$; it depends only on $\rho_A$.
If the $A$-rank of $\coker(i)$ is zero, then $t^{d_3}\alpha \in \im(i)$, 
and we take $d_2 = d_3$.

Assume that the $A$-rank of $\coker(i)$ is one, 
choose $\beta \in H^1(L/K, \ad)$ such that $\beta$ maps to a generator of $\coker(i)$ 
modulo its $t$-power torsion. 
Multiplying $\beta$ by a unit of $A^\times$, if necessary, we can find $m \ge 0$ 
and $\delta \in H^1(L/K, \frz)$ such that $t^{d_3}\alpha = t^m \beta + \delta$. 
Since the kernel of $\rho_A$ is a nice prime, 
Lemma~\ref{lemma:Zar_dense} and Proposition~\ref{prop:cokernel_class} yield  $\sigma \in \Gal(L/K)$ such that:
\begin{itemize}
\item $\rho_A(\sigma) \in [\Gamma,\Gamma]$, hence $\delta(\sigma) = 0$.
\item $\overline{\rho}(\sigma)$ has distinct eigenvalues.
\item If $\kappa_\beta$ is a cocycle representing $\beta$, then $\kappa_\beta(\sigma) \notin (\sigma - 1)\ad_A$.
\end{itemize}
There is a nonnegative integer $d_4$ such that $\kappa_\beta(\sigma) \notin (\sigma - 1)\ad + t^{d_4+1}\ad$. 
By choosing $\sigma$ and $\kappa_\beta$ so that $d_4$ is minimal, it then depends only on $\rho_A$. 
By computing in a basis that diagonalizes $\rho_A(\sigma)$, we further see that $t^m\kappa_\beta(\sigma) \notin (\sigma-1)\ad + t^{m+d_4+1}\ad$.
The choice of $\kappa_{\beta}$ determines a choice of cocycle 
$\kappa_{\alpha'} = t^m\kappa_{\beta} + \delta$ for $\alpha'= t^{d_3}\alpha$, and 
	\[ \kappa_{\alpha'}(\sigma) = t^m\kappa_{\beta}(\sigma) \notin (\sigma-1)\ad + t^{m+d_4+1}\ad. \]
By Chebotarev density, we can find $v\in \cQ_N - V_N$ such that $\rho_{A/t^N}(\Frob_v) = \rho_{A/t^N}(\sigma)$ 
and such that $\kappa_{\alpha'}(\Frob_v) - \kappa_{\alpha'}(\sigma) \in t^N \ad$. 
Since $\alpha'$ lifts a multiple of $\gamma \in Sh_{\cQ_N - V_N}$, it has trivial image in $H^1(F_v, \ad_{A/t^N})$ and 
	\[ \kappa_{\alpha'}(\sigma) \in \kappa_{\alpha'}(\Frob_v) + t^N\ad \subseteq (\Frob_v - 1)\ad + t^N\ad = 
    (\sigma - 1)\ad + t^N\ad. \]
So $m \ge N - d_4$.
Setting $d_2 = d_3+d_4$, we  have 
	\[ t^{d_2}\alpha = t^{d_4}\alpha' = t^{d_4}\delta + t^{m+d_4}\beta \in \im(i) + t^N H^1(L/K, \ad), \]
which finishes the proof of Claim~\ref{claim:in_centre}.
\end{proof}

Now set $\gamma_2 = t^{c_2}\gamma_1 = t^{c_1+c_2}\gamma$. 
Let $\gamma_2' \in H^1(L/K, \frz_{A/t^N})$ be such that $i_N(\gamma_2') = \gamma_2$. 
The map $f_N$ factors as $f_N = i_N \circ f_N'$ with 
$f_N' : A/t^N \to H^1(L/K, \frz_{A/t^N})$ given by $f_N'(1) = \kappa_\zeta'$.
So to show some multiple of $\gamma_2$ lies in the image of $f_N$, 
it suffices to show some multiple of $\gamma_2'$ lies in the image of $f_N'$.

\begin{claim}\label{claim:in_im_fN}
There is a positive integer $c_3$, depending only on $\cS$ and $\rho_A$, 
such that $t^{c_3}\gamma_2'$ is contained in the image of $f_N'$.
\end{claim}

\begin{proof}
Recall that $K(\zeta)/K$ is the unique quadratic cyclotomic extension and
that $E/K$ is the unique quadratic extension such that $\overline{\rho}|_{G_E}$ is reducible.
Our assumption on $\overline{\rho}|_{G_{K_v}}$ for $v|p$ implies that $K(\zeta)$ and $E$ are distinct. 
First assume that $\gamma_2'(\Gal(L/E(\zeta))) \ne 0$.
The image of $\gamma_2'$ is an abelian group annihilated by $2$ and the image of $\overline{\rho}|_{G_E}$ 
contains elements with distinct (hence odd order) eigenvalues. 
So we further have $\gamma_2'(\Gal(L/E(\zeta_{p^N}))) \ne 0$ and we can find $\tau \in \Gal(L/K(\zeta_{p^N}))$ such that $\gamma_2'(\tau) \ne 0$ 
and $\overline{\rho}(\tau)$ has distinct eigenvalues. 
By Chebotarev density, we can then find a prime $v \in \cQ_N - V_N$ such that $\gamma_2'(\Frob_v) = \gamma_2'(\tau) \ne 0$. 
Then Lemma~\ref{lem:one_TW_prime} implies that the restriction of $\gamma_2'$ to $H^1(K_v, \ad_{A/t^N})$ is nonzero, 
contradicting the fact that $i_N(\gamma_2') = \gamma_2 = t^{c_1+c_2}\gamma \in Sh_{\cQ_N - V_N}$. 

So $\gamma_2'$ must factor through $\Gal(E(\zeta)/K)$. 
Since we are only interested in $\coker(f_N')$, we may further assume that $\gamma_2'$ factors through $\Gal(E/K)$. 
Since the kernel of $R_\cS \to A$ is a nice prime and $i_N(\gamma_2') \in H_{\cS^\perp, T}^1(\ad\rho_{A/t^N})$, we can apply part~\ref{local_dual_centre:not_in_selmer} of Lemma~\ref{lem:local_dual_centre} to find a nonnegative integer $c_3$, depending only on $\cS$ and $\rho_A$, such that $t^{c_3}\gamma_2' = 0$. 
\end{proof}

It follows from Claims~\ref{claim:in_im_cohom}, \ref{claim:in_centre}, and \ref{claim:in_im_fN} 
that $t^{c_1+c_2+c_3}\gamma \in \im(f_N)$. 
Taking $c = \max\{c_0, c_1+c_2+c_3\}$, the kernel and cokernel of $f_N$ 
is annihilated by $t^c$, which completes the proof of 
Lemma~\ref{lem:all_TW_primes}.
\end{proof}

\begin{lemma}\label{lem:sweet_prime_image_cohom}
Let $A \in \CNL_\Lambda$ be discrete valuation ring of characteristic $0$ with fraction field $F$ and let $t$ be a uniformizer for $A$. 
Let $\rho_A : G_K \to \GL_2(A)$ be a continuous representation unramified outside of finite many places such that $\rho_F = \rho_A \otimes_A F$ is absolutely irreducible and the image of $\rho(G_K)$ in $\PGL_2(F)$ is Zariski dense. 
Let $K_\infty = \cup_{n \ge 1} K(\zeta_{p^n})$ and $L/K_\infty$ be the extension cut out by $\ad\rho_A|_{G_{K_\infty}}$. 
Then the order of $H^1(L/K, \ad\rho_{A/t^N}(1))$ is bounded independently of $N \geq 1$.
\end{lemma}

\begin{proof}
It suffices to prove that $H^1(L/K, \ad\rho_A(1))$ is finite, or that $H^1(L/K, \ad\rho_F(1)) = 0$. Let $L_0/K$ be the extension cut out by $\ad\rho_F$. We may identify $\Gal(L_0 / K)$ with a compact, Zariski dense subgroup of $\PGL_2(F)$. In particular, by \cite[Theorem 0.2]{Pink_compact}, the closure $\overline{ [\Gal(L_0 / K), \Gal(L_0 / K) ]}$ of the commutator subgroup has finite index in $\Gal(L_0 / K)$, and consequently $L_0 \cap K_\infty$ is a finite extension of $K$. 

Since $F$ has characteristic 0, the restriction map  $H^1(L/K, \ad\rho_F(1)) \to H^1(L/ L_0 \cap K_\infty, \ad\rho_F(1))$ is injective, so we may assume without loss of generality that $L_0 \cap K_\infty = K$. We then have an isomorphism $\Gal(L/K) \cong \Gal(L_0/K) \times \Gal(K_\infty/K)$ and a K\"unneth decomposition
\begin{align*}
H^1(L/F, \ad\rho_F(1)) &= H^0(L_0/K, \ad\rho_F) \otimes H^1(K_\infty/K, F(1)) \\
& \quad \oplus H^1(L_0/K, \ad\rho_F) \otimes H^0(K_\infty/K, F(1)).
\end{align*}
It is easy to see that
\[ H^0(K_\infty/K, F(1)) = H^1(K_\infty/K, F(1)) = 0. \]
This completes the proof. 
\end{proof}

\begin{lemma}\label{lem:sweet_prime_ad}
    Let $A \in \CNL_\Lambda$ be a discrete valuation ring of characteristic $0$ and let $t$ be a uniformizer for $A$. 
    Let $\rho_A : G_K \to \GL_2(A)$ be a lift of $\overline{\rho}$ that is not dihedral. 
    Let $\ad^0\rho_A$ and $\frz_A$ be the submodule of trace zero matrices and the submodule of diagonal matrices, respectively, in $\ad\rho_A$. 
    There is a nonnegative integer $d$, depending only on $\rho_A$, with the following property. 
If $M \subseteq \ad\rho_A$ is an $A[G_K]$-submodule and $m \ge 0$ is such that $M \subseteq t^m \ad\rho_A$ but $M \not\subseteq t^{m+1}\ad\rho_A$, 
    then either $t^{m+d} \frz_A \subseteq M$ or $t^{m+d} \ad^0\rho_A \subseteq M$.
\end{lemma}

\begin{proof}
Let $F$ be the fraction field of $A$, and let $\rho_F = \rho_A \otimes F$. 
There is a decomposition $\ad\rho_F = \ad^0\rho_F \oplus \frz_F$. 
We first claim that $\ad^0\rho_F$ is an absolutely irreducible $G_K$-representation. 
Otherwise, letting $\overline{F}$ be an algebraic closure of $F$ and $\rho_{\overline{F}} = \rho_F \otimes\overline{F}$, 
there is a $G_K$-stable line $W \subseteq \ad^0\rho_{\overline{F}}$. 
Let $\chi$ denote the character giving the action of $G_K$ on $W$, and let $L/K$ be the abelian extension cut out by $\chi$. 
The restriction $\rho_{\overline{F}}|_{G_L}$ admits a nonscalar $G_L$-equivariant endomorphism, so is reducible. 
Since $\rho_{\overline{F}}$ is irreducible (as $\overline{\rho}$ is absolutely irreducible) and $L/K$ is abelian, 
Clifford theory implies that $\rho_{\overline{F}}$ is dihedral, a contradiction.

By Burnside's theorem, we can find an elements $e_1, e_2 \in A[G_K]$ and $c_1 \geq 0$ such that under the map $A[G_K] \to \End_A(\frz_A) \times \End_A(\ad^0 \rho_A)$, $e_1$ is mapped to $(t^{c_1}, 0)$ and $e_2$ is mapped to $(0, t^{c_1})$. In particular, $e_1 + e_2$ acts on $\ad \rho_A$ as multiplication by $t^{c_1}$.  By \cite[Lemma 6.3]{KisinGeoDefs}, we can find $c_2 \geq 0$ such that for any $A[G_K]$-submodule $X \subset \ad^0 \rho_A$ such that $X \not\subset t \ad^0 \rho_A$, we have $t^{c_2} \ad^0 \rho_A \subset X$. 

Let $M \subset \ad \rho_A$ be an $A[G_K]$-submodule as in the statement of the Lemma. Replacing $M$ by $t^{-m} M$, we can assume that $M \not\subset t \ad \rho_A$, and must show that there is $d \geq 0$, not depending on $M$, such that $t^d \frz_A \subset M$ or $t^d \ad^0 \rho_A \subset M$. We will show that this holds with $d = c_1 + c_2$. We have $t^{c_1} M = e_1 M \oplus e_2 M$, and $t^{c_1} M$ is not contained in $t^{c_1 + 1} \ad \rho_A$, so one of $e_1 M, e_2 M$ is not contained in $t^{c_1+1} \ad \rho_A$. If $e_1 M$ is not contained in $t^{c_1 + 1} \ad \rho_A$ then, as $e_1 M \subset \frz_A \cap M$, $\frz_A \cap M$ is not contained in $t^{c_1 + 1} \ad \rho_A \cap \frz_A = t^{c_1 + 1} \frz_A$, showing that $t^{c_1} \frz_A \subset M$.

If $e_2 M$ is not contained in $t^{c_1 + 1} \ad \rho_A$ then similarly $\ad^0 \rho_A \cap M$ is not contained in $t^{c_1 + 1} \ad \rho_A \cap \ad^0 \rho_A = t^{c_1 + 1} \ad^0 \rho_A$, in which case we see that we must have $t^{c_1 + c_2} \ad^0 \rho_A \subset M$. This completes the proof. 
\end{proof}

\begin{lemma}\label{lem:all_TW_primes_sweet_case}
Let $A \in \CNL_\Lambda$ be a discrete valuation ring of characteristic $0$ and let $\rho_A : G_K \to \GL_2(A)$ be a type $\cS$ deformation of $\overline{\rho}$ 
such that the kernel of the induced map $R_\cS \to A$ is a sweet prime.
Fix $t \in \ffrm_A$ with $p \in tA$. There is a positive integer $c$ depending only on $\cS$ and $\rho_A$ with the following property.

Let $N\ge 1$ and let $\rho_{A/t^N} = \rho_A \bmod t^N$. 
Let $\cQ_N$ be the set of finite places $v$ of $K$ that split in $K(\zeta_{p^N})$, are disjoint from $S$, have degree $1$ over $\bbQ$, and such that $\overline{\rho}(\Frob_v)$ has distinct eigenvalues. 
Let $T \subseteq S - S_p$ and let $V_N$ be any finite subset of $\cQ_N$. 
The kernel of the product of the restriction maps
 	\begin{equation}\label{eqn:sweet_prime_restriction} 
    H_{\cS^\perp, T}^1(\ad\rho_{A/t^N}(1)) \to \prod_{v\in \cQ_N - V_N} H^1(F_v, \ad\rho_{A/t^N}(1)).
    \end{equation}
is annihilated by $t^c$.
\end{lemma}

\begin{proof}
We can assume that $t$ is a uniformizer.
Let $\kappa$ be a cocycle representing an element in the kernel of \eqref{eqn:sweet_prime_restriction}.
Let $K_\infty = \cup_{n\ge 1} K(\zeta_{p^n})$, and let $L/K_\infty$ be the extension cut out by $\ad\rho_A|_{G_{K_\infty}}$. 
We also let $L_N$ be the extension of $K(\zeta_{p^N})$ cut out by $\ad\rho_{A/t^N}|_{G_{K(\zeta_{p^N})}}$. 
Since $K(\zeta_{p^N})/K$ is a $2$-power extension, we can and do fix some $\sigma \in G_{K(\zeta_{p^N})}$ such that $\overline{\rho}$ 
has distinct eigenvalues.
By Chebotarev density, we can find some $v\in \cQ_N - V_N$ such that $\rho_{A/t^N}(\Frob_v) = \rho_{A/t^N}(\sigma)$ and $\kappa(\Frob_v) = \kappa(\sigma)$. 
Then since $\kappa$ represents a class in the kernel of \eqref{eqn:sweet_prime_restriction} and $p^N \in t^NA$,
	\[
    \kappa(\sigma) = \kappa(\Frob_v) \in (\Frob_v - 1) \ad\rho_{A/t^N}(1) = 
    (\Frob_v - 1)\ad\rho_{A/t^N} = (\sigma - 1)\ad\rho_{A/t^N}.
    \]
Consider any $\tau \in \Gal(K_S/L)$. 
We can similarly find $\Frob_w \in \cQ_N - V_N$ such that $\rho_{A/t^N}(\Frob_w) = \rho_{A/t^N}(\tau\sigma)$ and $\kappa(\Frob_w) = \kappa(\tau\sigma)$, so
	\[
    \kappa(\tau\sigma) = \kappa(\Frob_w) \in (\Frob_w - 1) \ad\rho_{A/t^N} = (\tau\sigma - 1)\ad\rho_{A/t^N} = (\sigma - 1)\ad\rho_{A/t^N}.
    \]
On the other hand, $\kappa(\tau\sigma) = \kappa(\tau) + \kappa(\sigma)$, so we conclude 
	\[ \kappa(G_L) \subseteq (\sigma - 1)\ad\rho_{A/t^N}. \]
Now $\kappa|_{G_L}$ is an element of
	\begin{align*}
    H^1(K_S/L, \ad\rho_{A/t^N}(1))^{\Gal(L/K)} & = \Hom_{G_K}(\Gal(K_S/L), \ad\rho_{A/t^N}(1)) \\
    & \subseteq \Hom_{G_{K_\infty}}(\Gal(K_S/L), \ad\rho_{A/t^N}).
    \end{align*}
Let $M_N$ be the $A/t^N$-submodule of $\ad\rho_{A/t^N}$ generated by $\kappa(\Gal(K_S/L))$, 
and let $M$ be its inverse image in $\ad\rho_A$. 
Then $M_N \subseteq (\sigma - 1)\ad\rho_{A/t^N}$ and $t^N \ad\rho_A \subseteq M \subseteq (\sigma - 1)\ad\rho_A + t^N\ad\rho_A$.
Let $0\le m \le N$ be such that $M \subseteq t^m \ad\rho_A$ but $M \not\subseteq t^{m+1} \ad\rho_A$. 
Since the kernel of $R_{\cS} \to A$ is a sweet prime, 
Lemmas~\ref{lemma:Zar_dense} and~\ref{lem:sweet_prime_ad} yield a nonnegative integer $c_1$, depending only on $\cS$ and $\rho_A$, 
such that either $t^{c_1 +m}\frz_A \subset M$ or $t^{c_1 + m}\ad^0\rho_A \subset M$. 
So either 
	\[ t^{c_1 + m} \frz_A \subset (\sigma - 1)\ad\rho_A + t^N \ad\rho_A \]
or 
	\[ t^{c_1 + m} \ad^0\rho_A \subset (\sigma - 1)\ad\rho_A + t^N \ad\rho_A. \]
By computing in a basis that diagonalizes $\sigma$, we see that either case implies $c_1+m \ge N$. 
Thus, $t^{c_1}M_N = 0$, and the restriction of $t^{c_1}\kappa$ to $H^1(K_S/L, \ad\rho_{A/t^N}(1))$ is trivial.

Then $t^{c_1}\kappa$ represents a class in $H^1(L/K, \ad\rho_{A/t^N}(1))$ by the inflation-restriction exact sequence.
By Lemma~\ref{lem:sweet_prime_image_cohom}, there is a nonnegative integer $c_2$, depending only on $\cS$ and $\rho_A$, 
such that $t^{c_2}$ annihilates $H^1(L/K, \ad\rho_{A/t^N}(1))$. 
Taking $c = c_1+c_2$, $t^c$ annihilates any element in the kernel of \eqref{eqn:sweet_prime_restriction}. 
\end{proof}

The following is \cite[Lemma~6.5]{KisinGeoDefs} (\textit{loc. cit.} only treats the characteristic $0$ case, 
but the characteristic $p$ case is the same.)

\begin{lemma}\label{lem:Kisin_indices_lemma}
Let $A \in \CNL_\Lambda$ be a discrete valuation ring, and let $M$ be a finite length $A$-module 
which is generated by $r$ elements for some positive integer $r$. 
Let $\{M_i\}_{i \in I}$ be a collection of finite length $A$-modules and suppose there is an injective map 
$M \to \prod_{i \in I} M_i$. 
The we can find $r$ distinct indices $i_1,\ldots,i_r \in I$ such that the induced map $M \to \prod_{j = 1}^r M_{i_j}$ 
is injective.
\end{lemma}

\begin{lemma}\label{lem:dual_selmer_nice_and_sweet}
Let $T = S - S_p$ and assume $T \neq \emptyset$. 
Let $A_1,\ldots, A_n \in \CNL_\Lambda$ be discrete valuation rings, and for each $1 \le i \le n$, let $\rho_{A_i} : G_K \to \GL_2(A_i)$ 
be a type $\cS$ lift of $\overline{\rho}$ such that the kernel of the induced map $R_\cS \to A_i$ is either a nice prime or a sweet prime. 
For each $1\le i \le n$, fix nonzero $t_i \in \ffrm_{A_i}$ with $p \in t_i A_i$.

Then for each $1\le i \le n$, we can find integers $q_i \ge [K:\bbQ]$ and $c_i \ge 0$, 
depending only on $\cS$ and $\rho_{A_i}$, with the following property: 
for each integer $N\ge 1$ and $1\le i \le n$, 
there is a Taylor--Wiles datum $Q_N^i$ of level $N$ satisfying:
\begin{enumerate}
\item $Q_N^1,\ldots,Q_N^n$ are disjoint and $\lvert Q_N^i \rvert = q_i$ for each $i$.
We set $Q_N = \cup_{i=1}^n Q_N^i$. 
\item Each $v \in Q_N$ has degree $1$ over $\bbQ$.
\item\label{dual_selmer_nice_and_sweet:nice} If the kernel of $R_\cS \to A_i$ is a nice prime, 
then there is an $A_i/t_i^N$-module map $A_i/t_i^N \to H_{\cS_{Q_N}^\perp, T}^1(\ad\rho_{A_i/t_i^N}(1))$ 
with kernel and cokernel annihilated by $t^{c_i}$.
\item If the kernel of $R_\cS \to A_i$ is a sweet prime, then $H_{\cS_{Q_N}^\perp, T}(\ad\rho_{A_i/t_i^N}(1))$ is 
annihilated by $t_i^{c_i}$. 
\end{enumerate}
\end{lemma}

\begin{proof}
First, say the kernel of $R_\cS \to A_i$ is a nice prime. 
Let $Sh_{\cQ_N}$ be as in Lemma~\ref{lem:all_TW_primes} (with $V_N = \emptyset$). 
Then for any Taylor--Wiles datum $Q_N$ of level $N$, we have an inclusion 
$Sh_{\cQ_N} \subseteq H_{\cS_{Q_N}^\perp, T}(\ad\rho_{A_i/t_i^N}(1))$. 
So by Lemma~\ref{lem:all_TW_primes}, there is a nonnegative integer $d_1$, 
depending only on $\cS$ and $\rho_{A_i}$, and a map 
$f_{Q_N} : A_i/t_i^N \to H_{\cS_{Q_N}^\perp, T}^1(\ad\rho_{A_i/t_i^N}(1))$ with kernel 
annihilated by $t^{d_1}$. 
So in proving part~\ref{dual_selmer_nice_and_sweet:nice}, it only remains to show that 
the cokernel of $f_{Q_N}$ is annihilated by $t^{d_2}$ for some nonnegative integer $d_2$ 
depending only on $\cS$ and $\rho_{A_i}$. 

We induct on $n \ge 0$, with the $n = 0$ case being vacuous. 
Say we have constructed $Q_N^1,\ldots, Q_N^{n-1}$ with the desired properties, 
and set $V_N = \cup_{i = 1}^{n-1} Q_N^i$. 
Assume first that the kernel of $R_\cS \to A_n$ is a nice prime. 
To ease notation, write $A = A_n$ and $t = t_n$.
Letting $Sh_{\cQ_N - V_N} \subseteq H_{\cS^\perp, T}^1(\ad\rho_{A/t^N}(1))$ 
be as in Lemma~\ref{lem:all_TW_primes}, there is a nonnegative integer $d_3$, 
depending only on $\cS$ and $\rho_A$, and a map 
$f : A/t^N \to Sh_{\cQ_N - V_N}$ with kernel and cokernel annihilated by $t^{d_3}$. 
Let $F$ be the fraction field of $A$, and let 
$q_n = \operatorname{corank}_A H_{\cS^\perp, T}^1(\ad\rho_{F/A}(1)) - 1$. 
Note that $q_n \ge [K:\bbQ]$ by Lemma~\ref{lem:greenberg_wiles}.
By part~\ref{selmer_size:dual_selmer} of Lemma~\ref{lem:selmer_size}, we can find 
a nonnegative integer $d_4$, depending only on $\cS$ and $\rho_A$, 
and an injective $A/t^N$-module map $g : (A/t^N)^{q_n+1} \to H_{\cS^\perp, T}^1(\ad\rho_{A/t^N}(1))$ 
with cokernel annihilated by $t^{d_4}$.
Then $t^{d_4}\im(f) \subseteq \im(g)$ and the cyclic $A/t^N$-module $\im(f)\cap \im(g)$ is 
isomorphic to $A/t^m$ for some $N-d_3-d_4 \le m \le N$. 
Lifting a generator of $\im(f) \cap \im(g)$ to $(A/t^N)^{q_n+1}$ and applying a change of basis, 
if necessary, we can assume that $\im(f)\cap \im(g)$ is contained in 
$g(A/t^N \times 0 \times \cdots \times 0)$. 
Letting $W = g(0 \times (A/t^N)^{q_n})$, we have:
	\begin{enumerate}[label=(\alph*)]
    \item\label{Wproperties:generators}
    $W$ can be generated by $q_n$ elements.
\item\label{Wproperties:quotient} The quotient of $H_{\cS^\perp, T}^1(\ad\rho_{A/t^N}(1))$ by $\im(f)+ W$ is annihilated by $t^{d_3+2d_4}$. To see this, it suffices to show that $t^{d_3+d_4}\im(g) \subseteq \im(f) + W$, since the cokernel of $g$ is annihilated by $t^{d_4}$. It then suffices to show that $t^{d_3+d_4}g(A/t^N \times 0) \subseteq \im(f)$, which follows from the fact that $A/t^m \cong \im(f)\cap \im(g) \subseteq g(A/t^N \times 0)$ with $m \ge N - d_3 - d_4$. 
	\item\label{Wproperties:kernel} The product of the restriction maps 
    $r_W : W \to \prod_{\cQ_N - V_N} H^1(K_v, \ad\rho_{A/t^N}(1))$ satisfies $t^{d_3}\ker(r_W) \subseteq \im(f)$. Indeed, $\ker(r_W) \subseteq Sh_{\cQ_N - V_N}$ and the cokernel of $f: A/t^N \to Sh_{\cQ_N - V_N}$ is annihilated by $t^{d_3}$.
\end{enumerate}
Using \ref{Wproperties:generators} above and applying Lemma~\ref{lem:Kisin_indices_lemma} to $\im(r_W)$, we find $Q_N^n \subset \cQ_N$, disjoint from 
$V_N = \cup_{i=1}^{n-1} Q_N^i$, of cardinality $q_n$ such that 
$\im(r_W)$ injects into $\prod_{v\in Q_N^n} H^1(K_v, \ad\rho_{A/t^N}(1))$. 
Using \ref{Wproperties:quotient} and \ref{Wproperties:kernel} above, it then follows that $f_{Q_N^n} : A/t^N \to H_{\cS_{Q_N^n}^\perp, T}(\ad\rho_{A/t^N}(1))$ 
has cokernel annihilated by $t^{2d_3+2d_4}$.
Setting $Q_N = \cup_{i = 1}^n Q_N^n$, the map $f_{Q_N^n}$ factors as the composite of 
$f_{Q_N} : A/t^N \to H_{\cS_{Q_N}^\perp, T}(\ad\rho_{A/t^N}(1))$ with the inclusion 
$H_{\cS_{Q_N}^\perp, T}(\ad\rho_{A/t^N}(1)) \subseteq H_{\cS_{Q_N^n}^\perp, T}(\ad\rho_{A/t^N}(1))$, 
so the cokernel of $f_{Q_N}$ is also annihilated by $t^{2d_3+2d_4}$. 

The case when the kernel of $R_\cS \to A_n$ is a sweet prime is similar (and easier), 
appealing to Lemma~\ref{lem:all_TW_primes_sweet_case} instead of Lemma~\ref{lem:all_TW_primes}.
\end{proof}

\begin{lemma}\label{lem:normalization}
Let $A$ be a complete local integral domain of dimension $1$ with normalization $\widetilde{A}$, 
and let $t \in \ffrm_A$ be nonzero. 
Assume we are given nonnegative integers $b$ and $h$, and 
and for each integer $N \ge 1$ a finitely generated $A/t^N$-module $M_N$ 
and an $\widetilde{A}/t^N$-module map $(\widetilde{A}/t^N)^h \to M_N \otimes_A \widetilde{A}$ 
whose kernel and cokernel are annihilated by $t^b$. 
Then we can find a nonnegative integer $a$, depending only on $b$ and $A$, 
and for each $N \ge 1$, an $A/t^N$-module map $(A/t^N)^h \to M_N$ 
whose kernel and cokernel are annihilated by $t^a$. 
\end{lemma}

\begin{proof}
Let $d \ge 0$ be such that $t^d$ annihilates the $A$-module $\widetilde{A}/A$. 
Then for any $A$-module $G$, we have an exact sequence 
	\[ \Tor_1^A(G, \widetilde{A}/A) \to G \to G\otimes_A \widetilde{A} \to G \otimes_A (\widetilde{A}/A), \]
so $G \to G \otimes_A \widetilde{A}$ has kernel and cokernel annihilated by $t^d$.
In particular, this applies to $M_N \to M_N \otimes_A \widetilde{A}$ and $(A/t^N)^g \to (\widetilde{A}/t^N)^g$. 
So, 
	\[ \im(t^d(\widetilde{A}/t^N)^h \to M_N \otimes_A \widetilde{A}) \subseteq \im(M_N \to M_N \otimes_A \widetilde{A}),\]
and the composite $(A/t^N)^h \xrightarrow{t^d} (\widetilde{A}/t^N)^h \to M_N \otimes_A \widetilde{A}$ has image in $\im(M_N \to M_N \otimes_A \widetilde{A})$. 
We can lift this composite to an $A/t^N$-module map $(A/t^N)^h \to M_N$
such that 
	\[ \xymatrix@1{ 
    (A/t^N)^h \ar[r] \ar[d]_{t^d} & M_N \ar[d] \\ 
    (\widetilde{A}/t^N)^h \ar[r] & M_N \otimes_A \widetilde{A}
    }\]
commutes. 
Then $(A/t^N)^h \to M_N$ has kernel annihilated by $t^{b + 2d}$ and cokernel by $t^{b+3d}$.
\end{proof}

\begin{proposition}\label{cor_existence_of_taylor_wiles_data_for_a_collection_of_nice_and_sweet_primes}
Let $T = S - S_p$, and let $\frp_{1}, \dots, \frp_{m}$ (resp. $\frq_{1}, \dots, \frq_{ n}$) 
be nice (resp. sweet) primes of $R_\cS$. 
Choose a representative $\rho_\cS$ of the universal deformation, and let $R_\cS^T \to R_\cS$ be the corresponding map. 
For each $P \in \{ \frp_{1}, \dots, \frp_{m}, \frq_{1}, \dots, \frq_{n} \}$, 
let $P_T$ denote the pre-image of $P$ in $R_\cS^T$, and let $P_\loc$ denote the pre-image of $P_T$ in $A_\cS^T$. 
We assume that $T \neq \emptyset$ and that $A_{\cS}^T/P_\loc = R_\cS/P$ for each $P \in \{\frp_1,\ldots,\frp_m,\frq_1,\ldots,\frq_n\}$.

Then we can find integers $q \geq [K : \bbQ]$
and $g_0 \ge 0$, and for each $P \in \{\frp_1,\ldots,\frp_m,\frq_1,\ldots,\frq_n\}$, an integer $a_P \ge 0$ depending only on $\cS$ and $P$, with the following property: 
for any integer $N \geq 1$, there exists a Taylor--Wiles datum $Q_N$ of level $N$ satisfying the following conditions:
\begin{enumerate}
\item $\lvert Q_N \rvert = q$.
\item Every $v \in Q_N$ has degree $1$ over $\bbQ$.
\item\label{TWprimes:nice-and-sweet-tangent} Let $P \in \{ \frp_{1}, \dots, \frp_{m}, \frq_{1}, \dots, \frq_{n} \}$. 
Let $P_N, P_{T, N}$ denote the pre-images of $P$ in $R_{\cS_{Q_N}}$ and $P_T$ in $R_{\cS_{Q_N}}^T$, respectively. 
Let $A = R_\cS / P$ and let $t \in \ffrm_A$ be nonzero with $p \in tA$. 
Then there exists an $A/t^N$-module map $(A/t^N)^{2q - [K : \bbQ]} \to P_{T, N} / (P_{T, N}^2, P_\loc, t^N)$ with kernel and cokernel annihilated by $t^{a_P}$.
\item\label{TWprimes:max-tangent} $\dim_k(\ffrm_{R_{\cS_{Q_N}}}/(\varpi, \ffrm_{R_{\cS_{Q_N}}}^2)) \le g_0$.
\end{enumerate}
\end{proposition}

\begin{proof}
By Lemma~\ref{lem:dual_selmer_nice_and_sweet}, we can find 
\begin{itemize}
\item an integer $q = q_1 + \cdots + q_{n+m} \ge [K:\bbQ]$, 
depending only on $\cS$ and $\{\frp_1,\ldots,\frp_m,\frq_1,\ldots,\frq_n\}$,
\item for each $P \in \{\frp_1,\ldots,\frp_m,\frq_1,\ldots,\frq_n\}$, 
an integer $c_P \ge 0$, depending only on $\cS$ and $P$,
\end{itemize}
satisfying the following property: 
for every integer $N \ge 1$, there is a Taylor--Wiles datum 
$Q_N$ such that:
  \begin{itemize}
  \item $\lvert Q_N \rvert = q$.
  \item For $P \in \{\frp_1,\ldots,\frp_m\}$, letting $A = R_\cS/P$ and $\widetilde{A}$ be its normalization, 
  and choosing nonzero $t \in \ffrm_A$, there is an $\widetilde{A}/t^N$-module map 
  $\widetilde{A}/t^N \to H_{\cS_{Q_N}^\perp, T}^1(\ad\rho_{\widetilde{A}/t^N}(1))$ 
  with kernel and cokernel annihilated by $t^{c_P}$.
  \item For $P \in \{\frq_1,\ldots,\frq_n\}$, letting $A = R_\cS/P$ and $\widetilde{A}$ be its normalization, 
  and choosing $t \in \ffrm_A$ with $p \in tA$, the $\widetilde{A}/t^N$-module 
  $H_{\cS_{Q_N}^\perp, T}^1(\ad\rho_{\widetilde{A}/t^N}(1))$ 
  is annihilated by $t^{c_P}$.
  \end{itemize}
Then, for fixed $P \in \{\frp_1,\ldots,\frp_m,\frq_1,\ldots,\frq_m\}$, 
Proposition~\ref{prop:selmer_with_TW_primes} gives an integer $b_P$, 
depending only on $\cS$ and $P$, and an $\widetilde{A}/t^N$-module map
	\begin{equation}\label{eqn:size_of_selmer} H_{\cS_{Q_N}, T}^1(\ad\rho_{\widetilde{A}/t^N}) \to (\widetilde{A}/t^N)^{2q - [K:\bbQ]} \end{equation}
with kernel and cokernel annihilated by $t^{b_P}$. 
By Proposition~\ref{prop:tangent_space}, 
	\begin{align*}
    H_{\cS_{Q_N}, T}^1(\ad\rho_{\widetilde{A}/t^N}) & \cong \Hom_A(P_{T,N}/(P_{T,N}^2,P_\loc, t^N), \widetilde{A}/t^N) \\
    & \cong \Hom_{\widetilde{A}}((P_{T,N}/(P_{T,N}^2,P_\loc, t^N))\otimes_A \widetilde{A}, \widetilde{A}/t^N). 
    \end{align*} 
So the Pontryagin dual of \eqref{eqn:size_of_selmer} gives an $\widetilde{A}/t^N$-module map
	\[ (\widetilde{A}/t^N)^{2q - [K:\bbQ]} \to (P_{T,N}/(P_{T,N}^2,P_\loc, t^N))\otimes_A \widetilde{A}.\]
with kernel and cokernel annihilated by $t^{b_P}$.
Applying Lemma~\ref{lem:normalization} shows that the data we have constructed satisfy conditions 1.\ -- 3.\ of the Proposition. It remains to explain why condition 4.\ holds. We have
\[ \dim_k(\ffrm_{R_{\cS_{Q_N}}}/(\varpi, \ffrm_{R_{\cS_{Q_N}}}^2)) = \dim_k H^1_{\cS_{Q_N}, \emptyset}(K_{S \cup Q} / K, \ad \overline{\rho}), \]
and there is a short exact sequence
\[ 0 \to H^1_{\cS, \emptyset}(K_{S} / K, \ad \overline{\rho}) \to H^1_{\cS_{Q_N}, \emptyset}(K_{S \cup Q_N} / K, \ad \overline{\rho}) \to \oplus_{v \in Q_N} H^1(K_v, \ad \overline{\rho}) / H^1_{ur}(K_v, \ad \overline{\rho}). \]
We can therefore take $g_0 = \dim_k H^1_{\cS, \emptyset}(K_{S} / K, \ad \overline{\rho}) + 4 q$. 
\end{proof}

\begin{remark}\label{rmk:nice_and_sweet}
In the above, we have tried to give proofs for nice and sweet primes simultaneously. 
Since the dual Selmer groups for nice primes are more complicated, it is apparent that many of the above arguments would simplify if we only considered sweet primes.
Also, one would not need Lemma~\ref{lem:normalization} as one can just enlarge $\cO$.

On the other hand, if we only considered nice primes, then we could work with divisible 
coefficients, as opposed to finite level coefficients, 
which would simplify some of the above by arguing via coranks. 
We are forced to work with finite level coefficients when dealing with sweet primes because it is 
necessary for the cyclotomic character to be locally trivial at the Taylor--Wiles places.
\end{remark}

\section{Cohomology of locally symmetric spaces}\label{sec_cohom}
We first set up some general notation and establish routine lemmas for the cohomology groups of locally symmetric spaces and the Hecke algebras that act on them. 
We then recall some results on Galois representations attached to the cohomology of these locally symmetric spaces for the group $\GL_n$ over an imaginary CM field, and then derive from these the similar results for the group $\PGL_n$ over an imaginary CM field. 
These results for $\PGL_2$ are necessary for our proof of Theorem~\ref{thm_automorphy_at_odd_primes}, given in Appendix~\ref{sec:p-odd-proofs}.
We then prove a technical result (Theorem~\ref{thm_boundedness_of_good_dihedral_cohomology}) showing that the cohomology groups we are interested in are nonzero in only finitely many degrees after localizing at suitable maximal ideal of the Hecke algebra.

\subsection{Generalities}\label{sec_cohom_general}
We review some general constructions, referring the reader to \cite{New16} or \cite{10authors} for a more detailed summary. Let $p$ be a prime, and let $\cO$ be a coefficient ring. Let $R$ be a ring which is either $\cO$, $E = \operatorname{Frac}(\cO)$, or $\cO / (\varpi^c)$ for some $c \ge 1$. Let $K$ be a number field. We assume that $E$ is large enough that it contains the image of every embedding $K$ in $\overline{\bbQ}_p$.

Let $G$ be a split connected reductive group over $\cO_K$. 
We write $G^\infty = G(\bbA_K^\infty)$ and $G_\infty = G(K \otimes_\bbQ \bbR)$.
If $S$ is a finite set of finite places of $K$, we write $G^S = G(\bbA_K^{\infty,S})$, where $\bbA_K^{\infty,S}$ is the ring of finite adeles of $K$ deprived of its $S$ components, and $G_S = \prod_{v\in S} G(K_v)$; so $G^\infty = G^S \times G_S$. 
We write $\cJ^S$ for the set of open compact subgroups $U = \prod_v U_v \subset G(\widehat{\cO}_K)$ such that $U_v \subset G(\cO_{K_v})$ for all $v$ and $U_v = G(\cO_{K_v})$ for all $v \not\in S$. 
When we need to emphasize the group $G$, we will write $\cJ^S_G$ for $\cJ^S$. 
Any $U \in \cJ^S$ decomposes as $U = U^S U_S$ with $U^S = \prod_{v\notin S} G(\cO_{K_v}) \subset G^S$ and $U_S \subseteq \prod_{v\in S} G(\cO_{K_v}) \subset G_S$. 

Let $X_G = G(K \otimes_\bbQ \bbR) / U_\infty A(\bbR)$, with $U_\infty$ a fixed maximal compact subgroup of $G(K\otimes_\bbQ \bbR)$ and $A$ the maximal $\bbQ$-split torus in the centre of $\Res^K_\bbQ G$. We define a topological space $\mathfrak{X}_G = G(K) \backslash (G(\bbA_K^\infty) \times X_G)$, where in forming this quotient $G^\infty = G(\bbA_K^\infty)$ is endowed with the discrete topology. 

For any open compact subgroup $U \subset G^\infty$, and for any $R[U]$-module $M$, finite free as $R$-module, we write $A(U, M) \in \mathbf{D}(R)$ for the complex computing the $U$-equivariant cohomology $H^\ast_U(\mathfrak{X}_G, M)$ of $M$ on $\mathfrak{X}_G$. Similarly if $N$ is an $R[U]^\text{op}$-module, finite free as $R$-module, then we write $C(U, N) \in \mathbf{D}(R)$ for the complex computing the $U$-equivariant homology $H_\ast^U(\mathfrak{X}_G, N)$ of $N$ on $\mathfrak{X}_G$. 
(These can be defined as follows: let $C$ denote the complex of singular chains on $\mathfrak{X}_G$ with $R$-coefficients, and choose a quasi-isomorphism $P \to C$, where $P$ is a bounded-above complex of projective $R[U]^\text{op}$-modules. Then $C(U, N) = (P \otimes_R N)_U$, and $A(U, M) = \Hom_{R}(P, M)^U$. Here $P \otimes_R N$ gets its natural structure of $R[U]^\text{op}$-module, while $\Hom_R(P, M)$ is made into an $R[U]$-module by the formula $(u \cdot f)(p) = u f(up)$ for an element $u \in U$. The reason for using the ring $R[U]^\text{op}$ here (rather than using the group structure to turn a $R[U]^\text{op}$-module into an $R[U]$-module) is that we will want to consider actions of monoids $\Delta$ below when defining Hecke operators.)
We remind the reader of our convention in \S \ref{subsec_notation} for identifying cochain complexes and chain complexes. In particular, for any $i \in \bbZ$, $H^i(C(U, N)) = H_{-i}^U(\mathfrak{X}_G, N)$.

There is an isomorphism
\[ H^\ast_U(\mathfrak{X}_G, M) \cong \oplus_{g \in G(K) \backslash G^\infty / U} H^\ast( \Gamma_{g, U}, M ), \]
where by definition $\Gamma_{g, U} = G(K) \cap g U g^{-1}$ and the groups on the right-hand side are the usual group cohomology. We note that if $\Gamma_{g, U}$ contains $p$-torsion, then $H^\ast( \Gamma_{g, U}, M )$ need not be a finitely generated $R$-module (because it may be non-zero in infinitely many degrees, a phenomenon that can be observed already in the case $H^\ast(\SL_2(\bbZ), \bbF_2)$).

More generally, if $V \subset U \subset G^\infty$ are open compact subgroups with $V$ normal in $U$, then we write $A(U/V, M) \in \mathbf{D}(R[U/V])$ for the complex computing the $V$-equivariant cohomology of $M$ on $\mathfrak{X}_G$ as $R[U/V]$-modules, and $C(U/V, N) \in \mathbf{D}(R[U / V]^\text{op})$ for the complex computing the $V$-equivariant homology of $N$ on $\mathfrak{X}_G$ as $R[U/V]$-modules. (These can be computed in the same way as above.)

If we define $N = M^\vee = \Hom_R(M, R)$, then there are isomorphisms (where $\Gamma_{U/V}$ denotes the functor of $U/V$-invariants)
\begin{equation} R \Gamma_{U/V} A(U/V, M) \cong A(U, M) \end{equation}
and
\begin{equation} C(U/V, N) \otimes^\bbL_{R[U/V]^\text{op}} R \cong C(U, N) \end{equation}
in $\mathbf{D}(R)$ and
\begin{equation}\label{eqn_cohomology_is_dual_of_homology} R \Hom_{R[U/V]^\text{op}}(C(U/V, M^\vee), R[U/V]) \cong A(U/V, M) \end{equation}
in $\mathbf{D}(R[U/V])$.

We recall that if $U$ is neat, then it acts freely on $\mathfrak{X}_G$, and both $H^\ast_U(\mathfrak{X}_G, M)$ and $H^U_\ast(\mathfrak{X}_G, M^\vee)$ are finite $R$-modules. In fact, in this case both $A(U, M)$ and $C(U, M^\vee)$ are perfect complexes of $R$-modules (as follows from the existence of the Borel--Serre compactification of $\mathfrak{X}_G / U$).
\begin{lemma}\label{lem_universal_coefficient_theorem} Let $U \subset G^\infty$ be an open compact subgroup, and let $M$ be an $R[U]$-module, finite free as $R$-module.
    \begin{enumerate}
        \item For each $i \in \bbZ$, $H^i_U(\mathfrak{X}_G, M)$ and $H_i^U(\mathfrak{X}_G, M^\vee)$ are finite $R$-modules.
        \item Suppose that $R = \cO$, $E$, or $k$. Then for each $i \in \bbZ$ there is a short exact sequence
        \[ \xymatrix@1{ 0 \ar[r] & \Ext^1_R(H^U_{i-1}(\mathfrak{X}_G, M^\vee), R) \ar[r] & H^i_U(\mathfrak{X}_G, M) \ar[r] & \Hom_R(H_i^U(\mathfrak{X}_G, M^\vee), R) \ar[r] & 0.}  \]
    \end{enumerate} 
\end{lemma}
\begin{proof}
    We have just remarked that, if $U$ is neat, then the first part follows from the existence of the Borel--Serre compactification. In general we can find a neat, normal open compact subgroup $V \subset U$ and the finite generation follows from the existence of the Hochschild--Serre spectral sequence. The second part is just the universal coefficient theorem for equivariant (co)homology (i.e. isomorphism (\ref{eqn_cohomology_is_dual_of_homology}) in the case $U = V$.)
\end{proof}
We can define Hecke operators. Let $S$ be some finite (possibly empty) set of finite places of $K$, let $U$ be an open compact subgroup of $G^\infty$ that decomposes as $U = U^S U_S$ with $U \subset G^S$ and $U_S \subset G_S$, and let $\Delta^S$ be a submonoid of $G^S$ which contains $U^S$. Then the set $\cH(\Delta^S, U^S)$ of compactly supported $U^S$-biinvariant functions $f : \Delta^S \to \bbZ$ forms an algebra under convolution, with the indicator function $[U^S]$ of $U^S$ as multiplicative identity. There is a canonical isomorphism $\cH(\Delta^S, U^S) \cong \cH(\Delta^S \times U_S, U)$. Let $M$ be an  $R[\Delta^S \times U_S]$-module, finite free as $R$-module; we view it as an $R[U]$-module by restriction. Then there are algebra homomorphisms
\begin{equation}\label{eqn_hecke_hom_1} \cH(\Delta^S, U^S)^\text{op} \to \End_{\mathbf{D}(R)}( C(U, M^\vee) ) 
\end{equation}
and
\begin{equation}\label{eqn_hecke_hom_2} \cH(\Delta^S, U^S) \to \End_{\mathbf{D}(R)}( A(U, M) ). 
\end{equation}
These can be defined as follows (cf. \cite[\S 2.1.9]{10authors}): take the resolution $P \to C$ by a bounded above complex of projective $R[\Delta]^\text{op}$-modules.
Then $A(U, M) = \Hom_R(P, M)^U$ is a complex of $\cH(\Delta, U)$-modules and $C(U, M^\vee) = (P \otimes_R M^\vee)_U $ is a complex of $\cH(\Delta, U)^\text{op}$-modules, and the induced homomorphisms (\ref{eqn_hecke_hom_1}) and (\ref{eqn_hecke_hom_2}) are independent of the choice of $P$. Moreover, the isomorphism (\ref{eqn_cohomology_is_dual_of_homology}) respects the action of $\cH(\Delta, U)^\text{op}$ by endomorphisms of source and target.

More generally, if $V = V^S V_S \subset U = U^S U_S$ are open compact subgroups of $G^\infty$ with $V^S = U^S$ and $V$ normal in $U$, and $M$ is an $R[\Delta^S \times U_S]$-module, finite free as $R$-module, then there are algebra homomorphisms
\begin{equation}\label{eqn_hecke_hom_3} \cH(\Delta^S, U^S)^\text{op} \to \End_{\mathbf{D}(R[U/V])}( C(U/V, M^\vee) ) 
\end{equation}
and
\begin{equation}\label{eqn_hecke_hom_4} \cH(\Delta^S, U^S) \to \End_{\mathbf{D}(R[U/V])}( A(U/V, M) ). 
\end{equation}
For $U \in \cJ^S$, we write
\[ \bbT^S = \cH\big(G^S, U^S\big) \otimes_\bbZ \cO, \]
which depends only on $S$ and not on $U$. When we need to emphasize the group $G$, we will write $\bbT_G^S$ for $\bbT^S$. 
This is a commutative $\cO$-algebra, because of our assumptions on $G$ and $U^S$. If $C$ is an object of an $\cO$-linear category $\cC$ and $\bbT^S \to \End_\cC(C)$ is an $\cO$-algebra homomorphism, then we will write $\bbT^S(C)$ for its image. A typical use of this notation will be when $C = A(U / V, M) \in \mathbf{D}(R[U/V])$. 
We observe that there are canonical surjective homomorphisms
\[ \bbT^S(C(U/V, M^\vee)) \to \bbT^S(C(U, M^\vee)) \]
and
\[ \bbT^S(A(U/V, M)) \to \bbT^S(A(U, M)) \]
and
\[ \bbT^S(C(U/V, M^\vee)) \to \bbT^S(A(U/V, M)). \]
\begin{lemma}\label{lem_finiteness_of_Hecke_algebras_case_of_non_trivial_group_action}
    Let $U, V \in \cJ^S$ be such that $V \subset U$ is a normal subgroup and $U/V$ is abelian. Let $M$ be an $R[G^S \times U_S]$-module, finite free as $R$-module. 
    Then $\bbT^S(A(U/V, M))$ and $\bbT^S(C(U/V, M^\vee))$ are finite $R$-algebras. 
\end{lemma}
\begin{proof}
    We just prove this for $A(U/V, M)$, since the proof for $C(U/V, M^\vee)$ is similar. We will apply Lemma \ref{lem_finite_end_group}. If $U$ is neat, then $A(U / V, M)$ can be represented in $\mathbf{D}(R[U / V])$ by a bounded complex of (finite free) $R[U/V]$-modules. This implies that $\End_{\mathbf{D}(R[U/V])}(A(U/V, M))$ is a finite $R$-algebra, hence a fortiori that $\bbT^S(A(U/V, M))$ is a finite $R$-algebra. In general, we can choose $V' \in \cJ^S$ such that $V' \subset V$, $V$ is normal in $U$, and $V'$ is neat. Then the cohomology groups $H^\ast(A(U / V', M)) = H^\ast_{V'}(\mathfrak{X}_G, M)$ are finitely generated $R[U/V']$-modules, so Lemma \ref{lem_finite_end_group} again implies that $\End_{\mathbf{D}(R[U / V'])}(A(U / V', M))$ is a finitely generated $R$-module, hence that $\bbT^S(A(U / V', M))$ is a finite $R$-algebra. There is a surjective homomorphism $\bbT^S(A(U / V', M)) \to \bbT^S(A(U / V, M))$, which implies that $\bbT^S(A(U / V, M))$ is also a finite $R$-algebra, as required.
\end{proof}

We find it useful to introduce a book-keeping device to keep track of the various Hecke algebras we use later. 
\begin{definition}
    We define a Hecke datum $\cD$ for $G$ to consist of a tuple $\cD = (S, \{ (\Delta_v, U_v) \}_{v \in S}, \{ T_v \}_{v \in S})$, where:
    \begin{itemize}
        \item $S$ is a finite set of finite places of $K$, containing the $p$-adic places.
        \item For each $v \in S$, $U_v \subset G(\cO_{K_v})$ is an open compact subgroup and $\Delta_v \subset G(K_v)$ is a submonoid containing $U_v$.
        \item $T_v$ is a commutative $\cO$-algebra, endowed with an $\cO$-algebra homomorphism $T_v \to \cH(\Delta_v, U_v) \otimes_\bbZ \cO$.
    \end{itemize}
    If $\cD$ is a Hecke datum, then we define $\bbT_\cD = \bbT^S \otimes_\cO ( \otimes_{v \in S} T_v)$, $\Delta_\cD = \prod_{v \in S} \Delta_v \times \prod_{v \not\in S} G(K_v)$, and $U_\cD = \prod_{v \in S} U_v \times \prod_{v \not\in S} G(\cO_{K_v})$.
\end{definition}
We observe that if $\cD$ is a Hecke datum and $M$ is an $\cO[\Delta_\cD]$-module, then there are natural homomorphisms $\bbT_\cD \to \End_{\mathbf{D}(\cO)}(A(U, M))$ 
and $\bbT_\cD \to \End_{\mathbf{D}(\cO)}(C(U, M^\vee))$.

In \S\ref{sec_Hecke_Gal_rep}, we recall facts about Galois representations associated to $\bbT_\cD$ in the case of $G = \GL_n$ for certain $\cD$. It will be necessary for us to have similar results for $G = \PGL_n$ (in fact, we only need the $n = 2$, $p$ odd case). We will deduce these properties from those of $\GL_n$ by using the following proposition. Let $G^{\mathrm{ad}}$ be the adjoint group of $G$ and let $G^{\mathrm{der}}$ be the derived subgroup of $G$. These are both still split reductive groups over $\cO_K$. Let $Z_G$ denote the centre of $G$, which we now assume to be a split torus. For an open compact subgroup $U \subset G^\infty = G(\bbA_K^\infty)$,  we let $\overline{U}$ denote its image in $(G^\mathrm{ad})^\infty = G^\mathrm{ad}(\bbA_K^\infty)$, and similarly for an open compact subgroup $U_v \subset G(K_v)$ or open submonoid $\Delta_v \in G(K_v)$. 
For $U_v \subseteq \Delta_v \subseteq G(K_v)$ an open compact subgroup and submonoid, respectively, with $Z_G(\cO_{K_v}) \subset U_v$, we have a ring homomorphism $\cH(\Delta_v, U_v) \to \cH(\overline{\Delta}_v, \overline{U}_v)$ given by
\[ f \mapsto \left(g \mapsto \int_{Z_G(K_v)} f(gz) dz \right), \] 
where the measure is the one giving measure 1 to the maximal compact subgroup $Z_G(\cO_{K_v})$ of the split torus $Z_G(K_v)$. This homomorphism takes the double coset operator $[U_vgU_v]$ to the double coset operator $[\overline{U}_v\overline{g}\overline{U}_v]$. 

In particular, if $\cD = (S, \{ (\Delta_v, U_v) \}_{v \in S}, \{ T_v \}_{v \in S})$ is a Hecke datum for $G$, with $U_\cD$ containing $Z_G(\bbA_K^\infty)^c$, the maximal compact subgroup of $Z_G(\bbA_K^\infty)$,
then we obtain a Hecke datum 
    \[ \overline{\cD} = (S, \{ (\overline{\Delta}_v, \overline{U}_v) \}_{v\in S}, \{ T_v \}_{v \in S}) \] 
for $G^\mathrm{ad}$ with the corresponding algebra $\bbT_{\overline{\cD}} = \bbT^S_{G^\mathrm{ad}} \otimes_\cO ( \otimes_{v \in S} T_v)$, monoid $\overline{\Delta}_{\overline{\cD}} = \prod_{v \in S} \overline{\Delta}_v \times \prod_{v \not\in S} G^\mathrm{ad}(K_v)$, and subgroup $\overline{U}_{\overline{\cD}} = \prod_{v \in S} \overline{U}_v \times \prod_{v \not\in S} G^\mathrm{ad}(\cO_{K_v})$.

\begin{proposition}\label{prop_cohomology_for_adjoint_group}
Let $U \subset G^\infty$ be a compact open subgroup containing $Z_G(\bbA_K^\infty)^c$, the maximal compact subgroup of $Z_G(\bbA_K^\infty)$. 
Let $R = E$,  $\cO$, or $\cO/(\varpi^c)$ for some $c\ge 1$ and let $M$ be an $R[\overline{U}]$-module which is finite free as an $R$-module.
Suppose that the centre of $G$ is a split torus
and that the order of the centre of $G^{\mathrm{der}}$ (as a finite group scheme) is invertible in $R$.
 Then the map
    \[ H_{\overline{U}}^\ast(\mathfrak{X}_{G^{\mathrm{ad}}},M) \to H_U^\ast(\mathfrak{X}_G, M) \]
is injective.
\end{proposition}

\begin{proof}
Since the centre of $G$ is a split torus, the map $G^\infty \to (G^\mathrm{ad})^\infty$ is surjective (Hilbert 90). It therefore suffices to show that for each $g \in G^\infty$, the map 
\[ H^\ast(\Gamma_{g,\overline{U}}, M) \to H^\ast(\Gamma_{g,U}, M) \]
is injective. After replacing $U$ with $gUg^{-1}$, we are reduced to the case $g = 1$. For simplicity of notation, we write $\overline{\Gamma} = \Gamma_{1,\overline{U}}$ and $\Gamma = \Gamma_{1,U}$. We also set $U^d = U \cap (G^\mathrm{der})^\infty$ and $\Gamma^d = \Gamma_{1,U^d}$, so there are homomorphisms $\Gamma^d \to \Gamma \to \overline{\Gamma}$. It suffices to show that the map
\[ H^\ast(\overline{\Gamma}, M) \to H^\ast(\Gamma^d, M) \]
is injective. This in turn will follow if we can 
show that the map $f : \Gamma^d \to \overline{\Gamma}$ has kernel of finite order, invertible in $R$, and normal image of finite index, also invertible in $R$.

It follows from our assumption on the centre of $G^\mathrm{der}$ that $\Gamma^d \to \overline{\Gamma}$ has kernel of finite order which is invertible in $R$. The image $f(\Gamma^d)$ is of finite index in $\overline{\Gamma}$ because the class of arithmetic subgroups is preserved under isogeny. In fact, this image is normal with abelian quotient. Indeed, let $\overline{g}, \overline{h} \in \overline{\Gamma}$. We can lift these to elements $g,h \in G(K)$, and write $g = z_1 u_1$, $h = z_2 u_2$ with $z_i \in Z_G(\bbA_K^\infty)$, $u_i \in U$. Then $[g,h] = ghg^{-1}h^{-1} = [u_1,u_2]$ lies in $\Gamma^d$, showing that $[\overline{g}, \overline{h}] \in f(\Gamma^d)$. 

Let $n$ denote the order of the centre of $G^\mathrm{der}$. The proof will be complete if we can show that $n(\overline{\Gamma}/f(\Gamma^d)) = 0$. Let $\overline{g} \in \overline{\Gamma}$. Looking at the cohomology sequence of the short exact sequence 
\[ \xymatrix@1{ 1 \ar[r] & Z_{G^\mathrm{der}} \ar[r] & G^\mathrm{der} \ar[r] & G^\mathrm{ad} \ar[r] & 1,}  \]
we see that $\overline{g}^n$ lies in the image of $G^\mathrm{der}(K)$. If $g \in G(K)$ is a pre-image of $\overline{g}$, we can therefore write $g = zu$ with $z \in Z_G(\bbA_K^\infty)$, $u \in U$, and $g^n = wh$ for some $w \in Z_G(K)$ and $h \in G^\mathrm{der}(K)$. Putting these two expressions together gives $h \in G^\mathrm{der}(K) \cap (Z_G(\bbA_K^\infty)\cdot U)$. Suppose we knew that $G^\mathrm{der}(K)\cap (Z_G(\bbA_K^\infty)\cdot U) = G^\mathrm{der}(K) \cap U = \Gamma^d$. Then we could conclude that $h \in \Gamma^d$ and $\overline{g}^n = f(h)$, completing the proof of the proposition. 

It therefore remains to show that $G^\mathrm{der}(K) \cap (Z_G(\bbA_K^\infty)\cdot U) = G^\mathrm{der}(K) \cap U = \Gamma^d$. It is here that we use our assumption that $U$ contains $Z_G(\bbA_K^\infty)^c$. Let $C_G = G/G^\mathrm{der}$ denote the cocentre of $G$, and let $\nu : G \to C_G$ denote the canonical map. Then $\nu|_{Z_G}$ is an isogeny of split tori. Let $h \in G^\mathrm{der}(K)\cap (Z_g(\bbA_K^\infty)\cdot U)$, and write $h = zu$ with $z \in Z_G(\bbA_K^\infty)$, $u \in U$. Then $\nu(h) = 1 = \nu(z)\nu(u)$. We have $\nu(u) \in C_G(\bbA_K^\infty)^c$, hence $\nu(z) \in C_G(\bbA_K^\infty)^c$, hence $z\in Z_G(\bbA_K^\infty)^c$, hence $zu \in U$, hence $h \in G^\mathrm{der}(K)\cap U$. This completes the proof.
\end{proof}

\begin{corollary}\label{cor_hecke_surj_to_adjoint}
Let $\cD = (S, \{ (\Delta_v, U_v) \}_{v \in S}, \{ T_v \}_{v \in S})$ be a Hecke datum for $G$ such that $U_\cD$ contains $Z_G(\bbA_K^\infty)^c$, and let $\overline{\cD}$ be the corresponding Hecke datum for $G^\mathrm{ad}$.  
Let $R = E$, $\cO$ or $\cO/(\varpi^c)$ for some $c\ge 1$ and let $M$ be an $R[\overline{\Delta}_{\overline{\cD}}]$-module which is finite free as an $R$-module.
Suppose that the centre of $G$ is a split torus
and that the order of the centre of $G^{\mathrm{der}}$ (as a finite group scheme) is invertible in $R$.

Then, letting $\bbT_\cD(H_U^\ast(\mathfrak{X}_G, M))$ denote the image of $\bbT_\cD$ in $\End_R(H_U^\ast(\mathfrak{X}_G, M))$ and letting 
$\bbT_{\overline{\cD}}(H_{\overline{U}}^\ast(\mathfrak{X}_{G^\mathrm{ad}}, M))$ denote the image of $\bbT_{\overline{\cD}}$ in $\End_R(H_{\overline{U}}^\ast(\mathfrak{X}_{G^\mathrm{ad}}, M))$, there 
is a surjective $R$-algebra homomorphism
    \[ \bbT_\cD(H_U^\ast(\mathfrak{X}_G, M)) \to \bbT_{\overline{\cD}}(H_{\overline{U}}^\ast(\mathfrak{X}_{G^\mathrm{ad}}, M)). \]
\end{corollary}

\begin{proof}
This follows at once from Proposition \ref{prop_cohomology_for_adjoint_group}.
\end{proof}

\subsection{Hecke algebras and Galois representations}\label{sec_Hecke_Gal_rep}
We will assume throughout this section that $K$ is a CM field with maximal totally real subfield $K^+$ and that every $p$-adic place of $K^+$ splits in $K$. We fix a finite set of finite places $S$ of $K$ containing all the $p$-adic places. 

For a finite place $v$ of $K$, we write $I_v \subset \GL_n(\cO_{K_v})$ for the Iwahori subgroup defined by upper triangular modulo $\varpi_v$ matrices, and $I_v(1)$ for its subgroup of unipotent modulo $\varpi_v$ elements. More generally, if we are given integers $c \ge b \ge 0$, then we write $I_v(b,c)$ for the subgroup of $\GL_n(\cO_{K_v})$ of matrices that are upper triangular modulo $\varpi_v^c$ and upper triangular unipotent modulo $\varpi_v^b$. If $U \in \cJ_{\GL_n}^S$ and $v \not\in S$ is a finite place of $K$, then we will write $U_0(v) \in \cJ_{\GL_n}^{S \cup \{ v \}}$ for the group with $U_0(v)^v = U^v$ and $U_0(v)_v = I_v$, and $U_1(v) \in \cJ_{\GL_n}^{S \cup \{ v \}}$ for the group with $U_1(v)^v = U^v$ and $U_1(v)_v = I_v(1)$. We also make the same definition with $v$ replaced by a finite set $S'$ of finite places of $K$ such that $S \cap S' = \emptyset$.

For each $v \notin S$ and $1\le i \le n$, we write $\mathsf{T}_{v,i}\in \cH(\GL_n(K_v), \GL_n(\cO_{K_v}))$ for the double coset operator
\[ \mathsf{T}_{v,i} = [\GL_n(\cO_{K_v}) \alpha_{v, i} \GL_n(\cO_{K_v})], \]
where $\alpha_{v, i} = \diag(\varpi_1,\ldots,\varpi_v, 1, \ldots, 1)$ ($\varpi_v$ appearing $i$ times). Then $\bbT_{\GL_n}^S$ is generated over $\cO$ by the operators $\mathsf{T}_{v,1},\ldots,\mathsf{T}_{v,n},\mathsf{T}_{v,n}^{-1}$ for $v \notin S$. 
We define the Hecke polynomial
\[ P_v(X) = X^n + \sum_{i = 1}^n (-1)^i q_v^{i(i-1)/2} \mathsf{T}_{v,i}X^{n-i} \in \cH(\GL_n(K_v), \GL_n(\cO_{K_v}))[X]. \]
For any $\lambda \in (\bbZ_+^n)^{\Hom(K,E)}$, we define an $\cO[\prod_{v|p} \GL_n(\cO_{K_v})]$-module $\cV_\lambda$, finite free as an $\cO$-module, by the recipe of \cite[\S2.2]{Ger18}. It is an $\cO$-lattice in the irreducible algebraic representation $V_\lambda$ of $(\Res_{K/\bbQ}\GL_n)_E$ of highest weight $\lambda$. 

Given an open compact subgroup $U \subset \GL_n(\bbA_K^\infty)$, we let $\overline{U}$ be its image in $\PGL_n(\bbA_K^\infty)$. 
We use similar notation for open compact subgroups $U_v \subset \GL_n(K_v)$ as well as for monoids and elements. 
If $U \in \cJ_{\GL_n}^S$, then $\overline{U} \in \cJ_{\PGL_n}^S$. 
Say $V \subseteq U \subseteq \GL_n(\bbA_K^\infty)$ are open compact subgroups with $V$ normal in $U$, that $R = \cO$ or $\cO/\varpi^c$ for some $c \ge 1$, and that $M$ is an $R[\overline{U}]$-module, finite free as an $R$-module. 
Then we have complexes $A(\overline{V}, M) \in \cD(R)$ and $A(\overline{U}/\overline{V}, M) \in \mathbf{D}(R[\overline{U}/\overline{V}])$ computing $H^\ast_{\overline{V}}(\mathfrak{X}_{\PGL_n}, M)$, and complexes $C(\overline{V}, M^\vee) \in \mathbf{D}(R)$ and $C(\overline{U}/\overline{V}, M^\vee) \in \mathbf{D}(R[\overline{U}/\overline{V}])$ computing $H_\ast^{\overline{V}}(\mathfrak{X}_{\PGL_n}, M^\vee)$.

For each $v \notin S$, we write $\mathsf{T}_{v,1},\ldots,\mathsf{T}_{v,n} \in \cH(\PGL_n(K_v), \PGL_n(\cO_{K_v}))$ for the image of the operators of the same name under the ring homomorphism $\cH(\GL_n(K_v), \GL_n(\cO_{K_v})) \to \cH(\PGL_n(K_v), \PGL_n(\cO_{K_v}))$. 
It should be clear from the context whether we are in the $\GL_n$-case or $\PGL_n$-case, so we hope this does not cause confusion. 
Note that $\mathsf{T}_{v,n} = 1$ in $\cH(\PGL_n(K_v), \PGL_n(\cO_{K_v}))$, and $\bbT_{\PGL_n}^S$ is generated over $\cO$ by $\mathsf{T}_{v,1},\ldots,\mathsf{T}_{v,n-1}$ for $v \notin S$. 
We also again write 
\[ P_v(X) = X^n + \sum_{i = 1}^n (-1)^i q_v^{i(i-1)/2} \mathsf{T}_{v,i}X^{n-i} \in \cH(\PGL_n(K_v), \PGL_n(\cO_{K_v}))[X]. \]
We let $\bbZ_{+,0}^n\subset \bbZ_+^n$ be the subset of $(\lambda_1,\ldots,\lambda_n) \in \bbZ_+^n$ with $\lambda_1+\cdots + \lambda_n = 0$.
If $\lambda \in (\bbZ_{+,0}^n)^{\Hom(K,E)}$ then the action of the centre of $(\Res_{K / \bbQ} \GL_n)$ on $V_\lambda$ is trivial and $\cV_\lambda$ admits a natural structure of $\cO[\prod_{v|p} \PGL_n(\cO_{K_v})]$-module, finite free as $\cO$-module.
\begin{theorem}\label{thm_GLn_mod_p_Gal_rep}
    Let $V \subset U$ be elements of $\cJ_{\GL_n}^S$ with $V$ normal in $U$ and $U/V$ abelian. Let $R = \cO$ or $\cO/\varpi^c$ for some $c \ge 1$ and let $M$ be an $R[U_S]$-module, finite free as an $R$-module, such that $M \otimes_R k = \cV_\lambda \otimes_\cO k$ for some $\lambda \in (\bbZ_+^n)^{\Hom(K,E)}$. Let $\ffrm$ be a maximal ideal of $\bbT^S(A(U/V, M))$. Suppose that $S$ is stable under complex conjugation and satisfies the following condition:
    \begin{itemize}
        \item Let $v$ be a finite place of $K$ not contained in $S$, and let $l$ be its residue characteristic. Then either $S$ contains no $l$-adic places and $l$ is unramified in $K$, or there is an imaginary quadratic field $K_0 \subseteq K$ in which $l$ splits.
    \end{itemize}
    Then there exists a semisimple continuous representation $\overline{\rho}_\ffrm : G_K \to \GL_n(\bbT^S(A(U/V, M)) / \ffrm)$, unramified outside $S$, and such that for each $v \not\in S$,
        \[ \det(X - \overline{\rho}_\ffrm(\Frob_v)) = P_v(X) \bmod\ffrm. \]
    The same statement holds with $A(U/V, M)$ replaced by $C(U/V, M^\vee)$.
\end{theorem}
\begin{proof}
    We will deduce this from \cite[Theorem 2.3.5]{10authors} for $A(U/V, M)$, the case of $C(U/V, M^\vee)$ being similar. We choose $V' \in \cJ_{\GL_n}^S$ such that $V' \subset V$, $V$ is normal in $U$, and $V'$ is neat. By Lemma~\ref{lem_finiteness_of_Hecke_algebras_case_of_non_trivial_group_action}, there is a surjective homomorphism $\bbT^S(A(U / V', M)) \to \bbT^S(A(U / V, M))$. We can therefore assume without loss of generality that $V = V'$ and $V$ is neat. In this case $H^\ast_V(\mathfrak{X}_{\GL_n}, M)$ is a finite $R$-module, and its annihilator in $\bbT^S(A(U / V, M))$ is nilpotent. It follows that $\ffrm$ occurs in the support of $H^\ast_V(\mathfrak{X}, M)$; we can therefore assume without loss of generality that $U = V$, in which case the existence of  $\overline{\rho}_\ffrm$ follows from \cite[Theorem 2.3.5]{10authors}. (We note that the existence of $\overline{\rho}_\ffrm$ is almost contained in one of the main theorems of \cite{Sch15}. Here we are appealing to \cite{10authors}, using our additional assumption on $S$, in order to be able to assert that $\overline{\rho}_\ffrm$ is unramified outside $S$.)
\end{proof}

\begin{corollary}\label{cor_PGLn_mod_p_Gal_rep}
Assume that $p \nmid n$. 
Let $V \subset U$ be elements of $\cJ_{\GL_n}^S$ with $V$ normal in $U$ and $U/V$ abelian and such that $V$ contains $\widehat{\cO}_K^\times$. 
Let $R = \cO$ or $\cO/\varpi^c$ for some $c \ge 1$ and let $M$ be an $R[\overline{U}_S]$-module, finite free as an $R$-module, such that $M \otimes_R k = \cV_\lambda \otimes_\cO k$ for some $\lambda \in (\bbZ_{+,0}^n)^{\Hom(K,E)}$. 
Let $\ffrm$ be a maximal ideal of $\bbT^S(A(\overline{U}/\overline{V}, M))$. Suppose that $S$ is stable under complex conjugation and satisfies the following condition:
\begin{itemize}
    \item Let $v$ be a finite place of $K$ not contained in $S$, and let $l$ be its residue characteristic. Then either $S$ contains no $l$-adic places and $l$ is unramified in $K$, or there is an imaginary quadratic field $K_0 \subseteq K$ in which $l$ splits.
\end{itemize}
Then for any maximal ideal $\ffrm \subset \bbT^S(A(\overline{U}/\overline{V}, M))$, 
there exists a semisimple continuous representation $\overline{\rho}_\ffrm : G_K \to \GL_n(\bbT^S(A(\overline{U}/\overline{V}, M)) / \ffrm)$, unramified outside $S$, and such that for each $v \not\in S$,
    \[ \det(X - \overline{\rho}_\ffrm(\Frob_v)) = P_v(X) \bmod\ffrm. \]
The same statement holds with $A(\overline{U}/\overline{V}, M)$ replaced by $C(\overline{U}/\overline{V}, M^\vee)$.
\end{corollary}
\begin{proof}
We again prove this just for $A(\overline{U}/\overline{V}, M)$. 
As in the proof of Theorem~\ref{thm_GLn_mod_p_Gal_rep}, we can assume that $V$ is neat, so the annihilators of $H^\ast_V(\mathfrak{X}_{\GL_n}, M)$ and $H^\ast_{\overline{V}}(\mathfrak{X}_{\PGL_n}, M)$ in $\bbT^S(A(U/V, M))$ and $\bbT^S(A(\overline{U}/\overline{V}, M))$, respectively, are nilpotent. 
The corollary then follows for $A(\overline{U}/\overline{V}, M)$ from Theorem~\ref{thm_GLn_mod_p_Gal_rep} and Corollary~\ref{cor_hecke_surj_to_adjoint}. 

\end{proof}
We call a maximal ideal $\ffrm$ of a $\bbT_{\GL_n}^S$-algebra $R$ of Galois type if $R / \ffrm$ is finite and there exists a continuous, semisimple representation $\overline{\rho}_\ffrm : G_K \to \GL_n(R / \ffrm)$ satisfying the conclusion of Theorem \ref{thm_GLn_mod_p_Gal_rep}. If $\ffrm$ is of Galois type, we say that it is non-Eisenstein if $\overline{\rho}_\ffrm$ is absolutely irreducible. 
Appealing to Corollary~\ref{cor_PGLn_mod_p_Gal_rep}, we use the same terminology for $\bbT_{\PGL_n}^S$-algebras.

We now want to state a result from \cite{10authors} that asserts the existence of Hecke algebra-valued Galois representations and certain cases of local-global compatibility for them. 
To state the result, suppose given a Hecke datum $\cD = (S, \{ (\Delta_v, U_v) \}_{v \in S}, \{ T_v \}_{v \in S})$, where:
\begin{itemize}
	\item $S$ is a disjoint union $S = S_p \sqcup R \sqcup T$ and $T$ is stable under complex conjugation.
	\item There is an integer $c \geq 1$ such that for each $v \in S_p$, $I_v(c, c) \subset U_v \subset I_v(0, c)$. 
	\item For each $v \in R$, we have $I_v(1) \subset U_v \subset I_v$. 
	\item For each $v \in S_p$, $\Delta_v$ is the monoid $\cup_{\mu \in \bbZ^n_+} I_v \diag(\varpi_v^{\mu_1}, \dots, \varpi_v^{\mu_n}) I_v$.
	\item For each $v \in R$, $\Delta_v = \GL_n(F_v)$. 
	\item For each $v \in T$, $\Delta_v = U_v$. 
	\item For each $v \in S_p$, $T_v = \cO[ ((\cO_{K_v} / (\varpi_v^c))^\times)^n ][ \mathsf{U}_{v, 1}, \dots, \mathsf{U}_{v, n}]$.
	\item For each $v \in R$, $T_v = \cO[ (K_v^\times)^n / (1 + \varpi_v \cO_{K_v})^n ]$. 
	\item For each $v \in T$, $T_v = \cO$.
	\item For each $v \in S_p$, the homomorphism $T_v \to \cH(\Delta_v, U_v) \otimes_\bbZ \cO$ sends each $\mathsf{U}_{v, i}$ to the Hecke operator $\mathsf{U}_{v, i} = [U_v \alpha_{v, i} U_v]$ and each $z \in ((\cO_{K_v} / (\varpi_v^c))^\times)^n$ to $[U_v \diag(z) U_v]$. 
	\item For each $v \in R$, the homomorphism $T_v \to \cH(\Delta_v, U_v) \otimes_\bbZ \cO$ is the one defined just before \cite[Proposition 2.2.8]{10authors} (where we have a map $T_v \to \cO[\Xi_v]$ in the notation there).
\end{itemize}
We note that following \emph{loc. cit.} we can define group homomorphisms $t_{v, i} : K_v^\times \to T_v^\times$ for each $v \in R$ and $i = 1, \dots, n$. If $\lambda \in (\bbZ^n_+)^{\Hom(K, E)}$, then we can define a monoid homomorphism $\chi_{\lambda, v} : (K_v^\times)^n \cap \Delta_v \to T_v$ for each $v \in S_p$ by the formulae
\[ \chi_{\lambda, v}(1, \dots, u, \dots, 1) = \left( \epsilon^{1-i}(\Art_{K_v}(u)) \prod_{\tau \in \Hom(K_v, E)} \tau(u)^{-(w_0 \lambda)_{\tau, i}} \right) \cdot [(1, \dots, u, \dots, 1) \text{ mod } (1 + \varpi_v^c \cO_{K_v})^n] \]
for each $u \in \cO_{K_v}^\times$ and $i = 1, \dots, n$ (with $u$ occupying the $i^\text{th}$ position)
and
\[ \chi_{\lambda, v}(\alpha_{v, i}) = \epsilon^{i(1-i)/2}(\Art_{K_v}(u)) \mathsf{U}_{v, i} \]
for each $i = 1, \dots, n$. 

We let abuse notation by writing $\cV_\lambda$ for the $\cO$-module $\cV_\lambda$, endowed with its twisted $\cO[ \prod_{v \in S_p} \Delta_v]$-action, as defined in \cite[\S 2.2.4]{10authors}. Then there is a homomorphism
\[ \bbT_\cD \to \End_{\mathbf{D}(\cO)}(A(U_\cD, \cV_\lambda)) \]
and, by Proposition~\ref{prop_ord_summand}, there is a direct sum decomposition
\[ A(U_\cD, \cV_\lambda) = A(U_\cD, \cV_\lambda)^{\ord} \oplus A(U_\cD, \cV_\lambda)^{\mathrm{non\mbox{-}ord}} \]
 in $\mathbf{D}(\cO)$ with the property that such that $\mathsf{U}_p = \prod_{v \in S_p} \prod_{i = 1}^n \mathsf{U}_{v, i}$ acts invertibly on $H^\ast(A(U_\cD, \cV_\lambda)^{\ord})$ and topologically nilpotently on $H^\ast(A(U_\cD, \cV_\lambda)^{\mathrm{non\mbox{-}ord}})$. Then each operator $\mathsf{U}_{v, i}$ has unit image in $\bbT_\cD(A(U_\cD, \cV_\lambda)^{\ord})$, so the monoid homomorphism $\chi_{\lambda, v}$ extends uniquely to a homomorphism $\chi_{\lambda, v} = (\chi_{\lambda, v, 1}, \dots, \chi_{\lambda, v, n}) : (K_v^\times)^n \to \bbT_\cD(A(U, \cV_\lambda)^{\ord})^\times$.

We let $\overline{\cD} = (S, \{ (\overline{\Delta}_v, \overline{U}_v) \}_{v \in S}, \{ T_v \}_{v \in S})$ be the resulting Hecke datum for $\PGL_n$. 
If $\lambda \in (\bbZ_{+,0}^n)^{\Hom(K,E)}$, then the twisted $\cO[\prod_{v \in S_p} \Delta_v]$-action on $\cV_\lambda$ descends to $\cO[\prod_{v \in S_p} \overline{\Delta}_v]$.  
Then there is again a homomorphism
\[ \bbT_{\overline{\cD}} \to \End_{\mathbf{D}(\cO)}(A(\overline{U}_{\overline{\cD}}, \cV_\lambda)), \]
a direct sum decomposition
\[ A(\overline{U}_{\overline{\cD}}, \cV_\lambda) = A(\overline{U}_{\overline{\cD}}, \cV_\lambda)^{\ord} \oplus A(\overline{U}_{\overline{\cD}}, \cV_\lambda)^{\mathrm{non\mbox{-}ord}}, \]
and an extension $\chi_{\lambda, v} = (\chi_{\lambda, v, 1}, \dots, \chi_{\lambda, v, n}) : (K_v^\times)^n \to \bbT_{\overline{\cD}}(A(\overline{U}, \cV_\lambda)^{\ord})^\times$ of the monoid homomorphism $\chi_{\lambda, v}$.

\begin{theorem}\label{thm:local-global-GLn-ord}
Let $\cD$ be a Hecke datum for $\GL_n$ satisfying the above assumptions, and let $\ffrm \subset \bbT_\cD(A(U_\cD, \cV_\lambda)^{\ord})$ be a non-Eisenstein maximal ideal. 
Assume further:
\begin{itemize}
\item $K$ contains an imaginary quadratic field in which $p$ splits, and $[K^+ : \bbQ] > 1$.
\item Let $v$ be a finite place of $K$ not contained in $S$, and let $l$ be its residue characteristic. Then either $S$ contains no $l$-adic places and $l$ is unramified in $K$, or there is an imaginary quadratic field $K_0 \subseteq K$ in which $l$ splits. 
\item For each $v \in R$, there is an imaginary quadratic field in $K$ in which the residual characteristic of $v$ splits.
\item $\overline{\rho}_\ffrm$ is decomposed generic.
\end{itemize}
Then we can find an integer $\delta \ge 1$ depending only on $n$ and $[K:\bbQ]$, an ideal $J \subset \bbT_\cD(A(U, \cV_\lambda)^{\ord})_\ffrm$ with $J^\delta = 0$, and and a continuous representation $\rho_\ffrm : G_K \to \GL_n(\bbT_\cD(A(U_\cD, \cV_\lambda)^{\ord})_\ffrm / J)$ satisfying the following properties:
\begin{enumerate}
\item For any $v \notin S$, $\rho_\ffrm$ is unramified at $v$ and $\det(X - \rho_\ffrm(\Frob_v)) = P_v(X) \bmod J$.
\item For each $v \in R$ and $g \in W_{K_v}$, we have $\det(X - \rho_\ffrm(g)) = \prod_{i = 1}^n(X - t_{v, i}( \Art_{K_v}^{-1}(g))) \bmod J$.
\item For each $v \in S_p$ and $g \in W_{K_v}$, we have $\det(X - \rho_\ffrm(g)) = \prod_{i = 1}^n (X - \chi_{\lambda,v,i}(\Art_{K_v}^{-1}(g))) \bmod J$.
\item For each $v \in S_p$ and each $g_1,\ldots,g_n \in W_{K_v}$, we have
\[ (\rho_\ffrm(g_1) - \chi_{\lambda,v,1}(\Art_{K_v}^{-1}(g_1)))(\rho_\ffrm(g_2) - \chi_{\lambda,v,2}(\Art_{K_v}^{-1}(g_2)))\cdots(\rho_\ffrm(g_n) - \chi_{\lambda,v,n}(\Art_{K_v}^{-1}(g_n))) = 0. \]
\end{enumerate}
\end{theorem}
\begin{proof}    
The theorem in the case that $U$ is neat follows from \cite[Theorem~3.1.1]{10authors} and \cite[Theorem~5.5.1]{10authors}, together with an application of Carayol's lemma (\cite[Lemma 2.1.10]{Clo08}). 
In general, we choose a prime $l > 3$ not lying below any $v \in S$, and define a new Hecke datum $\cD'$ by adjoining the primes in $S_l = \{ v | l \}$ to $T$ and setting $U_{\cD', v} =  \ker(\GL_n(\cO_{K_v}) \to \GL_n(k(v)))$ if $v \in S_l$. Then $U_{\cD'}$ is a neat normal subgroup of $U_{\cD}$ and there is a natural homomorphism $\bbT_{\cD'} \to \bbT_{\cD}$. Arguing as in the proof of Lemma \ref{lem_finiteness_of_Hecke_algebras_case_of_non_trivial_group_action} with the second part of Lemma \ref{lem_finite_end_group} shows that there is a homomorphism
\[ \bbT_{\cD'}(A(U_{\cD'}, \cV_\lambda)^{\ord})_\ffrm \to \bbT_{\cD}(A(U_\cD, \cV_\lambda)^{\ord})_\ffrm / I, \]
where $I$ is an ideal satisfying $I^{\dim_\bbR X_G + 1} = 0$. We deduce the existence of a Galois representation valued with coefficients in a quotient of $\bbT_{\cD'}(A(U_\cD,\cV_\lambda)^{\ord})_\ffrm$ by a nilpotent ideal of bounded nilpotence degree. Repeating this argument with a different choice of $l$ and making another application of Carayol's lemma gives the desired result for the Hecke algebra $\bbT_{\cD}(A(U_\cD,\cV_\lambda)^{\ord})_\ffrm$.
\end{proof}

\begin{corollary}\label{cor:local-global-PGLn-ord}
Let $\cD$ be a Hecke datum for $\GL_n$ as above such that $U_\cD$ contains $\widehat{\cO}_K^\times$ and let $\lambda  \in (\bbZ_{+,0}^n)^{\Hom(K,E)}$. 
Let $\overline{\cD}$ be the Hecke datum for $\PGL_n$ determined by $\cD$ and let $\ffrm \subset \bbT_{\overline{\cD}}(A(\overline{U}_{\overline{\cD}}, \cV_\lambda)^{\ord})$ be a non-Eisenstein maximal ideal. 
    Assume further:
    \begin{itemize}
        \item $K$ contains an imaginary quadratic field in which $p$ splits, $p \nmid n$, and $[K^+ : \bbQ] > 1$.
        \item Let $v$ be a finite place of $K$ not contained in $S$, and let $l$ be its residue characteristic. Then either $S$ contains no $l$-adic places and $l$ is unramified in $K$, or there is an imaginary quadratic field $K_0 \subseteq K$ in which $l$ splits. 
        \item For each $v \in R$, there is an imaginary quadratic field in $K$ in which the residual characteristic of $v$ splits.
        \item $\overline{\rho}_\ffrm$ is decomposed generic.
    \end{itemize}
    Then there an integer $\delta \ge 1$ depending only on $n$ and $[K:\bbQ]$, an ideal $J \subset \bbT_{\overline{\cD}}(A(\overline{U}, \cV_\lambda)^{\ord})_\ffrm$ with $J^\delta = 0$, and and a continuous representation $\rho_\ffrm : G_K \to \GL_n(\bbT_{\overline{\cD}}(A(\overline{U}_{\overline{\cD}}, \cV_\lambda)^{\ord})_\ffrm / J)$ satisfying the following properties:
    \begin{enumerate}
        \item For any $v \notin S$, $\rho_\ffrm$ is unramified at $v$ and $\det(X - \rho_\ffrm(\Frob_v)) = P_v(X) \bmod J$. 
        In particular, $\det(\rho_\ffrm) = \epsilon^{n(1-n)/2}$.
        \item For each $v \in R$ and $g \in W_{K_v}$, we have $\det(X - \rho_\ffrm(g)) = \prod_{i = 1}^n(X - t_{v, i}( \Art_{K_v}^{-1}(g))) \bmod J$.
        \item For each $v \in S_p$ and $g \in W_{K_v}$, we have $\det(X - \rho_\ffrm(g)) = \prod_{i = 1}^n (X - \chi_{\lambda,v,i}(\Art_{K_v}^{-1}(g))) \bmod J$.
        \item For each $v \in S_p$ and each $g_1,\ldots,g_n \in W_{K_v}$, we have
        \[ (\rho_\ffrm(g_1) - \chi_{\lambda,v,1}(\Art_{K_v}^{-1}(g_1)))(\rho_\ffrm(g_2) - \chi_{\lambda,v,2}(\Art_{K_v}^{-1}(g_2)))\cdots(\rho_\ffrm(g_n) - \chi_{\lambda,v,n}(\Art_{K_v}^{-1}(g_n))) = 0. \]
    \end{enumerate}
\end{corollary}

\begin{proof}
This follows from Theorem~\ref{thm:local-global-GLn-ord} and Proposition~\ref{prop_cohomology_for_adjoint_group} using the same argument as in the proof of Corollary~\ref{cor_PGLn_mod_p_Gal_rep}. 
The fact that $\det (\rho_\ffrm) = \epsilon^{n(1-n)/2}$ follows from Chebotarev density since the constant term of $P_v(X)$ is $q_v^{n(n-1)/2}$ for any $v \notin S$.
\end{proof}

Finally, we will need a result on the rational cohomology for $\PGL_n$ analogous to \cite[Theorem~2.4.10]{10authors} (we will actually only use $\PGL_2$ in what follows).

\begin{theorem}\label{thm:aut_reps_and_cohom_PGLn}
Fix an isomorphism $\iota : \overline{\bbQ}_p \to \bbC$. 
Let $\overline{U} \in \cJ_{\PGL_n}^S$ and let $\lambda  \in (\bbZ_{+,0}^n)^{\Hom(K,E)}$. 
\begin{enumerate}
\item\label{PGLn:hecke_appears_in_cohomology} Let $\pi$ be a regular algebraic cuspical automorphic representation of $\PGL_n(\bbA_K)$ of weight $\iota\lambda$ such that $(\iota^{-1}\pi)^{\overline{U}} \ne 0$. 
The homomorphism $\bbT^S \to \overline{\bbQ}_p$ associated to the Hecke eigenvalues of $(\iota^{-1}\pi)^{\overline{U}}$ factors through $\bbT^S \to \bbT^S(H^\ast_{\overline{U}}(\mathfrak{X}_{\PGL_n}, \cV_\lambda))$. 
\item\label{PGLn:concentration_of_degrees} Let $q_0 = [K^+:\bbQ]n(n-1)/2$ and $l_0 = [K^+:\bbQ](n-1)$. 
Let $\ffrm \subset \bbT^S(H^\ast_{\overline{U}}(\mathfrak{X}_{\PGL_n}, \cV_\lambda))$ be a maximal ideal that is of Galois type and non-Eisenstein. 
Then 
\[ H^j_{\overline{U}}(\mathfrak{X}_{\PGL_n}, \cV_\lambda)_\ffrm[1/p] \ne 0 \]
only if $j \in [q_0, q_0 + l_0]$; moreover if one of the groups in this range is nonzero, then they all are. 
\item\label{PGLn:cohomology_class_is_automorphic} Let $\ffrm \subset \bbT^S(H^\ast_{\overline{U}}(\mathfrak{X}_{\PGL_n}, \cV_\lambda))$ be a maximal ideal that is of Galois type and non-Eisenstein. 
If $f : \bbT^S(H^\ast_{\overline{U}}(\mathfrak{X}_{\PGL_n}, \cV_\lambda)_\ffrm) \to \overline{\bbQ}_p$ is an $\cO$-algebra homomorphism, then there is a regular algebraic cuspidal automorphic representation $\pi$ of $\PGL_n(\bbA_K)$ of weight $\iota\lambda$ such that $f$ is associated to the Hecke eigenvalues of $(\iota^{-1}\pi)^{\overline{U}}$. 
\end{enumerate}
\end{theorem}

\begin{proof}
The proof is similar to \cite[Theorem~2.4.10]{10authors}. We give a sketch.

For part \ref{PGLn:hecke_appears_in_cohomology}, it suffices to show that the $\bbC$-valued $\cH(\PGL_n(\bbA_K^S), \overline{U}^S)$-eigensystem for $\pi^{\overline{U}}$ has kernel in the support of $H_{\overline{U}}^\ast(\mathfrak{X}_{\PGL_n}, V_{\iota\lambda})$. 
It similarly suffices to prove part \ref{PGLn:concentration_of_degrees} with $H_{\overline{U}}^\ast(\mathfrak{X}_{\PGL_n}, \cV_\lambda)_\ffrm$ replaced by $H_{\overline{U}}^\ast(\mathfrak{X}_{\PGL_n}, V_{\iota\lambda})_\ffrm := H_{\overline{U}}^\ast(\mathfrak{X}_{\PGL_n}, \cV_\lambda)_\ffrm \otimes_{\cO, \iota} \bbC$. 
And for part \ref{PGLn:cohomology_class_is_automorphic}, it suffices to show that $\bbC$-valued $\cH(\PGL_n(\bbA_K^S), \overline{U}^S)$-eigensystem supported in $H_{\overline{U}}^\ast(\mathfrak{X}_{\PGL_n}, V_{\iota\lambda})_\ffrm$ equals one of a regular algebraic cuspidal automorphic representation $\PGL_n(\bbA_K)$ of weight $\iota\lambda$.

There is a $\PGL_n(\bbA_K^\infty)$-equivariant decomposition (see \cite[\S2.2]{frankeschwermer}):
\[ H_{\overline{U}}^\ast(\mathfrak{X}_{\PGL_n}, V_{\iota\lambda}) \cong \bigoplus_{Q \in \mathcal{C}} H^\ast(\mathfrak{g}, \overline{U}_\infty; A_{V_{\iota\lambda}, \{Q\}} \otimes_\bbC V_{\iota\lambda})^{\overline{U}}. \]
In this formula, $\mathcal{C}$ is the set of associate classes of parabolic $\bbQ$-subgroups of $G := \Res_{K/\bbQ} \PGL_n$. 
The right hand side is relative Lie algebra cohomology with $\frg$ the Lie algebra of $G$, $\overline{U}_\infty$ our fixed choice of maximal compact subgroup of $G(\bbR)$, and $A_{V_{\iota\lambda}, \{Q\}}$ a certain space of automorphic forms for $G$.
We set $E_{\{Q\}} = H^\ast(\frg, \overline{U}_\infty; A_{V_{\iota\lambda}, \{Q\}} \otimes_\bbC V_{\iota\lambda})$. 
The summand with $Q = G$ is the cuspidal cohomology,
\[ E_{\{G\}} \cong \bigoplus_\pi (\pi^\infty)^{\overline{U}} \otimes_\bbC H^\ast(\frg, \overline{U}_\infty; \pi_\infty \otimes V_{\iota\lambda}) \]
where the summand runs over cuspidal automorphic representations of $G(\bbA_\bbQ) = \PGL_n(\bbA_K)$.

Let $Q \subset G$ be a parabolic subgroup with Levi subgroup $L_Q$.
Let $\mathfrak{M} \subset \cH(\PGL_n(\bbA_K^S), \overline{U}^S) \otimes_\bbZ \bbC$ be a maximal ideal in the support of $E_{\{Q\}}^{\overline{U}}$.
As in the proof of \cite[Thoerem~2.4.9]{10authors}, \cite[Proposition~3.3]{frankeschwermer} implies that $\mathfrak{M}$ corresponds to the Hecke eigensystem for the (unnormalized) parabolic induction $\Ind_{Q(\bbA_\bbQ^\infty)}^{G(\bbA_\bbQ^\infty)} \sigma^\infty$ with $\sigma$ a cuspidal automorphic representation of $L_Q(\bbA_\bbQ)$. 
Moreover, letting $\widetilde{L}_Q$ be the inverse image of $L_Q$ in $\Res_{K/\bbQ} \GL_n$ and viewing $\sigma$ as a cuspidal automorphic representation of $\widetilde{L}_Q(\bbA_\bbQ) \cong \prod_{i = 1}^r \GL_{n_i}(\bbA_K)$, we have $\sigma \cong \otimes_{i=1}^r \pi_i$ with $\pi_i$ a regular algebraic cuspidal automorphic representation of $\GL_{n_i}(\bbA_K)$. 
By the main theorem of \cite{hltt}, we have a Galois representation 
\[ r_\iota(\mathfrak{M}) = \bigoplus_{i=1}^r r_\iota(\pi_i) \otimes \epsilon^{-(n_{i+1} + \cdots + n_r)} : G_K \to \GL_n(\overline{\bbQ}) \]
unramified outside of $S$ and such that for $v \notin S$, the characteristic polynomial of $r_\iota(\mathfrak{M})(\Frob_v)$ equals $P_v(X) \bmod \mathfrak{M}$.

The direct summand $H_{\overline{U}}^\ast(\mathfrak{X}_{\PGL_n}, V_{\iota\lambda})_\ffrm$ of $H_{\overline{U}}^\ast(\mathfrak{X}_{\PGL_n}, V_{\iota\lambda})$ yields $\cH(\PGL_n(\bbA_K^S), \overline{U}^S)$-equivariant summands $E_{\{Q\}, \ffrm}^{\overline{U}}$ of $E_{\{Q\}}^{\overline{U}}$.
If $\mathfrak{M}$ is in the support of $E_{\{Q\}, \ffrm}^{\overline{U}}$, then we have an isomorphism $\overline{r_\iota(\mathfrak{M})} \cong \overline{\rho}_\ffrm$. 
Since we have assume $\overline{\rho}_\ffrm$ is irreducible, this implies that $E_{\{Q\},\ffrm}^{\overline{U}} = 0$ for any proper parabolic subgroup $Q \subset G$, so $H_{\overline{U}}^\ast(\mathfrak{X}_{\PGL_n}, V_{\iota\lambda})_\ffrm$ is a summand of the cuspidal cohomology. 
By \cite[Chapter II, Proposition 3.1]{borel-wallach}, the relative Lie algebra cohomology $H^\ast(\frg, \overline{U}_\infty; \pi_\infty \otimes V_{\iota\lambda})$ vanishes unless $\pi$ is regular algebraic of weight $\iota\lambda$. 
This finishes the proof of parts \ref{PGLn:hecke_appears_in_cohomology} and \ref{PGLn:cohomology_class_is_automorphic} of the proposition. 
For any such $\pi$, Clozel's purity lemma \cite[Lemma~4.9]{clozel-ann-arbor} implies that $\pi_v$ is a tempered principal series representation for each $v|\infty$.  
Part \ref{PGLn:concentration_of_degrees} of the proposition then follows from \cite[Chapter III, Theorem~5.1]{borel-wallach}. 
\end{proof}

\subsection{Boundedness of cohomology for some non-neat subgroups}\label{sec:bounded-cohomology}
For the remainder of this section, we assume that $G$ is either $\GL_2$ or $\PGL_2$. Abusing notation slightly, we will often write down elements in $\PGL_2$ by writing them as $2 \times 2$ matrices (i.e. lifts to $\GL_2$).
We fix a prime $p$ and a finite set $S$ of places of $K$ containing all $p$-adic places. 
\begin{theorem}\label{thm_boundedness_of_good_dihedral_cohomology} 
Let $U \in \cJ^S$ and let $R = \cO$ or $\cO/ \varpi^c$ for some $c \geq 1$. Let $M$ be an $R[U_S]$-module, finite free as $R$-module, such that $M \otimes_R k = k$. Let $\ffrm \subset \bbT^S(A(U, M))$ be a  maximal ideal of residue field $k$. Suppose that the following conditions are satisfied:
    \begin{enumerate}
        \item $K$ contains a non-trivial $p$th root $\zeta_p$ of $1$. If $G = \PGL_2$, then we assume that $p$ is odd.
        \item The maximal ideal $\ffrm$ is non-Eisenstein. 
\item If $G = \GL_2$, then we assume that $U \cap Z_G(K)$ contains no non-trivial $p$-torsion. 
    \end{enumerate}
    Then:
    \begin{enumerate}
        \item The groups $H^\ast_U(\mathfrak{X}, M)_\ffrm$ are zero outside of the range $[0, \dim_\bbR X_G]$.
        \item Let $V \in \cJ^S$ be a normal subgroup of $U$ such that $U / V$ is abelian of $p$-power order. Then there exists a maximal ideal $\frn \subset \bbT^S(A(U/V, M))$ such that $\overline{\rho}_\frn \cong \overline{\rho}_\ffrm$, and $C(U/V, M^\vee)_\frn$ is a perfect complex of $R[U/V]$-modules equipped with a $\mathbb{T}^S$-equivariant isomorphism
        \[ C(U/V, M^\vee)_\frn \otimes^\bbL_{R[U/V]} R \cong C(U, M^\vee)_\ffrm. \]\end{enumerate}
\end{theorem}
The rest of this section is devoted to the proof of this theorem. 
We observe that it suffices to prove the first part of the theorem in the case $M = k$; this then implies the first part for any $R$, and the second part follows from the first part and Lemma \ref{lem_existence_of_minimal_resolution}.
Presumably a similar result is true under more general hypotheses, but we just prove what we need.

We must first describe some preliminary results. 
We let $Z \subset B \subset G$ be the diagonal maximal torus and upper triangular Borel subgroup, respectively.
For any $v\notin S$, we let $\mathsf{T}_v$ denote the unramified Hecke operator
\[ \mathsf{T}_v = [ G(\cO_{K_v}) \diag(\varpi_v,1) G(\cO_{K_v}) ] \in \cH(G(K_v), G(\cO_{K_v})). \]
\begin{lemma}\label{lem_semisimple}
Let $k$ be an $\cO$-algebra field of characteristic different from $p$. Any $p$-torsion $\gamma \in G(k)$ is semisimple and conjugate to an element of $Z(k)$.
\end{lemma}
\begin{proof}
If $G = \GL_2$, then this follows easily from the fact that $k$ has characteristic different from $p$ and contains all $p$th roots of 1.
Assume $G = \PGL_2$ and choose a lift $\widetilde{\gamma} \in \GL_2(k)$ of $\gamma$. Then $\widetilde{\gamma}^p = \alpha \in k^\times$, so $X^p - \alpha$ and the characteristic polynomial of $\widetilde{\gamma}$ have nontrivial greatest common divisor. Since $p$ is odd and $k$ contains all $p$th roots of 1, Kummer theory implies that $\alpha$ is a $p$th power in $k^\times$. We can then assume that $\widetilde{\gamma}^p = 1$ and we are reduced to the $\GL_2$ case. 
\end{proof}

\begin{lemma}\label{lem_p_group_rank_1}
    Let $U \subset G^\infty$ be an open compact subgroup, and suppose that $U \cap Z_G(K)$ contains no non-trivial $p$-torsion element. Then for all $g \in G^\infty$, every non-trivial elementary abelian $p$-subgroup of $\Gamma_{g, U}$ has rank 1. 
\end{lemma}
\begin{proof}
Let $P \subset \Gamma_{g, U}$ be a non-trivial elementary abelian $p$-subgroup and let $\gamma \in P$ be a nontrivial element. By Lemma~\ref{lem_semisimple} and our assumption on $U$, $Z_G(\gamma)$ is a  maximal torus containing $P$ and the map $P \to (Z_G(\gamma)/Z_G)(K)$ is injective. Since $G$ has semisimple rank $1$ and $K$ contains $\zeta_p$, the $p$-torsion subgroup of $(Z_G(\gamma)/Z_G)(K)$ has rank 1.
\end{proof}
For any open compact subgroup $U \subset G(\bbA_K^\infty)$, we define $X(U) = G(\bbA_K^\infty) / U \times X$. Then the group $G(K)$ acts on $X(U)$ on the left, and there is a canonical isomorphism
\[ H^\ast_{G(K)}(X(U), k) \cong H^\ast_U(\mathfrak{X}, k). \]
This gives us another way to understand Hecke operators. For any $\eta \in G(\bbA_K^\infty)$, there is a diagram of spaces associated to the Hecke operator $[U \eta U]$, where the maps are $G(K)$-equivariant and $\pi_1, \pi_2$ are the canonical projections:
\[ \xymatrix{ & X(U \cap \eta^{-1} U \eta)  \ar[dl]_{\pi_1} & X(\eta U \eta^{-1} \cap U) \ar[l]_-\eta \ar[rd]^{\pi_2} \\ X(U) &&& X(U).} \]
We may define the unramified Hecke operator $\mathsf{T}_v$ on $H^\ast_{G(K)}(X(U), k)$ by the formula $\mathsf{T}_v= \pi_{2, \ast} \circ \eta^\ast \circ \pi_1^\ast$, where $\eta = \diag(\varpi_v,1)$. Let $\Delta \subset Z(K)$ be an order $p$ subgroup not contained in $Z_G(K)$. Taking $\Delta$-fixed points gives rise to another similar diagram, where the maps are now $Z(K)$-equivariant:
\[ \xymatrix{ & X(U \cap \eta^{-1} U \eta)^\Delta \ar[dl]_{\pi'_1} & X(\eta U \eta^{-1} \cap U)^\Delta  \ar[l]_-{\eta'} \ar[rd]^{\pi'_2} \\ X(U)^\Delta & & & X(U)^\Delta.} \]
The connected components of $X(U)$ are in bijection with the set $G(\bbA_K^\infty) / U$. Let $W_{\Delta,U} = \{ g \in G(\bbA_K^\infty) \mid \Delta \subset g U g^{-1} \}$, a right $U$-set. Then $X(U)^\Delta = W_{\Delta, U} / U \times X^\Delta$ and so $\pi_0(X(U)^\Delta) = W_{\Delta, U} / U$ (note that $X^\Delta$ is contractible, and equals the image of $Z(K \otimes_\bbQ \bbR)$ in $X = G(K\otimes_\bbQ \bbR)/U_\infty A(\bbR)$ if we identify $G(K \otimes_\bbQ \bbR) = \GL_2(\bbC)^{[K^+ : \bbQ]}$ and choose $U_\infty = \mathrm{U}_2(\bbR)^{[K^+ : \bbQ]}$).
\begin{lemma}\label{lem_X_Xdelta}
    Let $U \in \cJ^S$ be such that $U \cap Z_G(K)$ contains no $p$-torsion elements and let $\Delta \subset Z(K)$ be an order $p$ subgroup not contained in the centre of $G(K)$. Let $r$ be the rank of $G$ and let $s$ be the $\bbZ$-rank of $\cO_K^\times$. For any finite place $v \not\in S$ of $K$, and for all $i > rs$, the map $h : H^i_{G(K)}(X(U), k) \to H^i_{Z(K)}(X(U)^\Delta, k)$ is equivariant for the action of $\mathsf{T}_v$ on the source and $\mathsf{T}_v^\Delta = \pi'_{2, \ast} \circ (\eta')^\ast \circ (\pi_1')^\ast$ on the target.
\end{lemma} 
We note for future reference that $rs < \dim_\bbR X_G$.
\begin{proof}
    By definition, we have $\mathsf{T}_v = [U \eta U]$, where $\eta = \diag(\varpi_v,1) \in G(K_v) \subset G^\infty$. Let $U_0(v) = U \cap \eta^{-1} U \eta$ and $U^0(v) = \eta U \eta^{-1} \cap U$; this agrees with our previous notation for $U_0(v)$. We see that $X(U_0(v))^\Delta$ is a union of connected components of $\pi_1^{-1}(X(U)^\Delta)$. Similarly, $X(U^0(v))^\Delta$ is a union of connected components of $\pi_2^{-1}(X(U)^\Delta)$. Let $X_0 = \pi_1^{-1}(X(U)^\Delta) - X(U_0(v))^\Delta$, and let $X^0 = \pi_2^{-1}(X(U)^\Delta) - X(U^0(v))^\Delta$. Then $\eta$ restricts to an isomorphism $\zeta : X^0 \to X_0$, and we get a diagram by restriction of $\pi_1$ and $\pi_2$:
    \[ \xymatrix{ & X_0  \ar[dl]_{p_1} & \ar[l]_\zeta X^0 \ar[rd]^{p_2} \\  X(U) & & & X(U)^\Delta.} \]
    Then $h \circ \mathsf{T}_v - \mathsf{T}_v^\Delta \circ h = p_{2, \ast} \circ \zeta^\ast \circ p_1^\ast \circ \Res^{G(K)}_{Z(K)}$. Call this map $\beta$; we must show that $\beta = 0$ if $i > rs$. Since $\beta$ factors through $H^i_{Z(K)}(X_0, k)$, we see that it suffices to show that $H^i_{Z(K)}(X_0, k) = 0$ if $i > rs$. 
    
    We compute $H^i_{Z(K)}(X_0, k)$. There is a $Z(K)$-equivariant bijection
    \[ \pi_0(X_0) = (W_{\Delta, U}  - W_{\Delta, U_0(v)})/ U_0(v). \]
    If $g \in W_{\Delta, U}$, let $Z_{g, U} = Z(K) \cap g U g^{-1}$; it is a congruence subgroup of $Z(K)$, therefore contains a free abelian group of rank $rs$ as a subgroup of finite index. We have an isomorphism
    \[ H^\ast_{Z(K)}(X_0, k) \cong \prod_{g \in Z(K) \backslash (W_{\Delta, U}  - W_{\Delta, U_0(v)})/ U_0(v)} H^\ast(Z_{g, U_0(v)}, k). \]
    If $Z_{g, U_0(v)}$ contains no $p$-torsion elements for each $g \in W_{\Delta, U}  - W_{\Delta, U_0(v)}$, then these groups will be zero in degrees $i > rs$, and this will complete the proof. This is now elementary: suppose that $g \in W_{\Delta, U}  - W_{\Delta, U_0(v)}$ and $\gamma \in Z_{g, U_0(v)}$. Fixing a generator $\delta \in \Delta$,  $g^{-1} \delta g \in U_v = G(\cO_{K_v})$ and if $x \in G(k(v))$ denotes its reduction modulo $\varpi_v$, then  $x \notin B(k(v))$. Similarly, we have $g^{-1} \gamma g \in U_0(v)$ and if $y \in G(k(v))$ denotes its reduction modulo $\varpi_v$, then $y \in B(k(v))$.
    
    Assume that $\gamma$ is $p$-torsion. By our assumption on $U$, to prove $\gamma = 1$, it suffices to show that $\gamma \in Z_G(K)$. Since $v\nmid p$ and $\cO_{K_v}$ contains all $p$th roots of 1, reduction modulo $\varpi_v$ give isomorphisms $Z(\cO_{K_v})[p] \cong  Z(k(v))[p]$ and $Z_G(\cO_{K_v})[p] \cong Z_G(k(v))[p]$. In particular, if $y \in Z_G(k)$, then $\gamma \in Z_G(K)$. If $y \notin Z_G(k)$, then $y$ is regular semisimple by Lemma~\ref{lem_semisimple}, and commutes with $x$. This forces $x \in B(k(v))$, a contradiction.
\end{proof}

\begin{lemma}\label{lem_two_WUs}
Let $\delta \in \GL_2(K)$ be finite order and let $U \subset \GL_2(\bbA_K^\infty)$ be an open compact subgroup containing $\widehat{\cO}_K^\times$. Let $\overline{\delta}$ and $\overline{U}$ denote the images of $\delta$ and $U$, respectively, in $\PGL_2(\bbA_K^\infty)$. Then for any $h \in \PGL_2(\bbA_K^\infty)$ such that $\overline{\delta} \in h \overline{U} h^{-1}$ and any lift $g \in \GL_2(\bbA_K)$ of $h$, we have $\delta \in g U g^{-1}$.
\end{lemma}

\begin{proof}
There is $z \in (\bbA_K^\infty)^\times$ such that $\delta z = gug^{-1}$ for some $u \in U$, and it suffices to show that $z \in \widehat{\cO}_K^\times$. If $n$ is the order of $\delta$, we have $z^n = (gug^{-1})^n$. So $z^n$, hence also $z$, lies in an open compact of $(\bbA_K^\infty)^\times$.
\end{proof}

\begin{lemma}\label{lem_acts_by_2_on_delta}
    Suppose that $U \in \cJ^S$ and that $U$ contains a principal congruence subgroup of the form $U(\frn)$, for some non-zero ideal $\frn \subset \cO_K$. 
    Let $\Delta \subset Z(K)$ be an order $p$ subgroup not contained in the centre of $G(K)$.
Then for any finite place $v$ of $K$ such that $\frp_v$ is principal, generated by an element $\pi_v \in \cO_K$ such that $\pi_v \equiv 1 \text{ mod }p\frn$, the operator $\mathsf{T}_v^\Delta - 2$ acts as 0 on $H^\ast_{Z(K)}(X(U)^\Delta, k)$.
\end{lemma}
\begin{proof}
    Let $\eta = \diag(\varpi_v, 1) \in G(K_v) \subset G^\infty$. We begin by giving an explicit formula for $\mathsf{T}_v^\Delta$. It is helpful to write an element $f \in H^\ast_{Z(K)}(X(U)^\Delta, k)$ as a tuple $(f_g)_{g \in Z(K) \backslash W_{\Delta, U} / U}$ of elements $f_g \in H^\ast(Z_{g, U}, k)$. Define
    \[ w = \left( \begin{array}{cc} 0 & 1 \\ 1 & 0\end{array}\right) \in G(K_v) \subset G^\infty. \]
    The map $Z(K) \backslash W_{\Delta, U_0(v)} / U_0(v) \to Z(K) \backslash W_{\Delta, U} / U$ is 2-to-1, with an element $g$ having distinct pre-images $g$ and $g w$. A similar statement holds for $U^0(v)$. We find that we have for any $f \in H^\ast_{Z(K)}(X(U)^\Delta, k)$ the formula
    \[ (\mathsf{T}_v^\Delta f)_g = \operatorname{cores}^{Z_{g \eta, U_0(v)}}_{Z_{g, U}} \operatorname{res}^{Z_{g \eta, U}}_{Z_{g \eta, U_0(v)}} f_{g \eta} +  \operatorname{cores}^{Z_{g w \eta, U_0(v)}}_{Z_{g, U}} \operatorname{res}^{Z_{g  w \eta, U}}_{Z_{g w \eta, U_0(v)}} f_{g w \eta}. \]
    We see that $\mathsf{T}_v^\Delta - 2$ will act as 0 on  $H^\ast_{Z(K)}(X(U)^\Delta, k)$ provided that the following conditions hold:
    \begin{itemize}
        \item We have $g \eta = g$ in $Z(K) \backslash W_{\Delta, U} / U$. Equivalently, there exist elements $\gamma \in Z(K)$, $u \in U$ such that $g \eta = \gamma g u$.
        \item We have $g w \eta = g$ in $Z(K) \backslash W_{\Delta, U} / U$. Equivalently, there exist elements $\gamma' \in Z(K)$, $u' \in U$ such that $g w \eta = \gamma' g u'$.
        \item The inclusions $Z_{g \eta, U_0(v)} \subset Z_{g \eta, U}$ and $Z_{g \eta, U_0(v)} \subset Z_{g, U}$ are equalities.
        \item The inclusions $Z_{g w \eta, U_0(v)} \subset Z_{g w \eta, U}$ and $Z_{g w \eta, U_0(v)} \subset Z_{g, U}$ are equalities.
    \end{itemize}
    We will show that these conditions do hold under the hypotheses of the lemma for any $g \in W_{\Delta, U}$. We therefore fix a choice of $g \in W_{\Delta, U}$ for the rest of the proof. We also fix a generator $\delta = \diag(\zeta, \zeta') \in \Delta$, with $\zeta,\zeta'$ distinct $p$th roots of $1$. We are free to replace $g$ by any element with the same image in $Z(K) \backslash W_{\Delta, U} / U$. To simplify calculations, we therefore assume that $g$ has been chosen so that $g_v^{-1} \delta g_v = \diag(\zeta, \zeta')$ in $G(K_v)$, or equivalently that $g_v \in Z(K_v)$. This implies that 
    \[ Z(K_v) \cap (g_v \eta_v) U_v (g_v \eta_v)^{-1} = Z(K_v) \cap U_v = Z(K_v) \cap U_0(v)_v = Z(K_v) \cap (g_v \eta_v) U_0(v)_v (g_v \eta_v)^{-1}. \]
     The third point follows immediately from this. The fourth point follows in a very similar manner. 
     
     We now treat the first two points. In fact, it is enough now to consider only the case $G = \GL_2$. Indeed, if $G = \PGL_2$, define an open compact subgroup $\widetilde{U} = \prod_v \widetilde{U}_v \subset \prod_v \GL_2(\cO_{K_v})$ by setting $\widetilde{U}_v$ to be the pre-image of $U_v$ in $\GL_2(\cO_{K_v})$. Let $\widetilde{\delta} \in \GL_2(K)$ be a lift of $\delta$ of order $p$ (the existence of a lift follows from the proof of Lemma \ref{lem_semisimple}). Let $\widetilde{\Delta} = \langle \widetilde{\delta} \rangle$. Then there is a map $W_{\widetilde{\Delta}, \widetilde{U}} \to W_{\Delta, U}$, which  Lemma~\ref{lem_two_WUs}  shows to be surjective. Establishing the first two points above for $W_{\widetilde{\Delta}, \widetilde{U}}$ will therefore imply them for $W_{\Delta, U}$ also.
     
     We therefore assume for the remainder of the proof that in fact $G = \GL_2$.  For the first point, we note that it is equivalent to find $\gamma \in Z(K)$ such that $\eta^{-1} g^{-1} \gamma g \in U$. We choose $\gamma = (\pi_v (\delta -\zeta') - (\delta - \zeta))/(\zeta - \zeta')  = \diag(\pi_v, 1)$. We will show that in fact $\eta^{-1} g^{-1} \gamma g  \in U(\frn)$. It suffices to check this one place at a time. Since $\eta_v^{-1} g_v^{-1} \gamma g_v = \diag(\varpi_v^{-1}\pi_v, 1)$,
this is immediate at the place $v$. If $w \neq v$ is a finite place of $K$, then the $w$-entry of $\eta^{-1} g^{-1} \gamma g$ equals
    \[ g_w^{-1} \gamma g_w = \frac{\pi_v - 1}{\zeta - \zeta'} g_w^{-1} \delta g_w + \frac{\zeta - \pi_v\zeta'}{\zeta - \zeta'}. \]
    Note that $g_w^{-1} \delta g_w$ lies in $U_w \subset \GL_2(\cO_{K_w})$, by assumption. If $w \nmid \frn$ then, since $\pi_v \equiv 1 \text{ mod }p$, the above formula shows that $g_w^{-1} \gamma g_w$ has entries in $\cO_{K_w}$ and unit determinant, so lies in $U_w = \GL_2(\cO_{K_w})$. If $w | \frn$ then, since $\pi_v \equiv 1 \text{ mod }p\frn$, we see that $g_w^{-1} \gamma g_w$ is congruent to 1 modulo $\frn$ and has unit determinant, so lies in $U(\frn)_w$. This establishes the first point above. 

    The second point is very similar. We must find $\gamma' \in Z(K)$ such that $w \eta^{-1} w g^{-1} \gamma' g \in U$. We choose $\gamma' =  ((\delta - \zeta') - \pi_v (\delta - \zeta))/(\zeta - \zeta') = \diag(1, \pi_v)$. A very similar calculation shows that this element does the job.  All points having been established, this completes the proof of the lemma.
\end{proof}
\begin{lemma}\label{lem_Farrell}
    Let $U \in \cJ^S$ be such that $U \cap Z_G(K)$ contains no non-trivial $p$-torsion elements. Let $D$ be a set of representatives for the action of $N_G(Z)/Z$ on the set of order $p$ subgroups of $Z(K)$ not contained in the centre of $G(K)$. Then the direct sum of restriction maps $H^i_{G(K)}(X(U), k) \to \oplus_{\Delta \in D} H^i_{Z(K)}(X(U)^\Delta, k)$ is injective for all $i > \dim_\bbR X_G$.
\end{lemma}
\begin{proof}
    We describe these groups and the maps between them in more classical terms. We can decompose
    \[ H^\ast_{G(K)}(X(U), k) \cong \oplus_{g \in G(K) \backslash G(\bbA_K^\infty) / U} H^\ast(\Gamma_{g, U}, k). \]
    On the other hand, we can decompose
    \[ H^\ast_{Z(K)}(X(U)^\Delta, k) = \oplus_{g \in Z(K) \backslash W_{\Delta, U} / U} H^\ast(Z_{g, U}, k). \]
    There is a natural map $Z(K) \backslash W_{\Delta, U} / U \to G(K) \backslash G(\bbA_K^\infty) / U$, and the fibre above the double coset of an element $g \in G(\bbA_K^\infty)$ is identified with the set of $\Gamma_{g, U}$-conjugacy classes of $G(K)$-conjugates of $\Delta$ which lie in $\Gamma_{g, U}$. Thus the map in the lemma can be identified with a direct sum of maps
    \[ H^\ast(\Gamma_{g, U}, k) \to \prod_P H^\ast( Z_{\Gamma_{g, U}}(P), k), \]
    where the product runs over the set of $\Gamma_{g, U}$-conjugacy classes of $G(K)$-conjugates $P$ of some $\Delta$ which are contained in $\Gamma_{g, U}$.
    
    On the other hand, the theory of Farrell cohomology (see \cite[Corollary X.7.4]{Bro82} and \cite[X.3.4]{Bro82}) implies that when the elementary abelian $p$-subgroups of $\Gamma_{g, U}$ have rank at most 1, the product of restriction maps
    \[ H^i(\Gamma_{g, U}, k) \to \prod_P H^i(N_{\Gamma_{g, U}}(P), k) \]
    is an isomorphism for each $i > \dim_\bbR X_G$ (where the product runs over the set of $\Gamma_{g, U}$-conjugacy classes of elementary abelian $p$-subgroups of $\Gamma_{g, U}$). (We are using here that all the groups $\Gamma_{g, U}$ and $N_{\Gamma_{g, U}}(P)$ have virtual cohomological dimension at most $\dim_\bbR X_G$.) Our hypotheses imply that any elementary abelian $p$-subgroup $P$ of $\Gamma_{g, U}$ is $G(K)$-conjugate to a unique $\Delta \in D$. To finish the proof, we split into cases. If $p = 2$, then each group $P$ appearing in the product has order 2, and so satisfies $N_{\Gamma_{g, U}}(P) = Z_{\Gamma_{g, U}}(P)$; in this case, the map in the statement of the lemma is an isomorphism for $i > \dim_{\bbR} X_G$ (not just injective).
    
    If $p$ is odd, then the index of $Z_{\Gamma_g,U}(P)$ in $N_{\Gamma_g,U}(P)$ is prime to $p$ (as it divides $2$). The map $H^\ast(N_{\Gamma_{g, U}}(P), k) \to H^\ast(Z_{\Gamma_{g, U}}, k)$ is therefore injective, giving the desired statement. 
\end{proof}

Combining these lemmas gives us the following proposition which, as we have already observed, implies the truth of Theorem \ref{thm_boundedness_of_good_dihedral_cohomology}. 
\begin{proposition}\label{prop_vanishing_of_high_degree_non_Eisenstein_cohomology}
Let the assumptions be as in Theorem~\ref{thm_boundedness_of_good_dihedral_cohomology}. 
Then the groups $H^i(A(U, k))_\ffrm$ are zero whenever $i > \dim_\bbR X_G$.
\end{proposition}
\begin{proof}
We spell out the details. Let $\frn \subset \cO_K$ be a nonzero
ideal such that $U(\frn) \subset U$, and let $K(p \frn) / K$ be the ray class field of level $p \frn$. 
Combining Lemmas~\ref{lem_semisimple}, \ref{lem_p_group_rank_1}, \ref{lem_X_Xdelta}, \ref{lem_acts_by_2_on_delta}, and \ref{lem_Farrell}, we see that if $v \nmid p \frn$ is any finite place of $K$ that splits in $K(p \frn)$, then $T_v - 2$ annihilates $H^i(A(U, k))$.
    
On the other hand, the image of $\overline{\rho}_\ffrm$ is an irreducible finite subgroup of $\GL_2(k)$. First assume that the projective image of $\overline{\rho}$ is nonabelian. Using the classification of irreducible subgroups of $\PGL_2(k)$, we can find $g \in [\im(\overline{\rho}), \im(\overline{\rho})]$ such that $\tr(g) \ne 2$. 
If the projective image of $\overline{\rho}$ is abelian, then again using the classification of irreducible subgroups of $\PGL_2(k)$, it is isomorphic to the Klein $4$-group and $p >2$. Then there is a nontrivial scalar element $g \in [\im(\overline{\rho}), \im(\overline{\rho})]$ with $\tr(g) \ne 2$ (since $p >2$). 
So in either case, we can find $\sigma \in G_{K(p \frn)}$ such that $\tr \overline{\rho}_\ffrm(\sigma) \neq 2$. By the Chebotarev density theorem, there exists a finite place $v \not\in S$ which splits in $K(p \frn)$ and such that $\tr \overline{\rho}_\ffrm(\Frob_v) \neq 2$, and consequently $T_v - 2$ is invertible on $H^i(A(U, k))_\ffrm$. 

The only possibility is therefore $H^i(A(U, k))_\ffrm = 0$.
\end{proof}

\section{An application of the Taylor--Wiles method}\label{sec_application_of_TW_method}

In this section we combine the results of \S \ref{sec_cohom} with the Taylor--Wiles method  in order to prove (under restrictive hypotheses) an $R = \bbT$ result for 2-adic Galois representations.

\subsection{Set-up for the main technical automorphy lifting result}\label{subsec_setup_for_main_ALT}

We use the setup of \S\ref{sec_cohom_general}, specialized to the case $G = \GL_2$. In particular, we are given a CM field $K$, a prime $p$, and a coefficient field $E \subset \overline{\bbQ}_p$. We suppose that the following conditions hold:
\begin{itemize}
\item The prime $p = 2$.
\item We are given $\overline{\rho} : G_K \to \GL_2(\overline{\bbF}_2)$, a continuous, absolutely irreducible representation of soluble image. 
\item We assume that $\overline{\rho}$ is unramified away from the 2-adic places of $K$ and that $\overline{\rho}$ is decomposed generic.
\item For each place $v \in S_2$, $\overline{\rho}|_{G_{K_v}}$ satisfies the hypotheses of \S \ref{subsubsec_ordinary_deformation_problems}. In particular, it is ramified.
\item There exists a place $v_0 | 2$ of $K$ such that the inclusion $K \hookrightarrow K_{v_0}$ induces a bijection on 2-power roots of unity. We write $2^m$ for the number of 2-power roots of unity in $K$.
\end{itemize}
We suppose given the following data:
\begin{itemize}
\item A finite set $R$ of finite places of $K$ stable under complex conjugation such that for each $v \in R$, $\overline{\rho}|_{G_{K_v}}$ is the trivial representation. We assume that $[K^+ : \bbQ] > 8 |R| + 10$.
\item An isomorphism $\iota : \overline{\bbQ}_p \cong \bbC$ and a cuspidal automorphic representation $\pi$ of $\GL_2(\bbA_K)$ which is cohomological of weight 0, $\iota$-ordinary, which satisfies $\overline{r_\iota(\pi)} \cong \overline{\rho}$, and which satisfies the following additional conditions:
\begin{itemize}
\item For each place $v \in R$, there exists a non-trivial character $\chi_v : \cO_{K_v}^\times \to \cO^\times$ of 2-power order such that $\iota^{-1} \pi_v \cong i_B^G \widetilde{\chi}_{v, 1} \otimes \widetilde{\chi}_{v, 2}$, where $\widetilde{\chi}_{v, 1}|_{\cO_{K_v}^\times} = \chi_{v}$ and $\widetilde{\chi}_{v, 2}|_{\cO_{K_v}^\times} = \chi_{v}^{-1}$ and $\chi_v \neq \chi_v^{-1}$.
\item For each place $v | 2$ of $K$, $\pi_v$ has an Iwahori-fixed vector.
\item Let $S = S_2 \cup R$. We assume that for each place $v \in S$, there exists an imaginary quadratic subfield of $K$ in which the residue characteristic of $v$ splits; and that for each place $v \not\in S$, $\pi_v$ is unramified. We assume as well that if $l$ is a rational prime which ramifies in $K$, then there exists an imaginary quadratic subfield of $K$ in which $l$ splits.
\end{itemize} 
\end{itemize}
If $c \geq b \geq 1$ are integers, then we let $U(b, c) = \prod_v U(b, c)_v \in \cJ^S$ be the open compact subgroup of $G^\infty$ defined as follows:
\begin{enumerate}
\item If $v \in S_2$ and $v \neq v_0$, then $U(b, c)_v = I_v(b, c) \cdot (\mu_{2^\infty}(K_{v}) \times \mu_{2^\infty}(K_{v}))$, where $\mu_{2^\infty}(K_{v}) \times \mu_{2^\infty}(K_{v})$ is viewed as a subgroup of the diagonal matrices in $\GL_2(K_v)$. Thus $U(b, c)_v$ is a subgroup of $I_v(1, c)$.
\item If $v = v_0$, then $U(b, c)_v = I_v(b, c) \cdot(\mu_{2^\infty}(K_{v_0}) \times 1)$.
\item If $v \in R$, then $U(b, c)_v = I_v$.
\item If $v \not\in S$, then $U(b, c)_v = \GL_2(\cO_{K_v})$.
\end{enumerate}
We note that $U(c, c) \subset U(1, c)$ is a normal subgroup and there is a canonical isomorphism 
\begin{multline*} U(1, c) / U(c, c) \cong \left( \prod_{\substack{v \in S_2\\v \neq v_0}}  (1 + \varpi_v \cO_{K_v}) / ( \mu_{2^\infty}(K_v), 1 + \varpi_v^c \cO_{K_v}) \right)^2 \\ \times (1 + \varpi_{v_0} \cO_{K_{v_0}}) / (\mu_{2^\infty}(K_{v_0}), 1 + \varpi_v^c \cO_{K_{v_0}}) \times (1 + \varpi_{v_0} \cO_{K_{v_0}}) / (1 + \varpi_{v_0}^c \cO_{K_{v_0}}). \end{multline*}
For any $c \geq 1$ we define $\Lambda_{1, c} = \cO[U(1, c) / U(c, c)]$, and $\Lambda_c$ to be the quotient of $\Lambda_{1, c}$ corresponding to the quotient group
\[ \left(\prod_{v \in S_2} (1 + \varpi_v \cO_{K_v}) / (\mu_{2^\infty}(K_v), 1 + \varpi_v^c \cO_{K_v}) \right)^2. \]
Let $c_0 \geq 1$ be the smallest integer such that there are no non-trivial roots of unity in $K_{v_0}$ which are congruent to 1 modulo $\varpi_{v_0}^{c_0}$. For any $c \geq c_0$, we define $\Lambda_{0, c} = \cO[U(c_0, c) / U(c, c)]$. We set $\Lambda_1 = \varprojlim_c \Lambda_{1, c}$, $\Lambda = \varprojlim_c \Lambda_c$, and $\Lambda_0 = \varprojlim_c \Lambda_{0, c}$. We observe that there are maps
\[ \Lambda_0 \hookrightarrow \Lambda_1 \twoheadrightarrow \Lambda, \]
that $\Lambda_0$ and $\Lambda$ are regular local $\cO$-algebras of dimension $1 + 2[K : \mathbb{Q}]$, and that the composite map $\Lambda_0 \to \Lambda$ is injective. The reason for introducing these three algebras is that the complexes we use will naturally be complexes of $\Lambda_1$-modules; however, we will see that the cohomology groups will be $\Lambda$-modules, and the complexes will be perfect only as complexes of $\Lambda_0$-modules. We need all of these properties.

If $x : \prod_{v \in R} k(v)^\times \to \cO^\times$ is a character which is trivial modulo $\varpi$, then we write $\cO(x)$ for the $\cO[U(1, 1)]$-module, free of rank 1 over $\cO$, on which $U(1, 1)$ acts through the projection $U(1, 1) \to \prod_{v \in R} I_v \to \prod_{v \in R}( k(v)^\times \times k(v)^\times)$ via the character $x \boxtimes x^{-1}$. We thus have a complex for any $c \geq b \geq 1$:
\[ A(U(1, c) / U(b, c), \cO(x)) \in \mathbf{D}(\Lambda_{1, b}), \]
which is equipped with an algebra homomorphism
\[ \bbT^S[ \mathsf{U}_2 ]  \to \End_{\mathbf{D}(\Lambda_{1, b})}( A(U(1, c) / U(b, c), \cO(x)) ), \]
where by definition $\mathsf{U}_2 = \prod_{v \in S_2} \mathsf{U}_v$ (and $\mathsf{U}_v = \mathsf{U}_{v, 1}$, in the notation of \S \ref{sec_Hecke_Gal_rep}). 
\begin{lemma}
For any $c' \geq c \geq 1$, pullback induces a $\mathbb{T}^S[ \mathsf{U}_2]$-equivariant morphism
\[ A(U(1,c) / U(c, c), \cO(x)) \to  A(U(1, c') / U(c, c'), \cO(x)) \]
in $\mathbf{D}(\Lambda_{1, c})$. Consequently, there is an induced morphism of $\mathsf{U}_2$-ordinary parts 
\[ A(U(1, c) / U(c, c), \cO(x))^\text{ord} \to  A(U(1, c') / U(c, c'), \cO(x))^\text{ord} \]
which is in fact an isomorphism. 
\end{lemma}
\begin{proof}
This is a standard calculation in Hida theory. See either \cite[\S 6.3]{Kha17} or \cite[\S 5.2]{10authors} for a proof in our context. 
\end{proof}
Let $\ffrm \subset \bbT^S( A(U(1, 1), k )^\text{ord} )$ be a maximal ideal. For any $c \geq b \geq 1$ and  $x$ as above there are canonical surjective homomorphisms
\[ \bbT^S( A(U(1, c) / U(c, c), \cO(x))^\text{ord} ) \to \bbT^S( A(U(1, 1), k )^\text{ord} ) \]
and
\[ \bbT^S( A(U(1, c) / U(c, c), \cO(x))^\text{ord} ) \to \bbT^S( A(U(b, c) / U(c, c), \cO(x))^\text{ord} )  \]
which induce bijections on maximal ideals. We write abusively $\ffrm$ for the corresponding maximal ideal of $\bbT^S( A(U(b, c) / U(c, c), \cO(x))^\text{ord} )$. This in turns allows us to define localizations
\[ A(U(1, c) / U(c, c), \cO(x))^\text{ord}_\ffrm,\,\, C(U(1, c) / U(c, c), \cO(x))^\text{ord}_\ffrm \]
with the property that for $c' \geq c \geq 1$, there is a canonical isomorphism
\[ C(U(1, c') / U(c', c'), \cO(x))^\text{ord}_\ffrm \otimes^{\bbL}_{\Lambda_{1, c'}} \Lambda_{1, c} \cong C(U(1, c) / U(c, c), \cO(x))^\text{ord}_\ffrm \]
in $\mathbf{D}(\Lambda_{1, c})$. It follows that there is a canonical isomorphism for any $c' \geq c \geq c_0$:
\[ C(U(c_0, c') / U(c', c'), \cO(x))^\text{ord}_\ffrm \otimes^{\bbL}_{\Lambda_{0, c'}} \Lambda_{0, c} \cong C(U(c_0, c) / U(c, c), \cO(x))^\text{ord}_\ffrm \]
in $\mathbf{D}(\Lambda_{0, c})$. 
\begin{lemma}
Let $c \geq c_0$, and suppose that $\ffrm$ is non-Eisenstein. Then  $C(U(c_0, c) / U(c, c), \cO(x))^\text{ord}_\ffrm$ is a perfect complex of $\Lambda_{0, c}$-modules.
\end{lemma}
\begin{proof}
In fact, this is true even before projecting to ordinary parts. It follows from Theorem \ref{thm_boundedness_of_good_dihedral_cohomology}, after observing that $U(c_0, c) \cap Z_G(K)$ contains no non-trivial 2-torsion. (This is the reason for defining $U(c_0, c)$ in the way that we did.)
\end{proof}
In the situation of the lemma we can, arguing as in \cite[\S 6]{Kha17}, pass to the limit to obtain, for any character $x : \prod_{v \in R} k(v)^\times \to \cO^\times$ which is trivial mod $\varpi$, a minimal complex $F_{\ffrm, x}$ of $\Lambda_0$-modules with the following properties:
\begin{itemize}
\item There is a $\Lambda_0$-algebra homomorphism $\bbT^S \otimes_\cO \Lambda_1 \to \End_{\mathbf{D}(\Lambda_0)}(F_{\ffrm, x})$.
\item For each $c \geq c_0$, there is a $\bbT^S$-equivariant isomorphism 
\begin{equation}\label{eqn_classicality} F_{\ffrm, x} \otimes_{\Lambda_0} \Lambda_{0, c} \cong C(U(c_0, c) / U(c, c), \cO(x))^\text{ord}_\ffrm
\end{equation}
 in $\mathbf{D}(\Lambda_{0, c})$.
\item There is an isomorphism 
\begin{equation}\label{eqn_homology_is_inverse_limit} H^\ast(F_{\ffrm, x}) \cong \plim_c H_{-\ast}^{U(c, c)}(\mathfrak{X}, \cO(x))^\text{ord}_\ffrm 
\end{equation}
of $\bbT^S \otimes_\cO \Lambda_1$-modules.
\end{itemize}
It is possible, as in the statement of \cite[Proposition 6.6]{Kha17}, to give a list of properties which characterize the complex $F_{\ffrm, x}$ uniquely up to unique isomorphism in $\mathbf{D}(\Lambda_0)$ (and hence uniquely up to isomorphism as a complex of $\Lambda_0$-modules). Since we won't need these properties here, we have elected not to list them.

We now apply this construction to a particular choice of maximal ideal $\ffrm$, coming from our fixed cuspidal automorphic representation $\pi$. Indeed, after possibly enlarging $E$, $\pi$ determines (using e.g.\ \cite[Theorem 2.4.10]{10authors}) a homomorphism
\[ \bbT^S( C(U(1, 1), \cO(\chi^{-1}))^\text{ord} ) \to E, \]
which sends $\mathsf{T}_v$ to its eigenvalue on $\iota^{-1} \pi_v^{U_v}$. The kernel of this homomorphism is contained in a unique maximal ideal $\ffrm \subset \bbT^S( C(U(1, 1), \cO(\chi^{-1})) )$, which is in the support of $H^\ast_{U(1, 1)}(\mathfrak{X}, k)^\text{ord}$ and which satisfies $\overline{\rho}_\ffrm \cong \overline{r_\iota(\pi)}$. By hypothesis, $\overline{\rho}_\ffrm$ is absolutely irreducible, so the above discussion defines complexes $F_{\ffrm, x} \in \mathbf{D}(\Lambda_0)$. We define $\bbT_x$ to be the image of the homomorphism
\[ \bbT^S \otimes_\cO \Lambda_1 \to \End_{\mathbf{D}(\Lambda_0)}(F_{\ffrm, x}). \] 
Then $\bbT_x$ is a finite local $\Lambda_1$-algebra and we write $\ffrm_x$ for its unique maximal ideal. 
\begin{theorem}\label{thm_existence_of_Galois_base_case}
The structural map $\Lambda_1 \to \bbT_x$ factors through the quotient $\Lambda_1 \to \Lambda$. We can find an ideal $J_x \subset \bbT_x$ and a continuous homomorphism $\rho_x : G_K \to \GL_2(\bbT_x / J_x)$ satisfying the following conditions:
\begin{enumerate}
\item The ideal $J_x$ is nilpotent: there exists an integer $\delta = \delta(K) \geq 1$ depending only on $K$ such that $J_x^\delta = 0$.
\item We have $\rho_x \text{ mod }\ffrm_x = \overline{\rho}_\ffrm$.
\item For each place $v\not\in S$ of $K$, $\rho_x|_{G_{K_v}}$ is unramified and we have
\[ \det(X - \rho_x(\Frob_v)) = X^2 - \mathsf{T}_v X + q_v \mathsf{S}_v. \]
\item For each place $v \in S_2$, $\rho_x|_{G_{K_v}}$ is of type $\cD_v^\text{ord}$ (see \S \ref{subsubsec_ordinary_deformation_problems}).
\item For each place $v \in R$, $\rho_x|_{G_{K_v}}$ is of type $\cD_v^{x_v^{-1}}$ (see \S \ref{subsubsec_level_raising_deformations}).
\end{enumerate}
\end{theorem}
\begin{proof}
If $c \geq c_0$, let $\bbT_{x, c}$ denote the image of the map 
\[ \bbT^S \otimes_\cO \Lambda_1 \to\End_\cO(H_{-\ast}^{U(c, c)}(\mathfrak{X}, \cO(x))^\text{ord}_\ffrm). \]
Then there is a natural ring homomorphism
\begin{equation}\label{eqn_homomorphism_of_Hecke_algebras} \bbT_x \to \prod_{c \geq c_0} \bbT_{x, c}. 
\end{equation}
Suppose that there exists an integer $\delta_0  \geq 1$ depending only on $K$ and for each $c \geq c_0$ an ideal $J_{x, c} \subset \bbT_{x, c}$ satisfying $J_{x, c}^{\delta} \neq 0$ and a representation $\rho_{x, c} : G_K \to \GL_2(\bbT_{x, c} / J_{x, c})$ satisfying the analogues of the conditions in the statement of the Theorem. Then the Theorem itself follows, with $\delta = \delta_0 \cdot (\dim_\bbR X_{\GL_2} + 1)$, as we now explain. Let $J_x \subset \bbT_x$ denote the kernel of the induced $\bbT_x \to \prod_{c \geq c_0} \bbT_{x, c} / J_{x, c}$. For any elements $y_1, \dots, y_{\delta_0} \in J_x$, we find that the product $y_1 \dots y_{\delta_0}$ acts as 0 on each group $H_{-\ast}^{U(c, c)}(\mathfrak{X}, \cO(x))^\text{ord}_\ffrm$, and therefore as 0 on $H^\ast(F_{\mathfrak{m}, x})$ (using the isomorphism \ref{eqn_homology_is_inverse_limit}). Then \cite[Lemma 2.2.3]{10authors} implies that for any elements $x_1, \dots, x_\delta \in J_x$, the product $x_1 \dots x_\delta$ is zero in $\End_{\mathbf{D}(\Lambda_0)}(F_{\mathfrak{m}, x})$, or in other words that $J_x^\delta = 0$. By e.g.\ glueing determinants as in \cite[Example 2.32]{Che14}, we obtain a representation $\rho_x : G_K \to \GL_2(\bbT_x / J_x)$ which, after projection to  any $\GL_2(\bbT_{x, c} / J_{x, c})$ ($c \geq c_0$), is strictly equivalent to $\rho_{x, c}$ (look at characteristic polynomials of Frobenius elements). This $\rho_x$ satisfies the requirements 1.\ -- 3.\ of the Theorem; properties 4.\ and 5.\ follow from the analogous properties for the representations $\rho_{x, c}$ (that we will soon justify) and the representability of the local deformation problems $\cD_v^\text{ord}$ and $\cD_v^{x_v^{-1}}$.

We now establish the existence of the representations $\rho_{x, c}$. For $c \geq c_0$, let $\widetilde{\bbT}_{x, c}$ denote the image of the homomorphism
\[ \bbT^S[ \{ \mathsf{U}_v \}_{v \in S_2}] \otimes_\cO \Lambda_1 \to \End_\cO(H_{-\ast}^{U(c, c)}(\mathfrak{X}, \cO(x))^\text{ord}_\ffrm). \] 
It contains $\bbT_{x, c}$. For each place $v \in S_2$, let $\phi_{v, 1}, \phi_{v, 2} : I_{K_v}^\text{ab} \to \Lambda_1^\times$ be the tautological characters. It follows from Theorem \ref{thm:local-global-GLn-ord} that $\widetilde{\bbT}_{x, c}$ has a unique maximal ideal $\ffrm_{\widetilde{\bbT}_{x, c}}$ and that we can find for each $v \in S_2$ characters $\widetilde{\psi}_{v, 1}, \widetilde{\psi}_{v, 2} : G_{F_v} \to \widetilde{\bbT}_{x, c}^\times$, a nilpotent ideal $\widetilde{J}_{x, c} \subset \widetilde{\bbT}_{x, c}$ satisfying $\widetilde{J}_{x, c}^\delta = 0$, where $\delta$ is an integer depending only on $K$, and a continuous representation $\rho_{x, c} : G_{K, S} \to \GL_2(\widetilde{\bbT}_{x, c})$ satisfying the following conditions:
\begin{itemize}
    \item We have $\rho_{x, c} \text{ mod }\ffrm_{\widetilde{\bbT}_{x, c}} = \overline{\rho}_\ffrm$.
    \item For each place $v\not\in S$ of $K$, $\rho_{x, c}|_{G_{K_v}}$ is unramified and we have
    \[ \det(X - \rho_{x, c}(\Frob_v)) = X^2 - \mathsf{T}_v X + q_v \mathsf{S}_v. \]
    \item For each place $v \in S_2$, the restrictions $\widetilde{\psi}_{v, 1}|_{I_{F_v}}$ and $\widetilde{\psi}_{v, 2}|_{I_{F_v}}$ are equal to the respective pushforwards of the characters $\phi_{v, 1}$ and $\phi_{v, 2}$ to $\widetilde{\bbT}_{x, c} / \widetilde{J}_{x, c}$. 
    \item For each place $v \in S_2$, and for each $g_v \in G_{K_v}$, the characteristic polynomial of $\rho_{x, c}(g_v)$ equals $(X - \widetilde{\psi}_{v, 1}(g_v))(X - \widetilde{\psi}_{v, 2}(g_v))$.
    \item For each place $v \in S_2$, and for each $g_1, g_2 \in G_{K_v}$, we have
    \[ (\rho_{x, c}(g_1) - \widetilde{\psi}_{v, 1}(g_1))(\rho_{x, c}(g_2) - \widetilde{\psi}_{v, 2}(g_2)) = 0 \]
    in $M_2(\widetilde{\bbT}_{x, c} / \widetilde{J}_{x, c})$.
    \item For each place $v \in R$, $\rho_{x, c}|_{G_{K_v}}$ is of type $\cD_v^{x_v^{-1}}$.
    \end{itemize}
Let $J_{x, c} = \widetilde{J}_{x, c} \cap \bbT_{x, c}$. By Carayol's lemma \cite[Lemma 2.1.10]{Clo08}, we can moreover assume that $\rho_{x, c}$ takes values in $\GL_2(\bbT_{x, c} / J_{x, c})$. We need to show that for each place $v \in S_2$, $\rho_{x, c}|_{G_{K_v}}$ is of type $\cD_v^\text{ord}$. This will follow from Proposition \ref{prop_characterization_of_ordinary_deformations_by_pseudocharacter_relations} if we can show that for each $v \in S_2$, the map $\Lambda_{1, v} \to \bbT_x$ factors through the quotient $\Lambda_{1, v} \to \Lambda$; equivalently, if the characters $\phi_{v, 1} \circ \Art_{K_v}, \phi_{v, 2} \circ \Art_{K_v} : \cO_{K_v}^\times \to \bbT_x^\times$ are trivial on 2-power roots of unity. If $v \neq v_0$ this is true by definition of $\bbT_x$ (i.e.\ by choice of the level subgroups $U(c_0, c)_v$). If $v = v_0$, then it is true for $\phi_{v, 1} \circ \Art_{K_v}$ for similar reasons. 

It remains to observe that if $\zeta \in K_{v_0}$ is a 2-power root of unity, then the image of $\phi_{v, 2}(\Art_{K_v}(\zeta))$ is trivial. To show this, it is enough to show that if $\eta_{v_0} = \diag(1, \zeta) \in \GL_2(\cO_{K_{v_0}})$ then for any $c \geq c_0$, the image of the diamond operator $\langle \eta_{v_0} \rangle$ in $\End_{\mathbf{D}(\Lambda_{0, c})}(C(U(c_0, c) / U(c, c), \cO(x)))$ is trivial. 
Recall that the complex $C(U(c_0, c) / U(c, c))$ together with the action of $U(1,c)/U(c,c)$ is computed as follows. 
We let $C$ be the complex of singular chains on $\mathfrak{X}_G = \GL_2(K) \backslash \GL_2(\bbA_K)/U_\infty \bbR^\times$ with coefficients in $\cO$ and choose a quasi-isomorphism $P \to C$ with $P$ a bounded above complex of projective $\cO[U(1, c)]^\text{op}$-modules. 
Note that $P$ is also a bounded above complex of projective $\cO[U(c_0, c)]^\text{op}$-modules. 
Then $C(U(c_0, 0) / U(c, c)) = (P \otimes_{\cO} \cO(x))_{U(c,c)}$ and $\langle \eta_{v_0} \rangle$ is induced by the action of $\eta_{v_0}$ on $P$ and $\cO(x)$. 

Let $\zeta = \diag(\zeta,\zeta) \in \GL_2(K)$ and write $\zeta = \zeta^\infty \zeta_\infty$ with $\zeta^\infty \in G^\infty$ and $\zeta_\infty \in G_\infty$. 
Note that $(\zeta^\infty)^{-1}\eta_{v_0} \in U(c,c)$ and acts trivially on $\cO(x)$, so it acts trivially on $(P \otimes_{\cO} \cO(x))_{U(c,c)}$. 
Since $\zeta^\infty$ also acts trivially on $\cO(x)$, to conclude it suffices to show that multiplication by $\zeta^\infty$ on $P$ is homotopic to the identity. 
This follows from \cite[TAG064A]{stacks-project} and the fact that multiplication by $\zeta^\infty = \zeta \zeta_\infty^{-1}$ on $C$ equals the identity since $\zeta_\infty \in U_\infty$ and $\zeta \in Z_G(K)$.
\end{proof}
\begin{corollary}\label{cor_R_to_T}
Define a global deformation problem
\[ \cS_x = (K,\overline{\rho}, S, \{ \Lambda_v \}_{v \in S_2} \cup \{ \cO \}_{v \in R}, \{ \cD^{x_v}_v \}_{v \in R} \cup \{ \cD_v^\text{ord} \}_{v \in S_2} \}). \]
Then there is a unique surjective morphism of $\Lambda$-algebras $R_{\cS_x} \to \bbT_{x} / J_x$ with the property that for all $v \not\in S$, $\tr \rho_{\cS_x}(\Frob_v) \mapsto \mathsf{T}_v$.
\end{corollary}
Here is the main theorem of \S \ref{sec_application_of_TW_method}.
\begin{theorem}\label{thm_application_of_TW_method}
The morphism $R_{\cS_1} \to \bbT_1 / J_1$ has nilpotent kernel.
\end{theorem}
\begin{corollary}\label{cor_application_of_taylor_wiles}
Suppose given a continuous representation $\tau : G_K \to \GL_2(\overline{\bbQ}_2)$ satisfying the following conditions:
\begin{enumerate}
\item $\tau$ is unramified outside $S_2 \cup R$.
\item For each place $v \in R$, $\tau|_{G_{K_v}}$ is unipotently ramified.
\item For each place $v \in S_2$, $\tau|_{G_{K_v}}$ is ordinary of weight 0 and $\mathrm{WD}(\tau|_{G_{K_v}})$ is unipotently ramified.
\item There is an isomorphism $\overline{\tau} \cong \overline{\rho}$. 
\end{enumerate}
Then $\tau$ is automorphic: there exists an $\iota$-ordinary, cuspidal automorphic representation $\sigma$ of $\GL_2(\bbA_K)$ of weight 0 satisfying the following conditions:
\begin{enumerate}
\item $\sigma$ is everywhere unipotently ramified, and is unramified outside $S_2 \cup R$.
\item There is an isomorphism $r_\iota(\sigma) \cong \tau$.
\end{enumerate}
\end{corollary}
\begin{proof}[Proof of Corollary \ref{cor_application_of_taylor_wiles}]
This follows from Theorem \ref{thm_application_of_TW_method}, the classicality property (\ref{eqn_classicality}), and \cite[Theorem 2.4.10]{10authors}.
\end{proof}
The proof of Theorem \ref{thm_application_of_TW_method} will be given in \S \ref{subsec_proof_of_TW_method} below, using a patching argument. We must first introduce some auxiliary objects (which depend on auxiliary sets of Taylor--Wiles places) which will play a key role in the proof.

Let $Q$ be a Taylor--Wiles datum for $\cS_1$. We recall that by definition this means a finite set $Q$ of finite places of $K$, disjoint from $S$, together with the datum for each $v \in Q$ of a pair $\alpha_v, \beta_v : G_{K_v} \to k^\times$ of unramified characters such that $\overline{\rho}_\ffrm \cong \alpha_v \oplus \beta_v$. We assume that each place of $Q$ has residue characteristic split in some imaginary quadratic subfield of $K$, and moreover that $Q$ has level $N + m$, for some $N \geq 0$; by definition, this means that for each $v \in Q$, $q_v \equiv 1 \text{ mod }2^{N+m}$.

For any choice of $x$, we can define the auxiliary deformation problem
\[ \cS_{x, Q} = (K,\overline{\rho}, S, \{ \Lambda_v \}_{v \in S_2} \cup \{ \cO \}_{v \in R \cup Q}, \{ \cD^{x_v}_v \}_{v \in R} \cup \{ \cD_v^\text{ord} \}_{v \in S_2} \cup \{ \cD_v^{\text{TW}, N} \}_{v \in Q} \}), \]
together with a structure on $R_{\cS_x, Q}$ of $\cO[\Delta_Q]$-algebra, where $\Delta_Q = \prod_{v \in Q} (k(v)^\times \times k(v)^\times) / (2^{N})$. For any $c \geq b \geq 1$, we define subgroups $U_{Q, 1}(b, c) \subset U_{Q, 0}(b, c) \subset U(b, c)$ as follows:
\begin{itemize}
\item $U_{Q, 1}(b, c) = \prod_v U_{Q, 1}(b, c)_v$ and $U_{Q, 0}(b, c) = \prod_v U_{Q, 0}(b, c)_v$.
\item If $v \not\in Q$, then $U_{Q, 1}(b, c)_v = U_{Q, 0}(b, c)_v = U(b, c)_v$.
\item If $v \in Q$, then $U_{Q, 0}(b, c)_v = I_v$ and $U_{Q, 1}(b, c)_v$ is the kernel of the natural homomorphism $I_v \to (k(v)^\times \times k(v)^\times) / 2^{N}$.
\end{itemize}
Then $U_{Q, 1}(b, c) \subset U_{Q, 0}(b, c)$ is a normal subgroup and we can identify the quotient $U_{Q, 0}(b, c) / U_{Q, 1}(b, c)$ with $\Delta_Q$.

In what follows, we write $\bbT^{S \cup Q}_Q = \bbT^{S \cup Q} \otimes_\cO \otimes_{v \in Q} T_v$, where $T_v$ is the ring associated to places $v \in R$ just before the statement of Theorem \ref{thm:local-global-GLn-ord}. We observe that $\bbT^{S \cup Q}_Q$ has a natural structure of $\cO[\Delta_Q]$-algebra.
\begin{lemma}\label{lemma:max_ideal_with_TW_level}
	\begin{enumerate}
		\item 	There exists a (unique) maximal ideal $\ffrm_{Q, 0} \subset \bbT^{S \cup Q}(A(U_{Q, 0}(1, 1), k)^\text{ord})$ such that $\overline{\rho}_{\ffrm_{Q, 0}} \cong \overline{\rho}_\ffrm$. The trace map
		\[ H^\ast(A(U_{Q, 0}(1, 1), k)^\text{ord})_{\ffrm_{Q, 0}} \to H^\ast(A(U(1, 1), k)^\text{ord})_{\ffrm} \]
		is surjective.
		\item Let $\ffrm_{Q, 0, \alpha} \subset \bbT^{S \cup Q}_Q(A(U_{Q, 0}(1, 1,), k)^\text{ord})$ denote the ideal generated by $\ffrm_{Q, 0}$ and the elements $\mathsf{U}_v - \alpha_v(\Frob_v)$, $u - 1$ \($v \in Q$, $u \in (\cO_{K_v}^\times)^2 / (1 + \varpi_v \cO_{K_v})^2 \subset T_v^\times$\). Then $\ffrm_{Q, 0, \alpha}$ is a maximal ideal with residue field $k$ and the composite map
		\[ H^\ast(A(U_{Q, 0}(1, 1), k)^\text{ord})_{\ffrm_{Q, 0, \alpha}} \subset H^\ast(A(U_{Q, 0}(1, 1), k)^\text{ord})_{\ffrm_{Q, 0}} \to H^\ast(A(U(1, 1), k)^\text{ord})_{\ffrm} \]
		(inclusion followed by trace) is an isomorphism. 
	\end{enumerate}
\end{lemma}
\begin{proof}
	We just give the proof in the case that $Q$ contains a single place $v$.  The general case can be proved in the same way. Let $\ffrm'$ denote the pullback of $\ffrm$ to $\bbT^{S \cup Q}$. Let $\eta = \diag{(\varpi_v, 1)}$ and $U_Q^0(1,1) = \eta U_{Q, 0}(1,1) \eta^{-1}$. We consider the pullback
	\[ \pi_1^\ast : H^\ast(A(U(1, 1), k)^\text{ord})_{\ffrm'} \to H^\ast(A(U_{Q, 0}(1, 1), k)^\text{ord})_{\ffrm'}, \]
	the isomorphism
	\[ \eta^\ast : H^\ast(A(U_{Q, 0}(1, 1), k)^\text{ord})_{\ffrm'} \to H^\ast(A(U_Q^0(1, 1), k)^\text{ord})_{\ffrm'} \]
	and the trace
	\[ \pi_{2, \ast} : H^\ast(A(U_Q^0(1, 1), k)^\text{ord})_{\ffrm'} \to H^\ast(A(U(1, 1), k)^\text{ord})_{\ffrm'} \]
	Then $\mathsf{T}_v = \pi_{2, \ast} \circ \eta^\ast \circ \pi_1^\ast$,
by definition. The definition of Taylor--Wiles datum implies that $\mathsf{T}_v \text{ mod }\ffrm = \alpha_v(\Frob_v) + \beta_v(\Frob_v) \neq 0$, and hence that $\mathsf{T}_v$ acts invertibly on
	\[ H^\ast(A(U(1, 1), k)^\text{ord})_{\ffrm'} = H^\ast(A(U(1, 1), k)^\text{ord})_{\ffrm}, \]
	the equality of these two spaces being proved in the same way as in \cite[Lemma 6.20]{Kha17}. It follows that $\pi_1^\ast$ is injective and $\pi_{2, \ast}$ is surjective, and consequently that $\ffrm'$ is the pullback of a uniquely determined maximal ideal $\ffrm_{Q, 0}$ of $\bbT^{S \cup Q}(A(U_{Q, 0}(1, 1), k)^\text{ord})$. This establishes the first part of the lemma. To prove the second part, we follow the proof of \cite[Lemma 3.5]{Cal17}. 

Let $\eta' = \diag{(1, \varpi_v)}$; so we also have $\mathsf{T}_v = \pi_{1, \ast} \circ (\eta')^\ast \circ \pi_2^\ast$. We define maps
\[ \psi : H^\ast(A(U(1, 1), k)^\text{ord})_{\ffrm'}^2 \to H^\ast(A(U_{Q, 0}(1, 1), k)^\text{ord})_{\ffrm'} \]
and
\[ \phi : H^\ast(A(U_{Q, 0}(1, 1), k)^\text{ord})_{\ffrm'} \to H^\ast(A(U(1, 1), k)^\text{ord})_{\ffrm'}^2 \]
by $\psi = (\pi_1^\ast, (\eta')^\ast \circ \pi_2^\ast)$ and $\phi = \begin{psmallmatrix} \pi_{1, \ast} \\ \pi_{2, \ast} \circ \eta^\ast \end{psmallmatrix}$. 
The composite $\phi \circ \psi$ is given by
\[ \phi \circ \psi = \begin{pmatrix} q_v + 1 & \mathsf{T}_v \\ \mathsf{T}_v & \mathsf{S}_v(q_v + 1) \end{pmatrix}. \]
The determinant
\begin{align*} \det(\phi \circ \psi) &= \mathsf{S}_v(q_v + 1)^2 - \mathsf{T}_v^2 \\ 
& \equiv (\alpha_v(\Frob_v) + \beta_v(\Frob_v))^2 \mod {\ffrm'}. \end{align*}
is nonzero modulo $\ffrm'$ by the definition of a Taylor--Wiles datum, so $\phi \circ \psi$ is an isomorphism and $\psi$ identifies $H^\ast(A(U(1, 1), k)^\text{ord})_{\ffrm'}^2$ with a direct summand of $H^\ast(A(U_{Q, 0}(1, 1), k)^\text{ord})_{\ffrm'}$. 

A standard computation of the double cosets $[\GL_2(\cO_{K_v}) \eta \GL_2(\cO_{K_v})]$ and $[I_v \eta I_v]$ shows that $\pi_1^\ast \circ \mathsf{T}_v = \mathsf{U}_v \circ \pi_1^\ast + (\eta')^\ast \circ \pi_2^\ast$ and that $\mathsf{U}_v \circ (\eta')^\ast \circ \pi_2^\ast = \pi_1^\ast \circ (q_v\mathsf{S}_v)$. 
So the action of $\mathsf{U}_v$ on $H^\ast(A(U(1, 1), k)^\text{ord})_{\ffrm'}^2$ is given by the matrix
\begin{equation}\label{Up-matrix} \begin{pmatrix} \mathsf{T}_v & q_v \mathsf{S}_v \\ -1 & 0 \end{pmatrix}, \end{equation}
which has characteristic polynomial $X^2 - \mathsf{T}_v X + q_v\mathsf{S}_v \equiv (X - \alpha_v(\Frob_v))(X - \beta_v(\Frob_v)) \bmod {\ffrm'}$. 
Since $\alpha_v(\Frob_v) \ne \beta_v(\Frob_v)$, by Hensel's Lemma we can find $A_v, B_v \in \bbT^{S \cup Q}(H^\ast(A(U_{Q, 0}(1, 1), k)^\text{ord})_{\ffrm'})$ such that $(\mathsf{U}_v - A_v)(\mathsf{U_v} - B_v) = 0$ on $\im(\psi)$. 
Then $\mathsf{U}_v - B_v$ is a projector (up to unit) from $\im(\psi)_{\ffrm'}$ to $\im(\psi)_{\ffrm_{Q, 0, \alpha}}$. 
Using \eqref{Up-matrix}, it is easy to see that $(\mathsf{U}_v - B_v) \circ \pi_1^\ast : H^\ast(A(U(1, 1), k)^\text{ord})_{\ffrm'} \to \im(\psi)$ is injective, and it intersects the kernel of $\mathsf{U}_v - B_v$ trivially.
The similar claim holds with $B_v$ replaced by $A_v$ by symmetry, so then by counting dimensions we deduce that there is an isomorphism of $\bbT^{S \cup Q}$-modules
\[ H^\ast(A(U(1, 1), k)^\text{ord})_{\ffrm'} \cong \im(\psi)_{\ffrm_{Q, 0, \alpha}}. \]

We thus have a decomposition of $\bbT^{S \cup Q}$-modules
\[ H^\ast(A(U_{Q, 0}(1, 1), k)^\text{ord})_{\ffrm_{Q, 0, \alpha}} = H^\ast(A(U(1, 1), k)^\text{ord})_{\ffrm'} \oplus \ker(\phi). \]
To finish the proof, it suffices to show that $\ker(\phi)[\ffrm_{Q, 0, \alpha}] = 0$. 
Another double coset computation give the formula $\pi_1^\ast \circ \pi_{2, \ast} \circ \eta^\ast = \mathsf{U}_v + (w \eta)^\ast$, where $w = \begin{psmallmatrix} 0 & 1 \\ 1 & 0 \end{psmallmatrix}\in \GL_2(\cO_{K_v})$. 
We showed above that $\pi_1^\ast$ is injective on $H^\ast(A(U(1, 1), k)^\text{ord})_{\ffrm'}$. 
So if $f \in \ker(\phi)$, we have $\mathsf{U}_v(f) = -(w \eta)^\ast (f)$. 
On the other hand, if $f$ is annihilated by $\ffrm_{Q, 0, \alpha}$, then $\mathsf{U}_v(f) = \alpha_v(\Frob_v) f$ and $q_v \mathsf{S}_v(f)= \alpha_v(\Frob_v)\beta_v(\Frob_v) f$. 
Thus, if $f \in \ker(\phi)[\ffrm_{Q, 0, \alpha}]$,
\[ \alpha_v(\Frob_v)^2 f = \mathsf{U}_v^2(f) = ((w\eta)^2)^\ast(f) = \mathsf{S}_v (f) = q_v^{-1}\alpha_v(\Frob_v)\beta_v(\Frob_v) f. \]
By the definition of a Taylor--Wiles datum, this forces $f = 0$, concluding the proof. 
\end{proof}
There are surjective algebra homomorphisms
\[ \bbT^{S \cup Q}( A( U_{Q, 0}(1, 1) / U_{Q, 1}(1, 1), k )^\text{ord} ) \to \bbT^{S \cup Q}( A(U_{Q, 0}(1, 1), k )^\text{ord} ), \]
and
\[ \bbT^{S \cup Q}_Q( A( U_{Q, 0}(1, 1) / U_{Q, 1}(1, 1), k )^\text{ord} ) \to \bbT^{S \cup Q}_Q( A(U_{Q, 0}(1, 1), k )^\text{ord} ), \]
and we write 
\[ \ffrm_{Q, 1} \subset \bbT^{S \cup Q}( A( U_{Q, 0}(1, 1)/U_{Q, 1}(1, 1), k )^\text{ord} ) \]
 for the pullback of $\ffrm_{Q, 0}$ and 
 \[ \ffrm_{Q, 1, \alpha} \subset \bbT^{S \cup Q}_Q( A( U_{Q, 0}(1, 1)/U_{Q, 1}(1, 1), k )^\text{ord} ) \]
  for the pullback of $\ffrm_{Q, 0, \alpha}$. We also write $\ffrm_{Q, 1, \alpha}$ abusively for the pullback of $\ffrm_{Q, 1, \alpha}$ to any of the algebras $\bbT^{S \cup Q}_Q( A(U_{Q, 0}(1, c) / U_{Q, 1}(c, c), \cO(x)))^\text{ord}$. For any $c \geq 1$, there are $\bbT^{S \cup Q}$-equivariant isomorphisms 
\[ C(U_{Q, 0}(1, c) / U_{Q, 1}(c, c), \cO(x))^\text{ord}_{\ffrm_{Q, 1, \alpha}} \otimes^\bbL_{\Lambda_c[\Delta_Q]} \Lambda_c \cong C(U_{Q, 0}(1, c) / U_{Q, 0}(c, c), \cO(x))^\text{ord}_{\ffrm_{Q, 0, \alpha}} \cong C(U(1, c) / U(c, c), \cO(x))^\text{ord}_{\ffrm}. \]
In exactly the same way as before, a limiting process now gives rise to a minimal complex $F_{\ffrm, x, Q}$ of $\Lambda_0$-modules with the following properties:
\begin{itemize}
\item There is a $\Lambda_1[\Delta_Q]$-algebra homomorphism $\bbT^{S \cup Q} \otimes_\cO \Lambda_1[\Delta_Q] \to \End_{\mathbf{D}(\Lambda_0[\Delta_Q])}(F_{\ffrm, x, Q})$.
\item For each $c \geq c_0$, there is a $\bbT^{S \cup Q} \otimes_\cO \Lambda_1$-equivariant isomorphism 
\[ F_{\ffrm, x, Q} \otimes_{\Lambda_0[\Delta_Q]} \Lambda_{0, c}[\Delta_Q] \cong C(U_{Q, 0}(c_0, c) / U_{Q, 1}(c, c), \cO(x))^\text{ord}_{\ffrm_{Q, 1, \alpha}} \]
 in $\mathbf{D}(\Lambda_{0, c}[\Delta_Q])$.
\item There is an isomorphism $H^\ast(F_{\ffrm, x, Q}) \cong \plim_c H^\ast(C(U_{Q, 1}(c, c), \cO(x))^\text{ord}_{\ffrm_{Q, 1, \alpha}})$ of $\bbT^{S \cup Q} \otimes_\cO \Lambda_1[\Delta_Q]$-modules.
\item There is an isomorphism $F_{\ffrm, x, Q} \otimes_{\Lambda_0[\Delta_Q]} \Lambda_0 \cong F_{\ffrm, x}$ of complexes of $\Lambda_0$-modules which becomes  $\bbT^{S \cup Q} \otimes_\cO \Lambda_1$-equivariant when we consider the corresponding morphism in $\mathbf{D}(\Lambda_0)$.
\end{itemize}
 We define $\bbT_{x, Q}$ to be the image of the homomorphism
\[ \bbT^{S \cup Q} \otimes_\cO \Lambda_1[\Delta_Q] \to \End_{\mathbf{D}(\Lambda_0[\Delta_Q])}(F_{\ffrm, x, Q}). \] 
Then $\bbT_{x, Q}$ is a finite local $\Lambda_1[\Delta_Q]$-algebra and we write $\ffrm_{x, Q}$ for its unique maximal ideal. 
\begin{theorem}\label{thm_existence_of_Galois_auxiliary_case}
The structural map $\Lambda_1 \to \bbT_{x, Q}$ factors through the quotient $\Lambda_1 \to \Lambda$. We can find an ideal $J_{x, Q} \subset \bbT_{x, Q}$ and a continuous homomorphism $\rho_{x, Q} : G_K \to \GL_2(\bbT_{x, Q} / J_{x, Q})$ satisfying the following conditions:
\begin{enumerate}
\item The ideal $J_{x, Q}$ is nilpotent: there exists an integer $\delta = \delta(K) \geq 1$ (same $\delta$ as in Theorem \ref{thm_existence_of_Galois_base_case}) depending only on $K$ such that $J_{x, Q}^\delta = 0$.
\item We have $\rho_{x, Q} \text{ mod }\ffrm_{x, Q} = \overline{\rho}_\ffrm$.
\item For each place $v\not\in S$ of $K$, $\rho_{x, Q}|_{G_{K_v}}$ is unramified and we have
\[ \det(X - \rho_{x, Q}(\Frob_v)) = X^2 - \mathsf{T}_v X + q_v \mathsf{S}_v. \]
\item For each place $v \in S_2$, $\rho_{x, Q}|_{G_{K_v}}$ is of type $\cD_v^\text{ord}$ (see \S \ref{subsubsec_ordinary_deformation_problems}).
\item For each place $v \in R$, $\rho_{x, Q}|_{G_{K_v}}$ is of type $\cD_v^{x_v^{-1}}$ (see \S \ref{subsubsec_level_raising_deformations}).
\item For each $v \in Q$, $\rho_{x, Q}|_{G_{K_v}}$ is of type $\cD_v^{\text{TW}, N}$ (se \S \ref{subsubsec_taylor_wiles_deformation_problems}). The two induced $\cO[\Delta_Q]$-structures on $\bbT_{x, Q} / J_{x, Q}$ (i.e. the canonical one and the one arising from the existence of $\rho_{x, Q}$ as in \S \ref{subsubsec_taylor_wiles_deformation_problems}) coincide.
\end{enumerate}
\end{theorem}
\begin{proof}
The proof is essentially the same as that of Theorem \ref{thm_existence_of_Galois_base_case}, replacing $R$ by $R \cup Q$ (see also the proof of \cite[Proposition 6.5.11]{10authors}). We note that the reason for using the problem $\cD_v^{\text{TW}, N}$ is that this is exactly what forces the map $\Lambda_1 \to \bbT_{x, Q}$ to factor through the quotient $\Lambda_1 \to \Lambda$.
\end{proof}
\begin{corollary}\label{cor_R_to_T_TW_augmented}
There is a unique surjective morphism of $\Lambda[\Delta_Q]$-algebras $R_{\cS_{x, Q}} \to \bbT_{x, Q} / J_{x, Q}$ with the property that for all $v \not\in S \cup Q$, $\tr \rho_{\cS_{x, Q}}(\Frob_v) \mapsto \mathsf{T}_v$.
\end{corollary}

\subsection{The proof of Theorem \ref{thm_application_of_TW_method}}\label{subsec_proof_of_TW_method}

We continue with the notation and assumptions of the previous section. We set $C_1 = F_{\ffrm, 1}$ and $C_\chi = F_{\ffrm, \chi}$. We fix representatives $\rho_{\cS_1}$, $\rho_{\cS_\chi}$ for the universal deformations of type $\cS_1$ and $\cS_\chi$ that are the same modulo $(\varpi)$. If $T$ is any finite set of finite places of $K$, then these representatives determine augmentations $R_{\cS_1}^T \to R_{\cS_1}$ and $R_{\cS_\chi}^T \to R_{\cS_\chi}$ of the universal $T$-framed deformation rings which are compatible with the canonical isomorphisms $R_{\cS_1}^T / (\varpi) \cong R_{\cS_\chi}^T / (\varpi)$ and $R_{\cS_1} / (\varpi) \cong R_{\cS_\chi} / (\varpi)$.

We have constructed (Corollary \ref{cor_R_to_T}) surjective $\Lambda$-algebra morphisms $R_{\cS_1}  \to \bbT_1 / J_1$ and $R_{\cS_\chi} \to \bbT_\chi / J_\chi$, where the ideals $J_1, J_\chi$ are nilpotent. This allows us to think of $\Spec \bbT_1 \subset \Spec R_{\cS_1}$ and $\Spec \bbT_\chi \subset \Spec R_{\cS_\chi}$ as closed subspaces, which have equal intersection with the subspace $\Spec R_{\cS_1} / (\varpi) = \Spec R_{\cS_\chi} / (\varpi)$ (apply Lemma \ref{lem_support_and_derived_tensor_product} to the $\bbT^{S}$-equivariant isomorphism $C_\chi / (\varpi) \cong C_1 / (\varpi)$). We will say that a prime of $R_{\cS_1}$ (resp $R_{\cS_\chi}$) is in the support of a $\bbT_1$ (resp. $\bbT_\chi$)-module $M$ if it lies in the image of $\Supp_{\bbT_1} M$ in $\Spec R_{\cS_1}$ (resp. of $\Supp_{\bbT_1} M_\chi$ in $\Spec R_{\cS_\chi}$).

We write $\wp = \ker(\Lambda_0 \to \cO)$ for the augmentation ideal of $\Lambda_0$. We define $\overline{H}_1 = \text{im}(H^\ast(C_1) \to H^\ast(C_1)_{(\wp)})$, $\overline{H}_\chi = \text{im}(H^\ast(C_\chi) \to H^\ast(C_\chi)_{(\wp)})$. We write $\overline{\bbT}_\chi$ for the quotient of $\bbT_\chi$ which acts faithfully on $\overline{H}_\chi$. 
\begin{lemma}\label{lem_CG_for_ordinary_completed_coh}
$\overline{H}_\chi$ is a non-zero finite $\Lambda_0$-module of dimension at least $\dim \Lambda_0 + 1 - 2d = 2d + 2$, $d = [K^+ : \bbQ]$. If $Q$ is a prime of $\Lambda_0$ minimal in the support of $\overline{H}_\chi$, then $Q \subset \wp$.
\end{lemma}
\begin{proof}
Let $C = C_\chi$. Then $C_{(\wp)}$ is a bounded complex of finite free $\Lambda_{0,(\wp)}$-modules, and there is an isomorphism 
\[ H^\ast(C_{(\wp)} \otimes_{\Lambda_{0,(\wp)}} E) \cong \Hom_E(H^{-\ast}_{U(c_0, c_0)}(\mathfrak{X}, \cO(\chi^{-1}))^\text{ord}_\ffrm[1/p], E). \]
By \cite[Theorem 2.4.10]{10authors}), these last groups are non-zero precisely in degrees in the range $[-3d + 1, -d]$ 
(they are non-zero because $\pi$ exists). We can therefore apply \cite[Lemma 6.2]{Cal17} to deduce that $H^\ast(C_{(\wp)}) = \overline{H}_{\chi, (\wp)}$ has dimension at least
\[ \dim \Lambda_{0,(\wp)} - 2d + 1 = \dim \Lambda_0 - 2d \]
as $\Lambda_{0,(\wp)}$-module, hence that $\overline{H}_\chi$ has dimension at least $\dim \Lambda_0 + 1 - 2d$ as $\Lambda_0$-module. This proves the first claim of the lemma. The second claim follows immediately from the second part of Lemma \ref{lem_primes_and_supports}.
\end{proof}
For the statement of the next proposition, we recall from \S \ref{subsec_nice_and_sweet_primes} that a prime $\frp \subset R_{\cS_1}$ is said to be nice if $R_{\cS_1} / \frp$ is of dimension 1 and characteristic $p$, there exists $v \in S_2$ such that the quotient $\widetilde{\psi}_{v, 1} / \widetilde{\psi}_{v, 2} \text{ mod }\frp$ has infinite order, and there exists $\sigma \in G_K$ such that $\rho_{\cS_1}(\sigma) \text{ mod }\frp$ is a non-trivial unipotent element. A prime $\frq \subset R_{\cS_\chi}$ is said to be sweet if $R_{\cS_\chi} / \frq$ is of dimension 1 and characteristic 0, and there exists $v \in S_2$ such that the quotient $\psi_{v, 1} / \psi_{v, 2} \text{ mod }\frq$ has infinite order. 
\begin{proposition}\label{prop_patching}
Let $\frp_1, \dots, \frp_l$ be nice primes of $R_{\cS_1}$, and let $\frq_1, \dots, \frq_k$ be the primes of $R_{\cS_\chi}[1/p]$ in the support of $H^\ast_{U(c_0, c_0)}(\mathfrak{X}, \cO(\chi^{-1}))^\text{ord}_\ffrm[1/p]$. Then $\frq_1, \dots, \frq_k$ are sweet, and we can find the following:
\begin{enumerate}
\item A finite set of finite, prime-to-$S$ places $v_1, \dots, v_r$ of $K$ such that, setting $T = R \cup \{ v_1, \dots, v_r \}$, each map $A_{\cS_\chi}^T \to R_{\cS_\chi}^T \to R_{\cS_\chi} / \frp_i$ and $A_{\cS_\chi}^T \to R_{\cS_\chi}^T \to R_{\cS_\chi} / \frq_j$ is surjective. We set $\cT = \cO \llbracket T_1, \dots, T_{4 |T| - 1} \rrbracket$.
\item An integer $q \geq [K : \bbQ]$.
We set $S_{\infty} = \Lambda_0 \llbracket \bbZ_p^{2q} \rrbracket \widehat{\otimes}_\cO \cT$ and $g = 2q - [K : \bbQ]$. We define $\fra = \ker(S_\infty \to \Lambda_0)$.
\item Complete Noetherian $S_\infty \widehat{\otimes}_{\Lambda_0} A_{\cS_1}^T$- and $S_\infty \widehat{\otimes}_{\Lambda_0} A_{\cS_\chi}^T$-algebras $R_{1, \infty}$ and $R_{\chi, \infty}$ together with compatible isomorphisms $R_{1, \infty}  / (\fra) \cong R_{\cS_1}$, $R_{\chi, \infty}  / (\fra) \cong R_{\cS_\chi}$, and $R_{1, \infty} / (\varpi) \cong R_{\chi, \infty} / (\varpi)$. 
\item Bounded complexes $D_1$ and $D_\chi$ of finite free $S_{\infty}$-modules together with $S_\infty$-algebras $\bbT_{1, \infty}$ and $\bbT_{\chi, \infty}$ and inclusions of $S_{\infty}$-algebras $\bbT_{1, \infty} \subset \End_{\mathbf{D}(S_{\infty})}(D_1)$ and $\bbT_{\chi, \infty} \subset \End_{\mathbf{D}(S_{\infty})}(D_{\chi})$.
\item Isomorphisms $D_1 \otimes_{S_{\infty}} \Lambda_0 \cong C_1$ (sending $\bbT_{1, \infty}$ into $\bbT_1$), $D_\chi \otimes_{S_{\infty}} \Lambda_0 \cong C_\chi$ (sending $\bbT_{\chi, \infty}$ into $\bbT_\chi$) and $D_1 / (\varpi) \cong D_\chi / (\varpi)$ (sending the image of $\bbT_1$ to the image of $\bbT_\chi$).
\item Ideals $J_{1, \infty} \subset \bbT_{1, \infty}$ and $J_{\chi, \infty} \subset \bbT_{\chi, \infty}$ satisfying $J_{1, \infty}^\delta = 0$, $J_{\chi, \infty}^\delta = 0$, and surjective homomorphisms $R_{1, \infty} \to \bbT_{1, \infty} / J_{1, \infty}$ and $R_{\chi, \infty} \to \bbT_{\chi, \infty} / J_{\chi, \infty}$ (same $\delta$ as in Theorem \ref{thm_existence_of_Galois_base_case}).
\end{enumerate}
These can be chosen to satisfy the following additional conditions:
\begin{enumerate}
\item Let $\frp \in \{ \frp_1, \dots, \frp_l \}$, and let $\frp_\infty$ denote the pullback of $\frp$ to $R_{\chi, \infty}$, $\frp_\loc$ the pullback of $\frp$ to $A_{\cS_\chi}^T$. Then we have
\[ \dim_{\kappa(\frp)} ( \frp_\infty / ( \frp_\infty^2, \frp_\loc ) )_{(\frp_\infty)} \leq g. \]
\item Let $\frq \in \{ \frq_1, \dots, \frq_k \}$, and $\frq_\infty$ denote the pullback of $\frq$ to $R_{\chi, \infty}$, $\frq_\loc$ the pullback of $\frq$ to $A_{\cS_\chi}^T$.  Then we have
\[ \dim_{\kappa(\frq)} ( \frq_\infty / ( \frq_\infty^2, \frq_\loc ) )_{(\frq_\infty)} \leq g. \]
\end{enumerate}
\end{proposition}
\begin{proof}
We first note that each prime $\frq_1, \dots, \frq_k$ is sweet. 
Indeed, $\rho_{\cS_\chi} \text{ mod }\frq_i$ is Hodge--Tate regular by Theorem \ref{thm_existence_of_Galois_base_case}, so $\psi_{v, 1} / \psi_{v, 2} \text{ mod }\frq_i$ has infinite order for each $v \in S_2$. 

We next fix a choice of set $T$. Each ring $R_{\cS_\chi} / \frp_i$ and $R_{\cS_\chi} / \frq_j$ is a finite $\Lambda$-algebra, hence a finite $A_{\cS_\chi}^{R}$-algebra, which is topologically generated by the elements $\tr \rho_{\cS_\chi}(\Frob_v)$ ($v \not\in S$). We therefore choose $T$ so that each of these quotient rings is generated as a $\Lambda$-algebra by the elements $\tr \rho_{\cS_\chi}(\Frob_v)$ ($v \in T$). Then the associated maps $A_{\cS_\chi}^T \to R_{\cS_\chi} / \frp_i$ and $A_{\cS_\chi}^T \to R_{\cS_\chi} / \frq_j$ are surjective. We can (by the prime avoidance lemma) find an element $t \in \ffrm_\Lambda$ which has non-zero image in $R_{\cS_\chi} / \frp_i$ for each $i = 1, \dots, l$ and non-zero image divisible by $p$ in $ R_{\cS_\chi} / \frq_j$ for each $j = 1, \dots, k$. We fix this choice of element for the rest of the proof.

According to Proposition \ref{cor_existence_of_taylor_wiles_data_for_a_collection_of_nice_and_sweet_primes}, we can find integers $q \geq [K : \bbQ] = 2d$ and $a \ge 0$ with the following property: for any $N \geq 1$, there exists a Taylor--Wiles datum $Q_N$ of level $N+m$ for $R_{\cS_\chi}$ satisfying the following conditions:
\begin{itemize}
\item $Q_N \cap T = \emptyset$ and $|Q_N| = q$. We let $g = 2q - 2d$.
\item Let $\frp \in \{ \frp_1, \dots, \frp_l \}$, and let $\frp_N$, $\frp_N^T$, $\frp_\loc$ denote the respective pre-images of $\frp$ in $R_{\cS_{\chi, Q_N}}$, $R_{\cS_{\chi, Q_N}}^T$ and $A_{\cS_\chi}^T$.  Let $A = R_{\cS_\chi} / \frp$. Then there exists a map
\[ (A / t^N)^{g} \to \frp_N^T / ( (\frp_N^T)^2, \frp_\loc, t^N ) \]
with cokernel annihilated by $t^a$.
\item Let $\frq \in \{ \frq_1, \dots, \frq_k \}$, and let $\frq_N$, $\frq_N^T$, $\frq_\loc$ denote the respective pre-images of $\frq$ in $R_{\cS_{\chi, Q_N}}$, $R_{\cS_{\chi, Q_N}}^T$ and $A_{\cS_\chi}^T$. Let $A = R_{\cS_\chi} / \frq$. Then there exists a map
\[ (A / p^N)^{g} \to \frq_N^T / ( (\frq_N^T)^2, \frq_\loc, p^N ) \]
with cokernel annihilated by $p^a$.
\end{itemize}
In order to construct the objects in the statement of the proposition, we will patch these objects for varying Taylor--Wiles data $Q_N$ together. If $N \geq 1$, then we define $S_N =  \Lambda_{0, N} / (\varpi^N)[ (\bbZ / p^N \bbZ)^{2q} ][T_1, \dots, T_{4|T| - 1}] / (T_1^N, \dots, T_{4|T| - 1}^N)$, a quotient of $S_{ \infty}$. Fix an integer $g_0 \geq 1$ such that for any Taylor--Wiles datum $Q_N$ with $|Q_N| = q$, $R_{\cS_{\chi, Q_N}}^T$ can be written as a quotient of $\cO \llbracket Z_1, \dots, Z_{g_0} \rrbracket$ (for example, $g_0 = 2q + \dim_k(\ffrm_{R_{\cS_{\chi, Q_N}}^T}/(\varpi, \ffrm_{R_{\cS_{\chi, Q_N}}^T}^2)$ works). Let $s = \dim_k H^\ast(C_1 \otimes_{\Lambda_0} k)$, and for $N \geq 1$ let $f(N) = 4dsl(S_N)$, where $l(S_N)$ denotes the length of $S_N$ as $S_N$-module. We define a patching datum of level $N \geq 1$ to consist of the following data:
\begin{itemize}
\item Bounded complexes $D_{1, N}$ and $D_{\chi, N}$ of finite free $S_{N}$-modules, together with an isomorphism $D_{1, N} / (\varpi) \cong D_{\chi, N} / (\varpi)$.
\item Isomorphisms $\psi_1 : D_{1, N} \otimes_{S_{N}} \Lambda_{0, N} / (\varpi^N) \cong C_1 \otimes_{\Lambda_0} \Lambda_{0, N} / (\varpi^N)$ and $\psi_\chi : D_{\chi, N} \otimes_{S_{N}} \Lambda_{0, N} / (\varpi^N) \cong C_\chi \otimes_{\Lambda_0} \Lambda_{0, N} / (\varpi^N)$ of complexes of $\Lambda_{0, N} / (\varpi^N)$-modules.
\item An $S_\infty \widehat{\otimes}_{\Lambda_0} A_{\cS_1}^T$-algebra $R_{1, N}$ such that $\ffrm_{R_{1, N}}^{f(N)} = 0$, together with an isomorphism $R_{1, N} / (\fra) \cong R_{\cS_1} / (\ffrm_{R_{\cS_1}}^{f(N)})$.
\item An $S_\infty \widehat{\otimes}_{\Lambda_0} A_{\cS_\chi}^T$-algebra $R_{\chi, N}$ such that $\ffrm_{R_{\chi, N}}^{f(N)} = 0$, together with an isomorphism $R_{\chi, N}/ (\fra) \cong R_{\cS_\chi} / (\ffrm_{R_{\cS_\chi}}^{f(N)})$. If $\frp \in \{ \frp_1, \dots, \frp_l \}$, then we write $\frp(N)$ for the pre-image of $(\frp, \ffrm_{R_{\cS_\chi}}^{f(N)})$ in $R_{\chi, N}$. If $\frq \in  \{ \frq_1, \dots, \frq_k \}$, then we write $\frq(N)$ for the pre-image of $(\frq, \ffrm_{R_{\cS_\chi}}^{f(N)})$ in $R_{\chi, N}$.
\item An isomorphism $\eta : R_{1, N} / (\varpi) \cong R_{\chi, N} / (\varpi)$ of $S_\infty \widehat{\otimes}_{\Lambda_0} A_{\cS_1}^T / (\varpi)$-algebras.
\item $S_N$-algebras $\bbT_{1, N}$  and $\bbT_{\chi, N}$ equipped with inclusions $\bbT_{1, N} \subset \End_{\mathbf{D}(S_N)}(D_{1, N})$ and $\bbT_{\chi, N} \subset \End_{\mathbf{D}(S_N)}(D_{\chi, N})$ of $S_{N}$-algebras, together with ideals $J_{1, N} \subset \bbT_{1, N}$ and $J_{\chi, N} \subset \bbT_{\chi, N}$ such that $J_{1, N}^\delta = 0$ and $J_{\chi, N}^\delta = 0$.
\item Surjective morphisms $\phi_1 : R_{1, N} \to \bbT_{1, N} / J_{1, N}$ and $\phi_\chi : R_{\chi, N} \to \bbT_{\chi, N} / J_{\chi, N}$ of $S_\infty$-algebras.
\item For each $\frp \in \{ \frp_1, \dots, \frp_l \}$, elements $X_{\frp, 1}, \dots, X_{\frp, g} \in \frp(N)$ such that $\frp(N) / (\frp(N)^2, \frp_\loc, X_{\frp, 1}, \dots, X_{\frp, g}, t^N)$ is annihilated by $t^a$.
\item For each $\frq \in \{ \frq_1, \dots, \frq_k \}$, elements $X_{\frq, 1}, \dots, X_{\frq, g} \in \frq(N)$ such that
$\frq(N) / (\frq(N)^2, \frq_\loc, X_{\frq, 1}, \dots, X_{\frq, g}, p^N)$ is annihilated by $p^a$.
\item A surjective homomorphism $\zeta : \cO \llbracket Z_1, \dots, Z_{g_0} \rrbracket \to R_{1, N}$ of $\cO$-algebras. 
\end{itemize}
These data are required to satisfy the following conditions:
\begin{itemize}
\item The diagram 
\[ \xymatrix{ D_{1, N} / (\varpi) \ar[r] \ar[d] & D_{\chi, N} / (\varpi) \ar[d] \\
C_1 \otimes_{\Lambda_0} \Lambda_{0, N} / (\varpi) \ar[r] & C_\chi \otimes_{\Lambda_0} \Lambda_{0, N} /  (\varpi)} \]
commutes.
\item Let $\overline{\bbT}_N$ denote the image of $\bbT_{1, N}$ in $\End_{\mathbf{D}(S_{N} / (\varpi))}(D_{1, N} / (\varpi))$. Then $\overline{\bbT}_N$ equals the image of $\bbT_{\chi, N}$ under the identification $D_{1, N} / (\varpi) = D_{\chi, N} / (\varpi)$, and the diagram
\[ \xymatrix{ R_{1, N} / (\varpi) \ar[d] \ar[r] & \bbT_{1, N} / (\varpi, J_{1, N}) \ar[r] & \overline{\bbT}_{N} / (J_{1, N}, J_{\chi, N}) \ar[d]^= \\ 
R_{\chi, N} / (\varpi) \ar[r] & \bbT_{\chi, N} / (\varpi, J_{\chi, N}) \ar[r] & \overline{\bbT}_{N} / (J_{1, N}, J_{\chi, N})}  \]
commutes.
\end{itemize}
Note that the function $f(N)$ is defined so that any element $x \in \ffrm_{\bbT_{1, N}}$ satisfies $x^{f(N)} = 0$ (and similarly for $\bbT_{\chi, N}$), cf. the first few lines of the proof of \cite[Proposition 3.1]{Kha17}.

We can make the collection of patching data of level $N$ into a category by declaring a morphism
\[ (D_{1, N}, D_{\chi, N}, R_{1, N}, R_{\chi, N}, \bbT_{1, N}, \bbT_{\chi, N}, J_{1, N}, J_{\chi, N}, (X_{\frp, i})_{\frp, i}, (X_{\frq, j})_{\frq, j}, \psi_1, \psi_\chi, \eta, \phi_1, \phi_\chi, \zeta) \]
\[ \to (D'_{1, N}, D'_{\chi, N}, R'_{1, N}, R'_{\chi, N}, \bbT'_{1, N}, \bbT'_{\chi, N}, J'_{1, N}, J'_{\chi, N}, (X'_{\frp, i})_{\frp, i}, (X'_{\frq, j})_{\frq, j}, \psi'_1, \psi'_\chi, \eta', \phi_1', \phi_\chi', \zeta') \]
to be a tuple of isomorphisms
\[ D_{1, N} \to D'_{1, N}, D_{\chi, N} \to D_{\chi', N}, R_{1, N} \to R'_{1, N}, R_{\chi, N} \to R'_{\chi, N} \]
which takes $\bbT_{1, N}$ to $\bbT'_{1, N}$, $\bbT_{\chi, N}$ to $\bbT'_{\chi, N}$, $J_{1, N}$ to $J'_{1, N}$, $J_{\chi, N}$ to $J'_{\chi, N}$, $X_{\frp, i}$ to $X'_{\frp, i}$, $X_{\frq, j}$ to $X'_{\frq, j}$, which intertwine $ \psi_1, \psi_\chi, \eta, \phi_1, \phi_\chi, \zeta$ with $\psi'_1, \psi'_\chi, \eta', \phi_1', \phi_\chi', \zeta'$, and which are compatible with the various identifications made modulo $\varpi$. 

Let $\cD(N)$ denote the category of patching data of level $N$. For any $N \geq 1$, $\cD(N)$ has only finitely many isomorphism classes of objects. For any $M \geq N \geq 1$, there is a functor $\cD(M) \to \cD(N)$ which sends a tuple 
\[ (D_{1, M}, D_{\chi, M}, R_{1, M}, R_{\chi, M}, \bbT_{1, M}, \bbT_{\chi, M}, J_{1, M}, J_{\chi, M}, (X_{\frp, i})_{\frp, i}, (X_{\frq, j})_{\frq, j}, \psi_1, \psi_\chi, \eta, \phi_1, \phi_\chi, \zeta)  \]
to a tuple
\[ (D'_{1, N}, D'_{\chi, N}, R'_{1, N}, R'_{\chi, N}, \bbT'_{1, N}, \bbT'_{\chi, N}, J'_{1, N}, J'_{\chi, N}, (X'_{\frp, i})_{\frp, i}, (X'_{\frq, j})_{\frq, j}, \psi'_1, \psi'_\chi, \eta', \phi_1', \phi_\chi', \zeta'), \]
 where:
 \begin{itemize}
 \item $D'_{1, N} = D_{1, M} \otimes_{S_{ M}} S_{ N}$ and $D'_{\chi, N} = D_{\chi, M} \otimes_{S_{ M}} S_{ N}$.
 \item $R'_{1, N} = R_{1, M} / (\ffrm_{R_{1, M}}^{g_0 f(N)})$ and $R'_{\chi, N} = R_{\chi, M} / (\ffrm_{R_{\chi, M}}^{g_0 f(N)})$.
 \item $\bbT'_{1, N}$ is the image of $\bbT_{1, M}$ in $\End_{\mathbf{D}(S_{ N})}(D'_{1, N})$, and $\bbT'_{\chi, N}$ is the image of $\bbT_{\chi, M}$ in $\End_{\mathbf{D}(S_{ N})}(D'_{\chi, N})$.
 \item $J'_{1, N}$ is the image of $J_{1, M}$ in $\bbT'_{1, N}$ and $J'_{\chi, N}$ is the image of $J_{\chi, M}$ in $\bbT'_{\chi, N}$.
 \item $X'_{\frp, i}$ is the image of $X_{\frp, i}$ and $X'_{\frq, j}$ is the image of $X_{\frq, j}$.
 \item The maps $\psi'_1, \psi'_\chi, \eta', \phi_1', \phi_\chi', \zeta'$ are the obvious ones induced by passage to quotient. (Here we are using the property of the function $f(N)$ mentioned above.)
 \end{itemize} 
 For any $N \geq 1$, we can construct infinitely many patching data of level $N$ by using the Taylor--Wiles data whose existence is asserted at the beginning of the proof of this proposition. Indeed, if $Q_M$ is such a Taylor--Wiles datum of level $M+m$ for some $M \geq 1$, then we can define a patching datum 
 \[ (D_{1, M}, D_{\chi, M}, R_{1, M}, R_{\chi, M}, \bbT_{1, M}, \bbT_{\chi, M}, J_{1, M}, J_{\chi, M}, (X_{\frp, i})_{\frp, i}, (X_{\frq, j})_{\frq, j}, \psi_1, \psi_\chi, \eta, \phi_1, \phi_\chi, \zeta)  \]
 as follows. First choose an isomorphism $\Delta_{Q_M} \to (\bbZ / p^M \bbZ)^{2q}$. Then make the following assignments: 
 \begin{itemize}
 \item $D_{1, M} = F_{\ffrm, 1, Q_M} \otimes_{\Lambda_0[\Delta_{Q_M}]} S_{M}$ and $D_{\chi, M} = F_{\ffrm, \chi, Q_M} \otimes_{\Lambda_0[\Delta_{Q_M}]} S_{ M}$.
 \item $R_{1, M} = R_{\cS_{1, Q_M}}^T / \ffrm_{R_{\cS_{1, Q_M}}^T}^{f(M)}$ and $R_{\chi, M} = R_{\cS_{\chi, Q_M}}^T / \ffrm_{R_{\cS_{\chi, Q_M}}^T}^{f(M)}$.
 \item $\bbT_{1, M}$ is the image of $\bbT_{1, Q_M} \otimes_{\Lambda_0[\Delta_{Q_M}]} S_{M}$ in $\End_{\mathbf{D}(S_{ M})}(D_{1, M})$, and $\bbT_{\chi, M}$ is the image of $\bbT_{\chi, Q_M} \otimes_{\Lambda_0[\Delta_{Q_M}]} S_{M}$ in $\End_{\mathbf{D}(S_{ M})}(D_{\chi, M})$.
 \item $J_{1, M}$ is the image of $J_{1, Q_M}$ in $\bbT_{1, M}$ and $J_{\chi, M}$ is the image of $J_{\chi, Q_M}$ in $\bbT_{\chi, M}$.
 \item The elements $X_{\frp, i}$ and $X_{\frq, j}$ can be chosen to be any ones satisfying the necessary conditions; their existence is asserted by Proposition \ref{cor_existence_of_taylor_wiles_data_for_a_collection_of_nice_and_sweet_primes} (see the beginning of the proof of this proposition).
 \item The maps $\psi_1, \psi_\chi, \eta, \chi_1, \phi_1, \phi_\chi, \zeta$ arise from the tautological ones by passage to quotient. (Once again we are using the property of the function $f(N)$ mentioned above.) 
 \end{itemize}
 By allowing $M$ to vary and taking image under the functors $\cD(M) \to \cD(N)$, we obtain infinitely many patching data of each level $N \geq 1$. Since each category $\cD(N)$ has only finitely many isomorphism classes of objects, we can diagonalize and pass to the limit to obtain the following objects:
 \begin{itemize}
 \item Bounded complexes $D_{1, \infty}$ and $D_{\chi, \infty}$ of finite free $S_{\infty}$-modules, together with an isomorphism $D_{1, \infty} / (\varpi) \cong D_{\chi, \infty} / (\varpi)$.
 \item Isomorphisms $\psi_1 : D_{1, \infty} \otimes_{S_{\infty}} \Lambda_0 \cong C_1$ and $\psi_\chi : D_{\chi, \infty} \otimes_{S_{\infty}} \Lambda_0 \cong C_\chi$ of complexes of $\Lambda_0$-modules.
 \item An $S_\infty \widehat{\otimes}_{\Lambda_0} A_{\cS_1}^T$-algebra $R_{1, \infty}$, together with an isomorphism $R_{1, \infty} / (\fra) \cong R_{\cS_1}$.
 \item An $S_\infty \widehat{\otimes}_{\Lambda_0} A_{\cS_\chi}^T$-algebra $R_{\chi, \infty}$, together with an isomorphism $R_{\chi, \infty}/ (\fra) \cong R_{\cS_\chi}$.
 \item An isomorphism $\eta : R_{1, \infty} / (\varpi) \cong R_{\chi, \infty} / (\varpi)$ of $A_{\cS_\chi}^T / (\varpi)$-algebras.
 \item $S_\infty$-algebras $\bbT_{1, \infty}$ and $\bbT_{\chi, \infty}$, together with inclusions $\bbT_{1, \infty} \subset \End_{\mathbf{D}(S_\infty)}(D_{1, \infty})$ and $\bbT_{\chi, \infty} \subset \End_{\mathbf{D}(S_\infty)}(D_{\chi, \infty})$ of $S_{\infty}$-algebras and ideals $J_{1, \infty} \subset \bbT_{1, \infty}$ and $J_{\chi, \infty} \subset \bbT_{\chi, \infty}$ such that $J_{1, \infty}^\delta = 0$ and $J_{\chi, \infty}^\delta = 0$.
\item Surjective morphisms $\phi_1 : R_{1, \infty} \to \bbT_{1, \infty} / J_{1, \infty}$ and $\phi_\chi : R_{\chi, \infty} \to \bbT_{\chi, \infty} / J_{\chi, \infty}$ of $S_\infty$-algebras.
\item A surjective homomorphism $\zeta : \cO\llbracket Z_1, \dots, Z_{g_0} \rrbracket \to R_{1, \infty}$. In particular, both $R_{1, \infty}$ and $R_{\chi, \infty}$ are complete Noetherian local $\cO$-algebras.
\item If $\frp \in \{ \frp_1, \dots, \frp_l \}$, then let $\frp_\infty \subset R_{1, \infty}$ denote the pre-image of $\frp \subset R_{\cS_1}$. Then we are given elements $X_{\frp, 1}, \dots, X_{\frp, g} \in \frp_\infty$ such that $\frp_\infty / (\frp_\infty^2, \frp_\loc, X_{\frp, 1}, \dots, X_{\frp, g})$ is annihilated by $t^a$.
\item If $\frq \in \{ \frq_1, \dots, \frq_l \}$, then let $\frq_\infty \subset R_{\chi, \infty}$ denote the pre-image of $\frq \subset R_{\cS_\chi}$. Then we are given elements $X_{\frq, 1}, \dots, X_{\frq, g} \in \frq_\infty$ such that $\frq_\infty / (\frq_\infty^2, \frq_\loc, X_{\frq, 1}, \dots, X_{\frq, g})$ is annihilated by $p^a$.
 \end{itemize}
 These objects satisfy the following conditions:
 \begin{itemize}
\item The diagram 
\[ \xymatrix{ D_{1, \infty} / (\varpi) \ar[r] \ar[d] & D_{\chi, \infty} / (\varpi) \ar[d] \\
C_1 \otimes_{\Lambda_0} \Lambda_0 / (\varpi) \ar[r] & C_\chi \otimes_{\Lambda_0} \Lambda_0 / (\varpi)} \]
commutes.
\item Let $\overline{\bbT}_\infty$ denote the image of $\bbT_{1, \infty}$ inside $\End_{\mathbf{D}(S_{\infty} / (\varpi))}(D_{1, \infty} / (\varpi))$. Then $\overline{\bbT}_\infty$ equals the image of $\bbT_{\chi, \infty}$ under the identification $D_{1, \infty} / (\varpi) = D_{\chi, \infty} / (\varpi)$, and the diagram
\[ \xymatrix{ R_{1, \infty} / (\varpi) \ar[d] \ar[r] & \bbT_{1, \infty} / (\varpi, J_{1, \infty}) \ar[r] & \overline{\bbT}_\infty / (J_{1, \infty}, J_{\chi, \infty}) \ar[d]^= \\ 
R_{\chi, \infty} / (\varpi) \ar[r] & \bbT_{\chi, \infty} / (\varpi, J_{\chi, \infty}) \ar[r] & \overline{\bbT}_\infty / (J_{1, \infty}, J_{\chi, \infty})}  \]
commutes.
 \end{itemize}
 The proof is now complete on setting $D_1 = D_{1, \infty}$ and $D_\chi = D_{\chi, \infty}$.
\end{proof}

\begin{corollary}\label{cor_consequences_of_patching}
In the situation of Proposition \ref{prop_patching}, let $H_{1, \infty} = H^\ast(D_1)$, $H_{\chi, \infty} = H^\ast(D_\chi)$. Let $\wp_\infty \subset S_\infty$ denote the pullback of $\wp$, and let $\overline{H}_{\chi, \infty} = \text{im}(H_{\chi, \infty} \to H_{\chi, \infty, (\wp_\infty)})$. Let $\overline{\bbT}_{\chi, \infty}$ denote the quotient of $\bbT_{\chi, \infty}$ which acts faithfully on $\overline{H}_{\chi, \infty}$. Then:
\begin{enumerate}
\item $\overline{H}_{\chi, \infty}$ is a non-zero $\cO$-flat finite $S_\infty$-module which is concentrated in cohomological degree $-d$.
Every prime of $S_\infty$ minimal in the support of $\overline{H}_{\chi, \infty}$ has dimension $1 + 2d + 2q + 4|T|$.
\item Let $\frp \in \{ \frp_1, \dots, \frp_l \}$, and suppose that $\frp_\infty \in \Spec \overline{\bbT}_{\chi, \infty}$ (viewed as a closed subspace of $\Spec R_{\chi, \infty}$). For example, this is true if $\frp \in \Spec \overline{\bbT}_\chi$, viewed as a closed subspace of $\Spec R_{\cS_\chi}$. Suppose that for each $v \in R$, $\rho_\frp|_{G_{K_v}}$ is the trivial representation (where $\rho_\frp : G_K \to \GL_2(R_{\cS_\chi} / \frp)$ is a representative of the tautological deformation). Then:
\begin{enumerate}
\item $\Spec R_{\chi, \infty, \frp_\infty}$ is irreducible and $\Spec R_{\chi, \infty, \frp_\infty} = \Spec \overline{\bbT}_{\chi, \infty, \frp_\infty} = \Spec \bbT_{\chi, \infty, \frp_\infty}$. In particular, there is a unique minimal prime $Q \in \Spec \bbT_{\chi, \infty}$  that is contained inside $\frp_\infty$, and any prime ideal of $R_{\chi, \infty}$ which is contained inside $\frp_\infty$ is in the image of $\Spec \overline{\bbT}_{\chi, \infty}$.
\item $H_{\chi, \infty, (Q)}$ is non-zero exactly in cohomological degree $-d$.
\end{enumerate}
\end{enumerate}
\end{corollary}
\begin{proof}
The same argument as in the proof of Lemma \ref{lem_CG_for_ordinary_completed_coh} implies that $\overline{H}_{\chi, \infty}$ is a non-zero $\cO$-flat finite $S_\infty$-module of dimension at least 
\[ \dim S_\infty - (2d-1) = 1 + 2d + 2q + 4|T|. \]
If we can show that its dimension is exactly $1 + 2d + 2q + 4|T|$, then \cite[Lemma 6.2]{Cal17} will show that $\overline{H}_{\chi, \infty}$ is moreover concentrated in cohomological degree  $-d$, and that $H_{\chi, \infty, (\wp_\infty)}$ is a Cohen--Macaulay $S_\infty$-module. Since $\Ass_{S_\infty} \overline{H}_{\chi, \infty} = \Ass_{S_{\infty, (\wp_\infty)}} H_{\chi, \infty, (\wp_\infty)}$, this will show that each prime of $S_\infty$ minimal in the support of $\overline{H}_{\chi, \infty}$ has the same dimension (and moreover that $\overline{H}_{\chi, \infty}$ has no embedded primes). It will suffice to show that 
\[ \dim_{S_{\infty, (\wp_\infty)}} H_{\chi, \infty, (\wp_\infty)} \leq 2d + 2q + 4|T|. \]

We will in fact show that 
$\dim \bbT_{\chi, \infty, (\wp_\infty)} \leq 2d + 2q + 4 |T|$.  We note that $\bbT_{\chi, \infty, (\wp_\infty)}$ is a semi-local ring, its maximal ideals $\frq_\infty$ being in bijection with the primes $\frq \in \{ \frq_1, \dots, \frq_k \}$ of $R_{\cS_\chi}[1/p]$ which are in the support of $H^\ast_{U(c_0, c_0)}(\mathfrak{X}, \cO(\chi^{-1}))^\text{ord}_\ffrm[1/p]$. If $\frq \in \{ \frq_1, \dots, \frq_k \}$, then we can find, by construction, a surjective morphism
\[ A_{\cS_\chi, \frq_\loc}^T \llbracket X_1, \dots, X_g \rrbracket \to R_{\chi, \infty, \frq_\infty}. \]
Since $\dim A^T_{\cS_\chi, \frq_\loc} = 4 |T|$, and $\dim \bbT_{\chi, \infty, \frq_\infty} \leq \dim R_{\chi, \infty, \frq_\infty}$, it follows that 
\[ \dim \bbT_{\chi, \infty, \frq_\infty} \leq \dim A_{\cS_\chi, \frq_\loc}^T \llbracket X_1, \dots, X_g \rrbracket = 4 |T| + 2q - 2d. \]
This completes the proof of the first part of the corollary. 

The proof of the second part is similar. Let $\frp$ be as in the statement of the corollary. By construction, we can find a surjection $A_{\cS_\chi, \frp_\loc}^T \llbracket X_1, \dots, X_g \rrbracket \to R_{\chi, \infty, \frp_\infty}$ of $A_{\cS_\chi, \frp_\loc}^T$-algebras.  According to Proposition \ref{prop_irreducibility_of_completed_local_lifting_ring}, the ring $A_{\cS_\chi, \frp_\loc}^T \llbracket X_1, \dots, X_g \rrbracket$ has irreducible spectrum (here we are using our hypothesis that $\rho_\frp|_{G_{K_v}}$ is trivial for each $v \in R$).

There are surjective homomorphisms 
\[ A_{\cS_\chi, \frp_\loc}^T \llbracket X_1, \dots, X_g \rrbracket \to R_{\chi, \infty, \frp_\infty} \]
and
\[ R_{\chi,  \infty, \frp_\infty} \to \bbT_{\chi, \infty, \frp_\infty} / (J_{\chi, \infty}) \]
and
\[ \bbT_{\chi, \infty, \frp_\infty} / (J_{\chi, \infty}) \to \overline{\bbT}_{\chi, \infty, \frp_\infty} / (J_{\chi, \infty}). \]
By the first part of the corollary, $\dim A_{\cS_\chi, \frp_\loc}^T \llbracket X_1, \dots, X_g \rrbracket = \dim \overline{\bbT}_{\chi, \infty, \frp_\infty} / (J_{\chi, \infty})$. It follows that each of these surjective homomorphisms has nilpotent kernel, and that $\Spec  R_{\chi,  \infty, \frp_\infty} = \Spec \bbT_{\chi, \infty, \frp_\infty}$ is irreducible. In particular, there is a unique minimal prime of $\bbT_{\chi, \infty}$ contained inside $\frp_\infty$ (use e.g.\ Lemma \ref{lem_primes_and_supports}). This shows (a).

To show (b), let $Q$ be the unique minimal prime of $\bbT_{\chi, \infty}$ which is contained in $\frp_\infty$. Then $Q$ is a minimal prime in the support of $\overline{H}_{\chi, \infty}$. Let $Q_0$ denote the pullback of $Q$ to $S_\infty$. Then $Q_0 \in \Ass_{S_\infty} \overline{H}_{\chi, \infty}$, hence $Q_0 \subset \wp_\infty$ (Lemma \ref{lem_primes_and_supports} again), hence $\overline{H}_{\chi, \infty, (Q_0)} = H_{\chi, \infty, (Q_0)}$ is concentrated in cohomological degree $-d$. Since $H_{\chi, \infty, (Q)}$ is a localization of $H_{\chi, \infty, (Q_0)}$, this implies that $H_{\chi, \infty, (Q)} = \overline{H}_{\chi, \infty, (Q)}$ is concentrated in cohomological degree $-d$, as desired.
\end{proof}

\begin{lemma}\label{lem_nice_primes_in_big_quotients}
	Let $I \subset R_{\cS_\chi} / (\varpi)$ be an ideal such that $\dim R_{\cS_\chi} / I \geq 2d - 2$ and $R_{\cS_\chi} / I$ is a finite $\Lambda$-algebra. Then there exists a nice prime $\frp \subset R_{\cS_\chi}$ such that $I \subset \frp$ and for each $v \in R$, $\rho_\frp|_{G_{K_v}}$ is the trivial representation.
\end{lemma}
\begin{proof}
	Let $\overline{\Lambda}$ denote the completed group algebra of the quotient of  $(\prod_{v \in S_2} (\cO_{K_v}^\times)^f)^2$ by the closure of the diagonally embedded units $\cO_{K}^\times \cap U(c_0, c_0)$. Then the map $\Lambda \to R_{\cS_\chi} / (\varpi)$ factors through the quotient $\Lambda \to \overline{\Lambda}$ (consider determinants). Let $\sigma \in G_K$ be an element such that $\overline{\rho}_\ffrm(\sigma)$ has order 2, and let $t_\sigma, d_\sigma \in R_{\cS_\chi}$ denote the trace and determinant of $\sigma$ in the universal deformation. Let $Z \subset \Spec \overline{\Lambda} / (\varpi)$ denote the image of $\Spec R_{\cS_\chi} / (I, \{ \ffrm_{R_v}^{\chi_v} \}_{v \in R}, t_\sigma, d_\sigma-1)$; let $W \subset \Spec \overline{\Lambda} / (\varpi)$ denote the subspace determined by the equations $\psi_{v, 1}|_{I_{K_v}}(2) = \psi_{v, 2}|_{I_{K_v}}(2)$ for each $v \in S_2$. To prove the lemma, it suffices to show that $\dim Z > \dim W$. 
	
	We have $\dim W = d + 1 + \delta_K$ (where $\delta_K$ is the Leopoldt defect of $K$) and $\dim Z \geq 2d - 4 - 4 |R|$ (because of our assumption that $R_{\cS_\chi} / I$ is a finite $\Lambda$-algebra, and $\dim R_v^{\chi_v} / (\varpi) = 4$). The proof will therefore be complete if we can show that $2d - 4 - 4 |R| > d + 1 + \delta_K$. Using the result (\cite[Theorem 1.1]{Wal84}) that $\delta_K < d / 2$, it suffices to show that $d/2 > 5 + 4 |R|$. This follows from our hypotheses (we have assumed at the start of \S \ref{subsec_setup_for_main_ALT} that $d > 10 + 8 |R|$).
\end{proof}
\begin{corollary}\label{cor_first_nice_prime}
	There exists a nice prime $\frp$ of $R_{\cS_\chi}$ in the support of $\overline{H}_{\chi}$ such that for each $v \in R$, $\rho_\frp|_{G_{K_v}}$ is the trivial representation.
\end{corollary}
\begin{proof}
	By Lemma \ref{lem_CG_for_ordinary_completed_coh}, the ring $\overline{\bbT}_\chi$ is a finite $\Lambda$-algebra of dimension at least $2d + 2$, hence $\overline{\bbT}_{\chi} / (\varpi, J_\chi)$ is a finite $\Lambda$-algebra of dimension at least $2d + 1$. We apply Lemma \ref{lem_nice_primes_in_big_quotients} to the corresponding quotient of $R_{\cS_\chi}$.
\end{proof}
	
\begin{proposition}\label{prop_connectedness_of_nice_prime_graph}
	Define a graph as follows: the vertices are the minimal primes of $R_{\cS_\chi} / (\varpi)$, and the edges between two vertices $Q, Q'$ correspond to the nice primes $\frp \supset (Q, Q')$ with the property that for each $v  \in R$, $\rho_\frp|_{G_{K_v}}$ is the trivial representation. Then this graph is connected. Moreover, $R_{\cS_\chi}$ is a finite $\Lambda$-algebra.
\end{proposition}
	\begin{proof}
Let $\frp_1 \subset R_{\cS_\chi}$ be a nice prime which is in the support of $\overline{H}_\chi$ (which exists by Corollary \ref{cor_first_nice_prime}), and let $Q_1 \subset R_{\cS_\chi} / (\varpi)$ be a minimal prime contained inside $\frp_1$. Let $Q_2, \dots, Q_m$ denote the other minimal primes of $R_{\cS_\chi} / (\varpi)$, and suppose for contradiction that the graph in question is not connected. After re-ordering $Q_2, \dots, Q_m$, we can assume that $Q_1, \dots, Q_k$ are the vertices in a connected component of the graph.

By Corollary \ref{cor_connectedness_dimension_Ihara_case}, we can find $Q \in \{ Q_1, \dots, Q_k \}$, $Q' \in \{ Q_{k+1}, \dots, Q_m \}$ such that $\dim R_{\cS_\chi} / (Q \cap Q') \geq 2d - 2$. If we can show that $R_{\cS_\chi} / Q$ is finite over $\Lambda$, then Lemma \ref{lem_nice_primes_in_big_quotients} will imply the existence of a nice prime connecting $Q$ and $Q'$, a contradiction. We will show in fact that $R_{\cS_\chi} / Q_i$ is finite over $\Lambda$ for each $i = 1, \dots, k$. This will also imply that $R_{\cS_\chi} / (\varpi)$ is a finite $\Lambda$-algebra, hence that $R_{\cS_\chi}$ is a finite $\Lambda$-algebra (by the completed version of Nakayama's lemma).

Let $\frp_2, \dots, \frp_l$ be nice primes such that the edges $\frp_1, \dots, \frp_l$ span the connected component $\{ Q_1, \dots, Q_k \}$. We apply Proposition \ref{prop_patching} to the set $\{ \frp_1, \dots, \frp_l \}$ of nice primes in order to obtain objects $D_\chi$, $H_{\chi, \infty}$, $\overline{H}_{\chi, \infty}$, $R_{\chi, \infty}$. We show by induction on the length of a shortest path between $Q_1$ and $Q_i$ that $Q_{i, \infty}$, the pullback of $Q_i$ to $R_{\chi, \infty}$, is in the image of $\Spec \overline{\bbT}_{\chi, \infty}$ in $\Spec R_{\chi, \infty}$. This will imply that $R_{\chi, \infty} / Q_{i, \infty}$ is a finite $S_\infty$-algebra, hence that $R_{\chi, \infty} / (Q_{i, \infty}, \fra) = R_{\cS_\chi} / Q$ is a finite $\Lambda$-algebra, as desired.

We first treat the base case $i = 1$. By part 2(a) of Corollary \ref{cor_consequences_of_patching}, any prime of $R_{\chi, \infty}$ which is contained inside $\frp_{1, \infty}$ is in the image of $\Spec \overline{\bbT}_{\chi, \infty}$. This applies in particular to $Q_{1, \infty}$.

We now treat the induction step. Suppose after renumbering that $Q_{1, \infty}, \dots, Q_{i, \infty}$ have been shown to be in the image of $\Spec \overline{\bbT}_{\chi, \infty}$, and let $\frp \in \{ \frp_1, \dots, \frp_l \}$ be a nice prime linking $Q_i$ and $Q_{i+1}$. Then $\frp_\infty$ is in the image of $\Spec \overline{\bbT}_{\chi, \infty}$, so applying part 2(a) of Corollary \ref{cor_consequences_of_patching} once more shows that $Q_{i+1, \infty}$ is also in the support. This completes the proof.
\end{proof}
\begin{corollary}\label{cor_nice_primes_for_S_1}
Any minimal prime $Q \subset R_{\cS_1}$ is contained in a nice prime $\frp$ with the property that for each place $v \in R$, $\rho_\frp|_{G_{K_v}}$ is the trivial representation.
\end{corollary}
\begin{proof}
	If $Q \subset R_{\cS_1}$ is a minimal prime, then $\dim R_{\cS_1} / (Q, \varpi) \geq 2d$ (by Corollary \ref{cor_connectedness_dimension_Ihara_case}). Therefore the corollary follows from the fact that $R_{\cS_\chi} / (\varpi) = R_{\cS_1} / (\varpi)$ is a finite $\Lambda$-algebra, together with Lemma \ref{lem_nice_primes_in_big_quotients}.
\end{proof}

We can now complete the proof of Theorem \ref{thm_application_of_TW_method}. Let $\frp_1$ be a nice prime of $R_{\cS_\chi}$ which is in the support of $\overline{H}_\chi$ and such that for each $v \in R$, $\rho_{\mathfrak{p}_1}|_{G_{K_v}}$ is the trivial representation, and let $Q_1$ be a minimal prime of $R_{\cS_\chi} / (\varpi)$ which is contained inside $\frp_1$. Let $Q_2, \dots, Q_k$ be the remaining minimal primes of $R_{\cS_\chi} / (\varpi)$. By Proposition \ref{prop_connectedness_of_nice_prime_graph}, we can find nice primes $\frp_2, \dots, \frp_{l_0}$ such that $\frp_1, \dots, \frp_{l_0}$ span the graph of minimal primes of $R_{\cS_\chi} / (\varpi)$ and such that for each $\frp \in \{ \frp_1, \dots, \frp_{l_0} \}$ and for each $v \in R$, $\rho_\frp|_{G_{K_v}}$ is the trivial representation. By Corollary \ref{cor_nice_primes_for_S_1}, we can find nice primes $\frp_{l_0+1}, \dots, \frp_l$ of $R_{\cS_\chi} / (\varpi) = R_{\cS_1} / (\varpi)$ such that each minimal prime $Q \subset R_{\cS_1}$ is contained in one of $\frp_{l_0+1}, \dots, \frp_l$. 

We apply Proposition \ref{prop_patching} to the set of nice primes $\{ \frp_1, \dots, \frp_l \}$ to get the associated objects $D_1$, $D_\chi$, $R_{1, \infty}$, $R_{\chi, \infty}$. Let $Q_{i, \infty}$ denote the pullback of $Q_i$ to $R_{\chi, \infty}$. By repeated application of part 2(a) of Corollary \ref{cor_consequences_of_patching} as in the proof of Proposition \ref{prop_connectedness_of_nice_prime_graph}, we find that each prime $Q_{i, \infty}$ is in the support of $H_{\chi, \infty}$ as $R_{\chi, \infty}$-module. Consequently, for any nice prime $\frp \in \{ \frp_1, \dots, \frp_l \}$ with pullback $\frp_\infty \subset R_{\chi, \infty} / (\varpi) = R_{1, \infty} / (\varpi)$, we have $\frp_\infty \in \Spec \overline{\bbT}_{\chi, \infty}$.

Let $P_\chi$ be any minimal prime of $\bbT_{\chi, \infty, (\frp_\infty)} / (\varpi)$, and let $Q_\chi$ be the unique minimal prime of $\bbT_{\chi, \infty, (\frp_\infty)}$. Then $\dim \bbT_{\chi, \infty, (P_\chi)} = 1$, $\dim \bbT_{\chi, \infty, (P_\chi)} / (\varpi) = 0$, and $Q_\chi \bbT_{\chi, \infty, (P_\chi)}$ is the unique minimal prime of $\bbT_{\chi, \infty, (P_\chi)}$. There are short exact sequences of $\bbT_{\chi, \infty}$-modules for every $q \in \bbZ$:
\[ \xymatrix@1{ 0 \ar[r] & H^q(D_\chi) / (\varpi) \ar[r] & H^q(D_\chi / (\varpi)) \ar[r] & H^{q+1}(D_{\chi})[\varpi] \ar[r] & 0. } \]
After localizing at $P_\chi$, we obtain short exact sequences
\[ \xymatrix@1{ 0 \ar[r] & H^q(D_\chi)_{(P_\chi)} / (\varpi) \ar[r] & H^q(D_\chi / (\varpi))_{(P_\chi)} \ar[r] & H^{q+1}(D_{\chi})_{(P_\chi)}[\varpi]\ar[r] & 0. } \]
According to part 2(b) of Corollary \ref{cor_consequences_of_patching}, we have $H^q(D_\chi)_{(Q_\chi)} = 0$ if $q \neq -d$ and $H^{-d}(D_\chi)_{(Q_\chi)} \neq 0$. Note that a finite $\bbT_{\chi, \infty, (P_\chi)}$-module $M$ is of finite length if it is $\varpi^\infty$-torsion; and for any such $M$, we have an equality
\[ \operatorname{length}_{\bbT_{\chi, \infty, (P_\chi)}}(M / (\varpi)) = \operatorname{length}_{\bbT_{\chi, \infty, (P_\chi)}}(M[\varpi]). \]
It follows that we have an equality
\[ \sum_{i \in \bbZ} (-1)^i \operatorname{length}_{\bbT_{\chi, \infty, (P_\chi)}}( H^i(D_\chi /(\varpi))_{(P_\chi)} ) = (-1)^{d} \operatorname{length}_{\bbT_{\chi, \infty, (P_\chi)}}H^{-d}(D_\chi)_{(P_\chi)}^\text{free} / (\varpi), \]
where the superscript `free' denotes the quotient by the $\varpi^\infty$-torsion subgroup. By Nakayama's lemma, this length is non-zero since 
\[ H^{-d}(D_\chi)_{(P_\chi)}^\text{free} [\varpi^{-1}] = H^{-d}(D_\chi)_{(P_\chi)}[\varpi^{-1}] = H^{-d}(D_\chi)_{(Q_\chi)} \]
is non-zero. Note that $\bbT_{\chi, \infty, (P_\chi)}[\varpi^{-1}] = \bbT_{\chi, \infty, (Q_\chi)}$.

The minimal prime $P_{\chi}$ of $\bbT_{\chi, \infty} / (\varpi)$ corresponds to a minimal prime $P_{1}$ of $\bbT_{1, \infty}/ (\varpi)$. A similar line of reasoning shows that the groups $H^q(D_1)_{(P_1)}$ are not all $\varpi^\infty$-torsion. Indeed,  they would otherwise be of finite length as $\bbT_{1, \infty, (P_1)}$-modules, implying that we have
\[ \sum_{i \in \bbZ} (-1)^i \operatorname{length}_{\bbT_{1, \infty, (P_1)}}( H^i(D_1 / (\varpi))_{(P_1)} ) = \sum_{i \in \bbZ} (-1)^i \operatorname{length}_{\bbT_{\chi, \infty, (P_\chi)}}( H^i(D_\chi / (\varpi))_{(P_\chi)} ) = 0, \]
contradicting what we have just seen. It follows that there is a prime $Q_1 \subset P_1$ of $\bbT_{1, \infty}$ of characteristic 0.

Let us recap. We have shown that for any nice prime $\frp \in \{ \frp_1, \dots, \frp_l \}$, and for any minimal prime $P_1$ of $\bbT_{1, \infty} / (\varpi)$ which is contained in $\frp_\infty$, there is a prime $Q_1 \subset \bbT_{1, \infty}$ of characteristic 0 such that $Q_1 \subset P_1$. It follows that for any minimal prime $P_1'$ of $\bbT_{1, \infty, \frp_\infty} / (\varpi)$, there is a prime $Q_1' \subset \bbT_{1, \infty, \frp_\infty}$ of characteristic 0 such that $Q_1' \subset P_1'$.

We now use this to finish the proof of the theorem. Let $Q \subset R_{\cS_1}$ be any minimal prime. We will show that $Q \in \Spec \bbT_1$. By construction, we can find a nice prime $\frp \in \{ \frp_1, \dots, \frp_l \}$ such that $Q \subset \frp$ (viewing $\frp$ as a prime of $R_{\cS_1} / (\varpi)$). We have surjective maps 
\[ A^T_{\cS_1, \frp^\loc} \llbracket X_1, \dots, X_g \rrbracket \to R_{1, \infty, \frp_\infty} \]
and
\[ R_{1, \infty, \frp_\infty} \to \bbT_{1, \infty, \frp_\infty} / (J_{1, \infty}). \]
We consider the action of $A^T_{\cS_1, \frp^\loc} \llbracket X_1, \dots, X_g \rrbracket$ on $H^\ast(D_1)_{\frp_\infty} / (J_{1, \infty})$; it will be enough to show that this action is nearly faithful, in the sense that its support contains every minimal prime of $A^T_{\cS_1, \frp^\loc} \llbracket X_1, \dots, X_g \rrbracket$ (see \cite[Definition 2.1]{Tay08}). 

The support contains the whole of $\Spec A^T_{\cS_1, \frp^\loc} \llbracket X_1, \dots, X_g \rrbracket / (\varpi)$ (by comparison with the $\chi$ situation). Let $Q_1$ be a minimal prime of $A^T_{\cS_1, \frp^\loc} \llbracket X_1, \dots, X_g \rrbracket$, and let $P_1$ be a prime of $A^T_{\cS_1, \frp^\loc} \llbracket X_1, \dots, X_g \rrbracket$ minimal over $(\varpi, Q_1)$. By Proposition \ref{prop_irreducibility_of_completed_local_lifting_ring}, $P_1$ has height 1 and $Q_1$ is the unique prime ideal of $A^T_{\cS_1, \frp^\loc} \llbracket X_1, \dots, X_g \rrbracket$ properly contained in $P_1$. Our work so far shows that the image of $\Spec \bbT_{1, \infty, \frp_\infty}$ in $\Spec A^T_{\cS_1, \frp^\loc} \llbracket X_1, \dots, X_g \rrbracket$ contains a prime $Q_1' \subset P_1$ of characteristic 0. We must have $Q_1 = Q_1'$; and this completes the proof.

\section{Deduction of the main $2$-adic automorphy lifting theorem}\label{sec_deduction_of_main_ALTs}

We can now state the first main theorem of this paper.
\begin{theorem}\label{thm_main_automorphy_theorem}
Let $K$ be a CM number field, and let $\rho : G_K \to \GL_2(\overline{\bbQ}_2)$ be a continuous representation satisfying the following conditions: 
\begin{enumerate}
\item $\overline{\rho}$ is decomposed generic.
\item $\overline{\rho}$ is dihedral. We write $L / K$ for the quadratic extension corresponding to $\overline{\rho}$.
\item $\overline{\rho}$ extends to a homomorphism $\overline{\rho}' : G_{K^+} \to \GL_2(\overline{\bbF}_2)$.
\item For all but finitely many places $v$ of $K$, $\rho|_{G_{K_v}}$ is unramified.
\item For each place $v | 2$ of $K$, $\rho|_{G_{K_v}}$ is ordinary of weight 0 and $\mathrm{WD}(\rho|_{G_{K_v}})$ is unipotently ramified. Equivalently, there is an isomorphism
\[ \rho|_{G_{K_v}} \sim \left( \begin{array}{cc} \alpha_v & \ast \\ 0 & \epsilon^{-1} \beta_v \end{array}\right), \]
where $\alpha_v, \beta_v$ are unramified characters. Moreover, $v$ is ramified in $L$ and $L_v = K_v(\sqrt{\beta})$, where $\beta \in K_v^\times$ has odd valuation. Finally, the extension $K_v(\sqrt{-1}) / K_v$ is unramified.
\end{enumerate}
Then $\rho$ is automorphic: there exists an isomorphism $\iota : \overline{\bbQ}_2 \cong \bbC$ and a cuspidal automorphic representation $\pi$ of $\GL_2(\bbA_K)$, cohomological of weight 0 and $\iota$-ordinary, such that $\rho \cong r_\iota(\pi)$.
\end{theorem}
We make some remarks on the hypotheses in the above theorem. The first condition is used in the proof of local-global compatibility in \cite{10authors}. The second and third conditions ensure that we can verify the residual automorphy of $\overline{\rho}$ by using automorphic induction over $K^+$. The fourth condition is a necessary one for Galois representations to be related to automorphic forms. The fifth condition ensures that the variable weight local lifting rings at 2 are smooth. It is similar to imposing a `tr\`es ramifi\'ee' type condition. 

We will deduce Theorem \ref{thm_main_automorphy_theorem} from the following result, which is essentially a reformulation of Corollary \ref{cor_application_of_taylor_wiles}.
\begin{theorem}\label{thm_preliminary_automorphy_lifting_theorem}
Let $K$ be a CM number field, and let $\rho :  G_K \to \GL_2(\overline{\bbQ}_2)$ be a continuous representation satisfying the following conditions:
\begin{enumerate}
\item $-1$ is a square in $K$.
\item $\overline{\rho}$ is decomposed generic. 
\item $\overline{\rho}$ is dihedral.
\item For all but finitely many places $v$ of $K$, $\rho|_{G_{K_v}}$ is unramified. For every finite place $v \nmid 2$ of $K$, $\rho|_{G_{K_v}}$ is unipotently ramified. 
\item For each place $v | 2$ of $K$, $\rho|_{G_{K_v}}$ is ordinary of weight 0, $\mathrm{WD}(\rho|_{G_{K_v}})$ is unipotently ramified, and there is an isomorphism 
\[ \overline{\rho}|_{G_{K_v}} \sim \left( \begin{array}{cc} 1 & \psi_v \\ 0 & 1 \end{array} \right), \]
where $\psi_v : G_{K_v} \to \overline{\bbF}_2$ cuts out a quadratic extension of $K_v$ of the form $K_v(\sqrt{\beta})$, where $\beta \in K_v^\times$ has odd valuation.
\item There exists a place $v_0 | 2$ of $K$ such that the embedding $K \hookrightarrow K_{v_0}$ is bijective on 2-power roots of unity.
\item There exists an isomorphism $\iota : \overline{\bbQ}_2 \cong \bbC$ and a cuspidal automorphic representation $\pi$ of $\GL_2(\bbA_{K})$ satisfying the following conditions:
\begin{enumerate}
\item $\pi$ is cohomological of weight 0 and $\iota$-ordinary. If $v | 2$, then $\pi_v$ is unipotently ramified.
\item There is an isomorphism $\overline{r_\iota(\pi)} \cong \overline{\rho}$.
\item Let $R$ denote the set of prime-to-2 places of $K$ at which $\pi$ or $\rho$ is ramified.  If $v \in R$, then $\overline{\rho}|_{G_{K_v}}$ is trivial and there is an isomorphism $\iota^{-1}\pi_v \cong i_B^G \chi_{v, 1} \otimes \chi_{v, 2}$, where $\chi_{v, 1}, \chi_{v, 2} : K_v^\times \to \overline{\bbZ}_p^\times$ are characters such that $\chi_{v, 1}|_{\cO_{K_v}^\times} = \chi_{v, 2}|_{\cO_{K_v}^\times}^{-1}$, $\chi_{v, 1}|_{\cO_{K_v}^\times} \equiv  1 \text{ mod }\ffrm_{\overline{\bbZ}_p}$, and $\chi_{v, 1}|_{\cO_{K_v}^\times} \neq \chi_{v, 2}|_{\cO_{K_v}^\times}$.
\item For each rational prime $l$ which is either even or which lies below a place of $R$, or which is ramified in $K$, there exists an imaginary quadratic subfield of $K$ in which $l$ splits.
\end{enumerate}
\item We have $[K^+ : \bbQ] > 10 + 8 |R|$. 
\end{enumerate}
Then $\rho$ is automorphic.
\end{theorem}
\begin{proof}[Proof of Theorem \ref{thm_main_automorphy_theorem}]
Let $K' / K$ be a soluble CM extension in which the 2-adic places of $K$ are unramified, and such that $\overline{\rho}|_{G_{K'}}$ is decomposed generic. Then the following properties hold:
\begin{itemize}
	\item $\rho|_{G_{K'}}$ satisfies the assumptions of Theorem \ref{thm_main_automorphy_theorem}.
	\item If $\rho|_{G_{K'}}$ is automorphic, then so is $\rho$ (by soluble descent). 
\end{itemize}
Indeed, $K' / K$ must in this case be linearly disjoint from the extension of $K$ cut out by $\overline{\rho}$. In proving the theorem we are therefore free to replace $K$ by $K'$ and $\rho$ by $\rho|_{G_{K'}}$. We will make several such reductions.

If $v \in S_2$, let $m_v$ denote the number of 2-power roots of unity in $K_v$, and let $m$ be the greater of 4 and $\min\{ m_v | v \in S_2 \}$. Then each place $v \in S_2$ is unramified in the extension $K(\zeta_m) / K$. After replacing $K$ by $K(\zeta_{m})$, we can therefore assume moreover that the following conditions are satisfied:
\begin{itemize}
	\item $-1$ is a square in $K$.
	\item There exists a place $v_0 | 2$ of $K$ such that the embedding $K \hookrightarrow K_{v_0}$ is bijective on 2-power roots of unity.
\end{itemize}
Here are are applying Lemma \ref{lem:decomposed_generic_after_extension} and Lemma \ref{lem:not_bad_dihedral_after_extension} in order to see that $\overline{\rho}$ remains decomposed generic and dihedral after replacing $K$ by $K(\zeta_m)$. Henceforth we will only make base changes in which the primes of $K$ above 2 split, in order to avoid disturbing the above condition on roots of unity. 

Let $l_0$ be a rational prime which is decomposed generic for $\overline{\rho}$ (note that this forces $l_0 \equiv 1 \text{ mod }4$). After a soluble base change in which the $2$-adic and $l_0$-adic places of $K$ split, we can assume that the following additional conditions is satisfied:
\begin{itemize}
\item Each place of $K$ at which $\rho$ is ramified is split over $K^+$.
\item For each finite place $v \nmid 2$ of $K$, either $\rho$ is unramified at $v$ and $v^c$ or $\overline{\rho}|_{G_{K_v}}$ is trivial, $\overline{\rho}|_{G_{K_{v^c}}}$ is trivial, and $\rho|_{G_{K_v}}$, $\rho|_{G_{K_{v^c}}}$ are unipotently ramified and $q_v \equiv q_{v^c} \equiv 1 \text{ mod }4$.
\end{itemize}
Let $\overline{\rho}^+$ denote the extension of $\overline{\rho}$ to $G_{K^+}$, and let $R$ denote the set of prime-to-2 places of $K$ such that $\rho|_{G_{K_v}}$ or $\rho|_{G_{K_{v^c}}}$ is ramified. By Lemma \ref{lem_existence_of_residually_dihedral_cusp_form} below, we can find a finite set $T$ of finite places of $K^+$ and a cuspidal automorphic representation $\pi^+$ of $\GL_2(\bbA_{K^+})$ satisfying the following conditions:
\begin{itemize}
\item For each $v \in T$, $v \nmid 2$ and $v$ does not lie below a place of $R$ or a place dividing $l_0$. Moreover, $\pi^+_v$ is not a twist of the Steinberg representation. 
\item  $\pi^+$ is cohomological of weight 0 and $\iota$-ordinary, and $\overline{r_\iota(\pi^+)} \cong \overline{\rho}^+$.
\item For each place $v | 2$ of $K^+$, $\pi^+_v$ is unipotently ramified.
\item For each place $v$ of $K^+$ below a place of $R$, there is an isomorphism $\iota^{-1} \pi^+_v \cong i_B^G \chi_{v, 1} \otimes \chi_{v, 2}$, where $\chi_{v, 1}, \chi_{v, 2} : (K^+_v)^\times \to \overline{\bbZ}_p^\times$ are characters such that $\chi_{v, 1}|_{\cO_{K^+_v}^\times} \equiv \chi_{v, 2}|_{\cO_{K^+_v}^\times} \equiv 1 \text{ mod }\ffrm_{\overline{\bbZ}_p}$ and $\chi_{v, 1}|_{\cO_{K^+_v}^\times} = \chi_{v, 2}|_{\cO_{K^+_v}^\times}^{-1}$ but $\chi_{v, 1}|_{\cO_{K^+_v}^\times} \neq \chi_{v, 2}|_{\cO_{K^+_v}^\times}$.
\item For any finite place $v$ of $K^+$ not dividing $T$, 2, or $R$, $\pi^+_v$ is unramified. 
\end{itemize}
After a further soluble CM base change, we can assume that $T$ is empty and that $[K^+ : \bbQ] / 4 > 10 + 8 |R|$. 

After adjoining two imaginary quadratic fields to $K$ and replacing $R$ by the set of places dividing $R$, we can assume moreover that the following two conditions are satisfied:
\begin{itemize}
	\item $[K^+ : \bbQ] > 10 + 8 |R|$ (which replaces the assumption $[K^+ : \bbQ] / 4 > 10 + 8 |R|$).
	\item Let $S = R \cup \{ v | 2 \}$. If $l$ is a rational prime lying below an element of $S$, or which is ramified in $K$, then there exists an imaginary quadratic subfield of $K$ in which $l$ splits.
\end{itemize}
To see this, let $\Sigma$ denote the set of rational primes lying below an element of $R$, or ramified in $K$, together with 2 and $l_0$. Let $p$ be a prime such that $p \equiv -1 \text{ mod }l$ for all $l \in \Sigma$ and such that $p \equiv -1 \text{ mod }8$. Let $E_1 = \bbQ(\sqrt{-p})$. Then each prime of $\Sigma$ splits in $E_1$. Let $q, r$ be primes not in $\Sigma$ such that $q \equiv 1 \text{ mod }8$, $r \equiv -1 \text{ mod }8$, $q \equiv -1 \text{ mod }p$, $r \equiv 1 \text{ mod }p$ and $q \equiv r \equiv 1 \text{ mod }l_0$, and let $E_2 = \bbQ(\sqrt{-qr})$. Then $p$ splits in $E_2$, $q$ splits in $\bbQ(\sqrt{-1}) \subset K$, and $r$ splits in $E_1$ (use quadratic reciprocity). Moreover the primes 2 and $l_0$ split in $E_1$ and $E_2$. We can therefore replace $K$ by the compositum $K \cdot E_1 \cdot E_2$.

We have now reduced ourselves to a situation where all of the hypotheses of Theorem \ref{thm_preliminary_automorphy_lifting_theorem} are satisfied. This completes the proof.
\end{proof}
\begin{lemma}\label{lem_existence_of_residually_dihedral_cusp_form}
Let $K$ be a totally real field such that $[K : \bbQ]$ is even, and suppose given a continuous representation $\overline{\rho} : G_K \to \GL_2(\overline{\bbF}_2)$ and a finite set $R$ of finite places of $K$ satisfying the following conditions:
\begin{enumerate}
\item $\overline{\rho}$ is absolutely irreducible and its image is soluble. We write $L / K$ for the unique quadratic extension contained in the field cut out by $\overline{\rho}$.
\item For each place $v | 2$ of $K$, $v$ is ramified in $L$ and $\overline{\rho}|_{G_{K_v}}$ is non-trivial and factors through the group $\Gal(L_v / K_v)$. Moreover, $L_v$ has the form $L_v = K_v(\sqrt{\beta})$, where $\beta \in K_v^\times$ has odd valuation.
\item For each place $v \nmid 2$ of $K$, $\overline{\rho}|_{G_{K_v}}$ is unramified.
\item For each place $v \in R$, $\overline{\rho}|_{G_{K_v}}$ is trivial and $q_v \equiv 1 \text{ mod } 4$.
\end{enumerate}
Suppose given moreover another finite set of places $\Sigma$ of $K$, not containing any 2-adic place of $K$ or any place of $R$. Let $\iota : \overline{\bbQ}_2 \cong \bbC$ be an isomorphism. Then we can find a finite set $T$ of finite places of $K$ and a cuspidal automorphic representation $\pi$ of $\GL_2(\bbA_K)$ satisfying the following conditions:
\begin{enumerate}
\item $\pi$ is cohomological of weight 0 and $\iota$-ordinary.
\item For each place $v | 2$ of $K$, $\pi_v$ is unipotently ramified.
\item For each $v \in T$, $v \nmid 2$ and $v \not\in R \cup \Sigma$. Moreover, $\pi_v$ is not a twist of the Steinberg representation.
\item For each place $v \nmid 2$ of $K$ such that $v \not \in R \cup T$, $\pi_v$ is unramified.
\item For each place $v \in R$, there is an isomorphism $\iota^{-1} \pi_v \cong i_B^G \chi_{v, 1} \otimes \chi_{v, 2}$, where $\chi_{v, 1}, \chi_{v, 2} : K_v^\times \to \overline{\bbZ}_2^\times$ are smooth characters such that $\chi_{v, 1}|_{\cO_{K_v}^\times} \equiv \chi_{v, 2}|_{\cO_{K_v}^\times} \equiv 1 \text{ mod }\ffrm_{\overline{\bbZ}_2}$ and $\chi_{v, 1}|_{\cO_{K_v}^\times}= \chi_{v, 2}|_{\cO_{K_v}^\times}^{-1}$, but such that $\chi_{v, 1}|_{\cO_{K_v}^\times} \neq \chi_{v, 1}|_{\cO_{K_v}^\times}^{-1}$.
\item There is an isomorphism $\overline{r_\iota(\pi)} \cong \overline{\rho}$.
\end{enumerate}
\end{lemma}
\begin{proof}
Arguing along the same lines as in \cite[Lemma 5.1.2]{All14}, we can find a cuspidal automorphic representation $\pi_0$ of $\GL_2(\bbA_K)$ which satisfies the following conditions:
\begin{itemize}
\item $\pi_0$ is cohomological of weight 0 and $\iota$-ordinary.
\item There is an isomorphism $\overline{r_\iota(\pi_0)} \cong \overline{\rho}$.
\item For each place $v \in R \cup \Sigma$, $\pi_{0, v}$ is unramified.
\item Let $T$ denote the set of places $v \not\in R \cup \Sigma \cup \{ v | 2 \}$ of $K$ at which $\pi_v$ is ramified. Then for each $v \in T$, $\overline{\rho}(\Frob_v)$ is regular semisimple.
\end{itemize}
Indeed, it is enough to find a lift $\rho : G_K \to \GL_2(\overline{\bbQ}_2)$ with the following properties:
\begin{itemize}
	\item $\rho$ is totally odd and has finite dihedral image.
	\item For all $v \in R \cup \Sigma$, $\rho|_{G_{K_v}}$ is unramified.
	\item Let $T$ denote the set of places $v \not\in R \cup \Sigma \cup \{ v | 2 \}$ of $K$ such that $\rho|_{G_{K_v}}$ is ramified. Then $T$  for each $v \in T$, $\overline{\rho}(\Frob_v)$ is regular semisimple.
\end{itemize}
Then  we can appeal, as in \cite[Lemma 5.1.2]{All14}, to the existence of a Hida family containing the weight 1 form associated to $\rho$. To construct $\rho$, we fix an isomorphism $\overline{\rho} \cong \Ind_{G_L}^{G_K} \overline{\chi}$ for a character $\overline{\chi} : G_L \to k^\times$, and let $\chi : G_L \to \cO^\times$ denote the Teichm\"uller lift of $\overline{\chi}$ (therefore of odd order).  If $\Ind_{G_L}^{G_K} \chi$ is totally odd, then we're done: take $\rho = \Ind_{G_L}^{G_K} \chi$. Otherwise, let $\Sigma_\infty$ denote the set of places of $K$ at which $\Ind_{G_L}^{G_K} \chi$ is not odd; then each place of $\Sigma_\infty$ splits in $L$. Applying Lemma \ref{lem_even_characters} to $\chi / \chi^\sigma$ (where $\sigma \in \Gal(L / K)$ denotes the non-trivial element) and possibly enlarging $\cO$, we can find a finite set $T_0$ of finite places of $L$ and a character $\xi : G_L \to \cO^\times$ with the following properties:
\begin{itemize}
	\item For each $v \in T_0$, $v$ is split over $K$, prime to 2, and $\overline{\rho}(\Frob_v)$ is regular semisimple.
	\item $\xi \equiv 1 \text{ mod }(\varpi)$ and $\xi$ is unramified outside $T_0$.
	\item Let $v$ be a real place of $K$ split in $L$, and let $c, c' \in G_L$ be complex conjugations at the two real places of $L$ above $v$. Then $\chi \xi(c) \chi \xi(c') = -1$.
\end{itemize}
We can then take $\rho = \Ind_{G_L}^{G_K} \chi \xi$.

By switching to a definite quaternion algebra and using types, in the same way as in the proof of \cite[Lemma 3.4.1]{Clo08}, we can find a cuspidal automorphic representation $\pi$ of $\GL_2(\bbA_K)$ which satisfies the following conditions:
\begin{itemize}
\item $\pi$ is cohomological of weight 0 and $\iota$-ordinary.
\item For each place $v | 2$ of $K$, $\pi_v$ is unipotently ramified.
\item There is an isomorphism $\overline{r_\iota(\pi_0)} \cong \overline{\rho}$.
\item For each place $v \in R$, there is an isomorphism $\iota^{-1} \pi_v \cong i_B^G \chi_{v, 1} \otimes \chi_{v, 2}$, where $\chi_{v, 1}, \chi_{v, 2} : K_v^\times \to \overline{\bbZ}_2^\times$ are smooth characters such that $\chi_{v, 1}|_{\cO_{K_v}^\times} \equiv \chi_{v, 2}|_{\cO_{K_v}^\times} \equiv 1 \text{ mod }\ffrm_{\overline{\bbZ}_2}$ and $\chi_{v, 1}|_{\cO_{K_v}^\times}= \chi_{v, 2}|_{\cO_{K_v}^\times}^{-1}$, but such that $\chi_{v, 1}|_{\cO_{K_v}^\times} \neq \chi_{v, 1}|_{\cO_{K_v}^\times}^{-1}$. (Here we are using the condition $q_v \equiv 1 \text{ mod }4$ in order to ensure the existence of characters of $\cO_{K_v}^\times$ of order 4.)
\item For all $v  \not\in R \cup T \cup \{ v | 2 \}$, $\pi_v$ is unramified.
\end{itemize}
We note that this kind of argument requires the choice of a `sufficiently small' level subgroup: this can be achieved by making the level small at primes of $T$. (If $T$ is empty, just adjoint any place $v$ such that $\overline{\rho}(\Frob_v)$ is regular semi-simple.) The proof is complete on noting that if $v \nmid 2$ and $\pi_v$ is a twist of the Steinberg representation, then the characteristic polynomial of $\overline{\rho}(\Frob_v)$ must have a repeated root (because $q_v \equiv 1 \text{ mod }2)$.
\end{proof}
\begin{lemma}\label{lem_even_characters}
	Let $L / K$ be a quadratic extension of number fields, and let $\chi : G_L \to \bbC^\times$ be a non-trivial continuous character of odd order. Let $\Sigma$ be a set of finite places of $K$, and let $\Sigma_\infty$ be a set of real places of $K$ which split in $L$. Then we can find a finite set $T_0$ of finite places of $L$ and a continuous character $\xi : G_L \to \bbC^\times$ with the following properties:
	\begin{enumerate}
		\item No place of $T_0$ lies above a place of $\Sigma$.
		\item If $v \in T_0$, then $v$ is split over $K$, $\chi$ is unramified at $v$ and $\chi(\Frob_v) \neq 1$.
		\item The character $\xi$ has 2-power order. 
		\item Let $v \in \Sigma_\infty$, and let $c, c' \in G_L$ be complex conjugations at the two real places of $L$ lying above $v$. Then $\xi(c) \xi(c') = -1$.
		\item Let $v$ be a real place of $K$ split in $L$ and not lying in $\Sigma_\infty$, and let $c, c' \in G_L$ be the complex conjugations at the two real places of $L$ lying above $v$. Then $\xi(c) = \xi(c') = 1$.
		\item The character $\xi$ is unramified away from $T_0$.
	\end{enumerate}
\end{lemma}
\begin{proof}
	We may assume without loss of generality that $\Sigma$ contains all places of $K$ above which $\chi$ is ramified and all the 2-adic places of $K$. Let $u_1, \dots, u_r$ be representatives for the non-trivial classes of $\cO_L^\times / (\cO_L^\times)^2$. We choose for each $i = 1, \dots, r$ a finite place $v_i$ of $L$ with the following properties:
	\begin{itemize}
		\item $v_i$ does not lie above a place of $\Sigma$. 
		\item $v_i$ is split over $K$.
		\item $\chi(\Frob_{v_i}) \neq 1$.
		\item $v_i$ is inert in $L(\sqrt{u_i}) / L$.
	\end{itemize}
	This is possible by the Chebotarev density theorem, since $\chi$ has odd order. We set $T_0 = \{ v_1, \dots, v_r \}$. Then if $u \in \cO_L^\times$ and, for each $v \in T_0$, the image of $u$ in $k(v)^\times$ is a square, then $u \in (\cO_L^\times)^2$. Class field theory implies the existence of a character $\xi$ with the desired properties. 
\end{proof}
\section{Automorphy lifting theorems for $p > 2$}\label{sec_odd_aut_lift}
Theorem \ref{thm_main_automorphy_theorem} is the main technical input for our applications to the modularity of elliptic curves. We will also need the following automorphy lifting theorem, valid for odd primes $p$.
\begin{theorem}\label{thm_automorphy_at_odd_primes}
	Let $p$ be an odd prime, let $K$ be a CM number field, and let $\rho : G_K \to \GL_2(\overline{\bbQ}_p)$ be a continuous representation satisfying the following conditions:
	\begin{enumerate}
		\item $\overline{\rho}$ is decomposed generic and $\overline{\rho}|_{G_{K(\zeta_p)}}$ is absolutely irreducible.
		\item $\det \rho = \epsilon^{-1}$.
		\item For all but finitely many places $v$ of $K$, $\rho|_{G_{K_v}}$ is unramified. 
		\item There exists $\lambda \in (\bbZ_{+, 0}^2)^{\Hom(K, \overline{\bbQ}_p)}$ such that $\rho$ is ordinary of weight $\lambda$.
		\item There exists an isomorphism $\iota : \overline{\bbQ}_p \cong \bbC$ and a cuspidal, $\iota$-ordinary, cohomological automorphic representation $\pi$ of $\PGL_2(\bbA_{K})$ such that $\overline{r_\iota(\pi)} \cong \overline{\rho}$.
\item If $p = 5$ and the projective image of $\overline{\rho}(G_{K(\zeta_5)})$ is conjugate to $\PSL_2(\bbF_5)$, we assume further that the extension of $K$ cut out by the projective image of $\overline{\rho}$ does not contain $\zeta_5$. 
	\end{enumerate}
	Then $\rho$ is automorphic: there is a cuspidal automorphic representation $\Pi$ of $\PGL_2(\bbA_K)$, cohomological of weight $\iota\lambda$ and $\iota$-ordinary, such that $\rho \cong r_\iota(\Pi)$.
\end{theorem}

We prove this theorem in Appendix~\ref{sec:p-odd-proofs}.
It is not directly implied by the automorphy lifting theorems of \cite{10authors} because they contain the following two additional hypotheses:
\begin{itemize}
    \item $\overline{\rho}(G_{K(\zeta_p)})$ is enormous (\cite[Definition~6.2.28]{10authors}),
    \item there is $\sigma \in G_K - G_{K(\zeta_p)}$ such that $\overline{\rho}(\sigma)$ is scalar.
\end{itemize}
The second bullet point is problematic because we need to allow $\zeta_3 \in K$ in the $p = 3$ case (see Proposition~\ref{prop_modularity_of_certain_curves} below). 
This will force us to deal with non-neatness issues similar to those in \S\ref{sec_application_of_TW_method} and we again use Theorem~\ref{thm_boundedness_of_good_dihedral_cohomology}. 

The first bullet point is problematic because the Galois representations associated to elliptic curves generically satisfy $\overline{\rho}(G_{K(\zeta_p)}) = \SL_2(\bbF_p)$, which isn't enormous when $p = 3,5$.
We recall that for a finite field $k$ of characteristic $p$, letting $\ad^0$ denote the trace 0 matrices in $\mathrm{M}_{2 \times 2}(k)$,
a subgroup $H\subseteq \GL_2(k)$ is enormous if it satisfies the following:
\begin{enumerate}
    \item $H$ has no non-trivial $p$-power order quotient.
    \item $H^0(H, \ad^0) = H^1(H, \ad^0) = 0$.
    \item For any simple $k[H]$-submodule $W \subseteq \ad^0$, there is a regular semisimple $h \in H$ such that $W^h \ne 0$.
\end{enumerate}
The group $\SL_3(\bbF_3) \subset \GL_2(\bbF_3)$ is not enormous because it admits a quotient of order $3$. 
The group $\SL_2(\bbF_5) \subset \GL_2(\bbF_5)$ is not enormous because $H^1(\SL_2(\bbF_5), \ad^0) \ne 0$. 
In both cases, the group in question satisfies the remaining properties of the definition (see \cite[Appendix A]{Bar13}). 

To handle $p = 3$ and $\overline{\rho}(G_{K(\zeta_3)}) = \SL_2(\bbF_3)$, we work with fixed determinant deformation problems. 
The point here is that the order $3$ quotient of $\SL_2(\bbF_3)$ is only problematic in showing the existence of sufficiently many Taylor--Wiles primes because of the scalar subspace $k \subset \ad\overline{\rho}$ of the adjoint representation. 
Fixing determinants results in $\ad^0 \overline{\rho}$ coefficients in the Galois cohomology, avoiding this issue with the centre. 
Because we fix determinants, we work with the locally symmetric spaces for $\PGL_2$, as opposed to $\GL_2$, appealing to Corollaries~\ref{cor_PGLn_mod_p_Gal_rep} and~\ref{cor:local-global-PGLn-ord} for the requisite Galois representations and their local properties.

The non-vanishing of $H^1(\SL_2(\bbF_5), \ad^0)$ is more problematic, and we can only treat the case when the projective image of $\overline{\rho}|_{G_{K(\zeta_5)}}$ is conjugate to $\PSL_2(\bbF_5)$ under the additional assumption that the extension of $K$ cut out by the projective image of $\overline{\rho}$ does not contain $\zeta_5$. 
The point here is that one can still show the existence of sufficiently many Taylor--Wiles primes 
because the image of $H^1(K_S/K, \ad^0\overline{\rho}(1))$ in $H^1(K_S/K(\zeta_5), \ad^0\overline{\rho})$ does not contain $H^1(\SL_2(\bbF_5), \ad^0)$. 
This is the source of our assumption that $\zeta_5 \notin K$ in many of the main theorems of this paper. 

\section{The 2-3 switch}\label{sec_2-3_switch}

In this section we establish our main results, including Theorem \ref{thm_application_to_serre}, which asserts the existence of modular elliptic curves with prescribed mod $p$ Galois representations over CM fields for $p = 2, 3$ or $5$. 

Let $K$ be a CM field. We recall that we say an elliptic curve $E$ over $K$ is modular if either $E$ has CM, or there exists a cuspidal, regular algebraic automorphic representation $\pi$ of $\GL_2(\bbA_K)$, a prime $p$, and an isomorphism $\iota : \overline{\bbQ}_p \to \bbC$ such that $\rho_{E, p} \otimes \epsilon^{-1} \cong r_\iota(\pi)$.

We call $\pi$ the automorphic representation corresponding to $E$. The following lemma shows that this is a robust notion.	\begin{lemma}\label{lem_more_on_modularity}
	Let $K$ be a CM field, and let $E$ be a non-CM modular elliptic curve. Then:
	\begin{enumerate}
		\item $\pi$ is uniquely determined by $E$. Moreover, $\pi$ has weight 0. 
		\item For every prime $p$ and every isomorphism $\iota : \overline{\bbQ}_p \to \bbC$, there is an isomorphism $r_\iota(\pi) \cong \rho_{E, p} \otimes \epsilon^{-1}$.
		\item For every prime $p$ and every isomorphism $\iota : \overline{\bbQ}_p \to \bbC$, and every finite place $v \nmid p$ of $K$, there is an isomorphism
		\[ \mathrm{WD}(\rho_{E, p} \otimes \epsilon^{-1})^{F-ss} \cong \rec^T_{K_v}(\pi_v). \]
	\end{enumerate}
\end{lemma}
\begin{proof}
	If $\pi, \pi'$ are cuspidal, regular algebraic automorphic representations of $\GL_2(\bbA_K)$, both associated to $E$, then for all but finitely many places $v$ of $K$, we have $\pi_v \cong \pi'_v$. Strong multiplicity one then implies that we have $\pi \cong \pi'$. 
	
	By the Chebotarev density theorem, for any prime $p$ and any isomorphism $\iota : \overline{\bbQ}_p \to \bbC$, we have $\rho_{E, p} \otimes \epsilon^{-1} \cong r_\iota(\pi)$ (note that $r_\iota(\pi)$ is semisimple by definition; the irreducibility of $\rho_{E, p}$ is a theorem of Serre \cite{Ser98}). By \cite[Lemma 7.1.10]{10authors}, there exists $p, \iota$ such that for each place $v | p$ of $K$, $r_\iota(\pi)|_{G_{K_v}}$ is crystalline of Hodge--Tate weights $\mathrm{HT}_\tau(r_\iota(\pi)) = \{ \lambda_{\iota^{-1} \tau, 1} + 1, \lambda_{\iota^{-1} \tau, 2} \}$ for each  $\tau \in \Hom_{\bbQ_p}(K_v, \overline{\bbQ}_p)$ (where $\pi$ is of weight $\lambda \in (\bbZ^2_+)^{\Hom(K, \bbC)}$). Since the Hodge--Tate weights of $\rho_{E, p} \otimes \epsilon^{-1}$ are $\{0, 1 \}$ (for any $\tau$), this shows that $\pi$ is necessarily of weight $\lambda = 0$.
	
	This proves the first two parts of the lemma. For the third, we use the  \cite[Theorem~1]{ilavarma}, which shows (with notation as in the statement of the lemma) that
	\[ \mathrm{WD}(\rho_{E, p} \otimes \epsilon^{-1}|_{G_{K_v}})^{F-ss} \prec \rec_{K_v}^T(\pi_v). \]
	Here we recall that `$\prec'$ means that the representations agree up to semisimplification, and (since we working with 2-dimensional representations) if the right-hand Weil--Deligne representation has trivial $N$, then so does the left-hand side. It remains to show that if $v \nmid p$ is a place such that $\rec_{K_v}^T(\pi_v)$ has non-trivial $N$, then $\mathrm{WD}(\rho_{E, p} \otimes \epsilon^{-1}|_{G_{K_v}})^{F-ss}$ also has non-trivial $N$. This is a consequence of purity. 
\end{proof}
\begin{corollary}\label{cor_tate_to_steinberg}
	Let $K$ be a CM field, and let $E$ be a non-CM modular elliptic curve. Let $\pi$ be the corresponding automorphic representation of $\GL_2(\bbA_K)$. Suppose that there exists a rational prime $p$ such that for each place $v | p$ of $K$, $E_{K_v}$ has potentially multiplicative reduction. Then for any isomorphism $\iota : \overline{\bbQ}_p \to \bbC$, $\pi$ is $\iota$-ordinary. 
\end{corollary}
\begin{proof}
	Under these hypotheses, Lemma \ref{lem_more_on_modularity} implies that for each place $v | p$ of $K$, $\pi_v$ is a twist of the Steinberg representation. By \cite[Lemma~5.1.5]{Ger18}, $\pi$ is $\iota$-ordinary for any $\iota : \bbQ_p \cong \bbC$.
\end{proof}
The argument that follows is inspired by the `3-5' switch of Wiles \cite[Theorem 5.2]{Wil95}. We begin with some preliminary lemmas.
\begin{lemma}\label{lem_when_Kummer_extensions_are_CM}
Let $K$ be an imaginary CM field, and let $n \geq 2$ be an integer. Let $a \in K$ satisfy $\mathbf{N}_{K / K^+}(a) = 1$, and let $K_a$ denote the splitting field of $x^n - a$. Then $K_a / K$ is a soluble CM extension.
\end{lemma}
\begin{proof}
Let $C$ denote the norm 1 torus of $K / K^+$; thus it is a form of $\bbG_m$ over $K^+$, and the map $[n] : C \to C$ has degree $n$. For any place $v | \infty$ of $K^+$, $C(K^+_v)$ is isomorphic to $S^1$, and the map $[n] : C(K^+_v) \to C(K^+_v)$ is surjective, with kernel of order $n$. We find that for any $a \in C(K^+) = \ker \mathbf{N}_{K / K^+}$, the field $K^+([n]^{-1}(a))$ is totally real. Since $K_a = K \cdot K^+([n]^{-1}(a))$, we find that $K_a$ is CM. 
\end{proof}
N. Fakhruddin has shown us a generalization of Lemma \ref{lem_when_Kummer_extensions_are_CM} which gives a necessary and sufficient criterion on $a$ for the field $K_a$ to be CM. \begin{corollary}\label{cor_real_up_to_nth_powers}
Let $K$ be an imaginary CM field, and let $n \geq 2$ be an integer. Let $a \in K$, and let $L / K$ denote the splitting field of the polynomial $x^n - a / a^c$. Then $L / K$ is a soluble CM extension and there exists $z \in L^+$ such that $a \equiv z \text{ mod }(L^\times)^n$.
\end{corollary}
\begin{proof}
The extension $L / K$ is soluble and CM, by Lemma \ref{lem_when_Kummer_extensions_are_CM}. Let $\beta \in L$ be such that $\beta^n = a / a^c$. Then $\mathbf{N}_{L / L^+}(\beta) = \beta \beta^c$ is a totally positive element of $L^+$, and $(\beta \beta^c)^n = 1$. It follows that $\beta \beta^c = 1$, and hence (by Hilbert 90) there exists $b \in L^\times$ such that $b^c / b = \beta$. Then the element $z = a b^n$ lies in $L^+$.
\end{proof}
If $E$ is an elliptic curve over a field $K$, then we write $\Delta_E \in K^\times / (K^\times)^{12}$ for the discriminant of $E$.
\begin{lemma}\label{lem_quotients_of_X(6)}
Let $K$ be a number field, and let $E$, $A$ be elliptic curves over $K$. Let $Y_E$  denote the (fine) moduli space of pairs $(F, \varphi)$, where $F$ is an elliptic curve and $\varphi : F[6] \to E[6]$ is a symplectic isomorphism (i.e.\ respecting Weil pairings). Define $Y_A$ similarly. Then:
\begin{enumerate}
\item Let $V_E \subset \PSp(E[3]) \subset \Aut(Y_E)$ denote the 2-Sylow subgroup (more precisely, the $K$-group with $V_E(\overline{K})$ equal to the 2-Sylow subgroup of $\PSp(E[3])(\overline{K})$.) Define $V_A$ similarly. Suppose that $\Delta_E \equiv \Delta_A \text{ mod }(K^\times)^3$. Then there is an isomorphism $Y_E / V_E \cong Y_A / V_A$ over $K$.
\item Let $C_E \subset \PSp(E[2]) \subset \Aut(Y_E)$ denote the 3-Sylow subgroup (more precisely, the $K$-group with $C_E(\overline{K})$ equal to the 3-Sylow subgroup of $\PSp(E[2])(\overline{K})$.) Define $C_A$ similarly. Suppose that $\Delta_E \equiv \Delta_A \text{ mod }(K^\times)^2$. Then there is an isomorphism $Y_E / C_E \cong Y_A / C_A$ over $K$.
\end{enumerate}
\end{lemma}
\begin{proof}
Both parts are similar, so we just prove the first part. Let $T_3(E)$ denote the set of partitions of $\bbP(E[3](\overline{K}))$ into 2-element subsets. According to \cite[Lemma 4.1]{Fuk17}, there is a $G_K$-equivariant bijection between $T_3(E)$ and the set of cube roots of $\Delta_E$ in $\overline{K}$. The same is true for $T_3(A)$.  Choose $\lambda \in K^\times$ such that $\lambda^3 \Delta_E = \Delta_A$. Multiplication by $\lambda$ determines a $G_K$-equivariant bijection between the set of cube roots of $\Delta_E$ and the set of cube roots of $\Delta_A$, hence a $G_K$-equivariant bijection $f_\lambda : T_3(E) \to T_3(A)$.

The map $\PSp(E[3](\overline{K})) \to \mathrm{Sym}(T_3(E))$ is surjective, so can choose a symplectic isomorphism $f : E[3]_{\overline{K}} \to A[3]_{\overline{K}}$ such that the induced map $T_3(E) \to T_3(A)$ is equal to $f_\lambda$. 
Let us use an overline to denote base change to the algebraic closure $\overline{K}$. We can define a map $\overline{Y}_E \to \overline{Y}_A$ by the formula $(F, \varphi) \mapsto (F, f \circ \varphi)$. Since $V_E \subset \PSp(E[3])$ is the subgroup which acts trivially on $T_3(E)$, we see that the map $\overline{Y}_E / V_E(\overline{K}) \to \overline{Y}_A / V_A(\overline{K})$ is defined over $K$, hence determines a map $Y_E / V_E \to Y_A / V_A$. 
\end{proof}

\begin{lemma}\label{lem_existence_of_elliptic_curves}
Let $K$ be a number field, and let $A, E$ be elliptic curves over $K$ such that $\Delta_A \equiv \Delta_E \text{ mod }(K^\times)^6$ and $E[3]$ is irreducible. Then there exists an elliptic curve $F$ over $K$ such that $\Delta_F \equiv \Delta_A \text{ mod }(K^\times)^6$, and there are (symplectic) isomorphisms $F[2] \cong A[2]$ and $F[3] \cong E[3]$.
\end{lemma}
\begin{proof}
Let $X_A$, $X_E$ denote the compactified moduli space of elliptic curves with symplectic level 6 structure given either by $A[6]$ or $E[6]$, respectively. Let $X$ denote the compactified moduli space of elliptic curves with symplectic level 6 structure given by $A[2] \times E[3]$. Thus $X_A, X_E$ and $X$ are smooth, projective, geometrically connected curves of genus 1 over $K$ which all become isomorphic over $\overline{K}$. With notation as in Lemma \ref{lem_quotients_of_X(6)}, let $V = V_E$ and $C = C_A$; then there are isomorphisms
\[ X_A / V_A \cong X / V \]
and
\[ X / C \cong X_E / C_E. \]
It follows that the quotients $X / V$ and $X / C$ have $K$-rational points (because $X_A$ and $X_E$ do). Taking pre-images, we find that $X$ has $K$-rational divisors of degree 4 and 3, hence a $K$-rational divisor of degree 1. Since $X$ has genus 1, there is a canonical isomorphism $X \cong \Pic^1(X)$. It follows that $X$ has a $K$-rational point. Since we have assumed that $E[3]$ is irreducible, it follows that this $K$-rational point must correspond to an elliptic curve $F$ over $K$ with the required properties (in other words, none of the cusps of $X$ are defined over $K$).
\end{proof}
Let $p$ be an odd prime, let $K$ be a number field, and let $\overline{\rho} : G_K \to \GL_2(\bbF_p)$ be a continuous homomorphism of cyclotomic determinant. Let $V$ be the $\bbF_p[G_K]$-module on which $\overline{\rho}$ acts. A choice of isomorphism $\iota_0 : \wedge^2 V\cong \mu_p$ determines a modular curve $Y_{\overline{\rho}}$, which represents the functor of isomorphism classes of pairs $(E, \iota)$, where $E$ is an elliptic curve and $\iota$ is an isomorphism $\iota : V \cong E[p]$ such that $\wedge^2 \iota = \iota_0$ (where we identify $\wedge^2 E[p] \cong \mu_p$ using the Weil pairing) \cite[\S 1]{She97}.
\begin{lemma}\label{lem_curves_satisfying_local_conditions}
Let $K$ be a number field, and let $p \in \{ 3, 5 \}$. Let $\overline{\rho} : G_K \to \GL_2(\bbF_p) = \GL(V)$ be a homomorphism of cyclotomic determinant, and fix a choice of $\iota_0 : \wedge^2 V\cong \mu_p$. Let $q \neq p$ be a prime. Let $S$ be a finite set of finite places of $K$, and suppose given for each $v \in S$ a non-empty open subset $\Omega_v \subset Y_{\overline{\rho}}(K_v)$. Let $K' / K$ be a finite extension. Then we can find an elliptic curve $E$ over $K$ satisfying the following conditions:
\begin{enumerate}
\item There is an isomorphism $\iota : V \cong E[p]$ such that $\wedge^2 \iota = \iota_0$.
\item For each $v \in S$, $(E_{K_v}, \iota) \in \Omega_v$.
\item $\overline{\rho}_{E, q}(G_{K'(\zeta_q)}) = \SL_2(\bbF_q)$. 
\item There exists a prime $l$ which splits in $K'$ and which is decomposed generic for $\overline{\rho}_{E, q}$.
\item For each place $v | l$ of $K$, $E$ has good reduction at $v$ and the isomorphism class of the reduction (as a curve over $k(v) = \bbF_l$) is independent of the choice of $v$.
\end{enumerate}
\end{lemma}
\begin{proof}
There is an isomorphism between $Y_{\overline{\rho}}$ and an open subset of the projective line over $K$. This is \cite[Lemma 1.1]{She97} when $p = 5$, and the same reference explains how the same proof works in the case $p = 3$.  In particular, $Y_{\overline{\rho}}$ has infinitely many $K$-rational points. By \cite[Proposition 3.2.1]{Ser08}, the set of points in $Y_{\overline{\rho}}(K)$ for which the associated representation $\overline{\rho}_{E, q}|_{G_{K'(\zeta_q)}}$ does not have image $\SL_2(\bbF_q)$ is a thin set. By the Hilbert irreducibility theorem (e.g. in the form of \cite[Theorem 3.5.3]{Ser08}), we can find infinitely many points in $Y_{\overline{\rho}}(K)$ corresponding to pairs $(E, \iota)$ with the following properties:
\begin{itemize}
\item For each $v \in S$, $(E_{K_v}, \iota) \in \Omega_v$.
\item $\overline{\rho}_{E, q}(G_{K'(\zeta_q)}) = \SL_2(\bbF_q)$.
\end{itemize}
To prove the lemma, it suffices to show that we can enlarge $S$ so that the existence of $l$ follows from the above conditions. We recall that by definition, the rational prime $l$ is decomposed generic for $\overline{\rho}_{E, q}$ if $l$ splits in $K$ and for each place $v | l$ of $K$, $\overline{\rho}_{E, q}|_{G_{K_v}}$ is generic (in the sense that it is unramified, and the eigenvalues of $\overline{\rho}_{E, q}(\Frob_v)$ are not in ratio $q_v^{\pm 1} \text{ mod }q$). It therefore suffices to show that there exists a rational prime $l$ satisfying the following conditions:
\begin{itemize}
\item Let $M / K'(\zeta_{p^7 q})$ denote the extension cut out by $\overline{\rho}$. Then $l$ splits in $M$.
\item For each $v \in S$, $v \nmid l$. Moreover, $l$ is prime to $pq$.
\item There exists an integer $a \in [-2 \sqrt{l}, 2 \sqrt{l}]$ prime to $l$ such that the polynomial $X^2 - a X + l$ has discriminant prime to $q$, is a product of two linear factors in $\bbZ_p[X]$, and the two roots $\alpha, \beta \in \bbZ_p$ satisfy $\alpha \equiv \beta \equiv 1 \text{ mod }p$.
\end{itemize}
Indeed, it follows that we can find an elliptic curve $E_l$ over $\bbF_l$ such that $\overline{\rho}_{E_l, p}$ is the trivial representation and $\overline{\rho}_{E_l, q}(\Frob_l)$ is regular semisimple. If we adjoin each place $v | l$ to $S$ and choose $\Omega_v$ to be such that each point $x \in \Omega_v$ corresponds to an elliptic curve $E_x$ over $K_v$ with good reduction isomorphic to $E_l$, then any elliptic curve $E$ as above will automatically have the property that $\overline{\rho}_{E, q}$ is decomposed generic. 

We claim that any sufficiently large prime $l$ which splits in $M$ will do the job. We just need to justify the existence of $a$. By the Chinese remainder theorem, it suffices to note that if $a \in \bbZ$, $l \equiv 1 \text{ mod }p^7$ is a prime, and $a \equiv 2 + p^4 + p^6 \text{ mod }(p^7)$, then Hensel's lemma implies that the polynomial $X^2 - a X + l$ factors in $\bbZ_p[X]$, and the two roots $\alpha, \beta \in \bbZ_p$ satisfy $\alpha \equiv \beta \equiv 1 \text{ mod }p$ (in fact, there is a unique root congruent to $1 - p^2$ mod $(p^7)$). This completes the proof.
\end{proof}
When $p = 2$ the analogous functor of pairs $(E, \iota)$ is not representable. We record the following simple variant of Lemma \ref{lem_curves_satisfying_local_conditions} which applies in this case.
\begin{lemma}\label{lem_curves_satisfying_local_conditions_p_equals_2}
Let $K$ be a number field, and let $\overline{\rho} : G_K \to \GL_2(\bbF_2) = \GL(V)$ be a continuous homomorphism. Let $S$ be a finite set of places of $K$ such that for each $v \in S$, there is a Tate elliptic curve $E_v$ over $K_v$ such that $\overline{\rho}_{E_v, 2} \cong \overline{\rho}|_{G_{K_v}}$, and let $q$ be an odd prime. Then we can find an elliptic curve $E$ over $K$ satisfying the following conditions:
\begin{enumerate} 
\item There is an isomorphism $V \cong E[2]$.
\item For each $v \in S$, $E_{K_v}$ is a Tate elliptic curve.
\item $\overline{\rho}_{E, q}$ is decomposed generic and $\overline{\rho}_{E, q}(G_{K(\zeta_q)}) = \SL_2(\bbF_q)$.
\end{enumerate} 
\end{lemma}
\begin{proof}
Let $f(X) = X^3 + a X + b \in K[X]$ be a polynomial such that there is an isomorphism 
\[ \Hom_K(K[X] / (f(X)), \overline{K}) \cong V - \{ 0 \} \]
 of $G_K$-sets. In \cite[Theorem 1]{Rub01} appears, in terms of $f(X)$, the equation of an elliptic curve $\cE$ over a Zariski open subset $U \subset \bbA_K^2$ with the following property: for any field extension $F / K$ and for any $x \in U(F)$, there is an isomorphism $\overline{\rho}_{\cE_x, 2} \cong \overline{\rho}|_{G_F}$; and for any field extension $F / K$ and any elliptic curve $E / F$ with $\overline{\rho}_{E, 2} \cong \overline{\rho}|_{G_F}$, there is a point $x \in U(F)$ such that $\cE_x \cong E$. 

We claim that the homomorphism $\pi_1^{\text{\'et}}(U_{\overline{K}}) \to \SL_2(\bbF_q)$ associated to $\cE[q]$ is surjective. Indeed, consider the Legendre elliptic curve $E_\lambda : y^2 = x(x-1)(x-\lambda)$ over $F = \overline{K}(\lambda)$. There exists a point $x \in U(F)$ such that $\cE_x \cong E_\lambda$. It is well-known that the representation $\overline{\rho}_{E_\lambda, q} : G_F \to \SL_2(\bbF_q)$ associated to the Legendre curve is surjective for any odd prime $q$. This homomorphism factors through a homomorphism $G_F \to \pi_1^{\text{\'et}}(U_{\overline{K}})$, yielding the claim. 

By assumption, we can find for each $v \in S$ a point $x_v \in U(K_v)$ such that $\cE_{x_v}$ is a Tate elliptic curve. Let $\Omega_v \subset U(K_v)$ be an open neighbourhood of $x_v$ such that for any $y \in \Omega_v$, $\cE_y$ is a Tate elliptic curve. The same argument as in the proof of Lemma \ref{lem_curves_satisfying_local_conditions} shows that we can find an odd prime $l$, not dividing any element of $S$, that splits in $K(\zeta_q, \overline{\rho}_{E, 2})$, and such that for each place $v | l$ of $K$, there is an elliptic curve $E_v$ over $k(v) = \bbF_l$ such that $\overline{\rho}_{E_v, 2}$ is trivial and the discriminant of $\det(X - \rho_{E_v, q}(\Frob_v)) \in \bbZ[X]$ is prime to $q$. Let $T$ denote the set of $l$-adic places of $K$ and for $v \in T$, let $\Omega_v \subset U(K_v)$ be a non-empty open subset such that for each $y \in \Omega_v$, $\cE_y$ has good reduction isomorphic to $E_v$. 

Applying the Hilbert irreducibility theorem as in the proof of Lemma \ref{lem_curves_satisfying_local_conditions}, we can find a point $x \in U(K)$ such that for each $v \in S \cup T$, $x \in \Omega_v$, and such that $\overline{\rho}_{\cE_x, q}(G_{K(\zeta_q)}) = \SL_2(\bbF_q)$. The proof is complete on taking $E = \cE_x$. 
\end{proof}
\begin{lemma}\label{lem_existence_of_certain_S3_extensions}
Let $K$ be a number field, and let $z \in K^\times$ be an element which is not a square. Then there exists a thin set $\Omega \subset K$ with the following property: let $a \in K - \Omega$. Then:
\begin{enumerate}
\item The quantity $\beta(a) = - \frac{4}{27}a^3 - \frac{1}{16 \times 27} z$ is not a square in $K$. 
\item Let $b = \sqrt{\beta(a)}$, $L = K(b)$, and $f(x) = x^3 + a x + b \in L[x]$. Then $f(x)$ is irreducible and $\disc f$ is not a square in $L$. Consequently, the splitting field of $f(x)$ over $L$ is an $S_3$-extension of $L$.
\end{enumerate}
\end{lemma}
\begin{proof}
We consider the tower of affine varieties 
\[ Y \to X \to \bbA^1_K, \]
where $\bbA^1_K = \Spec K[a]$, $X$ is the curve in the $(a, b)$-plane given by the equation $z = -16(4a^3 + 27b^2)$, and $Y$ is given by the equation $x^3 + a x + b = 0$. It is clear that both morphisms $Y \to X$ and $X \to \bbA^1_K$ are finite and dominant, and hence $X$, $Y$ are $K$-schemes of dimension 1. We claim that $Y$ is geometrically irreducible.

We first show why this claim implies the truth of the lemma. Let $E = K(\sqrt{z})$. Then $Y_E \to \bbA^1_E$ is generically a Galois cover of degree 6. Indeed, it suffices to check that $\Aut(Y_E / \bbA^1_E)$ contains elements of order 2 and 3. The cover $Y_E \to X_E$ has an automorphism of order 3 because the discriminant of $x^3 + ax +b$ is a square. On the other hand, $(a, b,x) \mapsto (a, -b, -x)$ is an  element of $\Aut(Y / \bbA^1_K)$ of order 2. By \cite[Proposition 3.3.1]{Ser08}, there is a thin set $\Omega_E \subset E = \bbA^1_K(E)$ such that for all $a_0 \in E - \Omega_E$, $a_0$ has 6 pre-images in $Y(\overline{E})$ on which $\Gal(\overline{E} / E)$ acts transitively. Equivalently, $\beta(a_0)$ is not a square in $E$, and if $b_0$ is a square root and $L = E(b_0)$, then the polynomial $f_0(x) = x^3 + a_0 x + b_0 \in L[x]$ is irreducible. 

By \cite[Proposition 3.2.1]{Ser08}, $\Omega = \Omega_E \cap K \subset K = \bbA^1_K(K)$ is a thin set. We claim that this $\Omega$ satisfies the conclusion of the proposition. Indeed, if $a_0 \in K - \Omega$, then $\beta(a_0)$ is not a square in $K$ and, if $L = K(b_0)$, then the polynomial $f(x) = x^3 + a_0 x + b_0 \in L[x]$ is irreducible. We just need to justify why $z$ is not a square in $L$. However, if $z$ is a square in $L$ then $L = K(\sqrt{z}) = E$, implying that $a_0$ has a pre-image in $X(E)$. This contradicts the defining property of $\Omega_E$.

It remains to justify our claim that $Y$ is geometrically irreducible. It can be given equivalently by the equation
\[ F(x) = (x^3 + ax)^2 + \frac{4}{27}a^3 + \frac{1}{16 \times 27}z = 0 \]
in the $(a,x)$-plane. It therefore suffices to show that $F(x)$ is irreducible in $\bbC(a)[x]$ (after extending scalars along some embedding $K \hookrightarrow \bbC$). Let $G(y) = y^3 + 2 a y^2 + a^2 y + \frac{4}{27}a^3 + \frac{1}{16 \times 27}z \in \bbC[a][y]$ be such that $G(x^2) = F(x)$. It even suffices to show that $G(y)$ is irreducible in $\bbC(a)[y]$. Indeed, if this is so and if $y$ is a square in $\bbC(a)[y] / (G(y))$, then the norm of $y$, namely $-(\frac{4}{27}a^3 + \frac{1}{16 \times 27}z)$, is a square in $\bbC(a)$; and this is a contradiction.

To show the irreducibility of $G(y)$ in $\bbC(a)[y]$, it suffices to show that $G(y)$ has no root in $\bbC(a)$. If $y_0 \in \bbC(a)$ is a root, then $y_0 \in \bbC[a]$, as $G$ is monic, and $y_0$ divides $\frac{4}{27}a^3 + \frac{1}{16 \times 27}z$ in $\bbC[a]$. On the other hand, the relation
\[ y_0 (y_0 + a)^2 = - (\frac{4}{27}a^3 + \frac{1}{16 \times 27}z) \]
implies that $y_0$ equals $(\frac{4}{27}a^3 + \frac{1}{16 \times 27}z)$ up to squares in $\bbC(a)$. Together these two facts imply that $y_0$ equals $\frac{4}{27}a^3 + \frac{1}{16 \times 27}z$ up to $\bbC^\times$-multiple, hence that $a$ satisfies a degree 9 equation over $\bbC$.
This is false, so we see that $G(y)$ is irreducible in $\bbC(a)[y]$ and hence that $Y$ is geometrically irreducible, as desired.
\end{proof}
\begin{corollary}\label{cor_existence_of_certain_S3_extensions_with_a_genericity_condition}
Let $K$ be a number field, and let $z \in K^\times$ be an element which is not a square. Let $S$ be a finite set of places of $K$, and let $K' / K$ be a finite extension. Let $T \subset S$ be a subset such that for each $v \in T$, $\ord_v z > 0$. Then we can find an element $a \in K$ with the following property: let $\beta = - \frac{4}{27}a^3 - \frac{1}{16 \times 27} z$. Then:
\begin{enumerate}
\item $\beta$ is not a square in $K$, and each place of $S$ splits in $L = K(b)$, where $b = \sqrt{\beta}$.
\item Let $M / L$ denote the splitting field of $f(x) = x^3 + a x + b$.
Then $\Gal(M/L) = S_3$.
\item There exists a rational prime $l$ not dividing any element of $S$ such that $l$ splits in $L K'$ and for each place $v | l$ of $L$, $v$ is unramified in $M$ and $\Frob_v \in \Gal(M / L)$ has order 3. 
\item Let $A$ denote the elliptic curve over $L$ given by the equation $y^2 = x^3 + a x + b$. Then for each place $v$ of $L$ lying over an element of $T$, $A_{L_v}$ acquires split multiplicative reduction over $L_v(\sqrt{-1})$.
\end{enumerate}
\end{corollary}
\begin{proof}
Let notation be as in the proof of Lemma \ref{lem_existence_of_certain_S3_extensions}. The corollary follows from the lemma and from \cite[Theorem 3.5.3]{Ser08}, provided we can show that the following points:
\begin{itemize}
\item For any place $v$ of $K$, there is a non-empty open subset $\Omega_v \subset K_v$ such that for any $a \in \Omega_v$, $a$ has 2 pre-images in $X(K_v)$; in other words, the fibre above $a$ is totally split.
\item For any place $v$ of $K$ such that $\ord_v z > 0$, there is a non-empty open subset $\Omega_v \subset K_v$ such that for any $a \in \Omega_v$, $a$ has two pre-images in $X(K_v)$, and for each pre-image $b$, the elliptic curve given by the equation $y^2 = x^3 + ax + b$ acquires split multiplicative reduction over $L_v(\sqrt{-1})$.
\item There exists a prime $l$ split in $K'(\sqrt{z})$ and not dividing any place of $S$ such that for each place $v | l$ of $K$, there is a non-empty open subset $\Omega_v \subset K_v$ with the following property: for any $a \in \Omega_v$, $a$ has two pre-images in $X(K_v)$, and each of these points is inert in $Y(K_v)$, defining an unramified extension of $K_v$ of degree 3. 
\end{itemize}
The first point is clear: $X$ is a punctured elliptic curve, so its set of $K_v$-points is non-empty and stable under multiplication by $-1$. We can choose $\Omega_v$ to be the image of a suitable open subset of $X(K_v)$ in $\bbA^1_K(K_v)$. 

For the second point, it is enough to find a single pair $(a, b) \in X(K_v)$ such that the elliptic curve $y^2 = x^3 + a x + b$ has split multiplicative reduction. Indeed, then we can choose $\Omega_v$ to be any sufficiently small open neighbourhood of the point $a$ in $K_v$. (Note that the other pre-image of $a$ in $X(K_v)$ is $(a, -b)$, which gives a quadratic twist of a curve with split multiplicative reduction.) To find such a pair, choose $q \in (\varpi_v)$ such that $\Delta(q) = z$ (possible since $\Delta = q + O(q^2) \in \bbZ\llbracket q \rrbracket)$, and let $(a, b)$ be the co-ordinates of some short Weierstrass equation of the elliptic curve of Tate parameter $q$. 

For the third point, let $F$ denote the Galois closure of $K'(\sqrt{z}) / \bbQ$ and spread out $Y_F$, $X_F$ to schemes $\cY$, $\cX$ of finite type over over $\bbA^1_{\cO_F[1/N]}$ defined by the same equations (for some integer $N \geq 1$). Then $Y_F \to \bbA^1_F$ is Galois of group $S_3$, so we can apply \cite[Proposition 9.15]{Ser12} to deduce that for any sufficiently large prime $l$ which splits in $F$ and for any place $w | l$ of $F$, there exists a point $a_0 \in \bbA^1(k(w))$ which is split in $\cX(k(w))$ and inert in $\cY(k(w))$. If $v$ is a place of $K$ lying above $l$, we can then choose $\Omega_v$ by choosing a place $w$ of $F$ lying above $v$ and taking $\Omega_v \subset K_v = F_w$ to be a non-empty subset of $\cO_{F_w}$ contained in the pre-image of $a_0$ under reduction modulo $w$.
\end{proof}

\begin{lemma}\label{lem_modularity_when_discriminant_is_real_mod_6th_powers}
Let $K$ be an imaginary CM number field, and let $E$ be an elliptic curve over $K$ satisfying the following conditions:
\begin{enumerate}
\item The representation $\overline{\rho}_{E, 3}$ is decomposed generic, and $\overline{\rho}_{E, 3}({G_{K(\zeta_3)}})$ is conjugate to $\SL_2(\bbF_3)$.
\item Let $\Delta_E \in K^\times$ denote the discriminant of a Weierstrass equation of $E$. Then $\Delta_E \in (K^+)^\times \cdot (K^\times)^6$.
\item For each place $v | 3$ of $K$, $E_{K_v}$ is a Tate curve and $\ord_{v} \Delta_E \equiv 2 \text{ or } 4 \text{ mod }6$.
\item For each place $v | 2$ of $K$, $E_{K_v}$ is a Tate curve and $\ord_{v} \Delta_E \equiv 3 \text{ mod }6$, and the extension $K_v(\sqrt{-1}) / K_v$ is unramified.\end{enumerate}
Then $E$ is modular.
\end{lemma}
\begin{proof}
Let us write $\Delta_E = z w^6$ with $z \in K^+$, $w \in K$. We can assume that for each place $v | 2, 3$ of $K$, $\ord_v z >0$. Let $l$ be a prime which is decomposed generic for $\overline{\rho}_{E, 3}$. Consider the cubic equation
\[ b^2 = - \frac{4}{27}a^3 - \frac{1}{16 \times 27} z = \beta(a), \]
where $\beta(a) \in K^+[a]$. Giving a rational point $(a, b)$ on this curve is equivalent to giving a Weierstrass equation $y^2 = x^3 + a x + b$ with discriminant $z$. 

Let $S$ be the set of places of $K^+$ dividing $2$, $3$, $l$, or $\infty$. By Corollary \ref{cor_existence_of_certain_S3_extensions_with_a_genericity_condition}, we can find an element $a_0 \in K^+$ with the following properties:
\begin{itemize}
\item The element $\beta(a_0)$ is not a square in $K^+$. Let $b_0$ be a square-root, and let $K_0^+ = K^+(b_0)$, $K_0 = K(b_0)$. 
\item Each place of $S$ splits in $K_0^+$. In particular, $K_0^+$ is a totally real quadratic extension of $K^+$, and $K_0$ is a CM field, and each place $v$ of $K$ dividing $2$, $3$ or $l$ splits in $K_0$.
\item Let $f(x) = x^3 + a_0 x + b_0 \in K_0^+[x]$. Then the splitting field $M$ of $f(x)$ over $K_0^+$ has Galois group $S_3$. 
\item There exists a rational prime $q$ split in $K_0$ such that for each place $v | q$ of $K_0^+$, $v$ is unramified in $M$ and $\Frob_v \in \Gal(M / K_0^+)$ has order 3. 
\item Let $A_0$ denote the elliptic curve over $K_0$ defined by the equation $y^2 = x^3 + a_0 x + b_0$. Then for each place $v | 2, 3$ of $K_0$, $A_0$ acquires split multiplicative reduction over $K_{0, v}(\sqrt{-1})$.
\end{itemize}
By construction, the discriminant of $A_0$ equals $z$, so satisfies $\Delta_{A_0} \equiv \Delta_E \text{ mod }(K_0^\times)^6$.
We observe that the following conditions are satisfied:
\begin{itemize}
\item $\overline{\rho}_{A_0, 2} : G_{K_0} \to \GL_2(\bbF_2)$ is decomposed generic and surjective, and extends to a homomorphism $G_{K_0^+} \to \GL_2(\bbF_2)$. (It is decomposed generic because of the existence of the prime $q$, which splits in $K_0$. It is surjective because otherwise $K_0 = K(\sqrt{z})$, which would contradict the fact that the primes of $K$ above 2 split in $K_0$ but ramify in $K(\sqrt{z})$.)
\item $\overline{\rho}_{E, 3}|_{G_{K_0}}$ is decomposed generic.
\item For each place $w | 2$ of $K_0$, $\ord_{w} \Delta_{A_0}\equiv 3 \text{ mod }6$.
\item For each place $w | 3$ of $K_0$, $\ord_{w} \Delta_E \equiv 2 \text{ or } 4 \text{ mod }6$.
\end{itemize}
By Lemma \ref{lem_existence_of_elliptic_curves}, we can find an elliptic curve $A$ over $K_0$ with the following properties:
\begin{itemize}
\item $\overline{\rho}_{A, 2} \cong \overline{\rho}_{A_0, 2}$.
\item $\overline{\rho}_{A, 3} \cong \overline{\rho}_{E, 3}|_{G_{K_0}}$. 
\item $\Delta_{A} \equiv \Delta_E \text{ mod }(K_0^\times)^6$.
\end{itemize}
Let $w$ be a 2-adic place of $K_0$. Then $\ord_w \Delta_E \equiv 3 \text{ mod }6$, and so $\overline{\rho}_{E, 3}|_{G_{K_0, w}}$ is unramified ($E_{K_{0, w}}$ is a Tate curve). It follows that $A$  must be semistable at $w$. Since $\ord_w \Delta_A$ is odd, we see that $A$ in fact has multiplicative reduction, and $\overline{\rho}_{A, 2}|_{G_{K_{0, w}}}$ has image of order 2 and cuts out the ramified quadratic extension $K_{0, w}(\sqrt{\Delta_A}) / K_{0, w}$. We now see that all of the hypotheses of Theorem \ref{thm_main_automorphy_theorem} are satisfied for the representation $\rho_{A, 2} \otimes \epsilon^{-1}$, and hence that $A$ is a modular elliptic curve.

Now let $w$ be a 3-adic place of $K_0$. Then $\ord_w \Delta_{A_0} \equiv 0 \text{ mod }2$, so $\overline{\rho}_{A_0, 2}|_{G_{K_0, w}} \cong \overline{\rho}_{A, 2}|_{G_{K_0, w}}$ is unramified ($A_0$ acquires split multiplicative reduction over an unramified extension of $K_{0, w}$). We claim this forces $A$ to have potentially multiplicative reduction. Indeed, if $A$ has potentially good reduction then $\rho_{A, 2}(I_{K_0, w})$ has order $m$ dividing 4, and this integer satisfies $m \ord_w \Delta_A \equiv 0 \text{ mod }12$, which contradicts the fact that $\ord_{w} \Delta_A$ is prime to 3.

Using Corollary~\ref{cor_tate_to_steinberg} to verify ordinariness, we now see that all of the hypotheses of Theorem \ref{thm_automorphy_at_odd_primes} 
are satisfied for the representation $\rho_{E, 3}|_{G_{K_0}}\otimes \epsilon^{-1}$, and hence that $E_{K_0}$ is a modular elliptic curve. The modularity of $E$ itself now follows by soluble descent (i.e.\ \cite[Proposition 6.5.13]{10authors}).
\end{proof}
\begin{proposition}\label{prop_modularity_of_certain_curves}
Let $K$ be an imaginary CM number field, and let $E$ be an elliptic curve over $K$. Suppose that the following conditions are satisfied: 
\begin{enumerate}
\item For each place $v | 2$ of $K$, $E_{K_v}$ is a Tate curve, the extension $K_v(\sqrt{-1}) / K_v$ is unramified, $\ord_{K_v} \Delta_E \equiv 3 \text{ mod }6$, and $\Delta_E / \Delta_E^c \in (K_v^\times)^2$.
\item For each place $v | 3$ of $K$, $E_{K_v}$ is a Tate curve, $\ord_{K_v} \Delta_E \equiv 2, 4 \text{ mod }6$, and $\Delta_E / \Delta_E^c \in (K_v^\times)^3$.
\item Let $L = K(\zeta_{12}, (\Delta_E / \Delta_E^c)^{1/6})$. Then $\overline{\rho}_{E, 3}|_{G_L}$ is decomposed generic and  $\overline{\rho}_{E, 3}({G_{L(\zeta_3)}})$ is conjugate to $\SL_2(\bbF_3)$.
\end{enumerate}
Then $E$ is modular.
\end{proposition}
\begin{proof}
By Corollary \ref{cor_real_up_to_nth_powers}, $\Delta_E \in (L^+)^\times (L^\times)^6$. Our hypotheses imply that for each place $w | 2$ of $L$, $\ord_{w} \Delta_E \equiv 3 \text{ mod }6$ and for each place $w | 3$ of $L$, $\ord_{w} \Delta_E\equiv 2, 4 \text{ mod }6$. We see that the hypotheses of Lemma \ref{lem_modularity_when_discriminant_is_real_mod_6th_powers} apply to $E_L$, and hence that $E_L$ is modular. The modularity of $E$ itself follows by soluble descent. 
\end{proof}

\begin{proposition}\label{prop_residual_modularity_at_5}
Let $K$ be an imaginary CM number field, and let $\overline{\rho} : G_K \to \GL_2(\bbF_5)$ be a continuous homomorphism of cyclotomic determinant. Let $K^\text{avoid} / K$ be a finite Galois extension. Then we can find a soluble CM extension $L / K$ and an elliptic curve $E$ over $L$, all satisfying the following conditions:
\begin{enumerate}
\item The elliptic curve $E$ is modular.
\item The extension $L / K$ is linearly disjoint from  $K^\text{avoid} / K$.
\item For each place $w | 5$ of $L$, $E_{L_w}$ is a Tate curve. 
\item There is an isomorphism $\overline{\rho}|_{G_L} \cong \overline{\rho}_{E, 5}$.
\end{enumerate}
Moreover, if $\overline{\rho}$ is decomposed generic, then we can assume that $\overline{\rho}|_{G_L}$ is decomposed generic. 
\end{proposition}
\begin{proof}
We choose $L / K$ to be a soluble CM extension satisfying the following conditions:
\begin{itemize}
\item For each place $w | 2, 3, 5$ of $L$, $\overline{\rho}|_{G_{L_w}}$ is trivial and $w$ is split over $L^+$.
\item For each place $w | 2$ of $L$, the extension $L_w(\sqrt{-1}) / L_w$ is unramified.
\item $L / K$ is linearly disjoint from $K^\text{avoid} / K$.
\item If $\overline{\rho}$ is decomposed generic, then we choose a prime number $q > 5$ which is decomposed generic for $\overline{\rho}$, and ask moreover that $q$ split in $L$.
\end{itemize}
By Lemma \ref{lem_curves_satisfying_local_conditions}, we can find an elliptic curve $E$ over $L$ with discriminant  $\Delta_E \in L^\times$ satisfying the following conditions:
\begin{itemize}
\item There is an isomorphism $\overline{\rho}_{E, 5} \cong \overline{\rho}|_{G_L}$.
\item $\overline{\rho}_{E, 3}(G_{L(\zeta_{3})}) = \SL_2(\bbF_3)$.
\item There exists a rational prime $l > 5$ split in $L(\zeta_{12})$ such that $l$ is decomposed generic for $\overline{\rho}_{E, 3}$ and for each prime $w | l$ of $L$, the elliptic curve $E$ has good reduction at $w$, with reduction over $k(w) = \bbF_l$ independent of the choice of $w$. In particular, $\Delta_E / \Delta_E^c \in (L_w^\times)^{12}$ for each $w | l$.
\item For each place $w | 2$ of $L$, $E_{L_w}$ is a Tate curve, $\ord_{w} \Delta_E \equiv 3 \text{ mod }6$, and $\Delta_E / \Delta_E^c \in (L_w^\times)^2$.
\item For each place $w | 3$ of $L$, $E_{L_w}$ is a Tate curve, $\ord_{w} \Delta_E\equiv 2 \text{ mod }6$, and $\Delta_E / \Delta_E^c \in (L_w^\times)^3$.
\end{itemize}
Let $L' = L(\zeta_{12}, (\Delta_E / \Delta_E^c)^{1/6})$. We see that the modularity of $E$ follows from Proposition \ref{prop_modularity_of_certain_curves}, provided we can show that $\overline{\rho}_{E, 3}(G_{L'(\zeta_{3})})$ is conjugate to $\SL_2(\bbF_3)$ and $\overline{\rho}_{E, 3}|_{G_{L'}}$ is decomposed generic. The first requirement follows from the observation that the order 3 quotient of $\overline{\rho}_{E, 3}(G_{L(\zeta_3)})$ cuts out the cyclic degree 3 extension $L(\zeta_3, \sqrt[3]{\Delta_E}) / L(\zeta_3)$ which is ramified at 3, while the ramification index of $L'/L$ at the 3-adic primes is prime to 3, because $\Delta_E / \Delta_E^c$ is a cube locally. The second requirement follows since $l$ splits in $L'$ and is decomposed generic for $\overline{\rho}_{E, 3}$.

Proposition \ref{prop_modularity_of_certain_curves} therefore implies the modularity of $E_{L'}$, and the modularity of $E$ itself follows by soluble descent. 
\end{proof}
\begin{corollary}\label{cor_modularity_with_a_condition_at_5}
Let $K$ be an imaginary CM number field and let $E$ be an elliptic curve over $K$ satisfying the following conditions:
\begin{enumerate}
\item $\overline{\rho}_{E, 5}|_{G_{K(\zeta_5)}}$ is absolutely irreducible. If $\overline{\rho}_{E, 5}(G_{K(\zeta_5)}) = \SL_2(\bbF_5)$, then $\zeta_5 \not\in K$.
\item For each place $v | 5$ of $K$, $E_{K_v}$ is ordinary (i.e.\ has good ordinary reduction or potentially multiplicative reduction).
\item $\overline{\rho}_{E, 5}$ is decomposed generic.
\end{enumerate}
Then $E$ is modular.
\end{corollary}
\begin{proof}
By Proposition \ref{prop_residual_modularity_at_5}, we can find a soluble extension CM extension $L / K$ and a modular elliptic curve $A$ over $L$ all satisfying the following conditions:
\begin{itemize}
\item $\overline{\rho}_{E, 5}|_{G_{L(\zeta_5)}}$ is absolutely irreducible. If $\overline{\rho}_{E, 5}(G_{L(\zeta_5)}) = \SL_2(\bbF_5)$, then $\zeta_5 \not\in L$.
\item For each place $w | 5$ of $L$, $A_{L_w}$ is a Tate curve.
\item There is an isomorphism $\overline{\rho}_{E, 5}|_{G_L} \cong \overline{\rho}_{A, 5}$ and these representations are decomposed generic. 
\end{itemize}

Say $p = 5$ and $\overline{\rho}_{E, 5}(G_{L(\zeta_5)}) = \SL_2(\bbF_5)$.
Since $\zeta_5 \notin L$, $\det(\overline{\rho}_{E, 5})$ has order $2$ or $4$. In the former case, the projective image of $\overline{\rho}_{E, 5}$ is $\PSL_2(\bbF_5)$, and in the latter, $[L(\zeta_5) : L] > [\PGL_2(\bbF_5) : \PSL_2(\bbF_5)]$. 
In either situation, $\zeta_5$ is not contained in the extension cut out by the projective image of $\overline{\rho}_{E, 5}$.
Theorem \ref{thm_automorphy_at_odd_primes} then implies the modularity of $E_L$. The modularity of $E$ then follows by soluble descent. 
\end{proof}
\begin{proposition}\label{prop_residual_modularity_at_2_3}
Let $p \in \{ 2, 3 \}$ and let $K$ be an imaginary CM field such that $\zeta_5 \not\in K$, and let $\overline{\rho} : G_K \to \GL_2(\bbF_p)$ be a continuous homomorphism of cyclotomic determinant. Let $K^\text{avoid} / K$ be a finite Galois extension. Then we can find a soluble CM extension $L / K$ and an elliptic curve $E$ over $L$, all satisfying the following conditions:
\begin{enumerate}
\item The elliptic curve $E$ is modular.
\item The extension $L / K$ is linearly disjoint from  $K^\text{avoid} / K$.
\item For each place $w | p$ of $L$, $E_{L_w}$ is a Tate curve. 
\item There is an isomorphism $\overline{\rho}|_{G_L} \cong \overline{\rho}_{E, p}$.
\end{enumerate}
Moreover, if $\overline{\rho}$ is decomposed generic, then we can assume that $\overline{\rho}|_{G_L}$ is decomposed generic. 
\end{proposition}
\begin{proof}
We choose $L / K$ to be a soluble CM extension satisfying the following conditions:
\begin{itemize}
\item For each place $w | p, 5$ of $L$, $\overline{\rho}|_{G_{L_w}}$ is trivial.
\item $L / K$ is linearly disjoint from $K^\text{avoid}(\zeta_5) / K$.
\item If $\overline{\rho}$ is decomposed generic, then we choose a prime $q > 5$ which is decomposed generic for $\overline{\rho}$, and ask that $q$ split in $L$.
\end{itemize}
By Lemma \ref{lem_curves_satisfying_local_conditions} (if $p = 3$) or Lemma \ref{lem_curves_satisfying_local_conditions_p_equals_2} (if $p = 2$), we can find an elliptic curve $E$ over $L$ satisfying the following conditions:
\begin{itemize}
\item There is an isomorphism $\overline{\rho}_{E, p} \cong \overline{\rho}|_{G_L}$.
\item For each place $w | p, 5$ of $L$, $E_{L_w}$ is a Tate curve.
\item $\overline{\rho}_{E, 5}$ is decomposed generic and $\overline{\rho}_{E, 5}|_{G_{L(\zeta_5)}}$ is absolutely irreducible.
\end{itemize}
Corollary \ref{cor_modularity_with_a_condition_at_5} now implies that $E$ is modular. This completes the proof.
\end{proof}
\begin{corollary}\label{cor_modularity_with_a_condition_at_3}
Let $K$ be an imaginary CM number field such that $\zeta_5 \not\in K$, and let $E$ be an elliptic curve over $K$ satisfying the following conditions:
\begin{enumerate}
\item $\overline{\rho}_{E, 3}|_{G_{K(\zeta_3)}}$ is absolutely irreducible. 
\item For each place $v | 3$ of $K$, $E_{K_v}$ is ordinary.
\item $\overline{\rho}_{E, 3}$ is decomposed generic.
\end{enumerate}
Then $E$ is modular.
\end{corollary}
\begin{proof}
By Proposition \ref{prop_residual_modularity_at_2_3}, we can find a soluble extension CM extension $L / K$ and a modular elliptic curve $A$ over $L$ all satisfying the following conditions:
\begin{itemize}
\item $\overline{\rho}_{E, 3}|_{G_{L(\zeta_3)}}$ is absolutely irreducible. 
\item For each place $w | 3$ of $L$, $A_{L_w}$ is a Tate curve.
\item There is an isomorphism $\overline{\rho}_{E, 3}|_{G_L} \cong \overline{\rho}_{A, 3}$, and these representations are decomposed generic. 
\end{itemize}
Theorem \ref{thm_automorphy_at_odd_primes} then implies the modularity of $E_L$ (using Corollary~\ref{cor_tate_to_steinberg} to verify the ordinariness of the automorphic representation giving rise to $\overline{\rho}_{E, 3}|_{G_L} \cong \overline{\rho}_{A, 3}$). The modularity of $E$ then follows by soluble descent. 
\end{proof}
\begin{theorem}\label{thm_application_to_serre}
Let $K$ be an imaginary CM number field such that $\zeta_5 \not\in K$, let $p \in \{ 2, 3, 5 \}$, and let $\overline{\rho} : G_K \to \GL_2(\bbF_p)$ be a continuous representation of cyclotomic determinant. Suppose that the following conditions are satisfied:
\begin{enumerate}
\item If $p \in \{ 2, 3 \}$, then for each place $v | 5$ of $K$, there exists a Tate curve $E_v$ over $K_v$ such that $\overline{\rho}|_{G_{K_v}} \cong \overline{\rho}_{E_v, p}$.
\item If $p = 5$, then for each place $v | 3$ of $K$, there exists a Tate curve $E_v$ over $K_v$ such that $\overline{\rho}|_{G_{K_v}} \cong \overline{\rho}_{E_v, p}$.
\end{enumerate}
Then there exists a modular elliptic curve $E$ over $K$ satisfying the following conditions:
\begin{enumerate}
\item If $p \in \{ 2, 3 \}$ (resp. if $p = 5$), then for each place $v | 5$ (resp. each place $v | 3$) of $K$, $E_{K_v}$ is a Tate curve.
\item There is an isomorphism $\overline{\rho}_{E, p} \cong \overline{\rho}$.
\end{enumerate}
\end{theorem}
\begin{proof}
If $p = 2$ (resp. $p = 3$), then we apply Lemma \ref{lem_curves_satisfying_local_conditions_p_equals_2} (resp. Lemma \ref{lem_curves_satisfying_local_conditions}) to find an elliptic curve $E$ over $K$ satisfying the following conditions:
\begin{enumerate}
\item For each place $v | 5$ of $K$, $E_{K_v}$ is a Tate curve. 
\item $\overline{\rho}_{E, 5}$ is decomposed generic and $\overline{\rho}_{E, 5}|_{G_{K(\zeta_5)}}$ is absolutely irreducible.
\item There is an isomorphism $\overline{\rho}_{E, p} \cong \overline{\rho}$.
\end{enumerate}
Corollary \ref{cor_modularity_with_a_condition_at_5} implies the modularity of $E$, hence the statement in the theorem. The argument is very similar in the case $p = 5$, except we use the prime 3 and appeal instead to Corollary \ref{cor_modularity_with_a_condition_at_3}.\end{proof}
\section{Application to modularity of elliptic curves}\label{sec_application_to_modularity}
Let $K$ be a number field. If $X > 0$, then we write $\cO_{K, X}$ for the set of $\alpha \in \cO_K$ such that for every embedding $\sigma : K \hookrightarrow \bbC$, $|\sigma(\alpha)| < X$. We write  $\cE_{X}$ for the set of pairs $(a, b) \in \cO_{K, X^4} \times \cO_{K, X^6}$ with $\Delta(a, b) = -16(4 a^3 + 27 b^2) \neq 0$. If $(a, b) \in \cE_X$, then we write $E_{a, b}$ for the elliptic curve over $K$ given by the equation $y^2 = x^3 + a x + b$.
\begin{theorem}\label{thm_example_application}
	Suppose that $K$ is a CM field and $\zeta_5 \not\in K$. Then a positive proportion of elliptic curves over $K$ are modular. More precisely, we have
	\[ \liminf_{X \to \infty} \frac{ | \{ (a, b) \in \cE_X \mid E_{a,b} \text{ is modular } \} | }{ | \cE_X | }  > 0. \]
\end{theorem}
Theorem \ref{thm_example_application} is an easy consequence of the following result.
\begin{proposition}
	Suppose that $K$ is a CM field and $\zeta_5 \not\in K$. Then we can find a finite set $S$ of finite places of $K$, together with for each $v \in S$ a (non-empty) open compact subset $U_v \subset \cO_{K_v} \times \cO_{K_v}$, such that for any $(a, b) \in \cE_X \cap (\prod_{v \in S} U_v)$, $E_{a, b}$ is modular.
\end{proposition}
\begin{proof}
	We use Corollary \ref{cor_modularity_with_a_condition_at_3}. We need to show that we can choose $S$ and $U_v$ so that for any $(a, b) \in \cE_X \cap (\prod_{v \in S} U_v)$, $E_{a, b}$ satisfies the conditions of that corollary. 
	
	To this end, we choose for each $v \in S_3$ an ordinary elliptic curve $E_v$ over $k(v)$. We choose a prime $l > 3$ split in $K(\zeta_3)$ and for each $v \in S_l$ an elliptic curve $E_v$ over $k(v) = \bbF_l$ such that $\overline{\rho}_{E_v, 3}(\Frob_v)$ is regular semisimple. We fix places $w_1, w_2$ of $K$ split in $K(\zeta_{3})$ and elliptic curves $E_{w_1}, E_{w_2}$ over $k(w_1)$ and $k(w_2)$ respectively such that $\overline{\rho}_{E_{w_1}, 3}(\Frob_v)$ has order 4 and $\overline{\rho}_{E_{w_2}, 3}(\Frob_v)$ has order 6. 
	
	We let $S = S_3 \cup S_l \cup \{ w_1, w_2 \}$ and for each $v \in S$, define $U_v$ to be any non-empty open compact subset consisting of pairs $(a, b) \in \cO_{K_v} \times \cO_{K_v}$ which define an elliptic curve with good reduction isomorphic to $E_v$ over $k(v)$.
	
	It follows by construction that if $(a, b) \in \cE_X \cap (\prod_{v \in S} U_v)$, and $E = E_{a, b}$, then $E_{K_v}$ is ordinary for each $v \in S_3$ and $\overline{\rho}_{E, 3}$ is decomposed generic. Since any subgroup of $\SL_2(\bbF_3)$ containing an element of order 4 and an element of order 6 equals $\SL_2(\bbF_3)$, we see that $\overline{\rho}_{E, 3}(G_{K(\zeta_3)}) = \SL_2(\bbF_3)$. Thus $E$ satisfies all the conditions of  Corollary \ref{cor_modularity_with_a_condition_at_3}. This completes the proof.
\end{proof}

\appendix

\section{Automorphy lifting for $p>2$: proofs}\label{sec:p-odd-proofs}

In this appendix we prove Theorem~\ref{thm_automorphy_at_odd_primes}.

\subsection{Galois deformation theory, fixed determinant case}
We first need to set up Galois deformation theory in the fixed determinant case (in contrast to \S \ref{sec_Galois_deformation_theory}, where we did not fix determinants). We again let $K$ be a number field. 
We take an \emph{odd} prime $p$ and let $\cO$ be a coefficient ring with residue field $k$. 
We will primarily be interested in the case that $K$ is an imaginary CM field.
We fix a continuous, absolutely irreducible representation $\overline{\rho} : G_K \to \GL_2(k)$ and a character $\psi : G_K \to \cO^\times$ such that 
\[\det\overline{\rho} = \psi.\]
Let $S_p$ denote the set of $p$-adic places of $K$, and let $S$ be a finite set of finite places of $K$ containing $S_p$ and all the places at which $\overline{\rho}$ is ramified. We also fix for each $v \in S$ a ring $\Lambda_v \in \CNL_\cO$ and set $\Lambda = \widehat{\otimes}_{v \in S} \Lambda_v$, the completed tensor product being relative to $\cO$.

In this section, a global deformation problem is a tuple 
\[ \cS = ( K, \overline{\rho}, \psi, S, \{ \Lambda_v \}_{v \in S}, \{ \cD_v \}_{v \in S} ), \]
where for each $v \in S$, $\cD_v$ is a local deformation problem at $v$. If $R \in \CNL_\Lambda$, we say that a continuous lifting $\rho_R : G_K \to \GL_2(R)$ of $\overline{\rho}$ is of type $\cS$ if $\rho_R$ is unramified outside $S$, $\det\rho_R = \psi$, and $\rho_R \in \cD_v(R)$ for each $v \in S$. 
We can then define liftings and deformations, as well as their $T$-framed versions, as in \S\ref{sec_deformation_basics}.

Because we work with fixed determinants for our global deformations, we will always consider local deformation problems $\cD_v$ that are subfunctors of the local deformation problems $\cD_v^\psi$, where $\cD_v^\psi$ is the deformation problem consisting of lifts $\rho$ of $\overline{\rho}|_{G_{K_v}}$ with $\det\rho = \psi|_{G_{K_v}}$. 
Abusing notation slightly, we will simply write $\det\rho = \psi$. 
We let $R_v^\psi$ denote the quotient of $R_v^\square$ representing $\cD_v^\psi$. 
When the place $v$ is understood, we will write $\rho^\square$ and $\rho^\psi$ for the universal $R_v^\square$ and $R_v^\psi$ valued lifts, respectively.

We let $\ad \overline{\rho}$ be the space $M_2(k)$ of $2 \times 2$ matrices with entries in $k$ with adjoint $G_K$-action and let $\ad^0\overline{\rho}$ be the trace zero subspace.

\subsubsection{Ordinary deformations}\label{sec_ordinary_fixed_determinant_deformations}
Fix a prime $v \in S_p$. We assume that $\overline{\rho}|_{G_{K_v}}$ is trivial. 
Let $\cO_{K_v}^\times(p)$ be the maximal pro-$p$ quotient of $\cO_{K_v}^\times$, which we identify with a quotient of $I_{K_v}$ via class field theory. 
We similarly identify the pro-$p$ completion $K_v^\times(p)$ of $K_v^\times$ with a quotient of $G_{K_v}$.  
Fix a minimal prime $\frq \subset \cO\llbracket \cO_{K_v}^\times(p) \rrbracket$.
We then take $\Lambda_v = \cO\llbracket \cO_{K_v}^\times(p) \rrbracket/\frq$, and we write $\phi : I_{K_v} \to \Lambda_v^\times$ for the universal character.

Set $\widetilde{\Lambda}_v = \cO\llbracket K_v^\times(p) \rrbracket \widehat{\otimes}_{\cO\llbracket \cO_{K_v}^\times(p) \rrbracket} \Lambda_v$ and set $\widetilde{R}_v^\psi = R_v^\psi \otimes_{\Lambda_v} \widetilde{\Lambda}_v$. 
The character $\phi$ naturally extends to a character $\widetilde{\phi} : G_{K_v} \to \widetilde{\Lambda}_v^\times$. 
We define $\widetilde{R}_v^{\psi,\det,\ord}$ to be the maximal quotient of $\widetilde{R}_v^\psi$ over which the relations 
\[ \det(X - \rho^\psi(g)) = (X - \widetilde{\phi}(g))(X - \psi\widetilde{\phi}^{-1}(g)) \]
and
\[ (\rho^\psi(g_1) - \widetilde{\phi}(g_1))(\rho^\psi(g_2) - \psi\widetilde{\phi}^{-1}(g_2)) = 0 \]
hold for all $g,g_1,g_2 \in G_{K_v}$. 
We then define $R_v^{\psi,\det,\ord}$ to be the image of the $\Lambda_v$-algebra homomorphism
\[ R_v^\psi \to \widetilde{R}_v^{\psi,\det,\ord}. \]
Then $R_v^{\psi,\det,\ord}$ classifies a local deformation problem that we denote by $\cD_v^{\psi,\det,\ord}$.

\begin{proposition}\label{prop_det_ord_fixed_det}
	Assume that $p > 2$, that $[K_v : \bbQ_p] > 4$ and that $\overline{\rho}|_{G_{K_v}}$ is trivial. 
	Let $Z \subseteq \Spec\Lambda_v$ be the closed subscheme defined by $\phi^2 = \psi|_{I_{K_v}}$.
	Let $f : \Spec R_v^{\psi,\det,\ord} \to \Spec \Lambda_v$ be the structure map. 
	\begin{enumerate}
		\item There is a unique irreducible component $C'$ of $\Spec R_v^{\psi,\det,\ord}$ dominating $\Spec \Lambda_v$.
		It has dimension $4 + 2[K_v:\bbQ_p]$. 
		\item Let $C'$ be an irreducible component of $\Spec R_v^{\psi,\det,\ord}$ which does not dominate $\Spec \Lambda_v$. 
		Then $C' \subseteq f^{-1}(Z)$ and $C'$ has dimension at most $2 + 2[K_v:\bbQ_p]$. 
	\end{enumerate}
\end{proposition}

\begin{proof}
	We deduce this from \cite[Proposition~6.2.12]{10authors}, which treats the non-fixed determinant (and arbitrary rank) case.
	Let $\mu_{p^r}$ be the torsion subgroup of $\cO_{K_v}^\times(p)$. 
	Enlarging $\cO$ if necessary, we can assume that $\cO$ contains a primitive $p^r$th root of unity. 
	The minimal primes of $\cO\llbracket \cO_{K_v}^\times(p) \rrbracket$, resp. of $\cO\llbracket \cO_{K_v}^\times(p) \times \cO_{K_v}^\times(p) \rrbracket$, are then in bijection with characters $\eta : \mu_{p^r} \to \cO^\times$, resp. pairs $(\eta_1,\eta_2)$ of characters $\eta_i : \mu_{p^r} \to \cO^\times$. 
	Let $\eta : \mu_{p^r} \to \cO^\times$ be the character corresponding to our fixed minimal prime $\frq \subset \cO \llbracket \cO_{K_v}^\times(p) \rrbracket$, and we let $\frq' \subseteq \cO \llbracket \cO_{K_v}^\times(p) \times \cO_{K_v}^\times(p) \rrbracket$ be the minimal prime corresponding to the pair $(\eta, \eta^{-1}\psi|_{\mu_{p^r}})$. 
	We then set $\Lambda_v' = \cO \llbracket \cO_{K_v}^\times(p) \times \cO_{K_v}^\times(p) \rrbracket/\frq'$, and let $\phi_1, \phi_2 : I_{K_v} \to (\Lambda_v')^\times$ be the resulting universal characters. 
	Set $\widetilde{\Lambda}_v' = \cO\llbracket K_v^\times(p) \times K_v^\times(p) \rrbracket \otimes_{\cO\llbracket \cO_{K_v}^\times(p) \times \cO_{K_v}^\times(p) \rrbracket} \Lambda_v'$ and $\widetilde{R}_v^\square = R_v^\square \otimes_{\Lambda_v'} \widetilde{\Lambda}_v'$. 
	The characters $\phi_1,\phi_2$ naturally extends to characters $\widetilde{\phi}_1, \widetilde{\phi}_2 : G_{K_v} \to (\widetilde{\Lambda}_v')^\times$. 
	Define $\widetilde{R}_v^{\det,\ord}$ to be the maximal quotient of $\widetilde{R}_v^\square$ over which the relations 
	\[ \det(X - \rho^\square(g)) = (X - \widetilde{\phi}_1(g))(X - \widetilde{\phi}_2(g)) \]
	and
	\[ (\rho^\square(g_1) - \widetilde{\phi}_1(g_1))(\rho^\square(g_2) - \widetilde{\phi}_2(g_2)) = 0 \]
	hold for all $g,g_1,g_2 \in G_{K_v}$. We then define $R_v^{\det, \ord}$ to be the image of the $\Lambda_v$-algebra homomorphism $R_v^\square \to \widetilde{R}_v^{\det, \ord}$.
	
	Let $K_v^\times(p)^f$ and $\cO_{K_v}^{\times,f}$ be the maximal torsion free quotients of $K_v^\times(p)$ and $\cO_{K_v}^\times$, respectively, and
	let $\Psi : G_{K_v} \to \cO\llbracket K_v^\times(p)^f \rrbracket$ and $\Phi : I_{K_v} \to \cO\llbracket \cO_{K_v}^{\times,f} \rrbracket$ be the tautological characters (trivial on the torsion subgroups of $G_{K_v}^\mathrm{ab}$ and $I_{K_v^\mathrm{ab}/K_v}$, respectively). 
	Let $\rho: G_{K_v} \to \GL_2(R_v^{\psi,\det,\ord})$ and $\rho' : G_{K_v} \to \GL_2(R_v^{\det,\ord})$ be the universal lifts. 
	There is a commutative diagram
	\begin{equation}\label{eqn_two_det_ords} 
	\xymatrix@1{ R_v^{\det,\ord} \ar[r]^-g & R_v^{\psi,\det,\ord} \widehat{\otimes}_{\cO} \cO \llbracket K_v^\times(p)^f \rrbracket \\
		\Lambda_v' \ar[u] \ar[r]^-h & \Lambda_v \widehat{\otimes}_{\cO} \cO \llbracket \cO_{K_v}^{\times,f} \rrbracket, \ar[u] }    
	\end{equation}
	where $g$ is induced by the lift $\rho \otimes \Psi$ and $h$ by the pair of characters $(\phi\Phi, \phi^{-1}\psi|_{I_{K_v}}\Phi)$. 
	We claim that $g$ and $h$ are isomorphisms.
	Since $p > 2$, for any $\CNL_{\Lambda_v}$-algebra $R$ and character $\chi : G_{K_v} \to R^\times$ that is trivial modulo $\ffrm_R$, there is a unique square root character $\chi^{1/2} : G_{K_v} \to R^\times$ that is also trivial modulo $\ffrm_R$.
	On universal objects, the inverse to $g$ is then given by 
	\[ \rho' \mapsto (\rho' \otimes ((\det\rho')^{-1}\psi)^{1/2}, ((\det\rho')\psi^{-1})^{1/2}), \]
	and the inverse to $h$ by
	\[ (\phi_1,\phi_2) \mapsto (\phi_1(\phi_1^{-1}\phi_2^{-1}\psi|_{I_{K_v}})^{1/2}, (\phi_1\phi_2\psi^{-1}|_{I_{K_v}})^{1/2}). \]
	Since $\cO\llbracket \cO_{K_v}^{\times,f} \rrbracket$ and $\cO \llbracket F_v^\times(p)^f \rrbracket$ are power series rings over $\cO$ in $[K_v:\bbQ_p]$ and $1+[K_v:\bbQ_p]$ variables, respectively, the proposition now follows from \cite[Proposition~6.2.12]{10authors}.
\end{proof}

\subsubsection{Level raising deformations}\label{sec:level_raising_defs_fixed_det}
Let $v$ be a finite place of $K$ such that $q_v \equiv 1 \text{ mod } p$ and $\overline{\rho}|_{G_{K_v}}$ is trivial. Let $\Lambda_v = \cO$. 

Given a character $\chi_v : k(v)^\times \to \cO^\times$ which is trivial mod $(\varpi)$, we define $\cD_v^{\psi,\chi_v} \subset \cD_v^\psi$ to be the functor of liftings $\rho_R : G_{K_v} \to \GL_2(R)$ of $\overline{\rho}|_{G_{K_v}}$ such that $\det\rho_R = \psi$ and for all $\sigma \in I_{K_v}$, the characteristic polynomial of $\rho_R(\sigma)$ equals
\[ (X - \chi_v(\Art_{K_v}^{-1}(\sigma))) (X - \chi_v^{-1}\psi(\Art_{K_v}^{-1}(\sigma))). \]
Then $\cD_v^{\psi,\chi_v}$ is a local deformation problem, and we write $R_v^{\psi,\chi_v}$ for the representing object. 
As in \S\ref{subsubsec_level_raising_deformations}, the ring  $R_v^{\psi,\chi_v} / (\varpi)$ is canonically independent of $\chi_v$.

\begin{lemma}\label{lem_ihara_avoidance_ring_fixed_det}
	\begin{enumerate}
		\item Assume that $\chi = \psi|_{I_{K_v}} = 1$. Then $R_v^{\psi,1}$ is equidimensional of dimension $4$ and every minimal prime has characteristic $0$.
		Further, every prime of $R_v^{\psi,1}$ minimal over $\varpi$ contains a unique minimal prime of $R_v^{\psi,1}$.
		\item Assume that $\chi_v^2 \ne \psi|_{I_{K_v}}$. Then $R_v^{\psi,\chi_v}$ has a unique minimal prime and has dimension $4$. Its minimal prime has characteristic $0$.
	\end{enumerate}
\end{lemma}
\begin{proof}
	Let $R_v^{\chi_v}$ be the quotient of $R_v^\square$ representing liftings $\rho_R : G_{K_v} \to \GL_2(R)$ of $\overline{\rho}|_{G_{K_v}}$ such that for all $\sigma \in I_{K_v}$, the characteristic polynomial of $\rho_R(\sigma)$ equals
	\[ (X - \chi_v(\Art_{K_v}^{-1}(\sigma))) (X - \chi_v^{-1}\psi(\Art_{K_v}^{-1}(\sigma))). \]
	A twisting argument similar to (but easier than) that in the proof of Proposition~\ref{prop_det_ord_fixed_det}, shows that $R_v^{\chi_v} \cong R_v^{\psi,\chi_v} \widehat{\otimes}_{\cO} \cO \llbracket (K_v^\times/\cO_{K_v}^\times)(p) \rrbracket \cong R_v^{\psi,\chi_v}\llbracket X \rrbracket$. 
	The lemma then follows from \cite[Proposition~3.1]{Tay08}. 
\end{proof}

\subsubsection{Formally smooth deformations}

\begin{lemma}\label{lemma:smooth_defs_fixed_det}
	Assume that $v\nmid p$ and that $H^2(K_v, \ad^0\overline{\rho}) = 0$. 
	Then $R_v^\psi$ is formally smooth over $\Lambda_v$ of relative dimension $3$.
\end{lemma}

\begin{proof}
	This is standard.
\end{proof}

\subsubsection{Taylor--Wiles deformations}\label{sec:TW_p_odd}
Let $v$ be a finite place of $K$ and let $N \geq 1$ be an integer such that $q_v \equiv 1 \text{ mod }p^N$ and $\overline{\rho}|_{G_{K_v}} = \alpha_v \oplus \beta_v$ is a direct sum of distinct unramified characters. Let $\Lambda_v = \cO$. Then we define $\Delta_v = k(v)^\times(p)$, the maximal $p$-power quotient of $k(v)^\times$. 

As in \S\ref{subsubsec_taylor_wiles_deformation_problems}, there is a uniquely determined lift $A_v : G_{K_v} \to (R_v^\psi)^\times$ of $\alpha_v$ such that the universal determinant $\psi$ lift $\rho^\psi : G_{K_v} \to \GL_2(R_v^\psi)$ is $\widehat{\GL}_2(R_v^\psi)$-conjugate to $A_v \oplus A_v^{-1}\psi$. 
Consequently, $A_v \circ \Art_{K_v}$ gives $R_v^\psi$ the structure of an $\cO[\Delta]$-algebra. 
If $\cS = (K, \overline{\rho}, \psi, S, \{ \cD_v \}_{v \in S})$ is a deformation problem,  then we call a Taylor--Wiles datum for $\cS$ a tuple $(Q, N, \{ \alpha_v, \beta_v \}_{v \in Q})$, where:
\begin{enumerate}
	\item $Q$ is a finite set of finite places of $K$.
	\item $N \geq 1$ is an integer. 
	\item $\alpha_v, \beta_v : G_{K_v} \to k^\times$ are continuous characters.
\end{enumerate}
We require that the following conditions are satisfied:
\begin{enumerate}
	\item $Q \cap S = \emptyset$.
	\item For each $v \in Q$, $q_v \equiv 1 \text{ mod }p^N$.
	\item For each $v \in Q$, $\overline{\rho}|_{G_{K_v}} \cong \alpha_v \oplus \beta_v$ is a direct sum of distinct unramified characters.
\end{enumerate}
We call $N$ the level of the Taylor--Wiles datum. If $(Q, N, \{ \alpha_v, \beta_v \}_{v \in Q})$ is a Taylor--Wiles datum, then we define the augmented deformation problem
\[ \cS_Q = (K, \overline{\rho}, \psi, S \cup Q,  \{ \Lambda_v \}_{v \in S \cup Q}, \{ \cD_v \}_{v \in S} \cup \{ \cD_v^\psi \}_{v \in Q}). \]
The deformation ring $R_{\cS_Q}$ has a natural structure of $\cO[\Delta_Q]$-algebra (where $\Delta_Q = \prod_{v \in Q} \Delta_v$); by construction, $\Delta_Q$ surjects onto  $(\bbZ / p^N \bbZ)^{|Q|}$.

\subsection{Geometry of fixed determinant deformation rings}
Suppose given a deformation problem $\cS = (K, \overline{\rho}, \psi, S,  \{ \Lambda_v \}_{v \in S}, \{ \cD_v \}_{v \in S})$ with $\cD_v \subseteq \cD_v^\psi$ for reach $v \in S$. 
Let $T \subset S$, and let $\Lambda_T = \widehat{\otimes}_{v \in T} \Lambda_v$.
We let $R_v \in \CNL_{\Lambda_v}$ denote the representing object of $\cD_v$ ($v \in S$), and define $A_\cS^T = \Lambda \otimes_{\Lambda_T} ( \widehat{\otimes}_{v \in T} R_v )$, the completed tensor product being over $\cO$. Then $A_\cS^T \in \CNL_{\Lambda}$ and there is a canonical map of $\Lambda$-algebras $A_\cS^T \to R_\cS^T$, corresponding to the natural transformation $( ( \alpha_v )_{v \in T}, \rho ) \mapsto ( \alpha_v \rho \alpha_v^{-1} )_{v \in T}$ at the level of $T$-framed liftings.

For each $v \in S$, the map $c \mapsto (1 + \epsilon c)\overline{\rho}$ defines an isomorphism $Z^1(K_v, \ad^0\overline{\rho}) \cong \Hom_{\CNL_{\Lambda_v}}(R_v^\psi, k[\epsilon]/(\epsilon^2))$. 
We let $\cL_v^1 \subseteq Z^1(K_v, \ad^0\overline{\rho})$ be the subspace corresponding to the tangent space of $R_v$, and let $\cL_v$ be its image in $H^1(K_v, \ad^0\overline{\rho})$. 
We have the complex of \cite[\S 5.3]{Tho16}:
\[ C^i_{\cS, T}(\ad^0\overline{\rho}) = \left\{ \begin{array}{ll} C^0(K_S / K, \ad \overline{\rho}) & i = 0;   \\
C^1(K_S / K, \ad^0 \overline{\rho}) \oplus \bigoplus_{v \in T} C^0(K_v, \ad \overline{\rho}) & i = 1 ; \\
C^2(K_S / K, \ad^0 \overline{\rho}) \oplus \bigoplus_{v \in T} C^1(K_v, \ad^0 \overline{\rho}) \oplus \bigoplus_{v \in S - T} C^1(K_v, \ad^0 \overline{\rho}) / \cL_v^1 & i = 2; \\
C^i(K_S / K, \ad^0 \overline{\rho}) \oplus \bigoplus_{v \in S} C^{i-1}(K_v, \ad^0 \overline{\rho}) & i > 2,
\end{array}\right. \]
with differential given by $(\phi, (\psi_v)_v)) \mapsto (\partial \phi, (\phi|_{G_{K_v}} - \partial \psi_v)_v)$, and whose cohomology we denote by $H^i_{\cS, T}(\ad^0 \overline{\rho})$. 

Since we are assuming $p > 2$, the trace pairing $(X,Y) \to \tr(XY)$ defines a perfect pairing on $\ad^0\overline{\rho}$. 
For each $v \in S$, we let $\cL_v^\perp \subset H^1(K_v, \ad^0\overline{\rho}(1))$ be the orthogonal complement of $\cL_v$ with respect to the Tate duality pairing, 
and then define a group $H^1_{\cS^\perp, T}(\ad^0 \overline{\rho}(1))$ by the formula
\[ H^1_{\cS^\perp, T}(\ad^0\overline{\rho}(1)) = \ker \left( H^1(K_S / K, \ad^0 \overline{\rho}(1)) \to \prod_{v \in S - T} H^1(K_v, \ad^0 \overline{\rho}(1)) / \cL_v^\perp \right). \]
The following result, as with Proposition~\ref{prop_presentation_by_generators_and_relations}, is proved exactly as in \cite[\S 5.3]{Tho16}.
\begin{proposition}\label{prop_presentation_by_generators_and_relations_p_odd}
	\begin{enumerate}
		\item There is a surjective map $A_\cS^T \llbracket X_1, \dots, X_g \rrbracket \to R_\cS^T$ of $A_\cS^T$-algebras, where $g = h^1_{\cS, T}(\ad^0 \overline{\rho})$. If $R_v$ is a formally smooth $\Lambda_v$-algebra for each $v \in S - T$, then the kernel of this map can be generated by $r$ elements, where $r = h^1_{\cS^\perp, T}(\ad^0 \overline{\rho}(1))$.
		\item If $v \in S$, let $\ell_v = \dim_k \cL_v$. Let $\delta_T = 1$ if $T$ is empty, and $\delta_T = 0$ otherwise. Then we have the formula
		\begin{align*} h^1_{\cS, T}(\ad^0 \overline{\rho}) & = h^1_{\cS^\perp, T}(\ad^0 \overline{\rho}(1)) - h^0(K, \ad^0 \overline{\rho}(1)) + \delta_T - 1 + \lvert T \rvert\\ 
		& \quad - \sum_{v | \infty} h^0(K_v, \ad^0 \overline{\rho}) + \sum_{v \in S - T} (\ell_v - h^0(K_v, \ad^0 \overline{\rho})). \end{align*}
	\end{enumerate}
\end{proposition}

\begin{lemma}\label{lemma_TW_primes_p_odd}
	Assume that $\overline{\rho}|_{G_{K(\zeta_p)}}$ is absolutely irreducible. 
If $p = 5$ and the projective image of $\overline{\rho}(G_{K(\zeta_5)})$ is conjugate to $\PSL_2(\bbF_5)$, we assume further that the extension of $K$ cut out by the projective image of $\overline{\rho}$ does not contain $\zeta_5$. 
	
	Let $q \ge h^1_{\cS^\perp, T}(\ad^0 \overline{\rho}(1))$. 
	Then for any $N \ge 1$, there is a Taylor--Wiles datum $(Q, N, \{ \alpha_v, \beta_v \}_{v \in Q})$ of level $N$ such that 
	\begin{enumerate}
		\item $\lvert Q \rvert = q$,
		\item each $v \in Q$ has degree 1 over $\bbQ$,
		\item $h^1_{\cS_{Q_N}^\perp, T}(\ad^0 \overline{\rho}(1)) = 0$.
	\end{enumerate} 
\end{lemma}
\begin{proof}
	It suffices to show that for any cocycle $\kappa$ representing a nonzero element of $H^1_{\cS^\perp, T}(\ad^0\overline{\rho}(1))$, we can find a place $v \notin S$ of $K$ such that 
	\begin{itemize}
		\item $v$ has degree 1 over $\bbQ$ and splits in $K(\zeta_{p^N})$,
		\item $\overline{\rho}(\Frob_v)$ has distinct eigenvalues,
		\item the image of $\kappa$ in $H^1(K_v, \ad^0\overline{\rho})$ is nonzero.
	\end{itemize}
	The primes of degree 1 in $K$ have Dirichlet density 1. 
	So applying Chebotarev density, it suffices to show that for every $\kappa$ representing a nonzero element of $H^1_{\cS^\perp, T}(\ad^0\overline{\rho}(1))$, we can find $\sigma \in G_{K(\zeta_{p^N})}$ such that
	\begin{itemize}
		\item $\overline{\rho}(\sigma)$ has distinct eigenvalues,
		\item $\kappa(\sigma) \notin (\sigma - 1) \ad^0\overline{\rho}(1)$.
	\end{itemize}
	This can be accomplished by arguing as in the proof of \cite[Theorem 2.49]{ddt} (with \cite[Appendix A]{Bar13} as a useful additional reference).
\end{proof}

\begin{proposition}\label{prop_presentation_with_TW_data}
	Assume that $K = K^+K_0$ with $K^+$ totally real and $K_0$ an imaginary CM field. 
	Assume that $\overline{\rho}|_{G_{K(\zeta_p)}}$ is absolutely irreducible. 
	If $p = 5$ and the projective image of $\overline{\rho}(G_{K(\zeta_5)})$ is conjugate to $\PSL_2(\bbF_5)$, we assume further that the extension of $K$ cut out by the projective image of $\overline{\rho}$ does not contain $\zeta_5$. 

	Take $T = S$ and let $q \ge h^1_{\cS^\perp, T}(\ad^0 \overline{\rho}(1))$. Then for any $N \ge 1$, there is a choice of level $N$ Taylor--Wiles datum $(Q, N, \{\alpha_v, \beta_v\}_{v\in Q})$ such that
	\begin{enumerate}
		\item $\lvert Q \rvert = q$,
		\item for each $v \in Q$, the rational prime below $v$ splits in $K_0$,
		\item there is a surjective $\CNL_\Lambda$-morphism $A_\cS^T \llbracket X_1, \ldots, X_g \rrbracket \to R_{\cS_{Q_N}}^T$ with $g = q - 3[K^+:\bbQ] - 1 + \lvert T \rvert$. 
	\end{enumerate}
\end{proposition}

\begin{proof}
	We apply Lemma~\ref{lemma_TW_primes_p_odd} and obtain a Taylor--Wiles datum $(Q, N, \{ \alpha_v, \beta_v \}_{v \in Q})$ of level $N$ such that 
	\begin{itemize}
		\item $\lvert Q \rvert = q$,
		\item eavh $v \in Q$ has degree 1 over $\bbQ$,
		\item $h^1_{\cS_{Q_N}, T}(\ad^0 \overline{\rho}(1)) = 0$.
	\end{itemize}
	Since $v \in Q$ has degree 1 over $\bbQ$, it split in $K_0$. 
	By choice of out Taylor--Wiles datum, Proposition~\ref{prop_presentation_by_generators_and_relations_p_odd} implies that there is a surjective $\CNL_\Lambda$-morphism $A_\cS^T \llbracket X_1, \ldots, X_g \rrbracket \to R_{\cS_{Q_N}}^T$ with
	\[ g = - h^0(K, \ad^0 \overline{\rho}(1)) - 1 + \vert T \rvert - \sum_{v | \infty} h^0(K_v, \ad^0 \overline{\rho}) + \sum_{v \in Q} (h^1(K_v, \ad^0\overline{\rho}) - h^0(K_v, \ad^0 \overline{\rho})).\]
	Now $h^0(K, \ad^0 \overline{\rho}(1)) = 0$ since $\overline{\rho}|_{G_{K(\zeta_p)}}$ is absolutely irreducible and
	\[ \sum_{v | \infty} h^0(K_v, \ad^0 \overline{\rho}) = 3[K^+:\bbQ] \]
	since $K$ is imaginary CM. 
	Finally, since $q_v \equiv 1 \bmod p$ and $\overline{\rho}(\Frob_v)$ has distinct eigenvalues for each $v \in Q$, a straightforward computation using the local Euler characteristic shows that
	\[ \sum_{v \in Q} (h^1(K_v, \ad^0\overline{\rho}) - h^0(K_v, \ad^0 \overline{\rho})) = q.\qedhere \] 
\end{proof}

\subsection{An application of patching}\label{sec:patching_application_p_odd}
We use the setup and notation of \S\S\ref{sec_cohom_general} and \ref{sec_Hecke_Gal_rep}, specialized to the case $G = \PGL_2$. 
We fix the isomorphism of the diagonal maximal torus of $\PGL_2$ with $\bbG_m$ by $\diag(a,d) \mapsto ad^{-1}$, and identify $\bbZ_{\ge 0}$ with $\bbZ_{+,0}^2$ by $\lambda \mapsto (\lambda, -\lambda)$. 
Throughout this subsection, we assume given the following data:
\begin{itemize}
	\item An odd prime $p$ and a CM number field $K$. 
	\item A finite set of finite places $S$ of $K$ containing the $p$-adic places $S_p$, and a (possibly empty) subset $R \subset S - S_p$. 
	\item A continuous representation $\overline{\rho} : G_K \to \GL_2(\overline{\bbF}_p)$ with $\det(\overline{\rho}) = \epsilon^{-1}$.
	\item An isomorphism $\iota : \overline{\bbQ}_p \cong \bbC$ and a cohomological cuspidal automorphic representation $\pi$ of $\PGL_2(\bbA_K)$. 
	We let $\mu \in \bbZ_{\ge 0}^{\Hom(K, \overline{\bbQ}_p)}$ be such that $\pi$ has weight $\iota\mu$. 
\end{itemize}   
We assume that our data satisfies the following hypotheses:
\begin{itemize}
	\item For each $v \in S_p$, $[K_v : \bbQ_p] > 4$.
	\item If $l$ is a rational prime lying below some $v \in S$ or ramifying in $K$, then there exists an imaginary quadratic subfield of $K$ in which $l$ splits. 
	\item $\overline{\rho}$ is decomposed generic and $\overline{\rho}|_{G_{K(\zeta_p)}}$ is absolutely irreducible. If $p = 5$ and the projective image of $\overline{\rho}(G_{K(\zeta_5)})$ is conjugate to $\PSL_2(\bbF_5)$, we assume further that the extension of $K$ cut out by the projective image of $\overline{\rho}$ does not contain $\zeta_5$. 
\item $\overline{\rho}|_{G_{K_v}}$ is the trivial representation for each $v \in R$. 
	\item If $S = S_p \cup R$, then $\zeta_p \in K$.
	\item If $S \ne S_p \cup R$, then $S - (S_p \cup R)$ contains at least two places with distinct residue characteristics. For each $v \in S - (S_p \cup R)$, $\pi_v$ is unramified, $v \notin R^c$,  and $H^2(K_v, \ad^0\overline{\rho}) = 0$. 
	\item $\overline{\rho} \cong \overline{r_\iota(\pi)}$.
	\item $\pi$ is $\iota$-ordinary, $\pi_v^{I_v(1)} \ne 0$ for each $v \in S_p$, $\pi_v^{I_v} \ne 0$ for each $v \in R$, and $\pi_v$ is unramified for each $v \notin S$. 
\end{itemize}

\begin{theorem}\label{thm:automorphy_lifting_with_conditions_p_odd}
	Let the notation and hypothesis be as above, and suppose given a lifting $\rho : G_K \to \GL_2(\overline{\bbQ}_p)$ of $\overline{\rho}$ and a weight $\lambda \in (\bbZ_{\ge 0})^{\Hom(K, E)}$ satisfying the following conditions:
	\begin{enumerate}
\item $\det(\rho) = \epsilon^{-1}$.
		\item For each $v|p$, we have an isomorphism
		\[ \rho|_{G_{K_v}} \sim \begin{pmatrix} \psi_v & \ast \\ 0 & \epsilon^{-1}\psi_v^{-1} \end{pmatrix}, \]
		where the character $\psi_v : G_{K_v} \to \overline{\bbQ}_p^\times$ satisfies 
		\[ \psi_v(\sigma) = \prod_{\tau \in \Hom_{\bbQ_p}(K_v, \overline{\bbQ}_p)} \tau(\Art_{K_v}^{-1}(\sigma))^{\lambda_\tau} \]
		for all $\sigma \in I_{K_v}$.
		\item\label{condition:weights} For each $v|p$ and $p$-power root of unity $u \in \cO_{K_v}^\times$, we have $\prod_{\tau \in \Hom_{\bbQ_p}(K_v, \overline{\bbQ}_p)} \tau(u)^{\lambda_\tau - \mu_\tau} = 1$.
\item For each $v \not\in S$, $\rho|_{G_{K_v}}$ is unramified.
		\item For each $v \in R$, $\rho|_{G_{K_v}}$ is unipotently ramified.
	\end{enumerate}
	Then there is a cuspidal automorphic representation $\Pi$ of $\PGL_2(\bbA_K)$, cohomological of weight $\iota\lambda$ and $\iota$-ordinary, such that $\rho \cong r_\iota(\Pi)$.
\end{theorem}

We first establish some preliminaries. 
Recall we have a coefficient field $E/\bbQ_p$ with ring of integers $\cO$ and residue field $k$. 
We $E$ is taken large enough so that it contains all embedding of $K$ into $\overline{\bbQ}_p$ and so that $\overline{\rho}$ can be defined over $k$.

If $c \geq b \ge 0$ are integers with $c \ge 1$, then we let $U(b, c) = \prod_v U(b, c)_v \in \cJ^S$ be the open compact subgroup of $G^\infty = \PGL_2(\bbA_K^\infty)$ defined as follows:
\begin{itemize}
	\item If $v \in S_p$, then $U(b,c)_v = I_v(b, c)$.
	\item If $v \in R$, then $U(b,c)_v = I_v$.
	\item If $v \in S - (S_p \cup R)$, then $U(b,c)_v = I_v(1)$, the pro-$v$ Iwahori subgroup of $\PGL_2(\cO_{K_v})$. 
	\item If $v \not\in S$, then $U(b,c)_v = \PGL_2(\cO_{K_v})$.
\end{itemize}

We define $\Lambda_{1,c} = \cO[U(1,c)/U(c,c)]$ and $\Lambda_1 = \varprojlim_c \Lambda_{1,c}$. 
We fix the isomorphism 
\[ \prod_{v\in S_p} (1 + \varpi_v\cO_{K_v})/(1 + \varpi_v^c\cO_{K_v}) \cong U(1,c)/U(c,c),   \]
given by $(x_v)_{v\in S_p} \mapsto g$ with $g^p = 1$ and $g_p = (\diag(x_v, 1))_{v\in S_p}$.
Under this isomorphism $\Lambda_1 \cong \widehat{\otimes}_{v\in S_p} \Lambda_{1,v}$ with $\Lambda_{1,v} = \cO\llbracket \cO_{K_v}^\times(p) \rrbracket$, the completed tensor product being taken over $\cO$. 

For a tuple of integers $\lambda = (\lambda_\tau) \in \bbZ^{\Hom(K, E)}$, there is an $\cO$-algebra surjection $\Lambda \to \cO$ induced by the character $\prod_{v \in S_p} \cO_{K_v}^\times(p) \to \cO^\times$ given by
\[ (u_v)_{v \in S_p} \mapsto \prod_{v \in S_p} \prod_{\tau \in \Hom_{\bbQ_p}(K_v, E)} \tau(u_v)^{\lambda_\tau},  \]
and we let $p_\lambda \subset \Lambda_1$ denote its kernel. 
If $M$ is an $\cO$-module, we will also write $M(\lambda)$ for the $\Lambda_1$-module $M \otimes_\cO \Lambda_1/p_\lambda$.

For any $\lambda \in \bbZ_{\ge 0}^{\Hom(K, E)}$, we have the $\cO[U(0,1)]$-module $\cV_\lambda$ with twisted monoid action as in \cite[\S 2.2.4]{10authors} (defined there for $\GL_n$, but for the coefficient module considered here the centre indeed acts trivially). If $x : \prod_{v \in R} k(v)^\times \to \cO^\times$ is a character which is trivial modulo $\varpi$, then we write $\cV_\lambda(x) = \cV_\lambda \otimes_\cO \cO(x)$ where $\cO(x)$ is the $\cO[U(0,1)]$-module on which $U(0, 1)$ through the projection $U(0, 1) \to \prod_{v \in R} I_v \to \prod_{v \in R}( k(v)^\times \times k(v)^\times)/k(v)^\times$ via the character $x \boxtimes x^{-1}$. 
We thus have a complex for any $c \geq b \geq 1$:
\[ A(U(1, c) / U(b, c), \cV_\lambda(x)) \in \mathbf{D}(\Lambda_{1, b}), \]
which is equipped with a $\Lambda_{1,c}$-algebra homomorphism
\[ \bbT^S \otimes_\cO \Lambda_{1,c} \to \End_{\mathbf{D}(\Lambda_{1, b})}( A(U(1, c) / U(b, c), \cV_\lambda(x)) ). \]

A standard calculation in Hida theory (c.f. \cite[\S 6.3]{Kha17} or \cite[\S 5.2]{10authors}) yields the following lemma.
\begin{lemma}
	For any $c' \geq c \geq 1$, pullback induces a $\bbT^S \otimes_\cO \Lambda_1[\{\mathsf{U}_v\}_{v \in S_p}]$-equivariant morphism
	\[ A(U(1,c) / U(c, c), \cV_\lambda(x)) \to  A(U(1, c') / U(c, c'), \cV_\lambda(x)) \]
	in $\mathbf{D}(\Lambda_{1, c})$. Consequently, there is an induced morphism of $\mathsf{U}_p$-ordinary parts 
	\[ A(U(1, c) / U(c, c), \cV_\lambda(x))^\text{ord} \to  A(U(1, c') / U(c, c'), \cV_\lambda(x))^\text{ord} \]
	which is in fact an isomorphism. 
\end{lemma}

\begin{proposition}\label{prop:independence_of_weight}
	For any $\lambda \in \bbZ_{\ge 0}^{\Hom(K, E)}$ and integers $c \ge r \ge 1$, there is a $\bbT^S \otimes \Lambda_1$-equivariant isomorphism in $\mathbf{D}(\Lambda_{1,c}/\varpi^r)$:
	\[ A(U(1,c)/U(c,c), \cV_\lambda(x)/\varpi^r)^{\ord} \cong A(U(1,c)/U(c,c), \cO(x)/\varpi^r)^{\ord} \otimes_\cO \cO(-\lambda). \]
\end{proposition}

\begin{proof}
	Let $e_\lambda : \cV_\lambda \to \cO(-\lambda)$ be the projection onto the lowest weight space, and let $\mathcal{K}_\lambda = \ker(e_\lambda)$; it is a free $\cO$-module. 
	Let 
	\[ \Delta = \PGL_2(\bbA_K^S) \times \prod_{v\in S_p} \Delta_v \subseteq \PGL_2(\bbA_K^{S - S_p}) \]
	be the sub-monoid with $\Delta_v = \cup_{n \ge 0} I_v(0,c)\diag(\varpi_v^n, 1) I_v(0,c)$ for each $v\in S_p$.
	Then
	\[ \xymatrix@1{ 0 \ar[r] & \mathcal{K}_\lambda/\varpi^r \ar[r] & \cV_\lambda/\varpi^r \ar[r] & \cO(-\lambda)/\varpi^r \ar[r] & 0 } \]
	is an exact sequence of $\cO/\varpi^r[\Delta]$-modules and we obtain an exact triangle
	\[ \xymatrix@1{ A(U(1,c)/U(c,c), \mathcal{K}_\lambda/\varpi^r) \ar[r] & A(U(1,c)/U(c,c), \cV_\lambda/\varpi^r) \ar[r] & A(U(1,c)/U(c,c), \cO(-\lambda)/\varpi^r) } \]
	in $\mathbf{D}(\Lambda_{1,c}/\varpi^r)$, equivariant for $\bbT^S[\{\mathsf{U}_v\}_{v\in S_p}]$. 
	Since $\diag(\varpi_v, 1)$ is nilpotent on $\mathcal{K}_\lambda/\varpi^r$ for each $v|p$, Proposition~\ref{prop_ord_summand} implies that on $\mathsf{U}_p$-ordinary parts
	\[  \xymatrix@1{ A(U(1,c)/U(c,c), \cV_\lambda/\varpi^r)^{\ord} \ar[r] & A(U(1,c)/U(c,c), \cO(-\lambda)/\varpi^r)^{\ord} \cong A(U(1,c)/U(c,c), \cO/\varpi^r)^{\ord} \otimes_\cO \cO(-\lambda) } \]
	is an isomorphism in $\mathbf{D}(\Lambda_{1,c}/\varpi^r)$.
\end{proof}

By Theorem~\ref{thm:aut_reps_and_cohom_PGLn}, Proposition~\ref{prop:independence_of_weight}, and our assumptions on $\overline{\rho}$ and the cuspidal auotmorphic representation $\pi$, there is a non-Eisenstein maximal ideal 
\[ \ffrm \subset \bbT^S( A(U(1, 1), \cV_{\mu}/\varpi )^{\ord} ) \cong \bbT^S( A(U(1,1), k)^{\ord}) \]
with $\overline{\rho}_\ffrm \cong \overline{\rho}$.  
For any $c \ge b \ge 1$, weight $\lambda \in \bbZ_{\ge 0}^{\Hom(K, E)}$, and tuple of characters $x$ as above, there is a canonical surjective homomorphism
\[ \bbT^S( A(U(1, c) / U(c, c), \cV_\lambda(x))^\text{ord} ) \to \bbT^S( A(U(1, 1), k )^\text{ord} ) \]
which induces a bijection on maximal ideals. 
We write abusively $\ffrm$ for the corresponding maximal ideal of $\bbT^S( A(U(1, c) / U(c, c), \cV_\lambda(x))^\text{ord} )$. This in turns allows us to define localizations
\[ A(U(1, c) / U(c, c), \cV_\lambda(x))^\text{ord}_\ffrm,\,\, C(U(1, c) / U(c, c), \cV_\lambda(x)^\vee)^\text{ord}_\ffrm \]
with the property that for $c' \geq c \geq 1$, there is a canonical isomorphism
\[ C(U(1, c') / U(c', c'), \cV_\lambda(x)^\vee)^\text{ord}_\ffrm \otimes^{\bbL}_{\Lambda_{1, c'}} \Lambda_{1, c} \cong C(U(1, c) / U(c, c), \cV_\lambda(x)^\vee)^\text{ord}_\ffrm \]
in $\mathbf{D}(\Lambda_{1, c})$. 
\begin{lemma}
	Suppose that $\ffrm$ is non-Eisenstein. Then  $C(U(1, c) / U(c, c), \cV_\lambda(x)^\vee)^\mathrm{ord}_\ffrm$ is a perfect complex of $\Lambda_{1, c}$-modules.
\end{lemma}
\begin{proof}
	If $S - (S_p \cup R)$ is non-empty, then $U(1,c)$ is neat by the choice of $U(1,c)_v$ for $v \in S - (S_p \cup R)$ using the same argument as \cite[Lemma~6.5.2]{10authors}.  The proposition then follows from \cite[Lemma~2.1.7]{10authors}.
	If $S = S_p \cup R$, then $\zeta_p \in K$ and the proposition follows from Theorem \ref{thm_boundedness_of_good_dihedral_cohomology}.
\end{proof}

As in \S\ref{sec_application_of_TW_method}, the arguments of \cite[\S 6]{Kha17} (and using Proposition~\ref{prop:independence_of_weight}) allow us to construct a minimal complex $F_{\ffrm, x}$ of $\Lambda_1$-modules such that:
\begin{itemize}
	\item There is a $\Lambda_1$-algebra homomorphism $\bbT^S \otimes_\cO \Lambda_1 \to \End_{\mathbf{D}(\Lambda_1)}(F_{\ffrm, x})$.
	\item For any $\lambda \in \bbZ_{\ge 0}^{\Hom(K,E)}$, there is a $\bbT^S \otimes_\cO \Lambda_1$-equivariant isomorphism in $\mathbf{D}(\cO)$:
	\[ F_{\ffrm, x} \otimes_{\Lambda_1} \Lambda_1/p_\lambda \cong C(U(1, 1) , \cV_\lambda(x)^\vee)^\text{ord}_\ffrm \otimes_{\cO} \Lambda_1/p_\lambda. \] 
	\item There is an isomorphism of $\bbT^S \otimes_\cO \Lambda_1$-modules:
	\[ H^\ast(F_{\ffrm, x}) \cong \plim_c H_{-\ast}^{U(1, c)}(\mathfrak{X}, \cO(x))^\text{ord}_\ffrm \]
\end{itemize}

Attached to the weight $\mu \in \bbZ_{\ge 0}^{\Hom(K, E)}$, we defined the prime ideal $p_\mu \subset \Lambda_1$. 
The prime ideal $p_{\mu}$ contains a unique minimal prime $\wp \subset \Lambda_1$, which pulls back to unique minimal primes $\wp_v \subset \Lambda_{1,v}$ for each $v\in S_p$. 
We set $\Lambda_v = \Lambda_{1,v}/\wp_v$ for each $v\in S_p$, and $\Lambda = \widehat{\otimes}_{v\in S_p} \Lambda_v \cong \Lambda_1/\wp$, the completed tensor product being taken over $\cO$. 
Finally, we define $\bbT_x$ to be the image of the homomorphism
\[ \bbT^S \otimes_\cO \Lambda_1 \to \End_{\mathbf{D}(\Lambda)}(F_{\ffrm, x} \otimes_{\Lambda_1} \Lambda). \]

We now define a global deformation problem 
\[ \cS_x = ( K, \overline{\rho}_\ffrm, \epsilon^{-1}, S, \{ \Lambda_v \}_{v \in S_p} \cup \{\cO\}_{v \in S - S_p}, \{ \cD_v^{\epsilon^{-1}, \ord} \}_{v \in S_p} \cup \{ \cD_v^{\epsilon^{-1}, x} \}_{v \in R} \cup \{\cD_v^{\epsilon^{-1}} \}_{v \in S - (S_p \cup R)} ), \]
where $\cD_v^{\epsilon^{-1},\ord}$ is as in \S\ref{sec_ordinary_fixed_determinant_deformations}, and $\cD_v^{\epsilon^{-1},x}$ is as in \S\ref{sec:level_raising_defs_fixed_det}.
We have a corresponding universal (fixed determinant) deformation ring $R_\cS$, which is a $\Lambda$-algebra.

\begin{proposition}\label{prop:R_to_T_p_odd}
	There is an integer $\delta = \delta(K)$, depending only on $K$, a nilpotent ideal $J_x \subset \bbT_x$ with $J_x^\delta = 0$, and a unique $\Lambda$-algebra surjection $R_{\cS_x} \to \bbT_x/J_x$ with the property that for all $v \not\in S$, $\tr \rho_{\cS_x}(\Frob_v) \mapsto \mathsf{T}_v$.
\end{proposition}

\begin{proof}
	This is similar to Corollary~\ref{cor_R_to_T}, using Corollary~\ref{cor:local-global-PGLn-ord} instead of Theorem~\ref{thm:local-global-GLn-ord}.
\end{proof}

Let $(Q, N, \{ \alpha_v, \beta_v \}_{v \in Q})$ be a Taylor--Wiles datum for $\cS_1$ (see \S\ref{sec:TW_p_odd}). 
We assume that each place of $Q$ has residue characteristic split in some imaginary quadratic subfield of $K$. 
For any choice of $x$ as above, $(Q, N, \{ \alpha_v, \beta_v \}_{v \in Q})$ is also a Taylor--Wiles datum for $\cS_x$ and we can define the auxiliary deformation problem
\[ \cS_{x, Q} = (K,\overline{\rho}_\ffrm, \epsilon^{-1}, S \cup Q, \{ \Lambda_v \}_{v \in S_p} \cup \{ \cO \}_{v \in (S - S_p) \cup Q }, \{ \cD_v^{\epsilon^{-1}, \ord} \}_{v \in S_p} \cup \{ \cD_v^{\epsilon^{-1}, x} \}_{v \in R} \cup \{\cD_v^{\epsilon^{-1}} \}_{v \in (S \cup Q) - (S_p \cup R)} ) , \]
together with a structure on $R_{\cS_{x, Q}}$ of a $\cO[\Delta_Q]$-algebra, where $\Delta_Q = \prod_{v \in Q} k(v)^\times(p)$. 
For any $c \geq b \geq 1$, we define subgroups $U_{Q, 1}(b, c) \subset U_{Q, 0}(b, c) \subset U(b, c)$ as follows:
\begin{itemize}
	\item $U_{Q, 1}(b, c) = \prod_v U_{Q, 1}(b, c)_v$ and $U_{Q, 0}(b, c) = \prod_v U_{Q, 0}(b, c)_v$.
	\item If $v \not\in Q$, then $U_{Q, 1}(b, c)_v = U_{Q, 0}(b, c)_v = U(b, c)_v$.
	\item If $v \in Q$, then $U_{Q, 0}(b, c)_v = I_v$ and $U_{Q, 1}(b, c)_v$ is the kernel of the homomorphism $I_v \to k(v)^\times(p)$ given by composing $\begin{pmatrix} a & b \\ c & d \end{pmatrix} \mapsto ad^{-1} \bmod {(\varpi_v)}$ with the projection $k(v)^\times \to k(v)^\times(p)$.
\end{itemize}
Then $U_{Q, 1}(b, c) \subset U_{Q, 0}(b, c)$ is a normal subgroup and we can identify the quotient $U_{Q, 0}(b, c) / U_{Q, 1}(b, c)$ with $\Delta_Q$.

In what follows, we write $\bbT^{S \cup Q}_Q = \bbT^{S \cup Q} \otimes_\cO \otimes_{v \in Q} T_v$, where $T_v$ is the ring associated to places $v \in R$ just before the statement of Theorem \ref{thm:local-global-GLn-ord}.
\begin{lemma}
	\begin{enumerate}
		\item 	There exists a (unique) maximal ideal $\ffrm_{Q, 0} \subset \bbT^{S \cup Q}(A(U_{Q, 0}(1, 1), k)^\text{ord})$ such that $\overline{\rho}_{\ffrm_{Q, 0}} \cong \overline{\rho}_\ffrm$. The trace map
		\[ H^\ast(A(U_{Q, 0}(1, 1), k)^\text{ord})_{\ffrm_{Q, 0}} \to H^\ast(A(U(1, 1), k)^\text{ord})_{\ffrm} \]
		is surjective.
		\item Let $\ffrm_{Q, 0, \alpha} \subset \bbT^{S \cup Q}_Q(A(U_{Q, 0}(1, 1), k)^\text{ord})$ denote the ideal generated by $\ffrm_{Q, 0}$ and the elements $\mathsf{U}_v - \alpha_v$, $u - 1$ \($v \in Q$, $u \in (\cO_{K_v}^\times)^2 / (1 + \varpi_v \cO_{K_v})^2 \subset T_v^\times$\).
Then $\ffrm_{Q, 0, \alpha}$ is a maximal ideal with residue field $k$ and the composite map
		\[ H^\ast(A(U_{Q, 0}(1, 1), k)^\text{ord})_{\ffrm_{Q, 0, \alpha}} \subset H^\ast(A(U_{Q, 0}(1, 1), k)^\text{ord})_{\ffrm_{Q, 0}} \to H^\ast(A(U(1, 1), k)^\text{ord})_{\ffrm} \]
		(inclusion followed by trace) is an isomorphism. 
	\end{enumerate}
\end{lemma}
\begin{proof}
	This is proved in the same way as Lemma~\ref{lemma:max_ideal_with_TW_level}.
\end{proof}
There are surjective algebra homomorphisms
\[ \bbT^{S \cup Q}( A( U_{Q, 0}(1, 1) / U_{Q, 1}(1, 1), k )^\text{ord} ) \to \bbT^{S \cup Q}( A(U_{Q, 0}(1, 1), k )^\text{ord} ), \]
and
\[ \bbT^{S \cup Q}_Q( A( U_{Q, 0}(1, 1) / U_{Q, 1}(1, 1), k )^\text{ord} ) \to \bbT^{S \cup Q}_Q( A(U_{Q, 0}(1, 1), k )^\text{ord} ), \]
and we write 
\[ \ffrm_{Q, 1} \subset \bbT^{S \cup Q}( A( U_{Q, 0}(1, 1)/U_{Q, 1}(1, 1), k )^\text{ord} ) \]
for the pullback of $\ffrm_{Q, 0}$, and 
\[ \ffrm_{Q, 1, \alpha} \subset \bbT^{S \cup Q}_Q( A( U_{Q, 0}(1, 1)/U_{Q, 1}(1, 1), k )^\text{ord} ) \]
for the pullback of $\ffrm_{Q, 0, \alpha}$. We also write $\ffrm_{Q, 1, \alpha}$ abusively for the pullback of $\ffrm_{Q, 1, \alpha}$ to any of the algebras $\bbT^{S \cup Q}_Q( A(U_{Q, 0}(1, c) / U_{Q, 1}(c, c), \cV_\lambda(x))^{\ord})$. For any $c \geq 1$ and $\lambda \in \bbZ_{\ge 0}^{\Hom(K, E)}$, there are $\bbT^{S \cup Q}$-equivariant isomorphisms 
\[ C(U_{Q, 0}(1, c) / U_{Q, 1}(c, c), \cV_\lambda(x))^\text{ord}_{\ffrm_{Q, 1, \alpha}} \otimes^\bbL_{\Lambda_c[\Delta_Q]} \Lambda_c \cong C(U_{Q, 0}(1, c) / U_{Q, 0}(c, c), \cV_\lambda(x))^\text{ord}_{\ffrm_{Q, 0}} \cong C(U(1, c) / U(c, c), \cV_\lambda(x))^\text{ord}_{\ffrm}. \]
As before, a limiting process gives rise to a minimal complex $F_{\ffrm, x, Q}$ of $\Lambda_1[\Delta_Q]$-modules with the following properties:
\begin{itemize}
	\item There is a $\Lambda_1[\Delta_Q]$-algebra homomorphism $\bbT^{S \cup Q} \otimes_\cO \Lambda_1[\Delta_Q] \to \End_{\mathbf{D}(\Lambda_1[\Delta_Q])}(F_{\ffrm, x, Q})$.
\item There is an isomorphism $H^\ast(F_{\ffrm, x, Q}) \cong \plim_c H^\ast(C(U_{Q, 1}(c, c), \cO(x))^\text{ord}_{\ffrm_{Q, 1, \alpha}})$ of $\bbT^{S \cup Q} \otimes_\cO \Lambda_1[\Delta_Q]$-modules.
	\item There is an isomorphism $F_{\ffrm, x, Q} \otimes_{\Lambda_1[\Delta_Q]} \Lambda_1 \cong F_{\ffrm, x}$ of complexes of $\Lambda_1$-modules which becomes  $\bbT^{S \cup Q} \otimes_\cO \Lambda_1$-equivariant when we consider the corresponding morphism in $\mathbf{D}(\Lambda_1)$.
\end{itemize}
We define $\bbT_{x, Q}$ to be the image of the homomorphism
\[ \bbT^{S \cup Q} \otimes_\cO \Lambda_1[\Delta_Q] \to \End_{\mathbf{D}(\Lambda[\Delta_Q])}(F_{\ffrm, x, Q}). \] 
Then $\bbT_{x, Q}$ is a finite local $\Lambda[\Delta_Q]$-algebra and we write $\ffrm_{x, Q}$ for its unique maximal ideal. 
\begin{proposition}\label{prop:R_to_T_TW_augmented_p_odd}
	There is an integer $\delta = \delta(K)$, the same one as in Proposition~\ref{prop:R_to_T_p_odd}, a nilpotent ideal $J_{x, Q} \subset \bbT_{x, Q}$ with $J_{x, Q}^\delta = 0$, and a unique $\Lambda[\Delta_Q]$-algebra surjection $R_{\cS_{x, Q}} \to \bbT_{x, Q} / J_{x, Q}$ with the property that for all $v \not\in S \cup Q$, $\tr \rho_{\cS_{x, Q}}(\Frob_v) \mapsto \mathsf{T}_v$.
\end{proposition}

\begin{proof}
	This is proved in the same was as Corollary~\ref{cor_R_to_T_TW_augmented}.
\end{proof}

We now have enough in place to prove Theorem~\ref{thm:automorphy_lifting_with_conditions_p_odd}.

\begin{proof}[Proof of Theorem~\ref{thm:automorphy_lifting_with_conditions_p_odd}]
	The proof follows that of \cite[Theorem~6.6.2]{10authors}. 
	Enlarging $E$ if necessary, we can assume that $\rho$ takes values in $\GL_2(\cO)$ and that $E$ contains a primitive $p$th root of $1$.
	Then $\rho$ corresponds to a homomorphism $f : R_{\cS_1} \to \cO$. 
	Assumption \ref{condition:weights} of Theorem~\ref{thm:automorphy_lifting_with_conditions_p_odd} implies that $p_\lambda : \Lambda_1 \to \cO$ factors through $\Lambda$. 
	We will show that $\ker f$ is in the support of $H^\ast(F_{\ffrm, 1} \otimes_{\Lambda_1} \Lambda)$. 
	This will then imply that $\ker f$ is in the support of $H^\ast(C(U(1, 1) , \cV_\lambda^\vee)^\text{ord}_\ffrm \otimes_{\cO} \Lambda_1/p_\lambda)[1/p]$, hence in the support of 
	\[ \Hom_E(H^\ast_{U(1,1)}(\mathfrak{X}_{\PGL_2}, \cV_\lambda)_\ffrm[1/p], E(\lambda)). \]
	The conclusion of the theorem will then follow from Theorem~\ref{thm:aut_reps_and_cohom_PGLn}.

	Since $E$ contains a primitive $p$th root of $1$, for each $v \in R$, we can choose a character $\chi_v : \cO_{K_v}^\times \to \cO^\times$ that is trivial modulo $\varpi$ with $\chi_v^2 \ne 1$. 
	We write $\chi$ for the tuple $(\chi_v)_{v \in R}$ as well as for the induced character $\prod_{v \in R} I_v \to \cO^\times$. 
	For each integer $N \ge 1$, let $(Q_N, N, \{ \alpha_{N,v}, \beta_{N,v} \}_{v \in Q_N})$ be a choice of Taylor--Wiles datum of level $N$ as in Proposition~\ref{prop_presentation_with_TW_data}. 
	For convenience, we write $\cS_{1,{Q_0}} = \cS_1$ and $F_{\ffrm, 1, Q_0} = F_{\ffrm, 1}$, and similarly with $\chi$ in place of $1$. 
Then our choice of $\chi$, Proposition~\ref{prop_presentation_with_TW_data}, the properties of $F_{\ffrm, \chi, Q_N}$ above, and Propositions~\ref{prop:R_to_T_p_odd} and~\ref{prop:R_to_T_TW_augmented_p_odd}, imply that the following data, for integers $N \ge 0$, satisfy the setup of \cite[\S6.4.1]{10authors}:
	\begin{itemize}
		\item $q = h^1(K_S/K, \ad^0\overline{\rho}_\ffrm(1))$, $g = q - 3[K^+:\bbQ]  - 1 + \lvert S \rvert$, and $\delta$ is as in Proposition~\ref{prop:R_to_T_p_odd}.
		\item $\mathcal{T}$ is a power series over $\Lambda$ in $4 \lvert S \rvert - 1$ many variables, $\Delta_\infty = \bbZ_p^q$, and $S_\infty = \mathcal{T} \llbracket \Delta_\infty \rrbracket$. 
		We view $S_\infty$ as an augmented $\Lambda$-algebra and let $\mathfrak{a}_\infty$ be the augmentation ideal. 
		\item $R^\loc = A_{\cS_1}^S$ and $R^{'\loc} = A_{\cS_\chi}^S$. 
		We let $R_\infty$ and $R_\infty'$ be power series rings over $R^\loc = A_{\cS_1}^S$ and $R^{'\loc} = A_{\cS_\chi}^S$, respectively, in $g$ variables. 
\item $\Delta_N = \Delta_{Q_N}$.
		\item $C_N = F_{\ffrm, 1, Q_N} \otimes_{\Lambda_1[\Delta_N]} \Lambda[\Delta_N]$ and $C_N' = F_{\ffrm, \chi, Q_N} \otimes_{\Lambda_1[\Delta_N]} \Lambda[\Delta_N]$. We take $T_N = \bbT_{1, Q_N}$ and $T_N' = \bbT_{\chi, Q_N}$.
		\item $R_N = R_{\cS_{1, Q_N}}$ and $R_N' = R_{\cS_{\chi, Q_N}}$.
		\item $I_N = J_{1, Q_N}$ and $I_N' = J_{\chi, Q_N}$.
	\end{itemize}
	Then the results of \cite[\S6.4]{10authors} (in particular, Propositions~6.4.16 and~6.4.17 of \textit{loc. cit.}) yield:
	\begin{itemize}
		\item Perfect complexes $C_\infty$ and $C_\infty'$ in $\mathbf{D}(S_\infty)$ with an isomorphism 
		\begin{equation}\label{eqn:mod_pi_iso}
		C_\infty \otimes_{S_\infty}^\mathbb{L} S_\infty/\varpi \cong C_\infty' \otimes_{S_\infty}^\mathbb{L} S_\infty/\varpi
		\end{equation} 
		in $\mathbf{D}(S_\infty/\varpi)$, and isomorphisms $C_\infty \otimes_{S_\infty}^\mathbb{L} S_\infty/\fra_\infty \cong C_0$ and $C_\infty' \otimes_{S_\infty}^\mathbb{L} S_\infty/\fra_\infty \cong C_0'$ in $\mathbf{D}(\Lambda)$.
		\item $S_\infty$-subalgebras $T_\infty \subset \End_{\mathbf{D}(S_\infty)}(C_\infty)$ and $T_\infty' \subset \End_{\mathbf{D}(S_\infty)}(C_\infty')$ whose images in  $\End_{\mathbf{D}(S_\infty/\varpi)}(C_\infty \otimes_{S_\infty}^\mathbb{L} S_\infty/\varpi)$ and $\End_{\mathbf{D}(S_\infty/\varpi)}(C_\infty' \otimes_{S_\infty}^\mathbb{L} S_\infty/\varpi)$ are identified via \eqref{eqn:mod_pi_iso}.
		\item Nilpotent ideals $I_\infty \subset T_\infty$ and $I_\infty' \subset T_\infty'$, and $S_\infty$-algebra surjections $R_\infty \to T_\infty/I_\infty$ and $R_\infty' \to T_\infty'/I_\infty'$. Moreover, the actions of $R_\infty/\varpi \cong R_\infty'/\varpi$ (induced from $T_\infty$ and $T_\infty'$, respectively) on $H^\ast(C_\infty \otimes_{S_\infty}^\mathbb{L} S_\infty/\varpi)/I_\infty$ and $H^\ast(C_\infty' \otimes_{S_\infty}^\mathbb{L} S_\infty/\varpi)/I_\infty'$ are identified via \eqref{eqn:mod_pi_iso}. 
		\item $S_\infty$-algebra isomorphisms $R_\infty/\fra_\infty \cong R_0$ and $R_\infty'/\fra_\infty \cong R_0'$, and $S_\infty$-algebra maps $T_\infty \to T_0$ and $T_\infty' \to T_0$ such that the composite $R_\infty \to T_\infty/I_\infty \to T_0/(I_0, I_\infty)$ agrees with the composite $R_\infty \to R_0 \to T_0/(I_0, I_\infty)$ and similarly for the prime counterparts. 
	\end{itemize}
	By Proposition~\ref{prop_det_ord_fixed_det}, choice of $\chi$, Lemma~\ref{lem_ihara_avoidance_ring_fixed_det}, and Lemma~\ref{lemma:smooth_defs_fixed_det}, the rings $R_\infty$ and $R_\infty'$ satisfy:
	\begin{itemize}
		\item $\dim R_\infty = \dim R_\infty' = \dim S_\infty - [K^+ : \bbQ]$ and $\dim R_\infty/\varpi = \dim R_\infty'/\varpi = \dim S_\infty - [K^+ : \bbQ] - 1$.
		\item Each generic point of $\Spec R_\infty/\varpi$ of maximal dimension is the specialization of a unique generic point of dimension $\dim R_\infty$, 
		and $\Spec R_\infty'$ has a unique generic point of dimension $\dim R_\infty$. 
		Further, any generic points of $\Spec R_\infty$, $\Spec R_\infty'$, $\Spec R_\infty/\varpi$ which are not of maximal dimension have dimension $< \dim S_\infty - [K^+ : \bbQ] - 1$.  
		\item $\ker f \subset R_0 \cong R_\infty/\fra_\infty \subset \Spec R_\infty$ lies in an irreducible component of maximal dimension (since the weight $\lambda$ is regular).
	\end{itemize}
	Finally, using Theorem~\ref{thm:aut_reps_and_cohom_PGLn}, we also have:
	\begin{itemize}
		\item The dimension 1 characteristic 0 prime $\frp = (p_{\mu}, \fra_\infty) \subset S_\infty$ satisfies $H^\ast(C_\infty \otimes_{S_\infty}^\mathbb{L} S_\infty/\frp)[1/p] \ne 0$, and these groups are nonzero only for degrees in an interval with length $[K^+ : \bbQ]$. 
	\end{itemize}
	We have thus satisfied all the set-up and assumptions of \cite[\S6.3.5 and Assmption~6.3.6]{10authors}, and we can apply \cite[Corollary~6.3.9]{10authors} to deduce that $\ker f$ is in the support of
	\[ H^\ast(C_\infty \otimes_{S_\infty}^\mathbb{L} S_\infty/(p_\lambda,\fra_\infty))[1/p] \cong H^\ast(C(U(1, 1) , \cV_\lambda^\vee)^\text{ord}_\ffrm \otimes_{\cO} \Lambda_1/p_\lambda)[1/p]. \qedhere\]
\end{proof}

\subsection{The main automorphy lifting theorem for $p>2$}\label{sec_odd_aut_lift_statement}
We can now complete the proof of Theorem~\ref{thm_automorphy_at_odd_primes}. For the reader's convenience, we first recall the statement. 
\begin{theorem}[Theorem~\ref{thm_automorphy_at_odd_primes}]\label{thm_automorphy_at_odd_primes_redux}
	Let $p$ be an odd prime, let $K$ be a CM number field, and let $\rho : G_K \to \GL_2(\overline{\bbQ}_p)$ be a continuous representation satisfying the following conditions:
	\begin{enumerate}
		\item $\overline{\rho}$ is decomposed generic and $\overline{\rho}|_{G_{K(\zeta_p)}}$ is absolutely irreducible.
		\item $\det \rho = \epsilon^{-1}$.
		\item For all but finitely many places $v$ of $K$, $\rho|_{G_{K_v}}$ is unramified. 
		\item There exists $\lambda \in (\bbZ_{+, 0}^2)^{\Hom(K, \overline{\bbQ}_p)}$ such that $\rho$ is ordinary of weight $\lambda$.
		\item There exists an isomorphism $\iota : \overline{\bbQ}_p \cong \bbC$ and a cuspidal, $\iota$-ordinary, cohomological automorphic representation $\pi$ of $\PGL_2(\bbA_{K})$ such that $\overline{r_\iota(\pi)} \cong \overline{\rho}$.
\item If $p = 5$ and the projective image of $\overline{\rho}(G_{K(\zeta_5)})$ is conjugate to $\PSL_2(\bbF_5)$, we assume further that the extension of $K$ cut out by the projective image of $\overline{\rho}$ does not contain $\zeta_5$. 
\end{enumerate}
	Then $\rho$ is automorphic: there is a cuspidal automorphic representation $\Pi$ of $\PGL_2(\bbA_K)$, cohomological of weight $\iota\lambda$ and $\iota$-ordinary, such that $\rho \cong r_\iota(\Pi)$.
\end{theorem}

\begin{proof}
	It suffices to prove that $\rho \cong r_\iota(\Pi)$ for $\Pi$ a cuspidal automorphic representation $\GL_2(\bbA_K)$,
cohomological of weight $\iota\lambda$ and $\iota$-ordinary, since $r_\iota(\Pi) = \epsilon^{-1}$ implies that $\Pi$ has trivial central character.

We follow the proof of \cite[Theorem~6.1.2]{10authors}, using soluble base change to reduce to Theorem~\ref{thm:automorphy_lifting_with_conditions_p_odd}. 
	Let $L/K(\zeta_p)$ be the extension cut out by $\overline{\rho}|_{G_{K(\zeta_p)}}$. 
	If $F/K$ is any finite soluble extension, we let $\pi_F$ denote the base change of $\pi$ to $F$.
	Choose a finite set $V$ of finite places of $K$ satisfying:
	\begin{itemize}
		\item For any proper extension $L'/K$ contained in $L$, there is some $v \in V$ not splitting in $L'$.
		\item There ia a rational prime $q \ne p$ such that $\overline{\rho}$ is decomposed generic for $q$, and $V$ contains all $q$-adic places of $K$.
		\item For each $v \in V$, $v \nmid 2p$ and both $\rho$ and $\pi$ are unramified at $v$.
	\end{itemize}
	For any finite Galois extension $F/K$ in which every $v \in V$ splits, $\overline{\rho}|_{G_F}$ decomposed generic and $\overline{\rho}(G_{F(\zeta_p)}) = \overline{\rho}(G_{K(\zeta_p)})$.
	We choose a solvable Galois CM extension $F_0/K$ satisfying:
	\begin{itemize}
		\item Every $v \in V$ splits in $F$.
		\item For every finite place $w$ of $F_0$, $\pi_{F_0}^{I_w} \ne 0$.
		\item For every finite prime-to-$p$ place $w$ of $F_0$, either $\rho|_{G_{F_{0, w}}}$ is unramified, or $\rho|_{G_{F_{0, w}}}$ is unipotently ramified, $q_w \equiv 1 \bmod p$, and $\overline{\rho}|_{G_{F_{0, w}}}$ is trivial.
		\item For each $w|p$ in $F_0$, $\overline{\rho}|_{G_{F_{0, w}}}$ is trivial and $[F_{0, w} : \bbQ_p] > 4$.
		\item For each $v|p$ in $K$ and $w|v$ in $F_0$, 
		\[ \rho|_{G_{K_v}} \sim \begin{pmatrix} \psi_v & \ast \\ 0 & \epsilon^{-1}\psi_v^{-1} \end{pmatrix}, \]
		where the character $\psi_v|_{I_{K_v}} : I_{K_v} \to \overline{\bbQ}_p^\times$ agrees with the character
		\[ \sigma \mapsto \prod_{\tau \in \Hom_{\bbQ_p}(K_v, \overline{\bbQ}_p)} \tau(\Art_{K_v}^{-1}(\sigma))^{\lambda_{\tau, 1}} \]
		on the whole of the inertia subgroup $I_{F_{0, w}} \subseteq I_{K_v}$. 
		\item Let $\mu \in \Hom(F_0, \overline{\bbQ}_p)$ be such that $\iota\mu$ is the weight of $\pi_{F_0}$. Then for each $v|p$ in $K$, $w|v$ in $F_0$, and $p$-power root of unity $u \in F_{0, w}$, we have 
		\[ \psi_v(\Art_{F_{0, w}}(u)) = \prod_{\tau \in \Hom_{\bbQ_p}(F_{0, w}, \overline{\bbQ}_p)} \tau(u)^{\mu_{\tau, 1}}.  \]
	\end{itemize}
	Choose an imaginary quadratic field $F_a$ in which $2$, $p$, and any rational prime lying below a place of $V$ all split. 
	Then set $F_b = \bbQ(\sqrt{-p_b})$ with $p_b$ a prime satisfying $p_b \equiv 1 \bmod 4$ and $p_b \equiv -1 \bmod l$ for any $l$ lying below a place where either $\rho|_{G_{F_0}}$ or $\pi_{F_0}$ are ramified or that ramifies in $F_0 F_a$. 
	Now set $F_c = \bbQ(\sqrt{-p_c})$ where $p_c \equiv 1 \bmod {4p_b}$ is a prime. 
	By quadratic reciprocity, $p_b$ splits in $F_c$ and $p_c$ splits in $F_b$.
	We can further assume these choices are made so that $F = F_0 \cdot F_a \cdot F_b \cdot F_c$ is disjoint over $\bbQ$ from the extension cut out by $\overline{\rho}|_{G_{F_0}}$. 
	Then $F/K$ is a solvable CM extension satisfying:
	\begin{itemize}
		\item Let $R$ be the set of prime-to-$p$ finite places of $F$ at which either $\rho|_{G_F}$ or $\pi_F$ is ramified. 
		Let $S_p$ be the set of $p$-adic places of $F$ and let $S' = S_p \cup R$. 
		Then if $l$ is a rational prime lying below a place of $S'$ or which is ramified in $F$, there is an imaginary quadratic subfield of $F$ in which $l$ splits.
	\end{itemize} 
	If $\zeta_p \notin F$, then by \cite[Lemma~4.11]{ddt} and Chebotarev density there are infinitely many primes $v_0 \in F$ with the following properies: $v_0$ is unramified in $F$ and has degree $1$ over $\bbQ$, $q_{v_0} \not\equiv 1 \bmod p$, $\overline{\rho}|_{G_F}$ is unramified at $v_0$, and the ratio of the eigenvalues of $\overline{\rho}(\Frob_{v_0})$ does not equal $q_{v_0}^{\pm 1}$. 
Then $H^2(F_{v_0}, \ad^0\overline{\rho}) \cong H^0(F_{v_0}, \ad^0\overline{\rho}(1))^\vee = 0$, and $v_0$ splits in any quadratic subfield of $F$. We choose two such primes $v_0, v_0'$ with distinct residue characteristics and set $S = S' \cup \{v_0, v_0'\}$. 
	If $\zeta_p \in F$, then we set $S = S'$. 
	
	We have now satisfied the setup and hypotheses of Theorem~\ref{thm:automorphy_lifting_with_conditions_p_odd}, so we deduce that there is a cuspidal automorphic representation $\Pi_F$ of $\GL_2(\bbA_F)$ that is $\iota$-ordinary and cohomological of weight $\iota\lambda_F$, where $\lambda_F \in (\bbZ_{+, 0}^2)^{\Hom_{\bbQ_p}(F, \overline{\bbQ}_p)}$ is defined by $\lambda_{F,\tau'} = \lambda_\tau$ if $\tau'|_K = \tau$. 
	By \cite[Proposition~6.5.13]{10authors} and \cite[Lemma~5.1.6]{Ger18}, we have a descent of $\Pi_F$ to a cuspidal automorphic representation $\Pi$ of $\GL_2(\bbA_K)$, cohomological of weight $\iota\lambda$ and $\iota$-ordinary, such that $\rho \cong r_\iota(\Pi)$.
\end{proof}

\bibliographystyle{alpha}
\bibliography{ModLiftBib}

\end{document}